\documentclass[10pt,oneside,notitlepage]{article} 
\usepackage[french,english]{babel}
\usepackage{amsmath}
\usepackage{amsthm}
\usepackage{amsfonts}
\usepackage{amssymb,amscd,epsf,verbatim,mathtools,framed}
\usepackage{mathrsfs}
\usepackage{graphicx}
\usepackage{latexsym}
\usepackage{lscape}
\usepackage[colorlinks=true]{hyperref}
\hypersetup{colorlinks, citecolor=blue, filecolor=black, linkcolor=purple, urlcolor=violet}
\usepackage{epstopdf}
\usepackage{tikz}
\usetikzlibrary{calc}
\usetikzlibrary{matrix,arrows,decorations.pathmorphing}
\usepackage{tikz-cd}
\usepackage{color}
\usepackage{multirow}  
\usepackage{scalefnt}
\usepackage{fancyhdr}
\usepackage[margin=1.25in]{geometry}

\usepackage{relsize} 
\usepackage[bbgreekl]{mathbbol} 
\DeclareSymbolFontAlphabet{\mathbb}{AMSb} 
\DeclareSymbolFontAlphabet{\mathbbl}{bbold} 
\newcommand{\Prism}{{\mathlarger{\mathbbl{\Delta}}}}

\newcommand{\Z}{\mathbb{Z}}
\newcommand{\F}{\mathbb{F}}
\newcommand{\N}{\mathbb{N}}

\newcommand{\Q}{\mathbb{Q}}
\newcommand{\A}{\mathbb{A}}
\renewcommand{\L}{\mathbb{L}}

\newcommand{\G}{\mathbb{G}}

\newcommand{\mB}{\mathcal{B}}
\newcommand{\mC}{\mathcal{C}}
\newcommand{\mD}{\mathcal{D}}

\newcommand{\mF}{\mathcal{F}}

\newcommand{\mH}{\mathcal{H}}

\newcommand{\mM}{\mathcal{M}}

\newcommand{\mL}{\mathcal{L}}
\newcommand{\mO}{\mathcal{O}}
\newcommand{\mP}{\mathcal{P}}
\newcommand{\mS}{\mathcal{S}}

\newcommand{\fm}{\mathfrak{m}} 
\newcommand{\fU}{\mathfrak{U}}
\newcommand{\fX}{\mathfrak{X}}

\newcommand{\tu}{\textup}

\newcommand{\cl}{\overline}
\newcommand{\ul}{\underline}

\renewcommand{\iff}{\Leftrightarrow}
\newcommand{\ra}{\rightarrow}

\newcommand{\sq}{\widetilde}

\DeclareMathOperator{\Hom}{Hom}

\DeclareMathOperator{\Ainf}{\textup{A}_{\textup{inf}}}
\DeclareMathOperator{\Binf}{\textup{B}_{\textup{inf}}}
\DeclareMathOperator{\Ainfx}{\mathbb{A}_{\textup{inf}, X}}
\DeclareMathOperator{\Acris}{\text{A}_{\textup{cris}}} 
\DeclareMathOperator{\Bcrisp}{\text{B}_{\textup{cris}}^+}
\DeclareMathOperator{\Bcris}{\text{B}_{\textup{cris}}}

\DeclareMathOperator{\Bdrp}{\text{B}_{\textup{dR}}^+}

\DeclareMathOperator{\Bst}{\text{B}_{\textup{st}}}
\DeclareMathOperator{\logcrys}{logcrys}
\DeclareMathOperator{\crys}{crys}

\DeclareMathOperator{\psh}{psh}
\DeclareMathOperator{\qSyn}{qSyn}
\DeclareMathOperator{\QSyn}{QSyn}
\DeclareMathOperator{\qrsP}{qrsPerfd}
\DeclareMathOperator{\QRSP}{QRSPerfd} 

\DeclareMathOperator{\opp}{\textup{op}}
\DeclareMathOperator{\midoplus}{\mathlarger{\mathlarger{\mathlarger{\oplus}}}}
\DeclareMathOperator{\midotimes}{\mathlarger{\mathlarger{\mathlarger{\otimes}}}}
\newcommand{\injlra}{\lhook\joinrel\longrightarrow}
\newcommand{\surjlra}{\relbar\joinrel\twoheadrightarrow}
\DeclareMathOperator{\lra}{\: \longrightarrow \:}
\DeclareMathOperator{\OXplus}{\widehat \mO_X^+}
\DeclareMathOperator{\aeq}{\: \approx^a \:}
\DeclareMathOperator{\gp}{\textup{gp}}

\newcommand{\gr}[1]{\langle {#1} \rangle} 
\DeclareMathOperator{\spec}{Spec}
\DeclareMathOperator{\spf}{Spf}
\DeclareMathOperator{\spa}{Spa}
 
\DeclareMathOperator{\init}{init}
\DeclareMathOperator{\etale}{\textup{\'etale}}
\DeclareMathOperator{\ett}{\textup{\'et}}
\DeclareMathOperator{\ket}{\textup{k\'et}}
\DeclareMathOperator{\fket}{\textup{fk\'et}}

\DeclareMathOperator{\proket}{\textup{prok\'et}}
\DeclareMathOperator{\isom}{\;\xrightarrow{\: {}_{\sim} \:} \;}

\newcommand{\bi}{\begin{itemize}}
\newcommand{\ei}{\end{itemize}}
\newcommand{\bt}{\begin{theorem}}
\newcommand{\et}{\end{theorem}}
\newcommand{\bbt}{\begin{theorem*}}
\newcommand{\eet}{\end{theorem*}}
\newcommand{\bp}{\begin{proposition}}
\newcommand{\ep}{\end{proposition}}
\newcommand{\bl}{\begin{lemma}}
\newcommand{\el}{\end{lemma}}
\newcommand{\bbl}{\begin{lemma*}}
\newcommand{\eel}{\end{lemma*}}
\newcommand{\bc}{\begin{corollary}}
\newcommand{\ec}{\end{corollary}}
\newcommand{\beg}{\begin{example}}
\newcommand{\eeg}{\end{example}}
\newcommand{\br}{\begin{remark}}
\newcommand{\er}{\end{remark}}
\newcommand{\bbr}{\begin{remark*}}
\newcommand{\eer}{\end{remark*}}
\newcommand{\bd}{\begin{definition}}
\newcommand{\ed}{\end{definition}}
\newcommand{\be}{\begin{enumerate}}
\newcommand{\ee}{\end{enumerate}}
\newcommand{\bex}{\begin{exercise}}
\newcommand{\eex}{\end{exercise}}
\newcommand{\bproof}{\begin{proof}}
\newcommand{\eproof}{\end{proof}}

\newcommand{\nocontentsline}[4]{}
\newcommand{\tocless}[3]{\bgroup\let\addcontentsline=\nocontentsline#1{#2}\egroup}

\theoremstyle{theorem}
\newtheorem{theorem}{Theorem}[section]
\newtheorem{mainthm}{Theorem}[section]

\theoremstyle{definition}
\newtheorem{example}[theorem]{Example}
\newtheorem{convention}[theorem]{Convention}

\newtheorem{maindef}[mainthm]{Definition}
\newtheorem{definition}[theorem]{Definition}
\newtheorem{proposition}[theorem]{Proposition}

\newtheorem{maincor}[mainthm]{Corollary}
\newtheorem{mainprop}[mainthm]{Proposition}
\newtheorem{construction}[theorem]{Construction}

\newtheorem{lemma}[theorem]{Lemma}
\newtheorem{corollary}[theorem]{Corollary}

\newtheorem{notation}[theorem]{Notation}
\newtheorem{remark}[theorem]{Remark}
\newtheorem{mainrmk}[mainthm]{Remark}
\newtheorem{variant}[theorem]{Variant}
\newtheorem{outline}[theorem]{Outline}

 \usepackage{etoolbox}
\patchcmd{\section}{\scshape}{\bfseries}{}{}
\makeatletter
\newcommand{\@secnumfont}{\bfseries}
\makeatother

\numberwithin{equation}{section}

\usepackage{hyperref}
\newcommand{\footremember}[2]{%
    \footnote{#2}
    \newcounter{#1}
    \setcounter{#1}{\value{footnote}}%
}
 
\title{Logarithmic $\Ainf$-cohomology}
\author{%
  Hansheng Diao \footremember{alley}{Yau Mathematical Sciences Center, Tsinghua University}%
  \and Zijian Yao \footremember{trailer}{Department of Mathematics, University of Chicago}%
\date{\empty}
  }

\begin{document}
\maketitle

\begin{abstract}
We extend the construction of $\Ainf$-cohomology by Bhatt--Morrow--Scholze to the context of log $p$-adic formal schemes over a log perfectoid base. In particular, using coordinates, we prove comparison theorems between log $\Ainf$-cohomology with other $p$-adic cohomology theories, including log de Rham, log (q-)crystalline, log prismatic, and Kummer \'etale cohomology, as well as the derived $\Ainf$-cohomology of certain infinite root stacks. 

Along the way, we define and give a combinatorial characterization of a new class of maps between saturated log schemes, called pseudo-saturated maps, which is of independent interest. They are related to (and slightly weaker than) the notion of quasi-saturated maps and maps of Cartier type studied by Tsuji.
\end{abstract}

\tableofcontents

\setlength{\parskip}{0.2em}

 
\section{Introduction} \label{section:intro}

In this paper and its sequel \cite{log2} we extend recent results on integral $p$-adic cohomology theory  developed in \cite{BMS1, BMS2, CK_semistable, MT} to the context of certain smooth log $p$-adic formal schemes over an integral perfectoid base. 

Let us fix the following notation throughout the introduction. Let $C$ be a perfectoid field containing all $m$-th power roots of unity for all $m\in \mathbb{Z}_{\geq 1}$ and write $\mO = \mO_C$ for its ring of integers. Upon fixing a compatible choice of primitive roots of unity, we obtain elements $\epsilon = (1, \zeta_p, \zeta_{p^2}, ...) \in \mO^\flat$ and $\mu = [\epsilon] - 1$ in Fontaine's period ring $ \Ainf = W(\mO^\flat)$, where $\mO^\flat$ denotes the tilt of $\mO$. The kernel of Fontaine's map $\theta: \Ainf \twoheadrightarrow \mO$ is generated by $\xi := \mu /\varphi^{-1} (\mu)$ where $\varphi = \varphi_{\Ainf}$ denotes the Witt vector Frobenius (see Section \ref{section:review_Ainf} for more details). 

In \cite{BMS1, BMS2}, Bhatt--Morrow--Scholze develop an integral $p$-adic cohomology theory for a smooth formal scheme $\fX$ over $\mO$, which takes values in $\Ainf$-modules and ``interpolates'' essentially all known $p$-adic cohomology theories (notably de Rham, crystalline, and \'etale cohomologies). This theory has been generalized in several directions. In \cite{CK_semistable}, \v{C}esnavi\v{c}ius--Koshikawa extended this theory to the context of formal schemes with semistable reduction. More recently, Bhatt--Scholze give a powerful site theoretic construction in \cite{BS} which allows more general base by defining the notion of prisms and prismatic site. In \cite{MT}, Morrow--Tsuji develop $\Ainf$-cohomology with coefficients. While this paper is being written, a logarithmic version of the prismatic theory has also been developed in a series of papers \cite{Koshikawa, logprism} by Koshikawa and the second author.  

The primary goal of this paper is to generalize the constructions of  \cite{BMS1,CK_semistable} to the context of log formal schemes that are ``sufficiently log smooth'' over $\mO$, which allows more general log formal schemes than those of semistable reduction. Compared to \cite{logprism}, our approach makes the comparison isomorphisms more explicit and likely more useful in applications where computations with explicit coordinates are involved. Moreover, we prove a logarithmic analog of a certain primitive Hodge--Tate comparison (Theorem \ref{mainthm:primitive_HT}), which should be of independent interest. As a result, some of the definitions and proofs in this paper are highly combinatorial in flavor. 
We then compare our construction with other approaches -- one approach given by the logarithmic prismatic site \cite{Koshikawa, logprism} and another given by considering a certain infinite root stack which we describe later in this introduction. Along the way, we extend the theory of log-adic spaces from \cite{DLLZ} to include those with non-finitely generated charts. In the sequal \cite{log2} of this article, we study logarithmic $\Ainf$-cohomology with more general coefficient systems, parallel to the work of \cite{MT}. 

\subsection*{Pseudo-saturated maps}
First we briefly explain the notion of being ``sufficiently log smooth''. Suppose that $\mO$ is equipped with a monoid map $N_\infty \ra \mO$ where $N_\infty$ is uniquely $n$-divisible for every $n \in \Z_{> 0}$, which in particular makes $(\mO, N_\infty)$ a perfectoid pre-log ring.\footnote{In the paper we will make an additional simplifying assumption, namely we require the pre-log ring $(\mO, N_\infty)$ to be \emph{split} in the sense of Definition \ref{definition:definition_pre_log_rings}. For simplicity, we ignore this technical point in the rest of the introduction.}  Write 
\[\mS = (\spf \mO, N_\infty)^a\] for the base log $p$-adic formal scheme associated to the pre-log formal scheme $(\spf \mO, N_\infty)$ (see Definition \ref{def:log}). Let $\fX = (\fX, \mM_{\fX})$ be a  log $p$-adic formal scheme over $\mS$. The classical study of log schemes and their Kummer \'etale topology by Kato and Nakayama in \cite{Kato, Nakayama1, Nakayama2} typically requires all the relevant log structures to be fs (fine and saturated), in which case Kato has defined a reasonable notion of log smooth maps (see \cite[Section 3.3]{Kato}). However, the divisible monoid $N_\infty$ (resp. the log structure $\mM_{\fX}$) we consider in this article is typically not finitely generated (resp. not coherent), so we are forced to consider only saturated log formal schemes. One possible approach to define log smooth maps for general saturated log formal schemes is to define a notion of formally log smoothness using a lifting criterion and require a suitable notion of local finite presentation. Another approach, which is the approach we take in this article, is to require the morphism in picture to be the saturated base change of a log smooth map between fs log formal schemes.\footnote{The second approach is also used in \cite{Koshikawa, logprism}, where it is additionally required that the map is an integral map between log formal schemes. Such maps are simply called \emph{smooth} in \emph{loc.cit.}} \footnote{Not surprisingly, these two approaches give the same notion of log smooth maps under some mild conditions (on the existence of local charts). Some aspect of this is discussed in the forthcoming paper \cite{Dori_Yao}.} 

Now suppose that  $\fX$ is the saturated base change 
\begin{equation} \label{eq:diagram_defining_fX}
\begin{tikzcd} 
\fX \arrow[d] \arrow[r] & \fX_0 \arrow[d, swap, "\pi_0"]  
\\ 
\mS \arrow[r, "\kappa"]  & \mS_0  
\end{tikzcd}
\end{equation}
of an fs log $p$-adic formal scheme $\fX_0$ that is log smooth over $\mS_0 = (\spf \mO, N)^a$, where $\kappa$ is induced by a map $N \ra N_\infty$ of monoids where $N$ is fine and saturated. Following the strategy of \cite{BMS1}, the starting point of the construction of log $\Ainf$-cohomology is to define a certain period sheaf $\Ainfx$ on the pro-Kummer \'etale site of the generic fiber $X$ of $\fX_{\eta}$, which we want to view as a log adic space in the sense of \cite{DLLZ}. However, this generic fiber (even that of the underlying $p$-adic formal scheme) may not exist as a log adic space (resp. an adic space), essentially due to the saturation procedure in the base change. Thus we need further restrictions on the map $\pi_0$ to be able to formulate such a construction. For example, we may require the map $\pi_0$ to be \emph{saturated} in the sense of \cite{Tsuji}, which is closely related to the mod $p$ reduction of $\pi_0$ being of Cartier type. This is not a surprising condition, since our construction of logarithmic $\Ainf$-cohomology aims to provide a certain ``$p$-adic Cartier isomorphism" which lifts the usual Cartier isomorphism from characteristic $p$  (see also \cite{Morrow}). One slightly unsatisfying feature of saturated morphisms is that it excludes the standard Kummer maps of the form $\Z_{\ge 0} \xrightarrow{n} \Z_{\ge 0}$, which are building blocks of Kummer \'etale morphisms. However, we will see that, \textit{a posteriori}, inclusion of such maps does not increase the generality of our setup. 

In order to give a precise formulation, we define a new notion of maps between saturated monoids (and between saturated log schemes), called \textit{pseudo-saturated maps}, which relaxes the condition of being saturated and should be of independent interest in the study of log geometry on its own. In particular, this notion captures maps built out of saturated maps and standard Kummer maps via successive compositions. 

\begin{maindef} \label{maindef:pseudo_sat}
Let $u:P\rightarrow Q$ be a map of saturated monoids. We say that $u$ is \emph{pseudo-saturated} if there exists an integer $M\geq 1$ such that $u$ satisfies the following condition: 
  \begin{itemize}
     \item[] \!\textit{For any integer $n\geq 1$ and any ${x}\in P$, ${y}\in Q$ such that $u({x})\,|\, n{y}$, there exists \\ ${x}'\in P$ such that $M{x}\,|\, n{x}'$ and $u({x}')\,|\, M{y}$.}
    \end{itemize}
\end{maindef}

We refer the reader to the beginning of Subsection \ref{subsection: conventions on monoids} for notations used in the definition above. Let us remark that, in the slightly technical definition above,  quasi-saturated maps studied by Tsuji in \cite{Tsuji} are precisely the ones where $M$ can be chosen to be $1$.  
One of the key properties we prove is the following 
\begin{mainprop}
Let $u: P \ra Q$ be a pseudo-saturated homomorphism between saturated monoids. Then for any map $P \ra P'$ where $P'$ is a saturated and divisible monoid (which means that it is $n$-divisible for every positive integer $n$), the natural map 
\[Q \sqcup_P P' \lra Q \sqcup_P^{\tu{sat}} P' : = (Q \sqcup_P P')^{\tu{sat}}\]
from the naive pushout to the saturated pushout is of finite type. 
\end{mainprop}

This finiteness result eventually allows us to show that, for a log $p$-adic formal scheme $\fX$ as in the pullback diagram (\ref{eq:diagram_defining_fX}), its generic fiber $X$ is indeed a log adic space under the assumption that $\pi_0: \fX_0 \ra \mS_0$ is pseudo-saturated. In the case of fs monoids, we give a combinatorial criterion of pseudo-saturated maps in terms of polyhedral cones.  To fix notations, let us consider an injective map 
\[
u: P \hookrightarrow Q
\]
of torsion-free fs  monoids (we view $P$ as a submonoid of $Q$ via the inclusion $u$). This induces an inclusion $P_{\Q_{\ge 0}} \hookrightarrow Q_{\Q_{\ge 0}}$, where $P_{\Q \ge 0}$ is the monoid  $ \varinjlim_{n} \frac{1}{n} P$ (see Definition \ref{definition:perfectoid_monoids}), which we may view as a rational polyhedral cone in the $\Q$-vector space $Q_{\Q_{\ge 0}}^{\gp} \cong \Q^{\oplus d}$ for some $d$, where $(\:)^{\gp}$ stands for the group envelop. Let $y \in Q_{\Q_{\ge 0}}$ be an element, we say that a decomposition 
\begin{equation} \label{eq:min_decomposition_of_y} y=y'+x'
\end{equation}
with $y'\in Q_{\mathbb{Q}_{\geq 0}}$ and $x'\in P_{\mathbb{Q}_{\geq 0}}$ is a \emph{minimal decomposition} of $y$ (\emph{relative to the pair $(P, Q)$}) if for any $x \in P_{\mathbb{Q}_{\geq 0}}-\{0\}$, we have  $y'-x\not\in Q_{\mathbb{Q}_{\geq 0}}$. 
In other words, if $y'$ is 
``minimal'' in the sense that it cannot be further decomposed as in (\ref{eq:min_decomposition_of_y}) in a nontrivial fashion. We are now ready to state the following criterion of pseudo-saturated maps. 


\begin{mainthm} \label{mainthm:ps_criterion}
Let $Q$ be a torsion-free fs monoid and let $P$ be a toric submonoid of $Q$ such that $Q\cap (-P)= \{0\}$. Then the inclusion $u:P \hookrightarrow Q$ is pseudo-saturated if and only if every  $y \in Q_{\mathbb{Q}_{\geq 0}}$ admits a unique minimal decomposition  relative to  $(P, Q)$. 
\end{mainthm}

In particular, in the setup of Theorem \ref{mainthm:ps_criterion}, whether $P \hookrightarrow Q$ is pseudo-saturated or not only depends on the relative position of the rational polyhedral cones $P_{\Q_{\ge 0}} \hookrightarrow Q_{\Q_{\ge 0}}$.  We refer the reader to Section \ref{subsection: classification} (in particular Example \ref{example:min_decom1}, Example \ref{example:min_decom2} and Remark \ref{remark:geometric_intuition}) for a more geometric interpretation of the equivalence above. Somewhat surprisingly, 
this ``geometric'' characterization of pseudo-saturated maps turns out to be helpful when we make local computations for logarithmic $\Ainf$-cohomology using coordinates in Section \ref{section:etale_comparison} and Section \ref{section:HT_primitive}. On the other hand, it also implies the following result. 

\begin{mainprop}
Let $Q$ be a torsion-free fs  monoid and let $P$ be a toric submonoid of $Q$ such that $Q\cap (-P)= \{0\}$.  Suppose that the inclusion $u:P \hookrightarrow Q$ is pseudo-saturated and that the cokernel of induced map $u^{\mathrm{gp}}: P^{\mathrm{gp}} \ra Q^{\mathrm{gp}}$ on group envelops is torsion-free. Then there exists a positive integer $m$ such that the canonical homomorphism 
\[\frac{1}{m}P \lra 
Q\sqcup^{\mathrm{sat}}_P\frac{1}{m}P\] is quasi-saturated in the sense of \cite{Tsuji} defined by Tsuji.
\end{mainprop}

In particular, this lemma implies that, asking the log smooth map $\pi_0$  in diagram (\ref{eq:diagram_defining_fX}) to be pseudo-saturated is in fact equivalent to asking it to be quasi-saturated (namely, they give the same requirement on $\fX$). See \S \ref{section: admissible smoothness} for details. 

We now define the notion of being  ``sufficiently log smooth'' (in its crude form). 

\begin{maindef}
We say that $\fX$ is \emph{admissibly smooth} over $\mS$ if it is the saturated base change of $\pi_0: \fX_0 \ra \mS_0$ as in diagram (\ref{eq:diagram_defining_fX}), where $\kappa$ is induced by an injective map $N \hookrightarrow N_\infty$, and $\pi_0$ is an integral, pseudo-saturated log smooth (or equivalently, a saturated log smooth) map of fs log formal schemes. \end{maindef}

\subsection*{The comparison isomorphisms}
Now let $\fX$ be an admissibly smooth log $p$-adic formal scheme over $\mS$, then its generic fiber $X$ exists as a saturated log adic space by Proposition \ref{prop:adm_sm_fiber_product_exist}. These saturated log adic spaces are by definition   \emph{admissibly smooth} log adic spaces. For such $X$, we develop a well behaved theory of  Kummer \'etale site $X_{\ket}$ and pro-Kummer \'etale site $X_{\proket}$, extending results from \cite{DLLZ}. In particular, we define a suitable notion of log affinoid perfectoid objects in  $X_{\proket}$, which form a basis of the site. As in \cite{Scholze, DLLZ}, we have sheaves 
\[ \OXplus, \quad  \widehat{\mO}_X^{\flat+}= \lim_{x \mapsto x^p} \OXplus/p, \quad \Ainfx = W(\widehat{\mO}_X^{\flat+}), \quad \cdots 
\]
on $X_{\proket}$ (see Definition \ref{defn: structure sheaves} and Definition \ref{defn: AOmega}). Following the same recipe as in \cite{BMS1}, we are now able to define the following complexes of \'etale sheaves: 
\begin{align} \label{eq:def_of_AOmega_and_tilde}
 \qquad A \Omega^{\log}_{\fX} & : = L \eta_{(\mu)} R \nu_* \widehat{\Ainfx} \\
\nonumber \sq \Omega^{\log}_{\fX} & : = L \eta_{(\zeta_p - 1)} R \nu_* \OXplus. \end{align}
Here $\,\widehat{\empty}\,$ denotes the derived $p$-adic completion, $\nu$ denotes the natural projection of topoi $X_{\proket} \ra X_{\ket} \ra \fX_{\ett}$, and $L\eta$ denotes the d\'ecalage operator (see Section \ref{ss:decalage} for a brief review). Then we simply define the log $\Ainf$-cohomology of $\fX$ as
\[
R \Gamma_{\Ainf} (\fX) = R \Gamma_{\Ainf} (\fX/\mS) := R \Gamma (\fX_{\ett}, A \Omega_{\fX}^{\log}).
\]
This is an ($E_\infty$-$\Ainf$-algebra) object in $\mD(\Ainf)$ equipped with a $\varphi_{\Ainf}$-linear endomorphism $\varphi$, which induces a quasi-isomorphism 
\[
\varphi: R \Gamma_{\Ainf} (\fX) [\frac{1}{\xi}] \isom R \Gamma_{\Ainf} (\fX) [\frac{1}{\varphi(\xi)}]. 
\]

One of the main goals of this paper is to prove the following comparison theorems. 

\begin{mainthm} \label{mainthm:comparison}
Fix a base log $p$-adic formal scheme  $\mS$ as before, and let $\fX$ be an admissibly smooth log $p$-adic formal scheme over $\mS$. Then we have the following comparison isomorphisms for $R \Gamma_{\Ainf}(\fX)$. 
\be
\item \'etale comparison: suppose that $C$ is algebraically closed and that the generic fiber $X$ is proper, then there is a functorial quasi-isomorphism 
\[
R \Gamma_{\Ainf}(\fX) \otimes^\L_{\Ainf} \Ainf[\frac{1}{\mu}] \isom R \Gamma_{\ket} (X, \Z_p) \otimes^\L_{\Z_p} \Ainf [\frac{1}{\mu}].
\]
\item de Rham comparison: there is a functorial quasi-isomorphism 
\[
R \Gamma_{\Ainf}(\fX) \otimes^\L_{\Ainf} \Ainf/\xi \isom R \Gamma_{\tu{logdR}} (\fX/\mS).
\]
\item prismatic comparison: the triple $(\Ainf, (\xi), N_\infty)$ naturally upgrades to a pre-log prism (see Subsection \ref{ss:log_prism}), and there is a functorial, $\varphi$-equivariant quasi-isomorphism 
\[
R \Gamma_{\Prism} (\fX/(\Ainf, N_\infty)) \widehat \otimes^\L_{\Ainf, \varphi}
 \Ainf \isom R \Gamma_{\Ainf}(\fX),
\]
where $R \Gamma_{\Prism} (\fX/(\Ainf, N_\infty))$ denotes the log prismatic cohomology of $\fX$ relative to the pre-log prism $(\Ainf, (\xi), N_\infty)$ recalled in Subsection \ref{ss:log_prism}. Here the completion is the derived $(p, \mu)$-adic completion. 
\item crystalline comparison: let $k$ denote the residue field of $\mO^\flat$, and let $\fX_{k}$ denote the special fiber of $\fX$ over $\ul k =(k,N_\infty)$, in other words, the base change of $\fX$ along $\spec k \hookrightarrow \spf \mO$, equipped with its pullback log structure from $\fX$. Then we have a functorial,  $\varphi$-equivariant quasi-isomorphism 
\[
R \Gamma_{\Ainf}(\fX) \widehat \otimes^\L_{\Ainf} 
W(k) \isom R \Gamma_{\tu{logcrys}}(\fX_k/W(\ul k)).
\]
Here we take derived $p$-completion on the left hand side.
\item absolute crystalline comparison: let $\fX_{\mO/p}$ denote the base change of $\fX$ along the closed embedding  $\spec \mO/p \hookrightarrow \spf \mO$, equipped with its pullback log structure. Then there is a functorial, $\varphi$-equivariant quasi-isomorphism 
\[
R \Gamma_{\Ainf}(\fX) \widehat \otimes^\L_{\Ainf} \Acris \isom R \Gamma_{\tu{logcrys}}(\fX_{\mO/p}/(\Acris, N_\infty)),
\]
where the completion is again derived $p$-adic. 
\ee 
\end{mainthm}

Using these comparison results and the Hyodo--Kato isomorphism on log crystalline cohomology, we deduce the following corollary in  Section \ref{ss:BKF}. 
\begin{maincor}
Suppose that $\fX$ is in addition proper, then $R \Gamma_{\Ainf} (\fX)$ is a perfect complex in $\mD(\Ainf)$. Moreover, its log $\Ainf$-cohomology group $H^i_{\Ainf} (\fX) := H^i (R \Gamma_{\Ainf}(\fX))$ comes equipped with the structure of a Breuil--Kisin--Fargues module (see Section \ref{ss:BKF} for the definition). 
\end{maincor}  
This in turn allows us to extend results in \cite{BMS1, CK_semistable} on torsion discrepancies among various integral $p$-adic cohomology groups to the logarithmic setting. We refer the reader to Section \ref{ss:BKF} for more details.  

Next, let us briefly comment on the proof of these comparison theorems. The \'etale comparison is an immediate consequence of the construction, provided that we extend Scholze's primitive comparison theorem to our setting. In this direction, we give the following  generalization of \cite[Theorem 6.2.1]{DLLZ}. Let $X$ be the log adic generic fiber of $\fX$ constructed as above, and assume that it is proper (or more generally, let $X$ be a proper log adic space which is weakly admissibly smooth over $(\spa (C, \mO_C), N_\infty)^a$ in the sense of Definition \ref{defn: admissiblly log smooth}). 
\begin{mainthm}  \label{mainthm:primitive_comparison}
Suppose $C$ is algebraically closed. Let $\L$ be an $\F_p$-local system on $X_{\ket}$. Then we have the following:
\begin{enumerate}
\item $H^i (X_{\ket}, \L \otimes_{\F_p} (\mO_X^{+a} /p))$ is an almost finitely generated $\mO$-module for every $i\geq 0$, and is almost zero for $i\gg 0$.
\item There is a canonical almost isomorphism
\[H^i (X_{\ket}, \L) \otimes_{\F_p} (\mO^{a}/p) \xrightarrow[]{\sim} H^i (X_{\ket}, \L \otimes_{\F_p} (\mO_X^{+a} /p))\]
of almost $\mO$-modules, for every $i\geq 0$. Moreover, $H^i (X_{\ket}, \L)$ is a finite dimensional $\F_p$-vector space for every $i\geq 0$, and $H^i (X_{\ket}, \L)=0$ for $i\gg 0$.
\end{enumerate} 
\end{mainthm}

We remark that, while we still follow the strategy of \cite{Scholze}, our proof of this key technical result (Theorem \ref{mainthm:primitive_comparison}) uses some rather elaborate analysis of the combinatorics that are used to characterize pseudo-saturated morphisms (see Lemma \ref{lemma: Q}).  
Next, for the de Rham comparison, we first prove the following ``primitive Hodge--Tate comparison'' result.

\begin{mainthm}[The primitive Hodge--Tate comparison] \label{mainthm:primitive_HT}
Let $\sq \Omega^{\log}_{\fX}$ be the \'etale sheaf from (\ref{eq:def_of_AOmega_and_tilde}). There are functorial isomorphisms 
\[
\Omega_{\fX/\mS}^{\bullet} \isom \mH^\bullet (\sq \Omega^{\log}_{\fX}) \{\bullet\}
\]
of \'etale sheaves of differential graded algebras, where $\{i\}$ denotes the Breuil--Kisin--Fargues twist, and the differential on the right hand side is the Bockstein differential. 
\end{mainthm}

Assuming Theorem \ref{mainthm:primitive_HT}, it is not difficult to deduce the de Rham comparison isomorphism from the following  
\begin{mainthm}[The Hodge--Tate comparison] \label{mainthm:HT}
There is a natural quasi-isomorphism 
\[
A \Omega_{\fX}^{\log} \otimes^\L_{\Ainf} \Ainf/\varphi(\xi) \isom \sq \Omega_{\fX}^{\log}.  
\]
\end{mainthm}

The strategy to prove Theorem \ref{mainthm:primitive_HT} and \ref{mainthm:HT} is similar to that in \cite{BMS1} --- we first construct   functorial maps using log cotangent complexes (studied by Gabber and Olsson  in \cite{Olsson_log}), and then show that it is a quasi-isomorphism by choosing nice coordinates  locally to work with. This local analysis gets somewhat technical due to the presence of log structures, and we spell out the details in Sections \ref{section:etale_comparison}, \ref{section:HT_primitive} and \ref{section:Hodge_Tate_and_dR}. 

\subsection*{Derived log $\Ainf$-cohomology}

On the other hand, slightly deviating from the original proof in \cite{BMS1}, our proof of the prismatic and crystalline comparisons relies on the theory of  derived log $\Ainf$-cohomology and \emph{uses} de Rham comparison. The starting point is  to   ``derive'' or ``animate'' (in the derived $p$-complete sense) the functor that sends the $p$-completed ``log free'' algebra 
\begin{equation} \label{eq:log_free_algebra}
\Sigma(S, \ul{T}) \coloneqq (\mO_C [\N^{S}, \N^{T}]^{\wedge}, N_{\infty} \oplus \N^T)
\end{equation}
(here $(\empty -)^\wedge$ denotes the $p$-adic completion) to its  log $\Ainf$-cohomology complex
\[ 
A \Omega^{\log}_{\Sigma(S, \ul{T})} \coloneqq
R \Gamma_{\Ainf} (\spf (\Sigma(S, \ul{T})^a)  \in \mD (\Ainf).
\]
This gives rise to a functor from the category of derived $p$-complete (simplicial) pre-log rings over $\mO$ to the derived category $ \mD (\Ainf)$ of $\Ainf$-modules, sending 
\[
 (R, P) \longmapsto A \Omega_{(R, P)}^{\L, \log},
\] 
where $A \Omega_{(R, P)}^{\L, \log}$ is called the \emph{derived log $\Ainf$-cohomology} of $(R, P)$. 

This construction can be naturally globalized to a sheaf $A \Omega_{\fX}^{\L, \log}$ of $E_\infty$-$\Ainf$-algebras on the \'etale site of $\fX$.  From  Theorem \ref{mainthm:primitive_HT} and Theorem \ref{mainthm:HT} we obtain a derived version of Hodge--Tate comparison (see Corollary \ref{cor:derived_HT}), which essentially allows us to control the derived log $\Ainf$-cohomology of $(R, P)$ by studying  log cotangent complexes. 
To proceed, we first prove a comparison between the derived log $\Ainf$-cohomology $ A \Omega_{(R, P)}^{\L, \log}$ and the derived log prismatic cohomology $ \Prism^{\L}_{(R, P)/\Ainf}$ (constructed in \cite{logprism}), using log q-crystalline cohomology studied in \cite{Koshikawa} as a bridge. This comparison is stated as part of Theorem \ref{mainthm:derived_comparisons}. In particular, derived log $\Ainf$-cohomology comes equipped with a Nygaard filtration, coming from the Nygaard filtration of  (derived) log prismatic cohomology constructed in \cite{logprism}. We then compare the derived and non-derived log $\Ainf$-(resp. log prismatic) cohomology in our setting using the de Rham comparison.



Moreover, in \cite{logprism}, Koshikawa and the second author  introduce a logarithmic version of the quasisyntomic site studied by \cite{BMS2} and prove that log cotangent complex (and thus derived log $\Ainf$-cohomology) satisfies ``log quasisyntomic descent''. This theory is briefly reviewed in Section \ref{section: log quasisyntomic site}. Similar to the (nonlog) quasisyntomic site, an important feature of the log quasisyntomic site is that it is locally log quasiregular semiperfectoid (see Definition \ref{definition:qrspd_log}),  for which the derived log $\Ainf$-cohomology turns out to be a discrete ring. For such pre-log rings, we prove the following analogue of \cite[Theorem 8.14]{BMS2} and \cite[Lemma 3.8]{Yao_Acris}. 

\begin{mainprop} \label{mainprop:derived_crystalline}
Let $(S, M)$ be a quasiregular semiperfectoid pre-log $\mO$-algebra, then there are functorial isomorphisms of (discrete) $\Ainf$-algebras 
\begin{equation} \label{eq:derived_crystalline_equivalent_complexes}
\widehat \L \Omega_{(S, M)/\Z_p} \isom \L R \Gamma_{\crys} ((S, M)/\Z_p) \isom A \Omega^{\L, \log}_{(S, M)} \widehat \otimes^\L_{\Ainf} \Acris,
\end{equation}
where $\widehat \L \Omega_{(-)/\Z_p}$ denotes the $p$-completed derived log de Rham cohomology, $ \L R \Gamma_{\crys} ((-)/\Z_p)$ denotes the derived (absolute) log crystalline cohomology, and $A \Omega^{\L, \log}_{(S, M)}$ is the derived log $\Ainf$-cohomology from above. 
\end{mainprop}

\begin{mainrmk} 
There is a subtle difference compared with the non-log situation. Let $(S, M)$ be as in Proposition \ref{mainprop:derived_crystalline} and let $\Acris ((S, M))$ denote the $p$-completed logarithmic PD envelop of $(S/p, M)$. Unlike \cite[Theorem 8.14]{BMS2}, we no longer have a canonical isomorphism between $\Acris ((S, M))$ and 
$A \Omega_{(S, M)}^{\log} \widehat \otimes^\L_{\Ainf} \Acris$ (or, equivalently, any term in (\ref{eq:derived_crystalline_equivalent_complexes})).  This is caused by a certain exactification procedure and is related to the difficulty of constructing the correct Nygaard filtration on log prismatic cohomology. We refer the reader to Section \ref{ss:remark_on_Acris} and \cite[Section 5]{logprism} for more details. 
\end{mainrmk}

\subsection*{Derived log  $\Ainf$-cohomology via the infinite root stack}

The theory of animation also allows one to construct another functor from the category of derived $p$-complete simplicial pre-log $\mO$-algebras to $\mD(\Ainf)$, by ``deriving'' the prismatic cohomology of the \textit{infinite root stack} associated to the log free algebras $\Sigma(S, \ul{T})$ from (\ref{eq:log_free_algebra}) as a stack.\footnote{We learned some aspect of this construction from Mathew, based on the  joint work in progress of Bhatt--Clausen--Mathew.}

The infinite root stack $\sqrt[\infty]{\fX}$ of a log (formal) scheme $\fX = (\fX, \mM)$ is a certain non-algebraic stack associated to $\fX$ constructed in \cite{Talpo_Vistoli}, which functorially parametrizes certain ``$n^{th}$-roots'' of the log structure $\mM$ for all $n$. We refer the reader to Section \ref{section:infinite_root_stack} for the precise definition (in fact, we will consider a variant of this that only parametrizes $p$-power roots). Let us remark that, it can be regarded as an algebraic incarnation of the Kato--Nakayama space associated to log schemes over the complex numbers, and is designed to turn the additional logarithmic information into plain geometric information, at a cost of introducing non-algebraic stacks. However, in the case of toric log schemes, it can be written down rather explicitly (see Section \ref{section:infinite_root_stack} for the example of the log affine line and more generally see \cite[Section 3]{Talpo_Vistoli}).
Moreover, one can  resconstruct the log (formal) scheme $\fX$ from the infinite root stack $\sqrt[\infty]{\fX}$. To the infinite root stack $\sqrt[\infty]{\fX_{S, \ul T }}$ associated to $\spf (\Sigma(S, \ul{T}))^a$, one can attach the $\Ainf$ cohomology of $\sqrt[\infty]{\fX_{S, \ul T}}$ in a suitable sense using the derived prismatic/$\Ainf$-cohomology of simplicial rings (for example, from the previous subsection with the trivial log structure; see also the treatment in \cite{Kubrak}), and then further animate to obtain a functor from the category of derived $p$-complete simplicial pre-log rings over $\mO$ to the derived category $ \mD (\Ainf)$ of $\Ainf$-modules, sending 
\[
(R, P) \longmapsto \Prism^{\L, \infty}_{(R, P)/\Ainf}.
\]

In Section \ref{section:infinite_root_stack} we show that this construction agrees with the derived log prismatic cohomology considered in the previous subsection. 
\begin{mainthm} \label{mainthm:derived_comparisons} Let $(R, P)$ be a derived $p$-complete simplicial pre-log ring over $\mO$, then we have functorial,  $\varphi$-equivariant isomorphisms 
\[
  A \Omega^{\L, \log}_{(R, P)} \isom  \varphi_{\Ainf}^* \Prism^\L_{(R, P)/ \Ainf} \isom \varphi_{\Ainf}^* \Prism^{\L, \infty}_{(R, P)/\Ainf}
\]
compatible with the (derived) Hodge--Tate and de Rham comparisons. 
\end{mainthm}

As a result, we deduce a logarithmic version of the Beilinson fiber square (on graded terms) obtained in \cite{Beilinson_fiber}. To state the result, we introduce the logarithmic $p$-adic Tate-twist $\Z_p(n)$ as in the non-log case, which is defined by
\[
\Z_p (n) (\ul S ):= \tu{Fib} \big(\tu{Fil}_{N}^n \Prism^{\L}_{\ul S} \{n\} \xrightarrow{\varphi - \tu{can}} \Prism^{\L}_{\ul S} \{n\}\big)
\]
on a (simplicial) pre-log $\mO_C$-algebra $\ul S$. 
Here $\{n\}$ denotes the Breuil--Kisin--Fargues twist, and $\tu{Fil}_N^\bullet$ denotes the Nygaard filtration on derived log prismatic cohomology, which comes equipped with a canonical map  \[\varphi:\tu{Fil}_{N}^n \Prism_{\ul S} \ra (\xi)^n \Ainf \otimes_{\Ainf} \Prism_{\ul S}\] (see \cite[Section 5.5]{logprism}). We then set $\Q_p (n) = \Z_p (n)[1/p].$

\begin{maincor}
Let $\ul S = (S, M)$ be a pre-log ring that is quasisyntomic over $\Z_p$.
    \be   
      \item There exists a natural pullback square 
     \[ \begin{tikzcd}[row sep = 2em]
          \Q_p (n) (\ul S) \arrow[d] \arrow[r] & \Q_p (n) (\ul S/p) \arrow[d] \\ 
           \tu{Fil}_{H}^n \widehat \L \Omega_{\ul S/\Z_p} \{n\}_{\Q_p} \arrow[r]  &  \widehat \L \Omega_{\ul S/\Z_p} \{n\}_{\Q_p}.
       \end{tikzcd} \]
       in the derived $\infty$-category $\mD(\Q_p)$, where $  \tu{Fil}_{H}^n$ denotes the ($p$-completed) Hodge filtration on the derived log de Rham cohomology, and $()_{\Q_p}$ denotes the base change $\otimes_{\Z_p} \Q_p.$
      \item There is a functorial isomorphism
      \[ 
      \Q_p (n) (\ul S) \cong \text{Fib} \Big(
           \tu{Fil}_{H}^n \widehat \L \Omega_{\ul S/\Z_p} \{n\}  \xrightarrow{\varphi - p^n} 
           \widehat \L \Omega_{\ul S/\Z_p} \{n\} \Big)_{\Q_p}.
      \]
    \ee 
\end{maincor}





 
\addtocontents{toc}{\protect\setcounter{tocdepth}{0}}
\subsection*{Relative log BKF modules (subsequent development in \cite{log2})}

We end the introduction with a brief description of some subsequent development that will appear in the sequel \cite{log2} of this paper. In \cite{log2}, we generalize the main theorems of this paper to allow more general coefficients, partially following ideas from \cite{MT}. More precisely, we construct a log $\Ainf$-cohomology theory with coefficients in so-called \emph{relative logarithmic Breuil-Kisin-Fargues modules} (or, \emph{relative log BKF modules} in short).

\begin{maindef}
Let $\mathfrak{X}$ be a $p$-adic formal log scheme which is admissibly smooth over $\mS$ as before and let $X$ be its adic generic fiber. Let $X_{\proket}$ denote the pro-Kummer \'etale site on $X$ and let $\nu: X_{\proket}\rightarrow \mathfrak{X}_{\ett}$ be the natural projection.
\begin{enumerate}
\item Let $\mathbb{M}$ be a sheaf of $\Ainfx$-modules on $X_{\proket}$. For every integer $n\geq 1$, write $\xi_n$ for the element $\frac{[\varepsilon]-1}{[\varepsilon^{1/p^n}]-1}\in \Ainf$. We say that $\mathbb{M}$ is \emph{trivial modulo} $\xi_n$ if the sheaf of $\nu_*\big(\Ainfx/\xi_n\big)$-module $\nu_*(\mathbb{M}/\xi_n)$ is locally finite free and the counit
\[\nu^{-1}\nu_*(\mathbb{M}/\xi_n)\otimes_{\nu^{-1}\nu_*(\Ainfx/\xi_n)}\Ainfx/\xi_n\rightarrow \mathbb{M}/\xi_n\]
is an isomorphism of sheaves on $X_{\proket}$. We say that $\mathbb{M}$ is \emph{trivial modulo} $<\mu$ if it is trivial modulo $\xi_n$ for all $n\geq 1$.

\item A \emph{relative logarithmic Breuil-Kisin-Fargues module} over $\fX$ is a pair $(\mathbb{M}, \varphi_{\mathbb{M}})$ where $\mathbb{M}$ is a locally finite free $\Ainfx$-module which is trivial modulo $<\mu$ and 
\[\varphi_{\mathbb{M}}:(\varphi^*\mathbb{M})[\frac{1}{\varphi(\xi)}]\rightarrow \mathbb{M}[\frac{1}{\varphi(\xi)}]\] is an isomorphism of sheaves of $\Ainfx[\frac{1}{\varphi(\xi)}]$-modules. The category of such objects is denoted by $\mathrm{BKF}^{\log}(\fX, \varphi)$.
\end{enumerate}
\end{maindef}

Just like in the case of trivial coefficients, for a relative log BKF module $\mathbb{M}\in \mathrm{BKF}^{\log}(\fX, \varphi)$, we can associate a complex of \'etale sheaves
\[A\Omega^{\log}_{\fX}(\mathbb{M}):=L\eta_{(\mu)} ({R\nu_*\mathbb{M}})^{\wedge}\in \mD(\Ainf)\]
where the completion is the derived $p$-adic completion. We also define
\[R \Gamma_{\Ainf}(\fX, \mathbb{M}):= R \Gamma (\fX_{\ett}, A \Omega_{\fX}^{\log}(\mathbb{M})).\]

Moreover, there exist natural specialization functors
\begin{align*} 
\sigma^*_{\ket}: 
&\mathrm{BKF}^{\log}(\fX, \varphi)\rightarrow \Big\{\begin{array}{l} \textrm{Kummer \'etale } \mathbb{Z}_p\textrm{-local}\\ \textrm{systems on } X = \fX_{\eta}\end{array}\Big\}   \\
\sigma^*_{\textrm{logdR}}: &\mathrm{BKF}^{\log}(\fX, \varphi)\rightarrow \Big\{\begin{array}{l} \textrm{vector bundles with integrable}\\ \textrm{log connections on }\fX \end{array} \Big\} \\ 
\sigma^*_{\textrm{logcrys}}:  &\mathrm{BKF}^{\log}(\fX, \varphi)\rightarrow \Big\{\begin{array}{l} \textrm{locally finite free $F$-crystals}\\ \textrm{on } (\fX_k/W(\underline{k}))_{\textrm{logcrys}} \end{array} \Big\}
\end{align*}

The following expected result in \cite{log2} is a logarithmic analogue of \cite[Theorem 6.2]{MT}.

\begin{mainthm}[\cite{log2}]
Let $\fX$ and $\mathbb{M}$ be as above. Then we have the following comparison isomorphisms for $R \Gamma_{\Ainf}(\fX, \mathbb{M})$. 
\be
\item \'etale comparison: suppose that $C$ is algebraically closed and that the generic fiber $X$ is proper, then there is a functorial quasi-isomorphism 
\[
R \Gamma_{\Ainf}(\fX, \mathbb{M}) \otimes^\L_{\Ainf} \Ainf[\frac{1}{\mu}] \isom R \Gamma_{\ket} (X, \sigma^*_{\ket}(\mathbb{M})) \otimes^\L_{\Z_p} \Ainf [\frac{1}{\mu}].
\]
\item de Rham comparison: there is a functorial quasi-isomorphism 
\[
R \Gamma_{\Ainf}(\fX, \mathbb{M}) \otimes^\L_{\Ainf} \Ainf/\xi \isom R \Gamma_{\tu{logdR}} (\fX/\mS, \sigma^*_{\textrm{logdR}}(\mathbb{M})).
\]
\item crystalline comparison: there is a functorial, $\varphi$-equivariant quasi-isomorphism 
\[
R \Gamma_{\Ainf}(\fX, \mathbb{M}) \widehat \otimes^\L_{\Ainf} 
W(k) \isom R \Gamma_{\tu{logcrys}}(\fX_k/W(\ul k), \sigma^*_{\textrm{logcrys}}(\mathbb{M}))
\]
where we take derived $p$-adic completion on the left hand side.
\ee 
\end{mainthm}

As an application, we prove the $C_{\mathrm{st}}$-conjecture for general coefficients, as well as over the complement of horizontal normal crossing divisors. For instance, we expect to prove the following result.  

\begin{mainthm}[\cite{log2}] 
Let $\mathfrak{X}$ be a proper log smooth $p$-adic formal log scheme over $\mS$ with semistable reduction (in particular, $\fX$ is admissibly smooth over $\mS$) and let $X$ be its adic generic fiber. Let $\mathbb{L}$ be a semistable $\mathbb{Z}_p$-local system on $X$ (which we make precise in \cite{log2}) and let $\mathcal{E}$ be the associated filtered $F$-isocrystal on $(\fX_k/W(\underline{k}))_{\textrm{logcrys}}$. Then there is a natural isomorphism of $\Bst$-modules
\[H^i_{\ett}(X, \mathbb{L})\otimes_{\mathbb{Z}_p} \Bst\simeq H^i_{\mathrm{logcrys}}(\fX_k/W(\underline{k}), \mathcal{E})\otimes_{W(k)}\Bst\]
for all $i\geq 0$, which is compatible with the natural Galois actions, Frobenius actions, filtrations, and monodromy operators on both sides.
\end{mainthm}

\subsection*{Acknowledgement} 
 
We would like to thank Dori Bejleri, Kestutis \v{C}esnavi\v{c}ius, Yuchen Fu,  Ofer Gabber, Teruhisa Koshikawa, Akhil Mathew, and Takeshi Tsuji for helpful discussions in the preparation of this paper. Part of the project was completed during the  second author's visits to YMSC at Tsinghua university and CNRS at universit\'e Paris-Sud. He would like to thank the hospitality of these institutes. During the preparation of the article, the first author was partially supported by the National Key R{\&}D Program of China No. 2023YFA1009703 and No. 2021YFA1000704, and the second author was partially supported by ERS Grant Number 851146.

\newpage 

\addtocontents{toc}{\protect\setcounter{tocdepth}{2}}
\section{Preliminaries} \label{section:review_log}

In this section we collect some basic algebraic preliminaries. We first recall Fontaine's period ring $\Ainf$ and the d\'ecalage operator $L\eta$, mostly following \cite{BMS1}. We then recall definitions and fix conventions in log geometry in Section \ref{subsection: conventions on monoids}. Moreover, we give an explicit description of pushouts in the category of saturated monoids in Section \ref{ss:saturated_pushout}. 


\subsection{Fontaine's period ring $\Ainf$} 
 \label{section:review_Ainf}
 
Let us first recall Fontaine's period ring $\Ainf$. Throughout the paper, let $C$ be a perfectoid field containing all $p$-power roots of unity and let $\mO=\mO_C$ be its ring of integers. Define 
$$\Ainf := W(\mO_C^\flat) = W( \underset{x \mapsto x^p}{\varprojlim} \mO_C /p)$$
equipped with a natural Frobenius automorphism $\varphi = \varphi_{\Ainf}$. We fix a compatible system of primitive $p$-power roots of unity $\zeta_{p^m} \in \mO_C$. The system $(1, \zeta_p, \zeta_{p^2}, ...)$ defines an element $\epsilon \in \mO_C^\flat$. For all $n \ge 1$, let 
\begin{align} \label{eqn:elements_Ainf} 
\qquad & \mu := [\epsilon] - 1  \nonumber \\ 
&  \xi := \mu/\varphi^{-1}(\mu) \nonumber \\
&  \xi_n := \varphi^{-(n\text{-}1)}(\xi) \varphi^{-(n\text{-}2)} (\xi) \cdots \xi,  \nonumber \\
& \sq \xi_n := \varphi^n(\xi_n) = \varphi(\xi) \varphi^2 (\xi) \cdots \varphi^n (\xi)  \nonumber
\end{align}
be elements in $\Ainf$. Let $W_n (\mO_C)$ be the ($p$-typical) Witt vectors of $\mO_C$ of level $n$. From \cite[Section 3]{BMS1}, there are two classes of surjective maps 
$$\theta_n: \Ainf \surjlra  W_n(\mO_C), \qquad \sq \theta_n: \Ainf \surjlra W_n (\mO_C)$$ whose kernels are respectively generated by $\xi_n$ and $\sq \xi_n$  specified above. We denote $\theta_1$ (resp. $\sq \theta_1$) by $\theta$ (resp. $\sq \theta$). By \cite[Lemma 3.23]{BMS1}, the kernel of 
$$\theta_{\infty}: \Ainf \ra W(\mO_C) = \varprojlim W_n (\mO_C) $$
is generated by $\mu$. In particular, the ideal $(\mu)$ is independent of our choice of the roots of unity. Finally, recall the following notations. 
\bi
\item Let us write $\Binf := \Ainf[1/p]$; 
\item the discretely valued period ring $\Bdrp$ is defined as the $\xi$-adic completion of $\Binf$; 
\item  the crystalline period ring $\Acris$ is defined as the $p$-complete PD (=divided power) envelop of $\Ainf$ along $\theta$; 
\item define  $\Bcrisp : = \Acris[\frac{1}{p}] $ and $\Bcris:=\Acris[\frac{1}{\mu}] = \Bcrisp[\frac{1}{\mu}]$. 
\ei 
The Witt vector functoriality also supplies a map 
\begin{equation} \label{eqn:vartheta}
\vartheta: \Ainf = W(\mO_C^{\flat}) \lra W(k)
\end{equation} 
which is useful for our discussion on the comparison with log crystalline cohomology.  Let us also record the following result from \cite{BMS1}. 
\bp 
\be
\item For each $n \ge 1$, the ring $W_n (\mO_C)$ is coherent, i.e., every finitely generated ideal of $W_n (\mO_C)$ is finitely presented. 
\item For any finitely presented $W_n (\mO_C)$-module $M$, there are no non-zero
elements of $M$ that are killed by $W_m (\fm)$. 
\ee 
\ep 

\bproof 
This is \cite[Proposition 3.24 \& Corollary 3.29]{BMS1}. 
\eproof

\subsection{The d\'ecalage operator $L\eta$}   
\label{ss:decalage}
  
Now let us briefly recall the d\'ecalage operator (see \cite[Section 6]{BMS1}). Let $\mO_T$ be a ring (or more generally, let $(T, \mO_T)$ be a ringed topos) and $K(\mO_T)$ be the category of cochain complexes of $\mO_T$-modules. Let $f \in \mO_T$ be an element that generates an invertible ideal $(f)$.  

\bd 
Let $C^\bullet \in K(\mO_T)$ be an $f$-torsion free complex, in other words, assuming that the map $f: C^i \ra C^i$ is injective for every $i \in \Z$. The complex $\eta_f C^\bullet$ is defined as the subcomplex of $C^\bullet [\frac{1}{f}]$ given by 
$$ (\eta_f C)^i := \{x \in f^i C^i  \: |\:  dx \in f^{i+1} C^{i+1}\} $$
where the differential is induced from the differential on $C^\bullet [\frac{1}{f}]$. 
\ed

\br 
The complex $\eta_f C^\bullet$ depends only on the ideal $(f)$. In fact, slightly more generally, one may define $\eta_I C^\bullet$ for an invertible ideal sheaf $I \subset \mO_T$ (see \cite[Definition 6.2]{BMS1}). Moreover, 
Let $C^\bullet \in K(\mO_T)$ be an $f$-torsion-free complex, then there is a canonical isomorphism 
$$ H^i (\eta_f C^\bullet) \cong \big( H^i(C^\bullet)/ H^i(C^\bullet)[f]  \big) \otimes_{\mO_T} (f^i) $$
(see \cite[Lemma 6.4]{BMS1}).
In particular, $\eta_f$ sends quasi-isomorphisms between such complexes to quasi-isomorphisms, and induces a (non-exact!) functor (\emph{the d\'ecalage operator})  
\[
L \eta_f : D (\mO_T) \ra D(\mO_T)
\]on the derived category. In particular, $L \eta_f$ commutes with filtered colimits and canonical truncations. 
\er 
 
\bp Let us collect some basic properties of $L \eta$. 
\be
\item  There is a natural lax symmetric monoidal structure on $L \eta_f$:  
\begin{equation} \label{eq:lax_sym_mon}
L \eta_f C \otimes^\L_{\mO_T} L \eta_f D \lra  L \eta_f (C \otimes^\L_{\mO_T} D).
\end{equation}
Moreover, if $\mO_T$ is a valuation ring (for example if $\mO_T = \mO$ from the introduction), then (\ref{eq:lax_sym_mon}) is a quasi-isomorphism . 
\item Suppose that $C \in \mD^{\ge 0}(\mO_T)$ satisfies $H^0(C)[f] = 0$, then there is a canonical map \[L \eta_f C \lra C.\] 
\item For two non-zero divisors $f, g \in \mO_T$, we have a natural equivalence
\[L \eta_{f g} \isom L \eta_f \circ L \eta_g: \mD(\mO_T) \ra \mD (\mO_T).\]
\item The the d\'ecalage operator behaves well with completions. More precisely, suppose that $J \subset \mO_T$ is a locally finitely generated ideal, and $C \in \mD(\mO_T)$ is derived $J$-adically complete, then $L \eta_f C$ is derived $J$-adically complete. More generally, for $C \in \mD(\mO_T)$, there are natural quasi-isomorphisms 
\[
(L \eta_f C)^{\wedge} \isom L \eta_f (C^{\wedge}) \isom \tu{R}\!\varprojlim (C \otimes^\L_{\mO_T} \mO_T/f^n).
\]
Here $(-)^{\wedge}$ denotes the derived $J$-adic completion. 
\ee 
\ep 
\bproof 
These are respectively \cite[Proposition 6.7, Proposition 6.8, Lemma 6.10, Lemma 6.11, Lemma 6.19, and  Lemma 6.20]{BMS1}
\eproof

The following surprising observation of \cite{BMS1} plays an important role for us later on.

\bp \label{lemma:decalage_turns_almost_into_actual}
Let $f \in \mO_T$ be a nonzero divisor as above and let $I \subset \mO_T$ be an ideal containing $f$. Let $\beta: C \ra D$ be a map in $\mD (\mO_T)$ such that for each $i\in \Z$, both the kernel and the cokernel of $H^i (C) \ra H^i (D)$ are killed by $I$, and that $H^i (C)$ and $H^i(C)/f$ have no nonzero elements killed by $I$, then 
\[L \eta_f \beta: L \eta_f C \isom L \eta_g D\] is a quasi-isomorphism. 
\ep 

\bproof 
This is \cite[Lemma 8.11]{BMS1}
\eproof

\subsection{Conventions on monoids}\label{subsection: conventions on monoids}  
In this section we recall some basic definitions in log geometry, mostly following \cite{Kato, Ogus}. Let us first fix the following conventions.
\bi
\item By a monoid we mean a commutative monoid with identity. Note that the category of monoids is not abelian, but admits all limits and colimits. 
\item For a monoid $P$, we write $P^{\gp}$ for its group envelope and $P^\times$ for its subgroup of units. 
\item The quotient monoid $P/P^\times$ is denoted by $\overline{P}$. 
\item For ${x}, {x}'\in P$, we write ${x}\,|\,{x}'$ if there exists ${x}''\in P$ such that ${x}+{x}''={x}'$. 
\item We use $\mathbb{Z}_{\geq 0}$  
(resp. $\mathbb{Q}_{\geq 0}$) to denote the set of non-negative integers (resp. non-negative rational numbers) equipped with the additive monoid structure.
\item 
For a monoid $P$ and a ring $R$, let $R[P]$ denote the corresponding monoid algebra.
\item   For a Huber pair $(R, R^+)$ with a ring of definition $R_0$ which is adic with respect to a finitely generated ideal of definition $I$, we have the Huber pair $(R[P], R^+[P])$ (resp. the completed Huber pair $(R\gr{P}, R^+ \gr{P})$) with topology given by the ring of definition $R_0 [P]$ and the basis $\{I^m R_0 [P]\}_{m \ge 0}$.
\ei

For the convenience of the reader, 
let us recall the following definitions. 
\begin{definition}\label{defn: monoids}
\begin{enumerate}
\item A monoid $P$ is \emph{finitely generated} if there exists a surjection \[\mathbb{Z}^r_{\geq 0} \surjlra P\] for some integer $r\geq 0$. 
\item A homomorphism $P\rightarrow Q$ is of \emph{finite type} if there exists a homomorphism $\mathbb{Z}^r_{\geq 0} \ra Q$  for some $r\geq 0$, such that the induced homomorphism 
\[P\oplus \mathbb{Z}^r_{\geq 0} \surjlra  Q\] is surjective.
\item A monoid $P$ is \emph{integral} if the natural homomorphism $P \rightarrow P^{\gp}$ is injective. \\ A monoid is \emph{fine} if it is both integral and finitely generated.
\item A monoid $P$ is \emph{saturated} if it is integral and, for every $a \in P^{\gp}$ such that $na \in P$ for some integer $n \geq 1$, we have $a \in P$. \\ A monoid is \emph{fs}  if it is both fine and saturated.
\item A monoid $P$ is \emph{sharp} if $P^{\times}=\{0\}$. \\
A monoid is \emph{toric} if it is both sharp and fs.
\item A monoid $P$ is \emph{$n$-divisible} (resp. \emph{uniquely $n$-divisible}; \emph{almost $n$-divisible}) for an integer $n \geq 1$ if the multiplication-by-$n$ map $[n]: P \rightarrow P$ is surjective (resp. bijective; of finite type). 
\item For a saturated torsion-free monoid $P$ and an integer $n\geq 1$, let $\frac{1}{n}P$ denote the saturated torsion-free monoid such that the inclusion $P\hookrightarrow \frac{1}{n}P$ is isomorphic to the multiplication-by-$n$ map $[n]: P \rightarrow P$. Let 
\[P_{\mathbb{Q}_{\geq 0}}:=\varinjlim_n (\frac{1}{n}P)\] where the colimit runs through all positive integers $n$.
\end{enumerate} 
\end{definition}

\begin{remark}\label{remark: int and sat}
\begin{enumerate}
\item Quotient of an integral monoid (resp. saturated monoid) by a submonoid remains integral (resp. saturated).
\item For any monoid $P$, let $P^{\mathrm{int}}$ denote the image of the canonical homomorphism 
\[P \lra P^{\mathrm{gp}}.\] Note that the functor $P\mapsto P^{\mathrm{int}}$ is the left adjoint to the inclusion from the category of integral monoids into the category of all monoids.
\item For any integral monoid $P$, define its \emph{saturation} to be the monoid \[P^{\mathrm{sat}}:=\{a\in P^{\mathrm{gp}}\,|\, na\in P, \textrm{ for some }n\geq 1\}.\] Then the functor $P\mapsto P^{\mathrm{sat}}$ is the left adjoint to the inclusion from the category of saturated monoids into the category of integral monoids. 
\item More generally, for any monoid $P$, let $P^{\mathrm{sat}}:=(P^{\mathrm{int}})^{\mathrm{sat}}$. The functor $P\mapsto P^{\mathrm{sat}}$ is also the left adjoint to the inclusion from the category of saturated monoids into the category of all monoids. 
\end{enumerate}
\end{remark}

The following lemma will be used later.  
\begin{lemma}\label{lemma: Pgp and P}
Let $P$ be an fs monoid and identify $P$ with a submonoid of $P^{\mathrm{gp}}$ via the canonical homomorphism $P\rightarrow P^{\mathrm{gp}}$. Then, for any ${x}\in P^{\mathrm{gp}}$ and any positive integer $n$, there exists ${x}'\in P$ such that ${x}+n{x}'\in P$.
\end{lemma}

\begin{proof}
By \cite[Lemma 2.1.10]{DLLZ}, the natural surjection $P\rightarrow P/P^{\times}=\overline{P}$ admits a splitting. Hence, we reduce to the case when $P$ is toric. By \cite[pp. 6-7]{KKMSD}, there exist elements $x_1, \ldots, x_d$ which generate the $\mathbb{Q}$-vector space $P^{\mathrm{gp}}\otimes_{\mathbb{Z}}\mathbb{Q}$ such that $P_{\mathbb{Q}\geq 0}$ is identified with the subset \[\big\{\sum_{i=1}^d\lambda_ix_i\,|\, \lambda_i\in \mathbb{Q}_{\geq 0}\big\}\] of $P^{\mathrm{gp}}\otimes_{\mathbb{Z}}\mathbb{Q}$. Choosing ${x}'=\sum_{i=1}^d\lambda_ix_i\in P$ for sufficiently large $\lambda_i$'s, we can guarantee that ${x}+n{x}'\in P_{\mathbb{Q}\geq 0}$. Finally, using the saturated-ness of $P$, we must have \[{x}+n{x}'\in P^{\mathrm{gp}}\cap P_{\mathbb{Q}\geq 0}=P.\]
\end{proof}

\begin{definition}
A monoid homomorphism $u:P\rightarrow Q$ is \emph{exact} if the induced homomorphism \[P\rightarrow P^{\mathrm{gp}}\times_{Q^{\mathrm{gp}}}Q\] is an isomorphism. In particular, if $P$ and $Q$ are integral monoids, then $u$ is exact if and only if $(u^{\mathrm{gp}})^{-1}(Q)=P$, where $u^{\mathrm{gp}}:P^{\mathrm{gp}}\rightarrow Q^{\mathrm{gp}}$ denotes the induced homomorphism on the group envelopes. 
\end{definition}

Let us also recall the notion of perfectoid monoids from \cite{logprism}. 
\begin{definition}\label{definition:perfectoid_monoids}
\begin{enumerate} 
\item 
A monoid $P$ is \emph{perfectoid} if the canonical homomorphism 
$$P^\flat /(P^\flat)^\times \lra P/P^\times$$
is an isomorphism, where 
$$P^\flat := \varprojlim_{x\mapsto px} P$$ is the \emph{tilt} of the monoid $P$. 
\item A monoid $P$ is \emph{divisible} if it is uniquely $n$-divisible for all $n \ge 1$. 
\end{enumerate}
\end{definition}

\br 
\be 
\item Note that $P^\flat$, and hence $P^\flat /(P^\flat)^\times$,  is uniquely $p$-divisible. 
\item The condition of being perfectoid is stronger than requiring $P/P^\times$ to be uniquely $p$-divisible. For example\footnote{We learned this example from Teruhisa Koshikawa.}, let $P$ be the monoid 
\[P=\gr{x_0, x_1, x_2, \dots, \pm y_1, \pm y_2, \dots}/\gr{p x_1=(x_0+ y_1), p x_2=(x_1+y_2), \dots}
\]
Then $P^\flat$ is trivial, while $P/P^\times \cong \mathbb{N}[1/p]$. 
 \item For a saturated torsion-free monoid $P$,  $P_{\mathbb{Q}_{\geq 0}}$ is divisible.  
\ee 
\er 

\subsection{Pushouts of saturated monoids} \label{ss:saturated_pushout}
We end this section with explicit descriptions of pushouts of saturated monoids. These explicit descriptions will be used throughout the paper. 

\begin{construction}\label{construction: explicit descriptions of pushouts}
Suppose that $u:P\rightarrow Q$ and $v: P\rightarrow P'$ are homomorphisms of saturated monoids. One can consider the pushouts 
\[Q\sqcup_P P', \: Q\sqcup^{\mathrm{int}}_P P', \: Q\sqcup^{\mathrm{sat}}_P P'\] in the category of monoids, category of integral monoids, and category of saturated monoids, respectively.

By \cite[Proposition I.1.1.5]{Ogus}, $Q\sqcup_P P'$ can be identified with the quotient of $Q\oplus P'$ by the congruence relation consisting of pairs 
\[((q,r), (q',r'))\in (Q\oplus P')\times (Q\oplus P') 
\footnote{Here we use the notation $\times$ to simply denote an ordered pair of elements in $Q \oplus P'$.} \] such that there exists a sequence $(q_0, r_0), (q_1, r_1), \ldots, (q_n, r_n)\in Q\oplus P'$ and a sequence $p_0, p_1, \ldots, p_{n-1}\in P$ satisfying
\begin{itemize}
\item $(q_0, r_0)=(q,r)$,
\item $(q_n, r_n)=(q', r')$,
\item $q_i+u(p_i)=q_{i+1}$, $r_i=r_{i+1}+v(p_i)$, if $i$ is even, and
\item $q_i=q_{i+1}+u(p_i)$, $r_i+v(p_i)=r_{i+1}$, if $i$ is odd.
\end{itemize}

By \cite[Proposition I.1.3.4]{Ogus}, $Q\sqcup^{\mathrm{int}}_P P'$ can be identified with the image of the canonical homomorphism 
\[Q\sqcup_P P'\rightarrow Q^{\mathrm{gp}}\oplus_{P^{\mathrm{gp}}}(P')^{\mathrm{gp}}.\] Using the explicit description of $Q\sqcup_P P'$ above, we know that $Q\sqcup^{\mathrm{int}}_P P'$ can be identified with the quotient of the monoid
\[
T:=\Bigl\{ (q,r)\in Q^{\mathrm{gp}}\times (P')^{\mathrm{gp}}\,\Big|\, \begin{array}{l}\textrm{there exists }q'\in Q, r'\in P', \textrm{ and }p\in P^{\mathrm{gp}}\\ \textrm{such that }q=q'+u^{\mathrm{gp}}(p), r+v^{\mathrm{gp}}(p)=r'\end{array}\Bigr\}
\]
by the congruence relation consisting of pairs $((q,r), (q',r'))\in T\times T$ such that there exists $p\in P^{\mathrm{gp}}$ satisfying $q=q'+u^{\mathrm{gp}}(p)$ and $r+v^{\mathrm{gp}}(p)=r'$. 

In particular, every element in $Q\sqcup^{\mathrm{int}}_P P'$ admits a representative $(q,r)$ with $q\in Q$ and $r\in P'$. Moreover, the homomorphism $Q\rightarrow Q\sqcup^{\mathrm{int}}_P P'$ sends $q\in Q$ to $(q,0)$. Similarly, the homomorphism $P'\rightarrow Q\sqcup^{\mathrm{int}}_P P'$ sends $r\in P'$ to $(0,r)$.

Finally, the saturated pushout $Q\sqcup^{\mathrm{sat}}_P P'$ consists of 
\[(q,r)\in (Q\sqcup^{\mathrm{int}}_P P')^{\mathrm{gp}} \cong  Q^{\mathrm{gp}}\oplus_{P^{\mathrm{gp}}}(P')^{\mathrm{gp}}\] such that there exists an integer $m\geq 1$ and $q'\in Q$, $r'\in P'$, $p\in P^{\mathrm{gp}}$ satisfying 
\[mq=q'+u^{\mathrm{gp}}(p)  \quad \tu{and} \quad  mr+v^{\mathrm{gp}}(p)=r'.\] 
\end{construction}

\newpage
\addtocontents{toc}{\protect\setcounter{tocdepth}{2}}
\section{Pseudo-saturated maps}
In this section, we introduce a new notion on maps between monoids, called \emph{pseudo-saturated-ness}, which generalizes the notion of (quasi-)saturated-ness defined by \cite{Tsuji}. This will be helpful for our development of the Kummer \'etale site for non-fs log schemes.  

\subsection{(Quasi-)saturated and pseudo-saturated homomorphisms}

\noindent 
Let us first recall the notion of integral and saturated homomorphisms of monoids from \cite{Kato} (also see \cite{Tsuji}). We start with integral homomorphisms.

\begin{definition}
Let $u:P\rightarrow Q$ be a homomorphism of integral monoids. We say that $u$ is \emph{integral} if for any homomorphism $v:P\rightarrow P'$ of integral monoids, the pushout $Q \sqcup_P P'$ in the category of monoids is integral.
\end{definition}

\begin{proposition}[\cite{Tsuji}, Proposition I.2.3]
\begin{enumerate}
\item If $u:P\rightarrow Q$ and $v: Q\rightarrow R$ are integral homomorphisms of integral monoids, so is $v\circ u$.
\item Let $u:P\rightarrow Q$ and $v:P\rightarrow P'$ be homomorphisms of integral monoids and suppose $u$ is integral. Let $Q'$ be the pushout $Q \sqcup_P P'$ in the category of monoids. Then the canonical homomorphism $P'\rightarrow Q'$ is integral.
\end{enumerate}
\end{proposition}

\begin{proposition}[\cite{Kato}, Proposition 4.1]\label{prop:integral_Kato}
Let $u: P \rightarrow Q$ be a homomorphism of integral monoids. Then the following conditions are equivalent
\begin{enumerate}
\item $u$ is integral and injective. 
\item The induced ring homomorphism $\Z[P] \rightarrow \Z[Q]$ is flat. 
\item For any field $k$, the induced ring homomorphism $k[P] \rightarrow k[Q] $ is flat. 
\item $u$ is injective and has the following property: suppose $a_1, a_2 \in P$, $b_1, b_2 \in Q$ are elements satisfying $u(a_1)+b_1  = u(a_2)+b_2$, then there exist $a_3, a_4 \in P$ and $b \in Q$ satisfying $b_1 = u(a_3)+b$, $b_2 = u(a_4)+b$, and $a_1+a_3 = a_2+a_4$. 
\end{enumerate}
\end{proposition}

In particular, integrality of homomorphisms between integral monoids can be viewed as a certain form of flatness. For later use, we also consider a slightly more general notion of flatness, which also works for non-integral monoids.

\begin{definition}[\cite{Bhatt_dR}, Definition 4.8]\label{definition:flat}
A monoid homomorphism $u: P \rightarrow Q$ is \emph{flat} if for any homomorphism $v: P \rightarrow P'$ of monoids, the homotopy pushout $Q \sqcup_P^{\L} P'$ agrees with the naive pushout $Q \sqcup_P P'$ in the category of monoids.
\end{definition}

Next, we recall the notion of (quasi-)saturated homomorphisms.

\begin{definition}
\begin{enumerate}
\item Let $p$ be a prime number and let $u:P\rightarrow Q$ be a homomorphism of integral monoids. Let 
\[Q' = Q \sqcup_{P, [p]} P\] be the pushout of $u: P \ra Q$ along the multiplication-by-$p$ map  $[p]:P\rightarrow P$. Let $w:Q'\rightarrow Q$ be the unique homomorphism such that $w\circ v=[p]$ and $w\circ u'=u$. We say that $u$ is \emph{$p$-quasi-saturated} if $w$ is exact.
\item A homomorphism of integral monoids is \emph{quasi-saturated} if it is $p$-quasi-saturated for every prime number $p$.
\item A homomorphism of integral monoids is \emph{saturated} if it is both integral and quasi-saturated.
\end{enumerate}
\end{definition}

\begin{proposition}\label{prop: saturated morphisms}
\begin{enumerate}
\item If $u:P\rightarrow Q$ and $v:Q\rightarrow R$ are quasi-saturated (resp. saturated) homomorphisms of integral monoids, so is $v\circ u$.
\item Let $u:P\rightarrow Q$ and $v:P\rightarrow P'$ be homomorphisms of integral monoids and suppose $u$ is quasi-saturated (resp. saturated). Let $Q'$ be the pushout $Q \sqcup_P P'$ in the category of monoids. Then the canonical homomorphism $P'\rightarrow Q'$ is quasi-saturated (resp. saturated).
\item Let $u:P\rightarrow Q$ be a homomorphism of saturated monoids. Then $u$ is quasi-saturated if and only if for any homomorphism $v:P\rightarrow P'$ of saturated monoids, the pushout $Q \sqcup_P P'$ in the category of integral monoids is saturated.
\item Let $u:P\rightarrow Q$ be an integral homomorphism of saturated monoids. Then $u$ is saturated if and only if for any homomorphism $v:P\rightarrow P'$ of saturated monoids, the pushout $Q \sqcup_P P'$ in the category of monoids is saturated.
\end{enumerate}
\end{proposition}

\begin{proof}
The statements follow from Proposition I.3.6 (1), Proposition I.3.6 (2), Corollary I.3.11, and Proposition I.3.14 in \cite{Tsuji}, respectively.
\end{proof}

We have the following explicit criterion of quasi-saturated-ness.

\begin{proposition}\label{prop: criterion quasi-saturated}
Let $u:P\rightarrow Q$ be a homomorphism of saturated monoids. Then $u$ is quasi-saturated if and only if it satisfies the following condition: for any integer $n\geq 1$ and any ${x}\in P$, ${y}\in Q$ such that $u({x})\,|\, n {y}$, there exists ${x}'\in P$ such that ${x}\,|\, n{x}'$ and $u({x}')\,|\, {y}$.
\end{proposition}

\begin{proof}
This follows from \cite[Proposition I.4.1]{Tsuji}.
\end{proof}

For our purpose, we introduce a variant of quasi-saturated-ness.

\begin{definition}\label{defn: pseudo-saturated}
Let $u:P\rightarrow Q$ be a homomorphism of saturated monoids. We say that $u$ is \emph{pseudo-saturated} if there exists an integer $M\geq 1$ such that $u$ satisfies the following condition: for any integer $n\geq 1$ and any ${x}\in P$, ${y}\in Q$ such that $u({x})\,|\, n{y}$, there exists ${x}'\in P$ such that $M{x}\,|\, n{x}'$ and $u({x}')\,|\, M{y}$. In this case, we say that $u$ is \emph{pseudo-saturated with conductor $M$}.
\end{definition}

\begin{example}\label{example: pseudo-saturated}
\begin{enumerate}
\item By definition (and Proposition \ref{prop: criterion quasi-saturated}), quasi-saturated homomorphisms of saturated monoids are pseudo-saturated with conductor 1.
\item If a homomorphism $u:P\rightarrow Q$ of fs monoids is Kummer (i.e., $u$ is injective and for every ${y}\in Q$, there exists an integer $n\geq 1$ such that $n{y}\in u(P)$), then $u$ is pseudo-saturated.
\item Consider the homomorphism $u:\mathbb{Z}_{\geq 0}\rightarrow \mathbb{Z}_{\geq 0}^r$ sending 1 to $(n_1, \ldots, n_r)$ for some integers $n_1, \ldots, n_r\in \mathbb{Z}_{\geq 0}$ and assume that at least one of the $n_i$'s is non-zero. Then $u$ is pseudo-saturated with conductor $M$ where $M$ is the least common multiple of the nonzero $n_i$'s. 
\item In particular, the homomorphism $u: \Z_{\ge 0} \ra \Z_{\ge 0}^2$ sending $1 \mapsto (1, 2)$ is pseudo-saturated but not quasi-saturated. \footnote{Moreover, $u$ even satisfies the condition that the cokernel of $u^{\textup{gp}}: \Z \ra \Z^2$ is torsion-free (this condition appears in Proposition \ref{lemma: choose X_0 to be X} part (3)).} 
\end{enumerate}
\end{example}

\begin{lemma}\label{lemma: pseudo_satured_modulo_units}
Let $u:P \rightarrow Q$ be a homomorphism of saturated monoids.
\begin{enumerate}
\item Let $G$ and $H$ be subgroups of $P$ and $Q$, respectively, such that $u(G)\subset H$. Let 
\[u': P/G\rightarrow Q/H\] be the induced homomorphism. Then $u$ is quasi-saturated (resp. pseudo-saturated) if and only if $u'$ is. As an immediate corollary, $u:P\rightarrow Q$ is quasi-saturated (resp. pseudo-saturated) if and only if the corresponding homomorphism $\overline{P}\rightarrow \overline{Q}$ is.
\item Let $S$ and $T$ be submonoids of $P$ and $Q$, respectively, such that $u(S)\subset T$. Let 
\[u'': S^{-1}P\rightarrow T^{-1}Q\] be the induced homomorphism on the corresponding localizations. Then $u$ is quasi-saturated (resp. pseudo-saturated) if and only if $u''$ is.
\end{enumerate}
\end{lemma} 

\begin{proof}
\begin{enumerate}
\item This is straightforward from Proposition \ref{prop: criterion quasi-saturated} (resp. Definition \ref{defn: pseudo-saturated}).
\item We only prove the case of pseudo-saturated-ness. The proof for quasi-saturated-ness is the same. Identify $P$ and $Q$ as submonoids of $S^{-1}P$ and $T^{-1}Q$, respectively. Let $x\in S^{-1}P$ and $y\in T^{-1}Q$ such that $x\,|\,ny$ in $T^{-1}Q$ for some positive integer $n$. Pick $s\in S$ and $t\in T$ such that $x+s\in P$, $y+t\in Q$, and $x+s\,|\, n(y+t)$ in $Q$. By assumption, there exists $x'\in P$ such that $M(x+s)\,|\, nx'$ in $P$ and $x'\,|\,M(y+t)$ in $Q$. It follows that $Mx\,|\,nx'$ in $S^{-1}P$ and $x'\,|\,My$ in $T^{-1}Q$, as desired.
\end{enumerate}
\end{proof}

\begin{proposition}\label{prop: pseudo-saturated}
If $u:P\rightarrow Q$ and $v:Q\rightarrow R$ are pseudo-saturated homomorphisms of integral monoids with conductors $M$ and $N$, respectively, then $v\circ u$ is pseudo-saturated of conductor $MN$.
\end{proposition}

\begin{proof}
This is clear from the definition.
\end{proof}

The following lemma tells us that pseudo-saturated-ness is stable under (saturated) base change.

\begin{lemma}\label{lemma: pseudo-saturatedness stable under base change}
Let $u:P\rightarrow Q$ and $v: P\rightarrow P'$ be homomorphisms of fs monoids and suppose $u$ is pseudo-saturated. Let $Q' = Q\sqcup^{\textup{sat}}_P P'$  be the pushout in the category of saturated monoids. Then the canonical homomorphism $P'\rightarrow Q'$ is also pseudo-saturated.
\end{lemma}

\begin{proof}
Let $S$ be the pushout $Q\sqcup_P P'$ in the category of monoids. Then $Q'=S^{\mathrm{sat}}$ and the homomorphism $P'\rightarrow Q'= S^{\mathrm{sat}}$ decomposes as \[P'\xrightarrow[]{w}S^{\mathrm{int}}\rightarrow S^{\mathrm{sat}}.\] Since the homomorphism $S^{\mathrm{int}}\rightarrow S^{\mathrm{sat}}$ is a Kummer homomorphism of fs monoids, it is pseudo-saturated. It remains to show that $w: P'\rightarrow S^{\mathrm{int}}$ is pseudo-saturated.

Recall from Construction \ref{construction: explicit descriptions of pushouts} that $S^{\mathrm{int}}$ can be identified with the quotient of the monoid
\[
T:=\Bigl\{ (q,r)\in Q^{\mathrm{gp}}\times (P')^{\mathrm{gp}}\,\Big|\, \begin{array}{l}\textrm{there exists }q'\in Q, r'\in P', \textrm{ and }p\in P^{\mathrm{gp}}\\ \textrm{such that }q=q'+u^{\mathrm{gp}}(p), r+v^{\mathrm{gp}}(p)=r'\end{array}\Bigr\}
\]
by the congruence relation consisting of pairs $((q,r), (q',r'))\in T\times T$ such that there exists $p\in P^{\mathrm{gp}}$ satisfying $q=q'+u^{\mathrm{gp}}(p)$ and $r+v^{\mathrm{gp}}(p)=r'$. In particular, every $s\in S^{\mathrm{int}}$ admits a representative $s=(q,r)$ with $q\in Q$ and $r\in P'$. Moreover, the homomorphism $w:P'\rightarrow S^{\mathrm{int}}$ sends $r\in P'$ to $(0,r)$.

Suppose $u:P\rightarrow Q$ is pseudo-saturated with conductor $M$. We show that $w:P'\rightarrow S^{\mathrm{int}}$ is also pseudo-saturated with conductor $M$. That is, for any $r_0\in P'$ and $s\in S^{\mathrm{int}}$ such that $w(r_0)\,|\, ns$ for some integer $n\geq 1$, we have to show that there exists $r'\in P'$ such that $Mr_0\,|\,nr'$ and $w(r')\,|\, Ms$.

By assumption, there exists $s''\in S^{\mathrm{int}}$ such that $w(r_0)+s''=ns$. Pick a representative $s''=(q'', r'')$ with $q''\in Q$ and $r''\in P'$. Also choose a representative $s=(q,r)$ with $q\in Q$ and $r\in (P')^{\mathrm{gp}}$ (but we do not require that $r\in P'$). Using the explicit description of $S^{\mathrm{int}}$, the equality $w(r_0)+s''=ns$ means that there exists $p\in P^{\mathrm{gp}}$ such that 
\[q''+u^{\mathrm{gp}}(p)=nq, \tu{ and } r_0+r''=v^{\mathrm{gp}}(p)+nr.\] By Lemma \ref{lemma: Pgp and P}, there exists $p''\in P$ such that $p+np''\in P$ and we can replace $q$ with $q+u^{\mathrm{gp}}(p'')$, replace $p$ with $p+np''$, and replace $r$ with $r-v^{\mathrm{gp}}(p'')$. Hence, we may actually assume that $p\in P$. In particular, we have  $u(p)\,|\,nq$.

By the assumption that $u:P\rightarrow Q$ is pseudo-saturated with conductor $M$, we know that there exists $p'\in P$ such that $Mp\,|\, np'$ and $u(p')\,|\, Mq$. Namely, there exists $p_0\in P$ and $q_0\in Q$ such that 
\[np'=Mp+p_0, \tu{ and }  Mq=u(p')+q_0.\]
Let $r':=v(p')+Mr$. We have to check that 
\be
\item  $r'\in P'$; 
\item $Mr_0\,|\, nr'$; and 
\item $w(r')\,|\,Ms$. 
\ee For (1), clearly, $r'\in R^{\mathrm{gp}}$. Notice that 
\begin{align*}
nr' & =nv(p')+Mnr \\ & =nv(p')+Mr_0+Mr''-Mv(p)\\
&=Mr_0+Mr''+v(np'-Mp) \\ &  =Mr_0+Mr''+v(p_0)
\end{align*} which lives in $P'$. 
It follows from the saturated-ness of $P'$ that $r'\in P'$. (2) is also clear as \[ nr'=Mr_0+Mr''+v(p_0).\] Finally, for (3), since 
\[w(r')+(Mq, -v^{\mathrm{gp}}(p'))=Ms\] and 
\[(Mq, -v^{\mathrm{gp}}(p'))=(u^{\mathrm{gp}}(p')+q_0, -v^{\mathrm{gp}}(p'))=(q_0, 0)\in S^{\mathrm{int}},\]
we have $w(r')\,|\,Ms$, as desired. This completes the proof.
\end{proof}

For the rest of the section, we prove the following key technical lemma. 

\begin{lemma}\label{lemma: tech lemma}
Let $P$ and $Q$ be fs monoids and let $P'$ be a saturated monoid which is divisible. 
Let $u:P\rightarrow Q$ be a pseudo-saturated homomorphism and let $v:P\rightarrow P'$ be any monoid homomorphism. Let $S = Q\sqcup_P P'$  be the pushout in the category of monoids. Then the natural homomorphism \[S \lra  S^{\mathrm{sat}}\] is of finite type. If moreover $P'$ is torsion-free, then $S^{\mathrm{sat}}$ is almost $n$-divisible, for all $n\geq 1$.
\end{lemma}

\begin{proof}
By definition, $S\rightarrow S^{\mathrm{int}}$ is surjective. It remains to show that $S^{\mathrm{int}}\rightarrow S^{\mathrm{sat}}$ is of finite type.

Firstly, we have seen from Construction \ref{construction: explicit descriptions of pushouts} that $S^{\mathrm{int}}$ can be identified with the quotient of the monoid
\[
T:=\Bigl\{ (q,r)\in Q^{\mathrm{gp}}\times (P')^{\mathrm{gp}}\,\Big|\, \begin{array}{l}\textrm{there exists }q'\in Q, r'\in P', \textrm{ and }p\in P^{\mathrm{gp}}\\ \textrm{such that }q=q'+u^{\mathrm{gp}}(p), r+v^{\mathrm{gp}}(p)=r'\end{array}\Bigr\}
\]
by the congruence relation consisting of pairs $((q,r), (q',r'))\in T\times T$ such that there exists $p\in P^{\mathrm{gp}}$ satisfying $q=q'+u^{\mathrm{gp}}(p)$ and $r+v^{\mathrm{gp}}(p)=r'$. Also from Construction \ref{construction: explicit descriptions of pushouts}, we know that $S^{\mathrm{sat}}$ consists of $(q,r)\in (S^{\mathrm{int}})^{\mathrm{gp}}=Q^{\mathrm{gp}}\oplus_{P^{\mathrm{gp}}}(P')^{\mathrm{gp}}$ such that there exists an integer $m\geq 1$ and $q'\in Q$, $r'\in P'$, $p\in P^{\mathrm{gp}}$ satisfying $mq=q'+u^{\mathrm{gp}}(p)$ and $mr+v^{\mathrm{gp}}(p)=r'$. 

Secondly, the canonical injection $P'\rightarrow S^{\mathrm{int}}$ sending $r\mapsto (0,r)$ identifies the torsion subgroup $P'_{\mathrm{tors}}$ of $P'$ as a submonoid of $S^{\mathrm{int}}$. To show that the injection $S^{\mathrm{int}}\rightarrow S^{\mathrm{sat}}$ is of finite type, it suffices to show that the induced homomorphism \[S^{\mathrm{int}}/P'_{\mathrm{tors}}\rightarrow S^{\mathrm{sat}}/P'_{\mathrm{tors}}\] is so. In fact, we will construct a homomorphism 
\[w:(S^{\mathrm{int}}/P'_{\mathrm{tors}})\oplus U\rightarrow S^{\mathrm{sat}}/P'_{\mathrm{tors}}\] for some finitely generated monoid $U$. Later we will show that $w$ is surjective.

Suppose the pseudo-saturated homomorphism $u:P\rightarrow Q$ has conductor $M$. Consider the submonoid $U\subset Q\oplus P$ consisting of $(q,p)\in Q\oplus P$ such that 
\[Mq-u^{\mathrm{gp}}(p)\in Q.\] We claim that $U$ is finitely generated. For this, consider another submonoid $U'\subset Q\oplus P$ consisting of $(q,p)\in Q\oplus P$ such that $q+u(p)=Mq'$ for some $q'\in Q$. Also consider the submonoid $U_0\subset U$ consisting of $(q,0)$ such that $q\in Q$ and $Mq=0$. Then $U_0$ is a finitely generated torsion group and the monoid homomorphism $U\rightarrow U'$ sending $(q,p)\mapsto (Mq-p, p)$ identifies $U'$ as the quotient monoid $U/U_0$. It remains to prove that $U'$ is finitely generated. Notice that $U'$ is an exact submonoid of $Q\oplus P$ (i.e., the inclusion $U'\hookrightarrow Q\oplus P$ is exact). Indeed, by \cite[Proposition I.2.1.16 (5)]{Ogus}, it suffices to check that $(Q\oplus P)\backslash U'$ is stable under the action of $U'$. This follows from the saturated-ness of $Q$. By \cite[Theorem I.2.1.17 (2)]{Ogus}, $U'$ is fs, as desired.

Now, we consider the homomorphism $w:(S^{\mathrm{int}}/P'_{\mathrm{tors}})\oplus U\rightarrow S^{\mathrm{sat}}/P'_{\mathrm{tors}}$ defined by sending 
\[(s, (q,p)) \longmapsto s+\bigl(q, -\frac{1}{M}v(p)\bigr).\] Here ``$\frac{1}{M}v(p)$'' stands for any element $r\in P'$ such that $Mr=v(p)$. (Notice that such an element is well defined up to an element in $P'_{\mathrm{tors}}$.) To check that $w$ is well-defined, it remains to show that the element
\[\bigl(q, -\frac{1}{M}v(p)\bigr)\in Q^{\mathrm{gp}}\oplus_{P^{\mathrm{gp}}}(P')^{\mathrm{gp}}\] in fact lies in $S^{\mathrm{sat}}$. Indeed, we have
\[
M\bigl(q, -\frac{1}{M}v(p)\bigr)=\bigl(Mq, -v(p)\bigr)=\bigl(Mq-u(p),0\bigr)
\]
and $Mq-u(p)\in Q$ by assumption. Hence $M\bigl(q, -\frac{1}{M}v(p)\bigr)\in S^{\mathrm{int}}$ as desired.

To show that $S^{\mathrm{int}}\rightarrow S^{\mathrm{sat}}$ is of finite type, it remains to show that $w$ is surjective. Let $(q,r)\in S^{\mathrm{sat}}$ with $q\in Q^{\mathrm{gp}}$ and $r\in (P')^{\mathrm{gp}}$. By the explicit description given above, there exists an integer $m\geq 1$ and $q'\in Q$, $r'\in P'$, $p\in P^{\mathrm{gp}}$ satisfying \[mq=q'+u^{\mathrm{gp}}(p) \quad\tu{ and  } \quad  mr+v^{\mathrm{gp}}(p)=r'.
\] 
Using Lemma \ref{lemma: Pgp and P}, we can choose $p'\in P$ such that $p+p'\in P$ and then replace $q$ by $q+u^{\mathrm{gp}}(p')$, $r$ by $r-v^{\mathrm{gp}}(p')$, and $p$ by $p+p'$. Hence, we may actually assume $p\in P$. Consequently, $q\in Q$ by the saturated-ness of $Q$. 

Since $u(p)\,|\, mq$, there exists $p'\in P$ such that $Mp\,|\, mp'$ and $u(p')\,|\, Mq$. We have an identity
\[
(q, r)=\bigl(q, -\frac{1}{M}v(p')\bigr)+\Bigl(0, \frac{1}{Mm}\bigl(v(mp'-Mp)+Mr'\bigr)\Bigr)
\]
in $\big(Q^{\mathrm{gp}}\oplus_{P^{\mathrm{gp}}}(P')^{\mathrm{gp}}\big)/P'_{\mathrm{tors}}$. (Notice that both $\frac{1}{M}v(p')$ and $\frac{1}{Mm}\bigl(v(mp'-Mp)+Mr'\bigr)$ are well defined elements in $P'$ up to elements in $P'_{\mathrm{tors}}$.) Since $(q, p')\in U$, we conclude that $(q,r)$ lies in the image of $w$, as desired.

Finally, we show that, if $P'$ is torsion-free, then $S^{\mathrm{sat}}$ is almost $n$-divisible for all $n\geq 1$. Using the surjection $w: S^{\mathrm{int}}\oplus U\rightarrow S^{\mathrm{sat}}$, it suffices to show that $S^{\mathrm{int}}$ is almost $n$-divisible for all $n\geq 1$. This is clear using the explicit description of $S^{\mathrm{int}}$ and $n$-divisibility of $P'$.
\end{proof}

\begin{remark} Let us remark that the condition in Lemma \ref{lemma: tech lemma} on $n$-divisibility for all $n$ is necessary. Here is a counterexample otherwise: Let $P=Q=\mathbb{Z}_{\geq 0}$ and let $u: P\rightarrow Q$ be the multiplication-by-2 map. Let \[P'=\mathbb{Q}_{(2), \geq 0}:=\mathbb{Q}_{(2)}\cap \mathbb{Q}_{\geq 0}\] and let  $v:P\rightarrow P'$ the natural inclusion $\mathbb{Z}_{\geq 0}\hookrightarrow \mathbb{Q}_{(2), \geq 0}$. Then $S^{\mathrm{sat}}$ is generated over $S^{\mathrm{int}}$ by elements \[e_n:=(1, -\frac{n}{2n+1}),\] for all $n\in \mathbb{Z}_{\geq 1}$. But one cannot find a finite set of generators.
\end{remark}

\subsection{A combinatorial criterion for pseudo-saturated homomorphisms}\label{subsection: classification} \noindent 

\noindent The main goal of this subsection is to prove Theorem \ref{thm: classification of pseudo-saturated}, which provides a combinatorial criterion for pseudo-saturated homomorphisms in terms of the geometry of polyhedral cones. To state this criterion, we first define the notion of a minimal decomposition. 

\begin{definition}\label{defn: minimal decomposition}
Let $P$ be a torsion-free fs  monoid and let $N$ be an fs submonoid. Let $y \in P_{\mathbb{Q}_{\geq 0}}$. A decomposition $y=y'+x'$ with $y'\in P_{\mathbb{Q}_{\geq 0}}$ and $x'\in N_{\mathbb{Q}_{\geq 0}}$ is called a \emph{minimal decomposition} of $y$ (\emph{relative to the pair $(N, P)$}) if $y'$ is ``minimal'' in the following sense: for any $x\in N_{\mathbb{Q}_{\geq 0}}-\{0\}$, we have  $y'-x\not\in P_{\mathbb{Q}_{\geq 0}}$.
\end{definition}

Let us give some examples of minimal decompositions. 

\beg \label{example:min_decom1}
In this example let $P = \Z_{\ge 0}^{2}$ and let $N \subset P$ be the submonoid generated by $(1, 1)$. Then the decomposition 
\[
(3,1) = (2,0) + (1,1)
\]
is a minimal decomposition of $(3,1)$ relative to $(N, P)$. 
\begin{figure}[h]
\[
\includegraphics[scale=0.3]{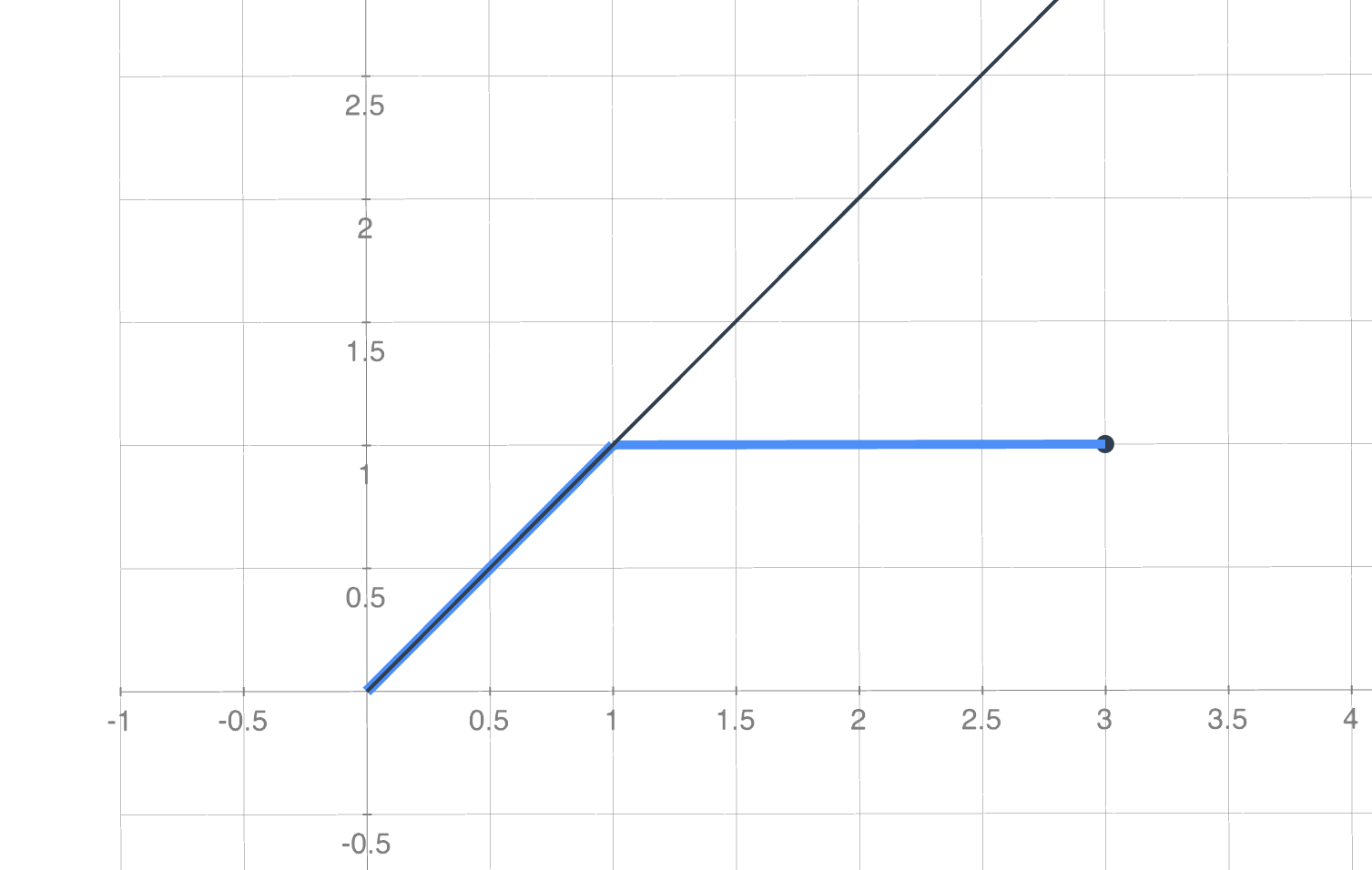}   
\]
\caption{minimal decomposition of $(3,1)$}
\label{figure:min_decom1}
\end{figure}
Note that minimal decompositions might not be unique. For example let $M \subset P =  \Z_{\ge 0}^{2}$ be the submonoid generated by $(1,2), (2, 1)$. Then 
\[
(3,1) = (5/2,0) + (1/2,1) = (1, 0) + (2, 1)
\]
give two such minimal decompositions of $(3,1)$ relative to $(M, P)$.  These examples can be represented geometrically as in Figure \ref{figure:min_decom1} and Figure \ref{figure:min_decom1_2}.

\begin{figure}[h]
\[
\includegraphics[scale=0.3]{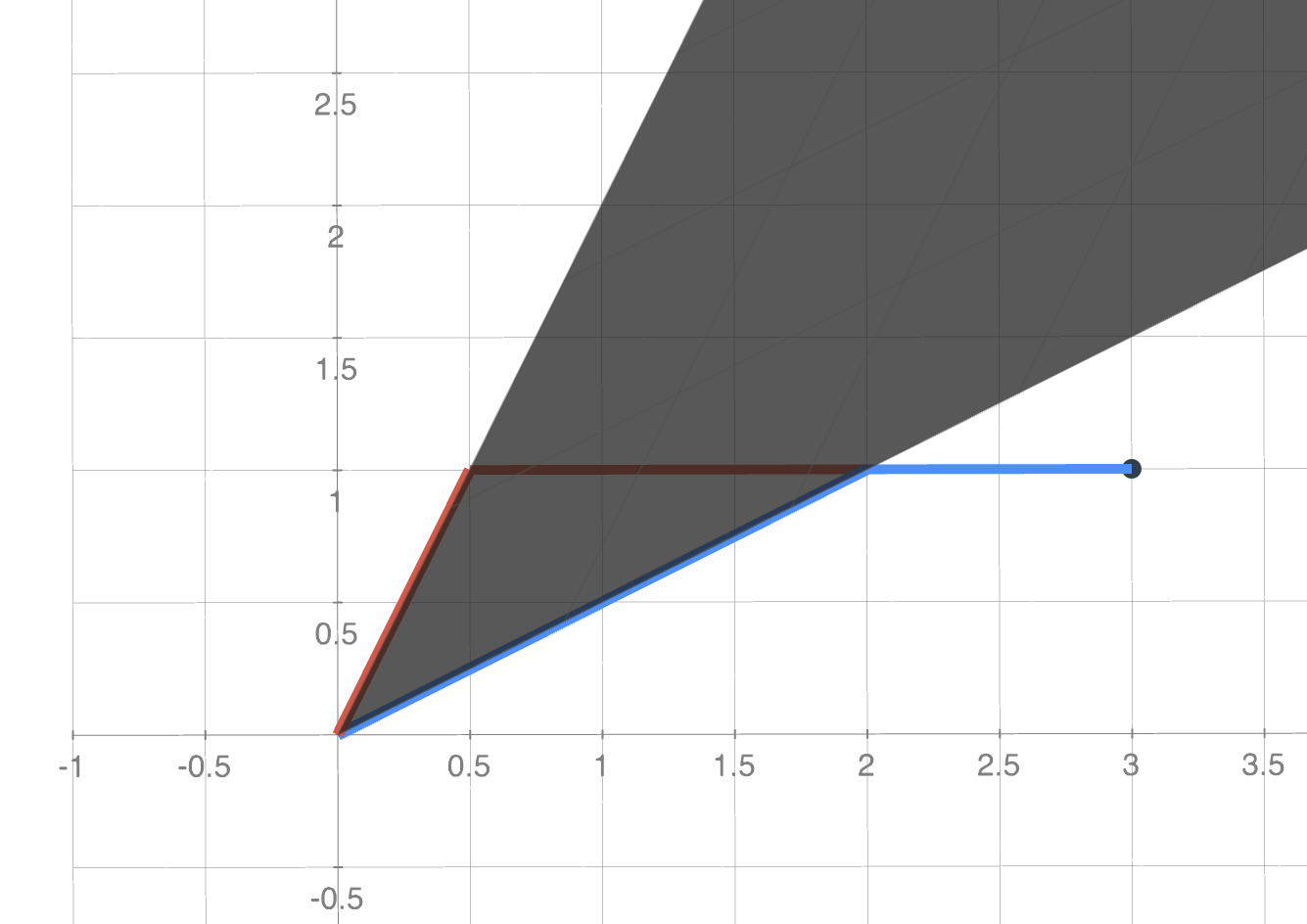}  
\]
\caption{two minimal decompositions of $(3,1)$}
\label{figure:min_decom1_2}
\end{figure}
\eeg 

\beg \label{example:min_decom2}
For another example, let $P = \Z_{\ge 0}^{3}$. Let $N \subset P$ be the submonoid generated by $(1,0,0), (0,1,1)$ and let $M \subset P$ be the submonoid given by
\[
M := \{(a,b,c) \in P | c = a+b\}. 
\]
Then 
\[
(2,2,1) = (0, 1, 0)+ (2, 1, 1)
\]
is a minimal decomposition of $(2, 2, 1)$ relative to $(N, P)$. 
\begin{figure}[h]
\[
\includegraphics[scale=0.35]{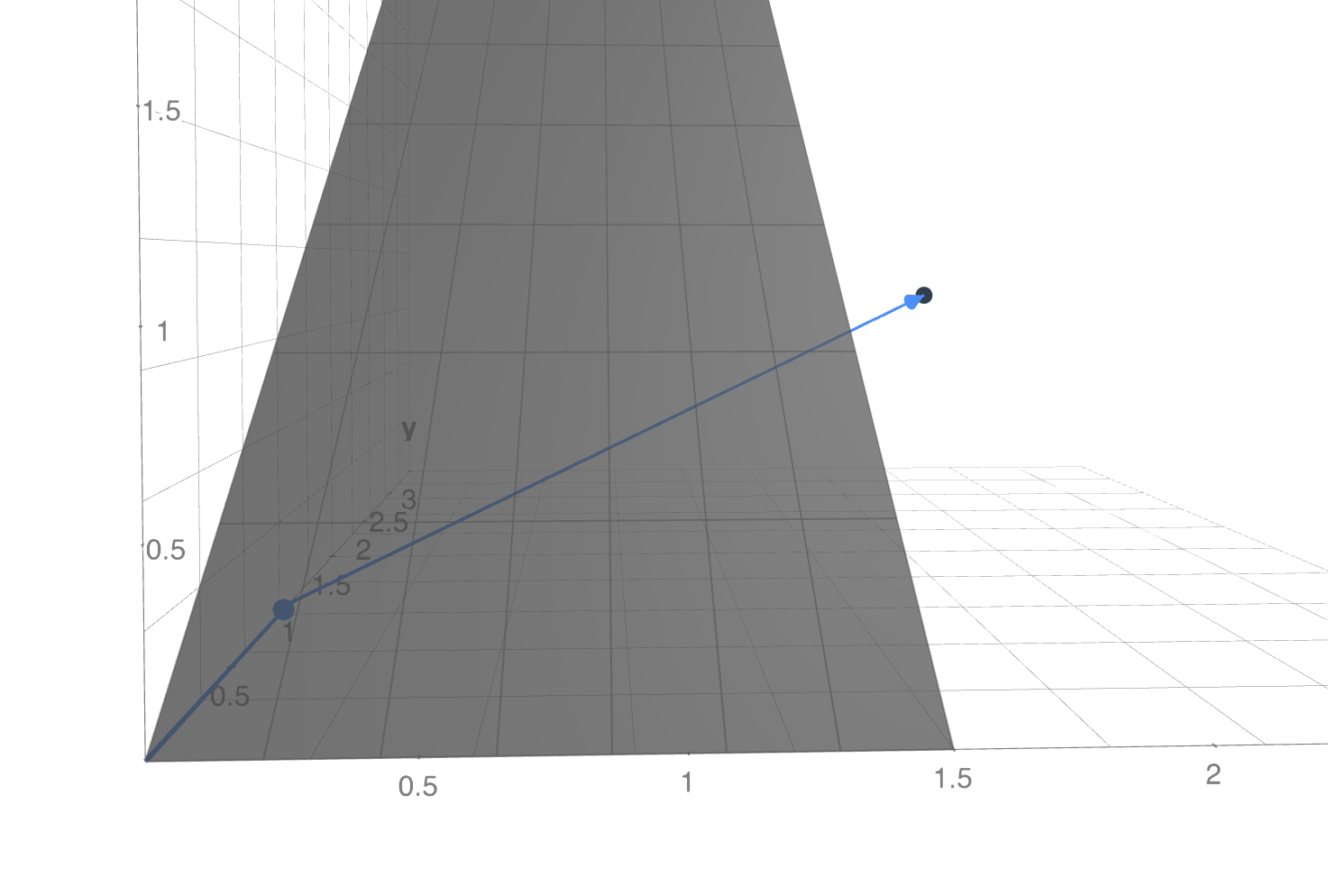}
\:\:\quad 
\]
\caption{minimal decomposition of $(2,2,1)$}
\label{figure:min_decom2_1}
\end{figure}

\noindent On the other hand, 
\begin{align*}
    (1, 1, 1) & = (0, 1, 0) + (1, 0, 1) \\
    & = (1, 0, 0) + (0, 1, 1)
\end{align*} 
give two minimal decompositions relative to $(M, P)$.  These examples are illustrated in  Figure \ref{figure:min_decom2_1} and Figure \ref{figure:min_decom2}, where $N_{\Q \ge 0}$ (resp. $M_{\Q \ge 0}$) are represented by the shaded ``rational cone'' inside the $\Q$-vector defined by $y=x$ on the left (resp. $z= x+y$ on the right).  
\begin{figure}[h]
\[
\includegraphics[scale=0.35]{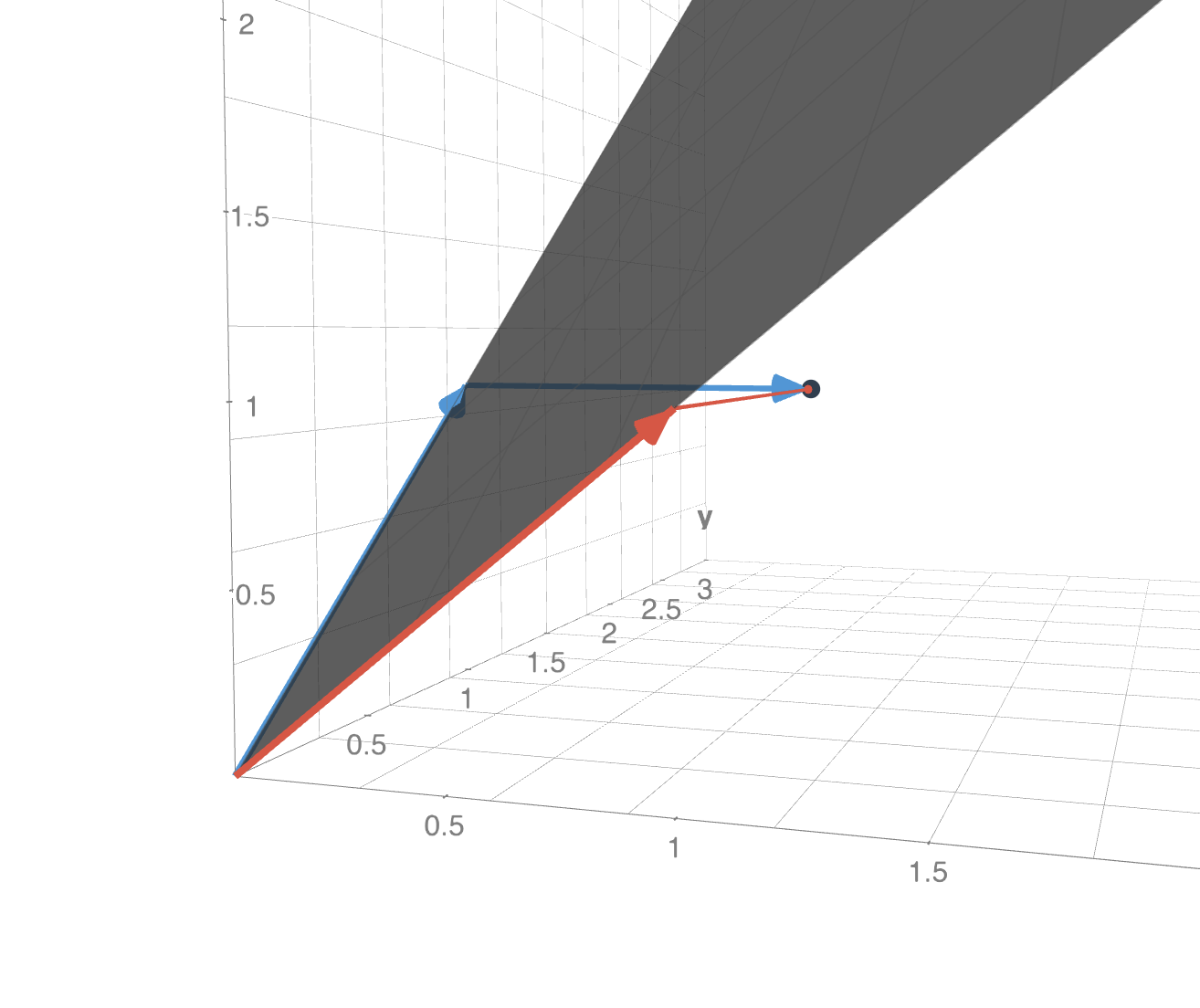} 
\:\:\quad 
\]
\caption{two minimal decompositions of $(1, 1, 1)$}
\label{figure:min_decom2}
\end{figure}
\eeg 

We are now ready to state the combinatorial criterion for pseudo-sataurated-ness.

\begin{theorem}\label{thm: classification of pseudo-saturated}
Let $P$ be a torsion-free fs  monoid and let $N$ be a toric submonoid such that $P\cap (-N)= \{0\}$. The following are equivalent:
\begin{enumerate}
\item The inclusion $u:N\hookrightarrow P$ is pseudo-saturated;
\item Every $x\in P_{\mathbb{Q}_{\geq 0}}$ admits a unique minimal decomposition \[x=g_N(x)+f_N(x)\] relative to $(N, P)$, where $g_N(x) \in N_{\Q \ge 0}$ and $f_N(x) \in P_{\Q \ge 0}$. 
\end{enumerate}
\end{theorem}

The theorem roughly says that, under mild assumptions, an injective monoid homomorphism $u:N\hookrightarrow P$ is pseudo-saturated if and only if there is a reasonable way to talk about the ``$N$-component'' (namely, $g_N(x)$) of every element $x\in P_{\mathbb{Q}_{\geq 0}}$. In particular, whether an injection $N\hookrightarrow P$ is pseudo-saturated only depends on the inclusion $N_{\mathbb{Q}_{\geq 0}}\hookrightarrow P_{\mathbb{Q}_{\geq 0}}$.

\br \label{remark:geometric_intuition} 
In Example \ref{example:min_decom1} (and similarly in Example \ref{example:min_decom2}), the map $N \ra P$ is pseudo-saturated while the map $M \ra P$ is not. Geometrically, Condition (2) in Theorem \ref{thm: classification of pseudo-saturated} means the following: if we ``parallel transport'' the rational cone $N_{\Q \ge 0}$ inside $\Q$-vector space $P_{\Q \ge 0}^{\tu{gp}}$ so that the transported cone contains the point $x$, then this cone should intersect with the edges of $P_{\Q \ge 0}$ only once. 
\er

To prove Theorem \ref{thm: classification of pseudo-saturated},  we need the following  result on rational polyhedral cones. In what follows, we will use the language of polyhedra and polytopes. Recall that a \emph{rational polyhedron} in a finite dimensional $\mathbb{Q}$-vector space is the intersection of finitely many closed half spaces. A \emph{rational polytope} is a bounded rational polyhedron. A \emph{(rational polyhedral) cone} is the intersection of finitely many closed affine half spaces. For a systematic treatment, see \cite[Section I.1]{BG}.

\begin{lemma}\label{lemma: polyhedral cone}
Let $X$ be a rational polyhedral cone and let $\mathrm{aff}(X)$ be the $\mathbb{Q}$-vector space spanned by $X$. Suppose $x\in X$ and $x'\in \mathrm{aff}(X)$ such that $x+\alpha x'\in X$ for all $\alpha\in \mathbb{Q}_{\geq 0}$. Then we must have $x'\in X$.
\end{lemma}

\begin{proof}
Write $X=H_{\lambda_1}\cap \cdots H_{\lambda_n}$ where each $\lambda_i: \mathrm{aff}(X)\rightarrow \mathbb{Q}$ is a linear form and \[H_{\lambda_i}=\{v\in \mathrm{aff}(X)\,|\, \lambda_i(v)\geq 0\}\] is a closed half space. By assumption, we have $\lambda_i(x+\alpha x')=\lambda_i(x)+\alpha\lambda_i(x')\geq 0$ for all $i$ and all $\alpha\in\mathbb{Q}_{\geq 0}$. This forces $\lambda_i(x')\geq 0$ for all $i$, which precisely means $x'\in X$.
\end{proof}

In this rest of this section we prove Theorem \ref{thm: classification of pseudo-saturated}. 
\begin{proof}[Proof of (1) $\Rightarrow$ (2) of Theorem \ref{thm: classification of pseudo-saturated}]
Firstly, for every $x \in P_{\mathbb{Q}_{\geq 0}}$, we show that such a minimal decomposition exists for $x$. Notice that $P_{\mathbb{Q}_{\geq 0}}$ is a rational polyhedral cone. Let $x-N_{\mathbb{Q}_{\geq 0}}$ denote the rational polyhedron 
\begin{equation} \label{eq:poly_Y_0}
x-N_{\mathbb{Q}_{\geq 0}} := \{x-q\,|\, q\in N_{\mathbb{Q}_{\geq 0}}\}
\end{equation} inside the $\mathbb{Q}$-vector space $P_{\mathbb{Q}}^{\mathrm{gp}}$. Consider the polyhedron 
\begin{equation} \label{eq:poly_Y} 
Y:= P_{\mathbb{Q}_{\geq 0}}\cap (x-N_{\mathbb{Q}_{\geq 0}}).
\end{equation} We claim that $Y$ is a polytope (i.e., a bounded polyhedron). Otherwise, $Y$ must contain a ray; namely, there exists $y\in Y$ and $y'\in P_{\mathbb{Q}}^{\mathrm{gp}}-\{0\}$ such that $y+\alpha y'\in Y$ for all $\alpha\in \mathbb{Q}_{\geq 0}$. By Lemma \ref{lemma: polyhedral cone}, we must have $y'\in P_{\mathbb{Q}_{\geq 0}}$. Since any two elements in $Y$ differ by an element in $N^{\mathrm{gp}}_{\mathbb{Q}}$, we must have $y'\in N^{\mathrm{gp}}_{\mathbb{Q}}$. Write $y=x-q$ for some $q\in N_{\mathbb{Q}_{\geq 0}}$. We know that $q-\alpha y'\in N_{\mathbb{Q}_{\geq 0}}$ for all $\alpha\in\mathbb{Q}_{\geq 0}$. By Lemma \ref{lemma: polyhedral cone} once again, we must have $-y'\in N_{\mathbb{Q}_{\geq 0}}$. Therefore, $y'\in P_{\mathbb{Q}_{\geq 0}}\cap (-N_{\mathbb{Q}_{\geq 0}})$. However, this means $my'\in P\cap (-N)$ for some positive integer $m$, a contradiction. Therefore, $Y$ is indeed a polytope. 

Now, there exists a linear form $\lambda: N^{\mathrm{gp}}_{\mathbb{Q}}\rightarrow \mathbb{Q}$ such that $\lambda(q)\geq 0$ for all $q\in N_{\mathbb{Q}_{\geq 0}}$ and \[\{q\in N_{\mathbb{Q}_{\geq 0}}\,|\,\lambda(q)=0\}=\{0\}.\]
By \cite[Corollary I.2.2.7]{Ogus}, there exists a basis $e_1, \ldots, e_d$ of the $\mathbb{Q}$-vector space $N^{\mathrm{gp}}_{\mathbb{Q}}$ such that 
\[N\subset \bigoplus_{i=1}^d\mathbb{Z}_{\geq 0}e_i.\] Using this basis, we can simply take the linear form $\lambda$ defined by $\lambda(e_i)=1$ for all $i$. Consider the polytope 
\begin{equation} \label{eq:poly_Y'} Y'=x-Y:=\{x-y\,|\, y\in Y\}\subset N_{\mathbb{Q}_{\geq 0}}.
\end{equation} Since $Y'$ is bounded, $\lambda$ takes maximal value at some $q_0\in Y'$. We first check that 
\[x=(x-q_0)+q_0
\]is a valid minimal decomposition of $x$. Indeed, suppose that $x-q_0-q\in P_{\mathbb{Q}_{\geq 0}}$ for some $q\in N_{\mathbb{Q}_{\geq 0}}-\{0\}$, then $q_0+q\in Y'$ and $\lambda(q_0+q)=\lambda(q_0)+\lambda(q)>\lambda(q_0)$, contradicting the maximality of $\lambda(q_0)$. This finishes the proof of existence.

It remains to prove uniqueness. Suppose that 
\[x=p'+q'=\tilde{p}'+\tilde{q}'\] are two different minimal decompositions. Without loss of generality, we assume $\tilde{q}'\neq 0$. By scalar multiplications, we may assume that $x, p', \tilde{p}' \in P$ and $q', \tilde{q}'\in N$. Let $M$ be the conductor of the pseudo-saturated homomorphism $u:N\rightarrow P$. Pick $m\gg M$ (to be determined later) and write $\epsilon:=M/m\in \mathbb{Q}_{>0}$. Next consider 
\begin{equation} \label{eq:def_q_bar}
\bar{q}:=q'+(m-1)\tilde{q}'.
\end{equation} Then clearly $\bar{q}\,|\,mx$ in $P$. By pseudo-saturated-ness, there exists $q_0\in N$ such that 
\begin{align} \label{eq:M_q_bar_divides} 
M\bar{q}\, &|\,mq_0   \quad \; \tu{in } N, \:\:  \tu{ and } \\
q_0\,& |\,Mx  \quad \tu{ in } P.  \label{eq:def_q_0_divides}
\end{align} 
By (\ref{eq:M_q_bar_divides}) and (\ref{eq:def_q_bar}), we know that $M (m-1) \sq q' | m q_0$. Now, by the construction of $\epsilon$, we have 
\begin{equation} \label{eq:epsilon_def}
M(m-1)\tilde{q}'=m(M-\epsilon)\tilde{q}'.
\end{equation} 
There, we obtain that \[m(M-\epsilon)\tilde{q}'\,|\,mq_0\] in $N$, which in turn implies that 
\[(M-\epsilon)\tilde{q}'\,|\, q_0\] in $N_{\mathbb{Q}\geq 0}$. Now, using the basis $e_1, ..., e_d$ of $N_{\Q}^{\tu{gp}}$ constructed above, we can write 
\[\tilde{q}'=a_1e_1+\cdots+a_de_d\] for some $a_1, \ldots, a_d\in \mathbb{Z}_{\geq 0}$. Choose a sufficiently large $m$ such that $\frac{1}{\epsilon}>a_i$ for all $i$. Since $N\subset \oplus_{i=1}^d\mathbb{Z}_{\geq 0}e_i$, we must have $M\tilde{q}'\,|\, q_0$.

Write $q_0=M\tilde{q}'+q_1$ for some $q_1\in N_{\mathbb{Q}_{\geq 0}}$. The condition $q_0\,|\,Mx$ implies that $(\tilde{q}'+\frac{1}{M}q_1)\,|\,x$ in $P_{\mathbb{Q}_{\geq 0}}$; namely, we have $\frac{1}{M}q_1\,|\,\tilde{p}'$. By assumption, $\tilde{p}'-q\not\in P_{\mathbb{Q}_{\geq 0}}$ for all $q\in N_{\mathbb{Q}_{\geq 0}}-\{0\}$. This forces $q_1=0$; i.e., $q_0=M\tilde{q}'$. However, we have 
\[M(q'+(m-1)\tilde{q}')\,| \,mq_0=Mm\tilde{q}'.\] This implies that $Mq'\,|\,M\tilde{q}'$ in $N$. Namely, $q'\,|\,\tilde{q}'$ in $N$, a contradiction.
\end{proof}

\begin{proof}[Proof of (2) $\Rightarrow$ (1) of Theorem \ref{thm: classification of pseudo-saturated}]
For every $x, x'\in P_{\mathbb{Q}_{\geq 0}}$, it follows from the definition that \[g_N(x)+g_N(x')\,|\,g_N(x+x')
\] in $N_{\mathbb{Q}_{\geq 0}}$. In particular, if $x_1\,|\, x_2$ in $P_{\mathbb{Q}_{\geq 0}}$, then $g_N(x_1)\,|\, g_N(x_2)$ in $N_{\mathbb{Q}_{\geq 0}}$. It is also clear from the definition that $f_N(\alpha x)=\alpha f_N(x)$ and $g_N(\alpha x)=\alpha g_N(x)$ for all $x\in P_{\mathbb{Q}_{\geq 0}}$ and $\alpha\in \mathbb{Q}_{\geq 0}$.
The key claim is the following. 

\vspace*{0.2cm}
\noindent \textbf{Claim}. There exists integer $M\geq 1$ such that $g_N(x)\in \frac{1}{M}N$ for all $x\in P$.
\vspace*{0.2cm}

Now let us prove the claim. 
Choose a basis $h_1, \ldots, h_m$ of the $\mathbb{Q}$-vector space $P_{\mathbb{Q}}^{\mathrm{gp}}$ such that $P\subset \oplus_{j=1}^m\mathbb{Z}h_j$. We can associate every element $x\in P_{\mathbb{Q}_{\geq 0}}$ with a coordinate $(a_1, \ldots, a_m)$ in $\mathbb{Q}^m$ so that $x=\sum_{j=1}^m a_jh_j$. Notice that elements in $P$ has integral coordinates.
Consider the boundary $\partial P_{\mathbb{Q}_{\geq 0}}$ of the rational polyhedral cone $P_{\mathbb{Q}_{\geq 0}}$. It can be uniquely written as a finite union of facets of $P_{\mathbb{Q}_{\geq 0}}$:
\[\partial P_{\mathbb{Q}_{\geq 0}}=\cup_{i\in I} P_i\]
where each $P_i$ is a rational polyhedral cone of dimension equal to $\mathrm{rk}_{\mathbb{Z}}(P^{\mathrm{gp}})-1=m-1$. By the definition of facet, we have $P_i=H_i\cap P_{\mathbb{Q}_{\geq 0}}$ for some hyperplane $H_i$. On the other hand, $N_{\mathbb{Q}}^{\mathrm{gp}}$ is an affine subspace of $P^{\mathrm{gp}}_{\mathbb{Q}}$. Hence, $N_{\mathbb{Q}}^{\mathrm{gp}}$ is the intersection of $m'$ hyperplanes $K_1, \ldots, K_{m'}$ where $m'=m-\mathrm{rk}_{\mathbb{Z}}(N^{\mathrm{gp}})$. If we express the boundary of $N_{\mathbb{Q}_{\geq 0}}$ as the finite union of its facets: 
\[\partial N_{\mathbb{Q}_{\geq 0}}=\cup_{l\in L} N_l,\]
then each $N_l$ can be identified with $J_l\cap N_{\mathbb{Q}_{\geq 0}}$ where $J_l$ is a hyperplane in $P_{\mathbb{Q}}^{\mathrm{gp}}$. 

We fix these hyperplanes $\{H_i\}_{i\in I}$, $\{K_j\}_{j=1}^{m'}$, $\{J_l\}_{l\in L}$ from now on. We relabel these planes as $\Lambda_1, \ldots, \Lambda_s$ with $s=|I|+m'+|L|$. Using the coordinates system we chose earlier, we can write \[\Lambda_i=\{(b_1, \ldots, b_m)\in \mathbb{Q}^m\,|\,c_{i,1}b_1+\ldots+c_{i,m}b_m=0\}\] with $c_{i,j}\in \mathbb{Z}$.

Now, suppose that $x\in P$. Consider the polytopes \[Y=P_{\mathbb{Q}_{\geq 0}}\cap (x-N_{\mathbb{Q}_{\geq 0}}) \quad \tu{and } \:\: \: Y'=x-Y
\] from (\ref{eq:poly_Y}) and (\ref{eq:poly_Y'}) in the proof of (1) $\Rightarrow$ (2) above. In the proof, we constructed a linear form $\lambda: N_{\mathbb{Q}}^{\mathrm{gp}}\rightarrow \mathbb{Q}$ such that every point in $Y'$ which achieves maximal value of $\lambda$ is a valid candidate for $g_N(x)$. By assumption, there is only one such point. In other words, $\lambda$ achieves maximal value at a unique point in $Y'$. This implies that $g_N(x)$ is a vertex of the polytope $Y'$.

Let $(a_1, \ldots, a_m)\in \mathbb{Z}^m$ be the coordinate of $x$ in the coordinate system we chose above. A vertex of $Y'$ must be the intersection of $m$ mutually non-parallel hyperplanes among $\{x+H_i\}_{i\in I}$, $\{K_j\}_{j=1}^{m'}$, $\{J_l\}_{l\in L}$. Using the fact that all $a_i$'s are integers, it follows from basic linear algebra that the denominator of the coordinate of such a vertex is bounded by the maximal value of $|\det T|$ where $T$ runs through all possible $m$-by-$m$ matrices whose entries are all selected from the set \[\{c_{i,j}\,|\, i=1, \ldots, s, j=1, \ldots, m\}.\] The claim thus follows.

Back to the proof. Suppose $s \in N$ and $t \in P$ such that $s\,|\,nt$ for some integer $n\geq 1$. We have $g_N(s)=s$ and $g_N(s)\,|\,ng_N(t)$. Take $s':=Mg_N(t)$. By the discussion above, we know that $s'\in N$. It is clear that $Ms\,| ns'$ in $N$ and $s'\,|\, Mt$ in $P$, as desired.
\end{proof}

\begin{remark}\label{remark: f_N}
Keep the notation as in Theorem \ref{thm: classification of pseudo-saturated}. It follows from the definition that, for every $x\in P_{\mathbb{Q}_{\geq 0}}$ and $q\in N_{\mathbb{Q}_{\geq 0}}$, we must have $f_N(x)=f_N(x+q)$.
\end{remark}

\begin{corollary}\label{cor: pseudo-saturated implies exact}
Let $u:N\hookrightarrow P$ be an injective homomorphism where $N$ is toric, $P$ is fs and torsion-free such that $P\cap (-N)= \{0\}$. If $u$ is pseudo-saturated, then $u$ must be exact.
\end{corollary}

\begin{proof}
Suppose $u$ is not exact; namely, $N^{\mathrm{gp}}\cap P\neq N$. Pick $x\in P\cap N^{\mathrm{gp}}$ such that $x\not\in N$. By the saturated-ness of $N$, we also have $x \not\in N_{\mathbb{Q}_{\geq 0}}$. Write $x =f_N(x)+g_N(x)$ as in Theorem \ref{thm: classification of pseudo-saturated}. Notice that $f_N(x)\neq 0$. 

On the other hand, since $x\in N^{\mathrm{gp}}$, we can write $x=q_1-q_2$ for some $q_1, q_2\in N$. Consider $x':=x+q_2$. By Remark \ref{remark: f_N}, we have $f_N(x')=f_N(x)\neq 0$. However, $x'=q_1\in N$ which implies that $f_N(x')=0$, a contradiction.
\end{proof}

\begin{remark}\label{remark: exact but not pseudo-saturated}
The converse of Corollary \ref{cor: pseudo-saturated implies exact} is false. For example, let $P=\mathbb{Z}_{\geq 0}^3$ and $M=\{(a,b,c)\in \mathbb{Z}_{\geq 0}^3\,|\, a+b=c\}$. Then the inclusion $u: M\hookrightarrow P$ is exact. But $u$ is not pseudo-saturated as $(1,1,1)=(1,0,0)+(0,1,1)=(0,1,0)+(1,0,1)$ are two different minimal decompositions (see Example \ref{example:min_decom2}). 
\end{remark}

As an application, we obtain the following corollary which roughly says that a pseudo-saturated homomorphism (with some extra conditions) becomes quasi-saturated after a suitable finite base change.

\begin{corollary}\label{corollary: pseudo-sat becomes quasi-sat after finite base change}
Let $P$ be a torsion-free fs  monoid and let $N$ be a toric submonoid such that $P\cap (-N)= \{0\}$. Suppose the inclusion $u:N\hookrightarrow P$ is pseudo-saturated. Moreover, we assume that the cokernel of $u^{\mathrm{gp}}$ is torsion-free. Then there exists a positive integer $m$ such that the canonical homomorphism $\frac{1}{m}N\rightarrow P\sqcup^{\mathrm{sat}}_N \frac{1}{m}N$ is quasi-saturated.
\end{corollary}

\begin{proof}
Since $u$ is pseudo-saturated, by the proof of Theorem \ref{thm: classification of pseudo-saturated}, there exists integer $M\geq 1$ such that $g_N(z)\in \frac{1}{M}N$ for all $z \in P$. We claim that the inclusion \[\frac{1}{M}N\rightarrow P\sqcup^{\mathrm{sat}}_N \frac{1}{M}N\] is quasi-saturated.

Let $x\in \frac{1}{M}N$ and $y\in P\sqcup^{\mathrm{sat}}_N \frac{1}{M}N$ such that $x\,|\, my$ for some positive integer $m$. Using the assumption that the cokernel of $u^{\mathrm{gp}}$ is torsion-free, one checks that the canonical homomorphism $P\sqcup^{\mathrm{sat}}_N \frac{1}{M}N\rightarrow P_{\mathbb{Q}_{\geq 0}}$ is injective. (Indeed, since all monoids here are integral, it suffices to prove the injectivity after passing to group envelopes. Namely, the homomorphism $P^{\mathrm{gp}}\oplus_{N^{\mathrm{gp}}}\frac{1}{M}N^{\mathrm{gp}}\rightarrow P^{\mathrm{gp}}_{\mathbb{Q}}$ is injective. But this is clear as $P^{\mathrm{gp}}\cong N^{\mathrm{gp}}\oplus \mathbb{Z}^r$ for some $r\geq 0$.) Express $y=f_N(y)+g_N(y)$ with $f_N(y)\in P_{\mathbb{Q}_{\geq 0}}$ and $g_N(y)\in N_{\mathbb{Q}_{\geq 0}}$ as in Theorem \ref{thm: classification of pseudo-saturated}. By assumption $g_N(y)\in \frac{1}{M}N$. Take $x'=g_N(y)$. It follows from the definition of minimal expression that $x\,|\,mx'$ in $N_{\mathbb{Q}_{\geq 0}}$, hence also in $\frac{1}{M}N$. On the other hand, \[f_N(y)=y-g_N(y)\in P_{\mathbb{Q}_{\geq 0}}\cap (P\sqcup^{\mathrm{sat}}_N \frac{1}{M}N)^{\mathrm{gp}}=P_{\mathbb{Q}_{\geq 0}}\cap \big(P^{\mathrm{gp}}\oplus_{N^{\mathrm{gp}}}\frac{1}{M}N^{\mathrm{gp}}\big)\] 
where the intersections are taken inside $P^{\mathrm{gp}}_{\mathbb{Q}}$. We claim that 
\[P_{\mathbb{Q}_{\geq 0}}\cap \big(P^{\mathrm{gp}}\oplus_{N^{\mathrm{gp}}}\frac{1}{M}N^{\mathrm{gp}}\big)=P\sqcup^{\mathrm{sat}}_N \frac{1}{M}N.\]
Indeed, it is clear that the right hand side is contained in the left hand side. For the other direction, suppose $(q,r)\in P^{\mathrm{gp}}\oplus_{N^{\mathrm{gp}}}\frac{1}{M}N^{\mathrm{gp}}$ is an element satisfying $q+r\in P_{\mathbb{Q}_{\geq 0}}$. Then 
\[z:=Mq+Mr\in P_{\mathbb{Q}_{\geq 0}}\cap P^{\mathrm{gp}}=P. \] 
This means that $M(q,r)=(Mq, Mr)=(z,0)\in P\sqcup^{\mathrm{int}}_N\frac{1}{M}N$, and hence $(q,r)\in P\sqcup^{sat}_N\frac{1}{M}N$. This finishes the proof of the claim.

Consequently, $f_N(y)\in P\sqcup^{sat}_N\frac{1}{M}N$ and thus $x'\,|\, y$ in $P\sqcup^{\mathrm{sat}}_N \frac{1}{M}N$, as desired.
\end{proof}

Let us demonstrate the phenomenon in Corollary \ref{corollary: pseudo-sat becomes quasi-sat after finite base change} with an example. 

\begin{example}
Let $N=\mathbb{Z}_{\geq 0}$ and $P=\mathbb{Z}_{\geq 0}^2$ and consider the injective homomorphism $u:N\rightarrow P$ sending $x$ to $(x, 2x)$. One checks that $u$ is pseudo-saturated with conductor 2. We can verify that the canonical homomorphism \[v:\frac{1}{2}N\rightarrow P\sqcup^{\mathrm{sat}}_N \frac{1}{2}N\] is quasi-saturated.
Using the explicit description at the end of Section \ref{subsection: conventions on monoids}, we have
\[P\sqcup_N^{\mathrm{int}} \frac{1}{2}N=\{(a,b,\frac{c}{2})\,|\, a,b,c\in \mathbb{Z}_{\geq 0}\}/\sim\]
where $(a,b,\frac{c}{2})\sim (a', b', \frac{c'}{2})$ if there exists $d\in \mathbb{Z}$ such that $a+d=a'$, $b+2d=b'$, and $c=c'+2d$. Choosing a representative from each equivalence class, we can express $P\sqcup_N^{\mathrm{int}} \frac{1}{2}N$ as
\[\{(a,0,\frac{c}{2})\,|\, a,c\in \mathbb{Z}_{\geq 0}\} \cup  \{(a,1,\frac{c}{2})\,|\, a\in \mathbb{Z}_{\geq 1}, c\in \mathbb{Z}_{\geq 0}\} \cup \{(0,b,\frac{c}{2})\,|\, b\in \mathbb{Z}_{\geq 1}, c\in \mathbb{Z}_{\geq 0}\}.\]
Taking saturation, we can express $P\sqcup_N^{\mathrm{sat}} \frac{1}{2}N$
\[\{(a,0,\frac{c}{2})\,|\, a,c\in \mathbb{Z}_{\geq 0}\} \cup  \{(a,1,\frac{c}{2})\,|\, a\in \mathbb{Z}_{\geq 1}, c\in \mathbb{Z}_{\geq -1}\} \cup \{(0,b,\frac{c}{2})\,|\, b\in \mathbb{Z}_{\geq 1}, c\in \mathbb{Z}_{\geq 0}\}.\]
Notice that the homomorphism $v:\frac{1}{2}N\rightarrow P\sqcup^{\mathrm{sat}}_N \frac{1}{2}N$ sends $\frac{c}{2}\mapsto (0,0,\frac{c}{2})$. Let $x=\frac{t}{2}\in \frac{1}{2}N$ and $y\in P\sqcup^{\mathrm{sat}}_N \frac{1}{2}N$ such that $v(x)\,|\,my$ for some $m\in \mathbb{Z}_{\geq 1}$. We have to show that there exists $x'\in \frac{1}{2}N$ such that $x\,|\, mx'$ in $\frac{1}{2}N$ and $v(x')\,|\, y$ in $P\sqcup^{\mathrm{sat}}_N\frac{1}{2}N$.
\be
\item[Case 1:] $y=(a,0,\frac{c}{2})$ with $a,c\in \mathbb{Z}_{\geq 0}$.\\
By assumption, $(0,0,\frac{t}{2})\,|\,(ma, 0, \frac{mc}{2})$. This implies $t\leq mc$. Take $x'=\frac{c}{2}$. One checks that \[v(x')=(0,0,\frac{c}{2})\,|\, y=(a,0,\frac{c}{2}).\]
\item[Case 2:] $y=(a,1,\frac{c}{2})$ with $a\in \mathbb{Z}_{\geq 1}$ and $c\in \mathbb{Z}_{\geq -1}$.\\ 
By assumption, $(0,0, \frac{t}{2})\,|\, (ma, m, \frac{mc}{2})$. We split into two more cases:
\be 
\item 
$m$ is even. In this case, we have $(ma, m, \frac{mc}{2})\sim(ma-\frac{m}{2}, 0, \frac{mc}{2}+\frac{m}{2})$. Hence, $\frac{t}{2}\leq \frac{mc}{2}+\frac{m}{2}$. Take $x'=\frac{c+1}{2}$. One checks that 
\[v(x')=(0,0,\frac{c+1}{2})\,|\, y=(a,1,\frac{c}{2})\] as $(a,1,-\frac{1}{2})\in P\sqcup^{\mathrm{sat}}_N \frac{1}{2}N$.
\item 
$m$ is odd. In this case, we have $(ma, m, \frac{mc}{2})\sim(ma-\frac{m-1}{2}, 1, \frac{mc}{2}+\frac{m-1}{2})$. This further implies that  $\frac{t}{2}\leq \frac{mc}{2}+\frac{m-1}{2}+\frac{1}{2}$; namely, $t\leq m(c+1)$. Take $x'=\frac{c+1}{2}$. One checks that \[v(x')=(0,0,\frac{c+1}{2})\,|\, y=(a,1,\frac{c}{2})\] as $(a,1,-\frac{1}{2})\in P\sqcup^{\mathrm{sat}}_N \frac{1}{2}N$.
\ee 
\item[Case 3:] $y=(0,b,\frac{c}{2})$ with $b\in \mathbb{Z}_{\geq 1}$ and $c\in \mathbb{Z}_{\geq 0}$. \\
By assumption, we have  $(0,0,\frac{t}{2})\,|\,(0,mb,\frac{mc}{2})$. This implies that $t\leq mc$. Take $x'=\frac{c}{2}$. One checks that 
\[v(x')=(0,0,\frac{c}{2})\,|\, y=(0,b,\frac{c}{2}).\]
\ee 
\end{example}

\newpage
\section{Log adic spaces}  \noindent

In this section, we  recall the notion of log adic spaces, mostly following \cite{DLLZ}. For an adic space $X$, we let $X_{\ett}$ denote the category of adic spaces \'etale over $X$. We say that $X$ is \emph{\'etale sheafy} if fiber products exist in $X_{\ett}$ (in particular, $X_{\ett}$ is a site) and the \'etale structure presheaf 
\[\mO_{X_{\ett}}: U\mapsto \mO_U(U)\] is a sheaf. \'Etale-sheafiness is known for locally noetherian adic spaces and perfectoid spaces. Everything in this section works equally in the context of schemes, formal schemes, or \'etale sheafy adic spaces. For convenience, we focus on the case of \'etale sheafy adic spaces.

\begin{convention}
All Huber pairs $(R, R^+)$ in the rest of this paper are assumed to be complete.
\end{convention}

\noindent Recall the following definitions from \cite{DLLZ} (see also \cite{Kato, Ogus}).
\begin{definition} \label{def:log}
Let $X$ be an \'etale sheafy adic space.
\begin{enumerate}
 \item A \emph{pre-log structure} on $X$ is a pair $(\mM_X, \alpha)$, where $\mM_X$ is a sheaf of monoids on $X_{\ett}$ and $\alpha: \mM_X \rightarrow \mO_{X_{\ett}}$ is a morphism of sheaves of monoids.  Notice that here $\mO_{X_{\ett}}$ is equipped with the multiplicative monoid structure.
\item A pre-log structure $(\mM_X, \alpha)$ on $X$ is called a \emph{log structure} if $\alpha^{-1}(\mO_{X_{\ett}}^\times) \rightarrow \mO_{X_{\ett}}^\times$ is an isomorphism.  If, moreover, $\alpha$ induces an isomorphism $\mM_X\cong \mO_{X_{\ett}}^{\times}$, we say that the log structure is \emph{trivial}.
\item For a log structure $(\mM_X, \alpha)$ on $X$, we write 
\[\overline{\mM}_X := \mM_X / \alpha^{-1}(\mO_{X_{\ett}}^\times).\]
 \item For a pre-log structure $(\mM_X, \alpha)$ on $X$, we have the \emph{associated log structure} 
 \[ (\mM_X, \alpha)^a = (\mM_X^a, \alpha^a)\] where $\mM_X^a$ is the pushout of \[\mO_{X_{\ett}}^\times \leftarrow \alpha^{-1}(\mO_{X_{\ett}}^\times) \rightarrow \mM_X\] in the category of sheaves of monoids on $X_{\ett}$, and $ \alpha^a:  \mM_X^a \rightarrow \mO_{X_{\ett}}$ is the morphism canonically induced by the natural map $\mO_{X_{\ett}}^\times \rightarrow \mO_{X_{\ett}}$ and the structure morphism $\alpha: \mM_X \rightarrow \mO_{X_{\ett}}$.  
 \item A \emph{pre-log adic space} (resp. \emph{log adic space}) is an \'etale sheafy adic space $X$ equipped with a pre-log structure (resp. log structure), denoted by the triple $(X, \mM_X, \alpha)$. When the context is clear, we simply write $(X, \mM_X)$ or $X$.
 \item Let $f: (Y, \mM_Y, \alpha_Y) \rightarrow (X, \mM_X, \alpha_X)$ be a morphism of log adic spaces. The \emph{pullback log structure} $f^*(\mM_X)$ on $Y$ is the log structure associated with the pre-log structure \[f^{-1}(\mM_X) \rightarrow f^{-1}(\mO_{X_{\ett}}) \rightarrow \mO_{Y_{\ett}}.\]  The morphism $f$ is called \emph{strict} if the induced morphism $f^*(\mM_X) \rightarrow \mM_Y$ is an isomorphism. The morphism is called \emph{exact} if, at each geometric point $\bar{y}$ of $Y$, the induced monoid homomorphism \[\bigl(f^*(\mM_X)\bigr)_{\bar{y}} \rightarrow \mM_{Y, {\bar{y}}}\] is exact.
\item If a strict morphism of log adic spaces is \'etale (resp. finite \'etale) on the underlying adic spaces, we say it is \emph{strictly \'etale} (resp. \emph{strictly finite \'etale}); or just \emph{\'etale} (resp. \emph{finite \'etale}) for simplicity.
\item A morphism of locally noetherian log adic spaces is \emph{locally of finite type} if the underlying morphism of adic spaces is locally of finite type in the sense of \cite[Definition 1.2.1]{Huber}.
\item A sheaf of monoids on $X_{\ett}$ is \emph{integral} (resp. \emph{saturated}) if it is a sheaf of integral (resp. saturated) monoids. We say that a pre-log adic space (resp. log adic space) $(X, \mM_X, \alpha)$ is \emph{integral} (resp. \emph{saturated}) if $\mM_X$ is.
\end{enumerate}
\end{definition}

\begin{proposition}\label{prop: integrality and saturatedness on stalks}
Let $(X, \mM_X, \alpha)$ be a log adic space. The following are equivalent.
\begin{enumerate}
\item $(X, \mM_X, \alpha)$ is integral (resp. saturated);
\item $\mM_{X, \bar{x}}$ is integral (resp. saturated) for every geometric point $\bar{x}$ of $X$;
\item $\overline{\mM}_{X, \bar{x}}$ is integral (resp. saturated) for every geometric point $\bar{x}$ of $X$.
\end{enumerate}
\end{proposition}

\begin{proof}
(1) $\iff$ (2) is \cite[Lemma 2.2.4]{DLLZ}. (2) $\iff$ (3) follows from Remark \ref{remark: int and sat} (1), \cite[Proposition I.1.3.4, Proposition I.1.3.5]{Ogus}, and the fact that $\mM_{X, \bar{x}}/\alpha^{-1}(\mO^{\times}_{X_{\ett}, \bar{x}})\cong \overline{\mM}_{X, \bar{x}}$.
\end{proof}

For our purpose, we need the notion of pseudo-saturated morphisms between log adic spaces. 

\begin{definition}
A morphism $f: (Y, \mM_Y, \alpha_Y) \rightarrow (X, \mM_X, \alpha_X)$ of saturated log adic spaces is \emph{pseudo-saturated} (resp. \emph{quasi-saturated}) if for every geometric point $\bar{y}\in Y$ and $\bar{x}=f(\bar{y})\in X$, the induces homomorphism 
\[\overline{\mM}_{X, \bar{x}}\rightarrow \overline{\mM}_{Y, \bar{y}}\] is pseudo-saturated (resp. quasi-saturated).
\end{definition}

Let us also recall the important notion of charts. 

\begin{definition} 
\begin{enumerate} 
\item Let $(X, \mM_X, \alpha)$ be a log adic space.  Let $P$ be a monoid and let $P_X$ denote the associated constant sheaf of monoids on $X_{\ett}$.  A \emph{chart} of $X$ \emph{modeled on} $P$ is a morphism of sheaves of monoids $\theta: P_X \to \mM_X$ such that $\alpha\bigl(\theta(P_X)\bigr) \subset \mO_{X_{\ett}}^+$ and such that the log structure associated with the pre-log structure $\alpha\circ\theta: P_X \rightarrow \mO_{X_{\ett}}$ is canonically isomorphic to $\mM_X$.  
\item A chart $\theta: P_X\rightarrow \mM_X$ of a log adic space $(X, \mM_X, \alpha)$ is \emph{exact} if, for every geometric point $\bar{x}$ of $X$, the induced monoid homomorphism $P\rightarrow \mM_{X,\bar{x}}$ is exact. This is equivalent to requiring the induced morphism $\cl{P}_X\rightarrow \cl \mM_X$ to be an isomorphism.
\item Let $f: (Y, \mM_Y, \alpha_Y) \to (X, \mM_X, \alpha_X)$ be a morphism of log adic spaces.  A \emph{chart} of $f$ consists of charts $\theta_X: P_X \rightarrow \mM_X$ and $\theta_Y: Q_Y \rightarrow \mM_Y$ and a monoid homomorphism $u: P \to Q$ such that the diagram        
\[ 
\begin{tikzcd} 
P_Y \arrow[d, "f^{-1}\theta_X"] \arrow[r, "u"] &  Q_Y \arrow[d, "\theta_Y"] 
\\ 
f^{-1}(\mM_X)\arrow[r] & \mM_Y
\end{tikzcd}
\]
commutes.
\end{enumerate}
\end{definition}

\begin{definition}
\begin{enumerate}
\item A log adic space $(X, \mM_X, \alpha)$ is \emph{quasi-coherent} (resp. \emph{coherent}) if, \'etale locally on $X$, it admits charts (resp. admits charts modeled on finitely generated monoids). 
\item A log adic space $(X, \mM_X, \alpha)$ is \emph{quasi-fine} (resp. \emph{quasi-fs}) if, \'etale locally on $X$, it admits charts modeled on integral monoids (resp. modeled on saturated monoids). 
\item A log adic space $(X, \mM_X, \alpha)$ is \emph{fine} (resp. \emph{fs}) if, \'etale locally on $X$, it admits charts modeled on fine monoids (resp. modeled on fs monoids). 
\end{enumerate}
\end{definition} 

\begin{proposition}
\begin{enumerate}
\item A quasi-fine log adic space is integral.
\item A quasi-fs log adic space is saturated.
\end{enumerate}
\end{proposition}

\begin{proof}
Suppose $(X, \mM_X, \alpha)$ admits a chart modeled on an integral (resp. saturated) monoid $P$. For every geometric point $\bar{x}\in X$, we have an isomorphism 
\[P/(\alpha\circ \theta)^{-1}(\mO^{\times}_{X_{\ett}, \bar{x}})\cong \overline{\mM}_{X, \bar{x}}.\] The desired statement then follows from Proposition \ref{prop: integrality and saturatedness on stalks} and \cite[Proposition I.1.3.4]{Ogus} (resp. \cite[Proposition I.1.3.5]{Ogus}).
\end{proof}

\begin{proposition}\label{prop: fs log adic space admits exact sharp fs chart}
Let $(X, \mM_X, \alpha)$ be an fs log adic space. Then, \'etale locally at every geometric point $\bar{x}$, $X$ admits a chart modeled on $\overline{\mM}_{X, \bar{x}}$. In particular, an fs log adic space, \'etale locally, always admits an exact chart modeled on a toric monoid.
\end{proposition}

\bproof 
This is \cite[Proposition 2.3.13]{DLLZ}.
\eproof 

\begin{proposition}\label{prop: pseudo-sat chart}
Let $f: (Y, \mM_Y, \alpha_Y) \to (X, \mM_X, \alpha_X)$ be a morphism of fs log adic spaces. Suppose $X$ admits an exact chart modeled on a toric monoid $P$. (According to Proposition \ref{prop: fs log adic space admits exact sharp fs chart}, such charts always exist \'etale locally on an fs log adic space.) The following are equivalent:
\begin{enumerate}
\item $f$ is integral (resp. quasi-saturated; resp. pseudo-saturated).
\item \'Etale locally on $Y$, $f$ admits a chart modeled on an integral (resp. quasi-saturated; resp. pseudo-saturated) homomorphism $u:P\rightarrow Q$ of fs monoids.
\end{enumerate}
\end{proposition}

\begin{proof}
We first prove (1) $\Rightarrow$ (2). By \cite[Proposition 2.3.21]{DLLZ} and the proof of \cite[Proposition 2.3.22]{DLLZ}, \'etale locally at every geometric point $\bar{y}\in Y$, the morphism $f$ is modeled on an fs chart $u:P\rightarrow Q$ for some fs monoid $Q$. By taking localization of $Q$ with respect to the kernel of the surjection $Q\rightarrow \overline{\mM}_{Y, \bar{y}}$, we may assume that $\overline{Q}=\overline{\mM}_{Y, \bar{y}}$ \'etale locally on $Y$. Also notice that $P=\overline{\mM}_{X, \bar{x}}$ for $\bar{x}=f(\bar{y})$. By \cite[Proposition I.2.5]{Tsuji} (resp. Lemma \ref{lemma: pseudo_satured_modulo_units} (1)), $P\rightarrow Q$ is integral (resp. quasi-saturated; resp. pseudo-saturated) if and only if $P\rightarrow \overline{Q}=\overline{\mM}_{Y, \bar{y}}$ is. But the last homomorphism is indeed integral (resp. quasi-saturated; resp. pseudo-saturated) by assumption.

Now we prove (2) $\Rightarrow$ (1). Suppose $u:P\rightarrow Q$ is an integral (resp. quasi-saturated; resp. pseudo-saturated) chart \'etale locally around $\bar{y}\in Y$. Let $\bar{x}=f(\bar{y})$. Taking the same type of localization as above and apply \cite[Proposition I.2.7]{Tsuji} (resp. Lemma \ref{lemma: pseudo_satured_modulo_units} (2)), we may assume that $\overline{Q}=\overline{\mM}_{Y, \bar{y}}$. In this case $P=\overline{\mM}_{X, \bar{x}}\rightarrow \overline{\mM}_{Y, \bar{y}}=\overline{Q}$ must be integral (resp. quasi-saturated; resp. pseudo-saturated) by applying \cite[Proposition I.2.5]{Tsuji} (resp. Lemma \ref{lemma: pseudo_satured_modulo_units} (1)) once again.
\end{proof}

Finally, let us recall the notion of log smooth and Kummer \'etale morphisms between (noetherian) fs log adic spaces. 

\begin{definition}[\cite{DLLZ}, Definition 3.1.1, Definition 4.1.2]\label{definition: Kummer etale morphism fs case}
Let $f: Y\rightarrow X$ be a morphism between locally noetherian fs log adic spaces.  
\begin{enumerate}
\item We say that $f$ is \emph{log smooth} if, \'etale locally on $Y$ and $X$, the morphism $f$ admits an fs chart $u: P \rightarrow Q$ such that
\begin{itemize}
\item the kernel and the torsion part of the cokernel of $u^{\mathrm{gp}}: P^{\mathrm{gp}} \rightarrow Q^{\mathrm{gp}}$ are finite groups of order invertible in $\mO_X$; and
\item the morphism \[Y \rightarrow X \times_{X\langle P\rangle} X\langle Q\rangle\] of log adic spaces (induced by $f$ and $u$) is strictly \'etale.
\end{itemize}
\item  A morphism (resp. finite morphism) $f: Y \rightarrow X$ of locally noetherian fs log adic spaces is called \emph{Kummer \'etale} (resp. \emph{finite Kummer \'etale}) if it admits, \'etale locally on $X$ and $Y$ (resp. \'etale locally on $X$), an fs chart modeled on an injective homomorphism $u: P \rightarrow Q$ such that 
\begin{itemize}
\item for any $q\in Q$, there exists some integer $n\geq 1$ such that $nq\in u(P)$; 
\item the order of the finite group $Q^{\mathrm{gp}} / u^{\mathrm{gp}} (P^{\mathrm{gp}})$ is invertible in $\mO_X$; and
\item the morphism 
\[Y \rightarrow X \times_{X\langle P\rangle} X\langle Q\rangle\] of log adic spaces, induced by $f$ and $u$, is strictly \'etale (resp. strictly finite \'etale).
\end{itemize}
\item A morphism $f: Y \rightarrow X$ of locally noetherian fs log adic spaces is called \emph{standard Kummer \'etale} if $X$ admits a chart modeled on an fs monoid $P$ and \[Y=X\times_{X\langle P\rangle}X\langle Q\rangle\] for some injective homomorphism $u:P\rightarrow Q$ satisfying the conditions in (2).
\end{enumerate}
\end{definition}

We end this section with some examples of log adic spaces.   

\begin{definition}\label{example: log points}
\begin{enumerate}
\item A \emph{log point} is a log adic space whose underlying adic space is $\spa(l,l^+)$ where $l$ is a non-archimedean local field. An \emph{fs log point} is a log point which is also an fs log adic space.
\item A \emph{split log point} is a log point of the form 
\[(X, \mM_X)\cong (\spa(l,l^+), \mO_{X_{\ett}}^{\times}\oplus P_X)\] for some (necessarily) sharp monoid $P$ and such that $P$ is a chart for $(X, \mM_X)$. Sometimes, we simply write $\spa(l,l^+)_P$ for such a split log point. 
\end{enumerate}
\end{definition}

\br \label{example: log points fs}
By Proposition \ref{prop: fs log adic space admits exact sharp fs chart}, every fs log point is necessarily of the form $\spa(l,l^+)_P$ where $P=\overline{M}$ is a toric monoid.
\er 

\br If $l$ is in addition separably closed, then the \'etale topos of the point  $\spa(l,l^+)$ is equivalent to the category of sets. 
In this case, giving a log structure $\mM$ on $\spa(l,l^+)$ is equivalent to giving a homomorphism of monoids $\alpha: M\rightarrow l$ inducing an isomorphism $\alpha^{-1}(l^{\times})\xrightarrow[]{\sim} l^{\times}$ where $M=\Gamma(\spa(l,l^+), \mM)$. We often write $(\spa(l,l^+), M)$ instead of $(\spa(l,l^+), \mM)$.
\er 

\begin{example}\label{example: chart}
\begin{enumerate}
\item Let $(R, R^+)$ be a Huber pair and let $P$ be a monoid. If 
\[X=\spa(R\langle P\rangle, R^+\langle P\rangle)\] is \'etale sheafy, then the natural monoid homomorphism $P\rightarrow R\langle P\rangle$ induces a pre-log structure $P_X\rightarrow \mO_{X_{\ett}}$ on $X$, whose associated log structure is denoted by $P^{\log}$. 
\item Let $(X, \mM_X)$ be a log adic space over some affinoid \'etale sheafy adic space $\spa(R, R^+)$ and suppose that $\spa(R\langle P\rangle, R^+\langle P\rangle)$ is \'etale sheafy for some monoid $P$. Then a morphism $P_X\rightarrow \mM_X$ of sheaves of monoids is a chart if and only if the induced morphism 
\[(X, \mM_X)\rightarrow (\spa(R\langle P\rangle, R^+\langle P\rangle), P^{\log})\] is strict. 
\end{enumerate}
\end{example}

\begin{example}\label{example: X<P>}
Let $X$ be a locally noetherian adic space equipped with the trivial log structure and let $P$ be a finitely generated monoid. For every affinoid open $\spa(R, R^+)$ of $X$, the Huber pair $(R\langle P\rangle, R^+\langle P\rangle)$ is \'etale sheafy. Glueing the morphisms \[(\spa(R\langle P\rangle, R^+\langle P\rangle), P^{\log})\rightarrow \spa(R, R^+),\] we obtain a morphism $Y\rightarrow X$ of log adic spaces, which we shall denote by 
\[X\langle P\rangle\rightarrow X.\]
\end{example}


\newpage

\section{Log cotangent complexes} \label{section: log quasisyntomic site}

In this section, we review some relevant properties of log cotangent complexes, following \cite{Bhatt_dR, Olsson_log} and \cite[Section 2]{logprism}. In particular, we recall that log cotangent complexes satisfy a certain form of ``flat descent'' (cf. Subsection \ref{ss:flat_descent}). 
We also review the logarithmic analogue of the quasisyntomic site developed in \cite{logprism}. In particular, we compute log cotangent complexes of certain perfectoid and ``quasiregular semiperfectoid'' pre-log rings, the latter form a basis of the log quasisyntomic site. These computations will be helpful for the setup of the Hodge--Tate comparison in Section \ref{section:HT_primitive}, and for the study of the derived version of log $\Ainf$-cohomology in Section \ref{section:derived}.

\subsection{Pre-log rings} \noindent 
\label{ss:cotangent_complex}

\begin{definition} \label{definition:definition_pre_log_rings}
\begin{enumerate}
\item A \emph{pre-log ring} is a triple $(R, M, \alpha)$ consisting of a ring $R$, a monoid $M$ together with a monoid homomorphism $\alpha: M \rightarrow R$ from $M$ to the underlying multiplicative monoid of $R$. When the context is clear, we simply write $(R, M)$ or just $\underline{R}$ for the pre-log ring. 
\item A pre-log ring $(R, M, \alpha)$ is a \emph{log ring} if the induced homomorphism \[\alpha^{-1} (R^\times) \ra R^\times\] is an isomorphism. For a pre-log ring $(R, M, \alpha)$, we write $(R, M, \alpha)^a$ for the associated log ring. 
\item A pre-log ring $(R, M, \alpha)$ is \emph{integral} (resp., \emph{saturated}) if $M$ is integral (resp., saturated). 
\item Let $R$ be a ring and let $P$ be a monoid. We write $R [\ul P] $ for the pre-log ring $(P, R[P], e)$ where the monoid homomorphism $e: P\rightarrow R[P]$ is the canonical map. Similarly, if $(R, R^+)$ is a Huber pair, we write $R \gr{\ul P}$ for the pre-log ring $(P, R \gr{P}, e)$. 
\item A pre-log ring $(R, M)$ is \emph{perfectoid} (resp. \emph{divisible perfectoid}) if both $R$ and $M$ are perfectoid (resp. if $R$ is perfectoid and $M$ is divisible).  
\item A pre-log ring $(R, M)$ is \emph{split} if the  log structure on $X = \spec R$ associated to the constant pre-log struture $M_X \ra \mO_{X_{\ett}}$ satisfies \[\mM_X \cong \mO_{X_{\ett}}^\times \oplus M_X.\] 
\end{enumerate}
\end{definition}

\begin{example}\label{example: pre-log rings}
Let $(R, R^+)$ be an \'etale sheafy Huber pair and let $X=\spa(R, R^+)$. Let $(R^+, P, \alpha)$ be a pre-log ring. Then there is a log structure on $X$ associated with the pre-log structure $P_X\rightarrow \mO_{X_{\ett}}$ induced from $P\xrightarrow[]{\alpha} R^+\subset R$. In particular, $P$ is a chart for this log structure. 
\end{example}

\begin{notation}\label{notation:forget_to_monoid_and_algebra}
 
\begin{enumerate}
\item  Let $\textup{Alg}^\textup{prelog}$ denote the category of pre-log rings with the obvious morphisms, and write 
\[
\textup{Forget}^{\textup{Alg}^\textup{prelog}}_{\textup{Set}\times \textup{Set}}:\textup{Alg}^\textup{prelog} \lra \textup{Set}\times \textup{Set}\]
and
\[
\textup{Forget}^{\textup{Alg}^\textup{prelog}}_{\textup{Mon}\times \textup{Alg}}:\textup{Alg}^\textup{prelog} \lra  \textup{Mon}\times \textup{Alg}  
\]
for the forgetful functors sending $(R, M, \alpha) \longmapsto (M,R)$. 
\item  For a map of pre-log rings 
\[\underline{R}=(R, M_R, \alpha_R)\to \underline{S}=(S, M_S, \alpha_S)\] in $ \textup{Alg}^\textup{prelog}$, we use $\Omega^{\mathrm{log}}_{\ul S/\ul R}$ to denote the $S$-module of \emph{log differentials} (cf. \cite[Section IV.1.1]{Ogus}). There is an identification
\[ \Omega^{\mathrm{log}}_{\underline{S}/\underline{R}} \cong  (\Omega^1_{S/R}\oplus (\textup{coker}(u^{\gp})\otimes_{\mathbb{Z}} S)/\big( (d\alpha_S(m),0)-(0,m\otimes \alpha_S(m)) \big),\] where $u: M_R \ra M_S$ denotes the map on monoids. 
\item For any $i\in \Z_{\geq 0}$, we write \[\Omega^{\mathrm{log}, i}_{\underline{S}/\underline{R}}:=\wedge^i_S\, \Omega^{\mathrm{log}}_{\underline{S}/\underline{R}}.\]  We  use $\Omega^{\mathrm{log}, \bullet}_{\underline{S}/\underline{R}}$ to denote the corresponding \emph{log de Rham complex}
which is equipped with a multiplication and a descending Hodge filtration. We often drop the supscript and write $\Omega_{\ul S/\ul R}^{(\bullet)}$ for $\Omega^{\log, (\bullet)}_{\ul S/\ul R}$ when the context is clear. 
\end{enumerate}
\end{notation}

We recall a particularly convenient class of morphisms of pre-log rings called \emph{homologically log flat morphisms} introduced in \cite{logprism} (also see  \cite{Bhatt_dR}). This notion will be used in the definition of the log quasisyntomic site in Section \ref{ss: log quasisyntomic site} and the descent results within the site.

\begin{definition}  \label{definition:hom_flat}
\be 
\item A morphism $\ul R \ra \ul S$ of pre-log rings is \emph{homologically log flat} if for every morphism $\underline{R}\to \underline{S'}$ of pre-log rings, the canonical morphism
\[ \underline{S'}\sqcup^{\L}_{\underline{R}} \underline{S} \lra \underline{S'}\sqcup_{\underline{R}} \underline{S}\]
from the homotopy pushout to the naive pushout is an isomorphism.  
\item A map $\underline{R}\to\underline{S}$ of pre-log rings is \emph{homologically log faithfully flat} if it is homologically log flat and the underlying map of rings $R\to S$ is faithfully flat.
\ee 
\end{definition}

\br[{\cite[Remark 2.41]{logprism}}]
The notion of homologically log flat is different from Kato's notion of log flatness from \cite[Definition 1.10]{Kato2} (thus the unfortunate choice of terminology). For example, a map $(k[P], P) \ra (k[Q], Q)$ induced from a map of monoids $P \ra Q$ that is injective but non-integral is log flat but not homologically log flat. In particular, we may take $P \subset Q = \N^2$ to be the submonoid generated by $(2, 0), (0, 2), (1, 1)$. On the other hand, let $(k, \N)$ (resp. $(k, \N^2)$) be the pre-log ring where all nonzero elements in the monoid map to $0 \in k$. Then the map $(k, \N) \ra (k, \N^2)$ induced by the diagonal map 
\[\N \xrightarrow{1 \mapsto (1, 1)} \N^2\] is homologically log flat but not log flat.
\er 
 
\br 
The functor $\textup{Forget}^{s\textup{Alg}^\textup{prelog}}_{s\textup{Mon}\times s\textup{Alg}}$ is in fact a left Quillen functor and commutes with homotopy colimits (see \cite[Proposition 5.5]{Bhatt_dR}). Therefore, a map 
\[\underline{R} = (R, M_R) \lra \underline{S} = (S, M_S)\]  in $\tu{Alg}^{\tu{prelog}}$ is homologically log flat if and only if $R\to S$ is a flat map of rings and $M_R\to M_S$ is a flat morphism of monoids (in the sense of Definition \ref{definition:flat}).
\er 

\bl 
An integral morphism between integral monoids is flat.  In particular, if $R\to S$ is flat and $M_R\to M_S$ is an integral morphism of integral monoids, then $\underline{R}\to\underline{S}$ is homologically log flat. 
\el 
\bproof 
This follows from  \cite[Proposition 4.1]{Kato} and \cite[Proposition 4.9 \& Example 4.10]{Bhatt_dR}. 
\eproof

\subsection{Log cotangent complexes}
\noindent 

\noindent 
The log differential ${\Omega}_{\underline{S}/\underline{R}}$ and log de Rham complex ${\Omega}^\bullet_{\underline{S}/\underline{R}}$ both have derived versions, which we now recall. 


 

Let $\underline{R}\to \underline{S}$ be a map in $\textup{Alg}^\textup{prelog}$. By a \emph{projective resolution} of $\underline{S}$ as an $\underline{R}$-algebra, we mean a trivial fibration $\underline{P}_\bullet\to \underline{S}$ with $\underline{P}_\bullet$ cofibrant in $s\tu{Alg}^\tu{log}_{\underline{R}/}$. Note that such resolutions always exist. For example, the adjunction 
$$
(\tu{Free}_{s\tu{Alg}^\tu{log}_{\underline{R}/}}^{s\tu{Set}\times s\tu{Set}},\tu{Forget}^{s\tu{Alg}^\tu{log}_{\underline{R}/}}_{s\tu{Set}\times s\tu{Set}})
$$ 
provides a projective resolution, which is called the canonical free resolution.
Moreover, any two such projective resolutions are homotopy equivalent.

\bd[Gabber]
For a projective resolution $\underline{P}_\bullet\to \underline{S}$ as above, consider the simplicial $P_\bullet$-module ${\Omega}_{\underline{P}_\bullet/\underline{R}}$. The \emph{log cotangent complex} $\mathbb{L}_{\underline{S}/\underline{R}}$ is defined as 
$$ {\Omega}_{\underline{P}_\bullet/\underline{R}}\otimes_{ P_\bullet } S,$$ viewed as an object in the derived $\infty$-category $D(S)$ of $S$-modules via the Dold-Kan correspondence.  The resulting object does not depend on the choice of the projective resolution.

\ed

\bd[Bhatt]
Similarly, the \emph{derived log de Rham complex} $\mathbb{L}{\Omega}_{\underline{S}/\underline{R}}$ is defined to be the simple complex (in other words,  homotopy colimit) associated to the simplicial complex $$ n \mapsto {\Omega}^\bullet_{ \underline{P}_n/\underline{R} }. $$ Equivalently, one may obtain a bicomplex from ${\Omega}^\bullet_{ \underline{P}_\bullet/\underline{R} }$ via the Dold-Kan correspondence, and take $\mathbb{L}{\Omega}_{\underline{S}/\underline{R}}$ to be the associated total complex. 
\ed 

\br 
By construction, $\L {\Omega}_{\ul S/\ul R}$ is an $E_\infty$-$R$-algebra equipped with a separated decreasing Hodge filtration whose graded piece is given by $\mathbb{L}{\Omega}^i_{\underline{S}/\underline{R}}:=\wedge^i_{S} \mathbb{L}_{\underline{S}/\underline{R}}$. We write $\widehat{\L} {\Omega}^{(h)}_{\underline{S}/\underline{R}} $  for the completion (in other words,  homotopy limit) of this Hodge filtration.
\er

\begin{remark} The above constructions generalize to the case when $\underline{S}$ is replaced by an object $\underline{S}_\bullet$ in $s\textup{Alg}^\textup{prelog}_{\underline{R}/}$. Namely, we form the simplicial complex $n\mapsto \mathbb L {\Omega}_{\underline{S}_n/\underline{R}}$ and then take homotopy colimit.
\end{remark}

\begin{variant}[$p$-completed log cotangent complex and $p$-completed derived log de Rham complex]
For a map $\ul R \ra \ul S$  in $\tu{Alg}^{\log}$,  we denote the derived $p$-adic completion of the log cotangent complex (resp. derived log de Rham complex) by $\widehat \L_{\ul S/\ul R}$ (resp. $\widehat \L {\Omega}_{\ul S/\ul R}$). 
\end{variant}

\br 
Gabber's log cotangent complexes as defined above enjoy similar functoriality properties as the usual cotangent complexes (see below). However, even in the log smooth case it may fail to be discrete. 
\er 

Nevertheless, we have the following 
\bl \label{lemma:compare_Gabber_with_Olsson}
Let $\ul R \ra \ul S$ be a log-smooth integral morphism of integral pre-log rings, then 
$$ \L_{\ul S/\ul R} \cong {\Omega}_{\ul S/\ul R}. $$
\el 

\bproof 
This is \cite[Lemma 2.10]{logprism}.
\eproof

\addtocontents{toc}{\protect\setcounter{tocdepth}{1}}
\subsection*{Functorial properties} 

\noindent 
Let $\ul R \ra \ul S$ be a map in $\tu{Alg}^{\log}$.  We have a morphism $\L_{S/R} \ra \L_{\ul S/\ul R}$ 
(resp. $\mathbb{L}\Omega_{S/R} \to  \mathbb L {\Omega}_{\underline{S}/\underline{R}}$) from the usual cotangent complex (resp. derived de Rham complex) to the log version induced by the obvious map $\Omega_{S/R}^1\to{\Omega}^1_{\underline{S}/\underline{R}}$.

\bl \label{lemma:no_additional_log}
The maps  
$$\L_{S/R} \ra \L_{\ul S/\ul R} \quad \text{ and } \quad \mathbb{L}\Omega_{S/R} \to  \mathbb L {\Omega}_{\underline{S}/\underline{R}}$$
are isomorphisms if $M_R\to M_S$ is an isomorphism.
\el 
\bproof This is \cite[Lemma 8.22]{Olsson_log} (also see \cite[Proposition 6.5]{Bhatt_dR}). 
\eproof

\bp[Invariance under passing to log structures] \label{prop:invariance_log}
Let $f: \ul R \ra \ul S $ be a map of pre-log rings and let $f^a: \ul R^a \ra \ul S^a$ be the induced map on the associated log rings. 
\be
\item The natural map $\L_{\ul S/\ul R} \isom \L_{\ul S^a/\ul R^a}$ is an equivalence.  
\item Suppose $\ul R, \ul S \in \textup{Alg}^{\log}_{\F_p /}$ and $M_R, M_S$ are integral, then the natural map 
$$\L{\Omega}_{\ul S/\ul R} \isom \L {\Omega}_{\ul S^a/\ul R^a}$$
is an equivalence. 
\ee
\ep

\bproof 
The first claim is \cite[Theorem 8.20]{Olsson_log}, the second is \cite[Corollary 7.5]{Bhatt_dR}. 
\eproof 

\br 
The second claim fails in characteristic $0$. For a counterexample one may take $\ul R = (\Q, 0)$ and $\ul S = (\Q[t, t^{-1}], 0)$. See \cite[Example 6.15]{Bhatt_dR} for details of this computation. 
\er

The log cotangent complex (resp. derived log de Rham complex) shares similar functoriality properties as the usual one, among which the following two are crucial for us. 

\begin{theorem} \label{thm:functorial_cotangent_complex}
\noindent 
\begin{itemize} 
	\item (transitivity triangle) For a sequence $\underline{R}\to \underline{S}_1 \to \underline{S}_2$ in $\textup{Alg}^\textup{prelog}$, the canonical sequence 
	$$\mathbb{L}_{\underline{S}_1/\underline{R}}\otimes^\L_{S_1} S_2 
	\lra \mathbb{L}_{\underline{S}_2/\underline{R}} 
	\lra  \mathbb{L}_{\underline{S}_2/\underline{S}_1}$$ is a homotopy fiber sequence.

	\item (base-change) Let $\underline{R}\to \underline{S}_i$ for $i=1,2$ be two objects in $\textup{Alg}^\textup{prelog}_{\underline{R}/}$, and let $\underline{R}\to \underline{S}$ be their homotopy coproduct, which is an object in $s\textup{Alg}^\textup{prelog}_{\underline{R}/}$. Then the natural map $\mathbb{L}{\Omega}_{\underline{S}_1/\underline{R}} \ra \mathbb{L}{\Omega}_{\underline{S}/\underline{S}_2}$ induces isomorphisms 
	$$\mathbb{L}{\Omega}_{\underline{S}_1/\underline{R}}\otimes^\L_{R} S_2 \simeq \mathbb{L}{\Omega}_{\underline{S}_1/\underline{R}}\otimes^\L_{S_1} S \simeq \mathbb{L}{\Omega}_{\underline{S}/\underline{S}_2}$$ between filtered $E_\infty$-$S$-algebras.
\end{itemize}
\end{theorem}

\begin{proof} 
The first claim is \cite[Theorem 8.18]{Olsson_log}, the second is \cite[Proposition 6.12]{Bhatt_dR}.  
\end{proof}

\addtocontents{toc}{\protect\setcounter{tocdepth}{2}}
\subsection{Homologically log flat descent}  \label{ss:flat_descent} 
\noindent 

\noindent In this subsection we recall a result from \cite{logprism} on ``log fpqc descent'' of log cotangent complexes.   
Recall from Definition \ref{definition:hom_flat} that a map $\underline{R}\to\underline{S}$ is homologically log flat if for all maps $\underline{R}\to \underline{S'}$, the homotopy pushout 
$ \underline{S'}\sqcup^{\L}_{\underline{R}} \underline{S} $
agrees with the naive pushout. In particular, homologically log flat maps are closed under base change and compositions. Therefore $\textup{Alg}^\textup{prelog}$ can be upgraded to a site with covering maps given by homologically log faithfully flat maps. We refer this Grothendieck topology as the $\text{hlf}$ (= homologically log flat) topology.

\bp \label{lem-descent-lfpqc-hodge-graded-piece} Fix a base log ring $\underline{R}$. For each $i\ge 0$, the functor $\underline{S}\mapsto \wedge^i_S \mathbb{L}_{\underline{S}/\underline{R}}$ is an $\text{hlf}$ sheaf on $\textup{Alg}^\textup{prelog}_{\underline{R}/}$ with values in $D(R)$. In other words,  if $\underline{S}^{-1}\to \underline{S}^0$ is a log faithfully flat map of $\underline{R}$-log-algebras, then we have a natural isomorphism
$$\wedge^i_{S^{-1}} \mathbb{L}_{\underline{S}^{-1}/\underline{R}} \simeq \tu{Tot}( \wedge^i_{S^\bullet} \mathbb{L}_{\underline{S}^\bullet/\underline{R}} )$$ where $\underline{S}^\bullet$ is the homotopy Cech nerve (which coincides with the naive Cech nerve by the log flatness assumption) of $\underline{S}^{-1}\to \underline{S}^0$.
\ep

\bproof 
This is \cite[Proposition 2.44 \& Corollary 2.45]{logprism}. 
\eproof

\subsection{The log quasisyntomic site} \label{ss: log quasisyntomic site}
\noindent 

\noindent  We freely use notations from Subsection \ref{ss:cotangent_complex}, in particular from Notation \ref{notation:forget_to_monoid_and_algebra} and their simplicial variants. 

\bd 
We say that a map of pre-log rings $\ul A \ra \ul B$ is \emph{$p$-completely homologically log flat} (resp. \emph{faithfully flat}) if 
\be
\item $ B \otimes^\L_A A/p \cong B/p$ is concentrated in degree $0$, and 
\item the map $\ul A/p \lra \ul B/p$ is homologically log flat (resp. faithfully flat). 
\ee
\ed 

\bd 
A pre-log ring $\ul A = (A, M_A)$ is \emph{log quasisyntomic} if 
\be 
\item $A$ is $p$-complete and has bounded $p^{\infty}$-torsion, and
\item the log cotangent complex $\L_{\ul A/\Z_p} \in D (A)$ has $p$-complete Tor amplitude in $[-1, 0]$, i.e., its mod $p$ reduction $\L_{\ul A/\Z_p} \otimes^\L_A A/p$  has Tor amplitude in $[-1, 0]$. 
\ee

Let $\QSyn^{\log}$ denote the category of log quasisyntomic pre-log rings. 
\ed

\bd 
Let $\ul A$ and $\ul B$ be $p$-complete pre-log rings with bounded $p^{\infty}$-torsion. A map $\ul A \ra \ul B$ is \emph{log quasisyntomic} (resp. \emph{log quasisyntomic cover}) if 
\be 
\item $\ul A \ra \ul B $ is $p$-completely homologically log flat (resp. faithfully flat), and 
\item $\L_{\ul B /\ul A} \otimes^\L_{B} B/p$ has Tor amplitude in $[-1, 0]$. 
\ee
\ed

 Let us give some examples of  quasisyntomic pre-log rings and  quasisyntomic maps between them. We again refer the reader to \cite{logprism} for details. 

\begin{example} \label{example:log_qsyn}
\be
\item For rings equipped with the trivial pre-log structures, log quasisyntomic objects and maps are the usual quasisyntomic rings and maps defined in \cite{BMS2}. In particular, $p$-adic completions of smooth algebras over $\Z_p$ or $\mO_C$ (equipped with trivial pre-log structures) are quasisyntomic; quasiregular semiperfectoid and perfectoid rings such as
\[ \mO_C/p, \quad \mO_C\gr{T^{1/p^\infty}}, \quad \mO_C\gr{T^{1/p^\infty}}/(T) \]
are quasisyntomic. 
\item If $R$ is a $p$-complete noetherian $\Z_p$-algebra, $M$ is an fs monoid and $\Z_p \ra \ul R = (R, M)$ is log smooth (on the associated log $p$-adic formal schemes), then $\ul R \in \QSyn^{\log}$. More generally, if $\ul{R}$ is a log complete intersection over $\Z_p$ where $R$ is $p$-complete\footnote{This means that $\Z_p \ra \ul R$ factors as $\Z_p \xrightarrow{f} (\sq R, \sq M) \xrightarrow{g} \ul R$ where $g$ is an exact surjection of fs pre-log rings, $\ker (\sq R \ra R)$ is ($p$-completely) regular, and $f$ is log smooth (on the associated log $p$-adic formal schemes).}, then $\ul R \in \QSyn^{\log}$ by Lemma \ref{lemma:compare_Gabber_with_Olsson} (see \cite[Subsection 3.1]{logprism}).  
\item The pre-log ring 
\[\Big( \alpha:  \N[\frac{1}{p}]\oplus_\N \N[\frac{1}{p}] \lra  \mO_C \gr{X^{1/p^\infty}, Y^{1/p^\infty}}/(X-Y)\Big)\] 
(equipped with the pre-log structure $\alpha: (a, 0) \mapsto X^{a}, (0, b) \mapsto Y^{b}$) is quasisyntomic. 
\ee 
\end{example}


\bp \label{lemma:composition_pushout_QSyn} Assume that the underlying rings of all the relevant pre-log rings are $p$-complete with bounded $p^{\infty}$-torsion. 
\be
\item Let $\ul A \ra \ul B$ be a quasisyntomic cover of pre-log rings. 
Then $\ul A \in \QSyn^{\log}$ if and only if $\ul B \in \QSyn^{\log}$.  
\item The composition $\ul A \ra \ul B \ra \ul C $ of two quasisyntomic maps is again  quasisyntomic. 
\item Let $\ul A \ra \ul B$ be a  quasisyntomic map and $\ul A \ra \ul C$ be any map of pre-log rings, then the $p$-completed pushout $\ul D = \ul B \widehat \otimes_{\ul A}^\L \ul C$ is a (discrete) pre-log ring where $D$ has bounded $p^{\infty}$-torsion, and the map 
$\ul C \lra \ul D$
is quasisyntomic. 
\ee
\ep
 
\bproof 
This is \cite[Lemma 3.5 and Lemma 3.6]{logprism}. 
\eproof

\bc 
$\QSyn^{\log, \opp}$ forms a site with the topology given by quasisyntomic covers. 
\ec 

\bproof 
This is an immediate consequence of Proposition \ref{lemma:composition_pushout_QSyn}. 
\eproof 

We call $\QSyn^{\log, \opp}$ the \emph{log quasisyntomic site}.

Let $\ul R = (R, M)$ be a perfectoid pre-log ring. We define $\QSyn^{\log, \opp}_{\ul R}$ to be the slice category of $\QSyn^{\log, \opp}$ over $\ul R$ with the quasisyntomic topology. For any object $\ul S = (S, N)$ in $\QSyn^{\log, \opp}_{\ul R}$, the derived $p$-completed complex of 
\[\widehat \L^i_{\ul S/\ul R}[-i] = (\wedge^i \L_{\ul S/\ul R}[-i])^{\wedge}\] lives in $D^{\geq 0}(S)$.

\br
Let $\ul R $ be a $p$-complete pre-log ring with bounded $p^\infty$-torsion. The small quasisyntomic site $\qSyn^{\log, \opp}_{\ul R}$ consists of quasisyntomic maps $\ul R \ra \ul S$ and is endowed with the quasisyntomic topology. For any object $\ul S \in \qSyn_{\ul R}^{\log, \opp}$, the complex $ \widehat \L^i_{\ul S/\ul R}[-i]$ lives in $D^{\geq 0}(S)$. If $\ul R$ is moreover quasisyntomic, then the objects of $\qSyn_{\ul R}^{\log, \opp}$ are quasisyntomic. 
\er


\subsection{Quasiregular semiperfectoid pre-log rings} \noindent 

\noindent 
Now we introduce a version of quasiregular semiperfectoid pre-log rings -- roughly these are given by a pair $(S, M)$ where $S$ is a quotient of perfectoid rings by a quasiregular ideal and $M$ is a quotient of a uniquely $p$-divisible monoid by a ``quasiregular submonoid".

\bd \label{definition:qrspd_log} 
A pre-log ring $\ul S = (S, M)$ is \emph{quasiregular semiperfectoid} if 
\be
\item $\ul S \in \QSyn^{\log}$ is quasisyntomic (in particular, $S$ is $p$-complete);
\item There exists a map $R \ra S$ from a perfectoid ring to $S$;
\item  $S/p$ is semiperfect (meaning that the map $\textup{Frob}: S/p \lra S/p$ is surjective); 
\item  the natural map  
\begin{align} \label{eq:semiperfect_monoid}
  M^\flat \lra M/M^\times 
\end{align}
is surjective. 
\ee
The category of quasiregular semiperfectoid pre-log rings is denoted by $\QRSP^{\log}$. (We use $\QRSP$ to denote the usual category of quasiregular semiperfectoid rings.) 
\ed 

\br
Pre-log rings satisfying conditions (3) and (4) in Definition \ref{definition:qrspd_log} are called \emph{log-semiperfectoid}. They imply that
\[
\L_{\ul S/ R}\otimes^\L_S {S/p}\in D^{\leq -1}(S/p)
\]
for any map $R\to S$. 
In particular, if $\ul S = (S, M)$ is quasiregular semiperfectoid, then $\widehat \L_{\ul S/ \Z_p}[-1]$ is $p$-completely flat.
We refer the reader to \cite{logprism} for variants of this notion (especially variants of Condition (4)). 
\er

\begin{example}
\be
\item Let $\ul R = (R, M)$ with $R \in \QRSP$ and $M$ is uniquely $p$-divisible, then $\ul R \in \QRSP^{\log}$. 
\item  The pre-log ring from Example \ref{example:log_qsyn} (3)
\[
\Big( \N[\frac{1}{p}]\oplus_\N \N[\frac{1}{p}] \lra  \mO_C \gr{X^{1/p^\infty}, Y^{1/p^\infty}}/(X-Y)\Big)
\]
is quasiregular semiperfectoid. 
\ee
\end{example}

The following results play an important role for our study of derived log  $\Ainf$-cohomology, as well as the comparison with log crystalline cohomology over $\Acris$ and the comparison with (a Frobenius twisted version of) log prismatic cohomology.

\bl \label{lemma:log_BMS_4.25}
Let $\ul S$ be a pre-log ring and assume $S$ is $p$-complete with bounded $p^{\infty}$-torsion and $\ul S/p$ is log-semiperfect. Then $\ul S \in \QRSP^{\log}$ if and only if there exists a map $R \ra \ul S$ from a perfectoid $R$ (equipped with the trivial pre-log structure) such that $\L_{\ul{S}/R} \otimes^\L_{S} S/p \in D(S/p)$ has Tor amplitude concentrated in degree $[-1]$. In this case, the latter condition holds for any choice of $R \ra \ul S$ with $R$ being perfectoid. 
\el  
 
\bproof 
This is \cite[Lemma 3.15]{logprism}.
\eproof

Quasiregular semiperfectoid pre-log rings form a basis for the log quasisyntomic site. More precisely, we have 

\bl
Let $\ul A \ra \ul B$ and $\ul A \ra \ul C$ be maps in $\QRSP^{\log}$,  where $\ul A \ra \ul B$ is a quasisyntomic cover.  Then their pushout exists in $\QRSP^{\log}$ and is a quasisyntomic cover of $\ul C$. In particular, the category $\QRSP^{\log, \opp}$ forms a site with the topology given by quasisyntomic covers. 
\el

\bproof 
This is \cite[Lemma 3.16]{logprism}.
\eproof

\bp  \label{prop:QRSP_forms_basis_for_QSyn}
Let $\mC$ be any presentable $\infty$-category, then the canonical restriction of sites $\QRSP^{\log, \opp} \ra \QSyn^{\log, \opp} $ induces an equivalence of sheaves valued in $\mC$
$$ \textup{Shv}_{\mC} (\QSyn^{\log, \opp}) \isom \textup{Shv}_{\mC} (\QRSP^{\log, \opp}).$$ 
As in the previous subsection, for a perfectoid pre-log ring $\ul R$, write $\QRSP^{\log, \opp}_{\ul R}$ for the slice category of $\QRSP^{\log, \opp}$ over $\ul R$. Then we have an equivalence 
$$ \textup{Shv}_{\mC} (\QSyn^{\log, \opp}_{\ul R}) \isom \textup{Shv}_{\mC} (\QRSP^{\log, \opp}_{\ul R}).$$ 
\ep 
 
\bproof 
This is \cite[Corollary 3.19]{logprism}.
\eproof

\subsection{Log cotangent complexes for (semi)perfectoid pre-log rings}  \noindent 

\noindent Fix a perfectoid pre-log base ring $\ul R$, the functor sending $\ul S  \in \QSyn^{\log, \opp}_{\ul R}$ to the derived $p$-complete complex $\widehat \L^i_{\ul S/\ul R} = (\wedge^i \L_{\ul S/\ul R})^{\wedge}   \in D (R)$ forms a sheaf in the quasisyntomic topology by Proposition \ref{lem-descent-lfpqc-hodge-graded-piece}. By Proposition \ref{prop:QRSP_forms_basis_for_QSyn}, to understand this functor it suffices, to some extent, to understand it on $\ul S \in \QRSP^{\log, \opp}_{\ul R}$. This often turns out to be convenient since for each $\ul S \in \QRSP^{\log, \opp}_{\ul R}$, $\widehat \L^i_{\ul S/\ul R}[-i]$ is a discrete (derived $p$-complete) object in $D(R)$ that lives in degree 0. In this subsection, we collect some relevant examples.

To start, let us recall the notion of Frobenius for pre-log rings. 
\be
\item 
Suppose that $\ul R = (M \ra R)$ is a pre-log ring over $\F_p$, then we have the \emph{(absolute) Frobenius map} $F_{\ul R}: \ul R \ra \ul R $ given by multiplication by  $p$ on $M$ and the usual $p$-power Frobenius on $R$.
\item For a map $f: \ul R \ra \ul S$ in $\textup{Alg}^{\log}_{\F_p/}$, denote  by $\ul S^{(1)}$  the homotopy pushout of $f$ along the Frobenius  $F_{\ul R}$ and by 
$$F =  F_{\ul S/\ul R}: \ul S^{(1)} \lra \ul S$$
the relative Frobenius. 
\item  A map $f$ in $\text{Alg}^{\log}_{\F_p/}$ as above is \emph{relatively perfect} if the relative Frobenius $F_{\ul S/\ul R}$ is an isomorphism.  A map $f$ in $\text{Alg}^{\log}_{\Z_p/}$ (or $\text{Alg}^{\log}_{\Z/p^n/}$) is \emph{relatively perfect mod $p$} if $f \otimes^\L \F_p$ is relatively perfect.  In particular, if $M_R$ and $M_S$ are uniquely $p$-divisible and $R, S$ are perfect $\F_p$-algebras, then $f: \ul R \ra \ul S$ is relatively perfect. 
\ee

\bl  \label{lemma:cotangent_for_perfect_maps}
If $f: \ul R \ra \ul S$ a relatively perfect map of pre-log $\F_p$-algebras, then $\L_{\ul S/\ul R} = 0 $ and the derived log de Rham complex is simply given by $\L {\Omega}_{\ul S/\ul R} \cong S$.   
\el 

\bproof 
This is \cite[Corollary 7.11]{Bhatt_dR}. 
\eproof 

The following observations will be useful for us. 

\bc \label{cor:cotangent_for_perfectoid}
Let $f: \ul R \ra \ul S$ be a map of pre-log rings. Suppose that $M_R$ and $M_S$ are uniquely $p$-divisible; $R$ and $S$ are (integral) perfectoid.  Then the $p$-completed log cotangent complex $ \widehat \L_{\ul S/\ul R} = 0 $. More generally, if both $\ul R$ and $\ul S$ are perfectoid pre-log rings, then  $ \widehat \L_{\ul S/\ul R} = 0 $. 
\ec 

\bproof 
This is \cite[Corollary 2.25]{logprism} 
\eproof

\bc \label{lemma:perfectoid_monoid}
 Let $\ul R = (P \ra R)$ be a perfectoid pre-log ring and let $\Z_p = (0 \ra \Z_p)$ be the trivial pre-log ring. Then the natural map 
$$\widehat \L_{R/\Z_p} \isom \widehat \L_{\ul R/\Z_p}$$
is an isomorphism. In particular, we have $\widehat \L_{\ul R/\Z_p}[-1] \{-1\} \cong R$. 
\ec 

\bproof 
This follows from Corollary \ref{cor:cotangent_for_perfectoid} by considering the exact triangle for $\Z_p \ra R \ra \ul R$. 
\eproof

The following examples are computed in \cite{logprism}. 

\bl \label{example:log_cotangent_complex}
\be
\item If we write $\ul \mO_C = (\Q_{\ge 0} \xrightarrow{\alpha \mapsto p^{\alpha}} \mO_C)$, then  
$$\widehat \L_{\ul \mO_C/\Z_p} [-1] \cong \mO_C \{1\} \cong \xi \Ainf/\xi^2 \Ainf. $$
 \item 
Let $\ul R \ra \ul S$ be natural inclusion of pre-log rings where 
\bi 
\item  $\ul S = (\mO_C \gr{t^{1/p^\infty}}, \N[\frac{1}{p}])$  with pre-log structure $\frac{a}{p^b} \mapsto T^{\frac{a}{p^b}}$,  and 
\item  $\ul R = (\N \xrightarrow{1 \mapsto T} \mO_C \gr{T})$.
\ei 
The $p$-completed log cotangent complex $\widehat \L_{\ul S/\ul R}$ is 
$$
\widehat \L_{\ul S/\ul R} \cong \mO_C \gr{T^{1/p^\infty}} [1].  
$$
In fact, $\widehat \L_{\ul S/\ul R}$ has Tor amplitude concentrated in $[-1]$. 
\footnote{
Note that although $\widehat \L_{\ul S/\ul R}$ is isomorphic to the usual cotangent complex $\widehat \L_{S/R}$ as a complex over $\mO_C \gr{T^{1/p^\infty}}$, the natural map $\widehat \L_{S/R} \ra \widehat \L_{\ul S/\ul R}$ is not an isomorphism. In fact this map is the base change of $\widehat \L_{R/\mO_C}[1] \ra \widehat \L_{\ul R/\mO_C}[1]$, given by $\mO_C \gr{T} dT \hookrightarrow \mO_C \gr{T} \textup{ dlog} T$ after a degree shift.  
}
\item Let $M^{(1)}:= \N[\frac{1}{p}] \oplus_{\N} \N[\frac{1}{p}]$ be the pushout of the natural inclusion $\N \ra \N[\frac{1}{p}]$ along itself. Let  
 \[S = \mO_C \gr{M^{(1)}} \cong \mO_C \gr{X^{1/p^\infty}, Y^{1/p^\infty}}/(X - Y)
 \]
 and $\ul S = (S, M^{(1)})$, where the pre-log structure is given by $(\alpha,\beta) \mapsto X^\alpha Y^\beta$. Let $\ul R = (\mO_C \gr{T}, \N)$ be as in the previous example and regard $\ul S$ as pre-log $\ul R$-algebra by the natural map $\N \ra M^{(1)}$.  Then the $p$-completed log cotangent complex $\widehat \L_{\ul S/\mO_C}$ is given by 
$\widehat \L_{\ul S/\mO_C} \cong  S [1],$ and both $\widehat \L_{\ul S/\Z_p}$ and $\widehat \L_{\ul S/\ul R}$ are concentrated in degree $-1$.
\ee 
\el

\bproof 
See \cite[Section 2]{logprism}. 
\eproof


\newpage

\section{Admissible smoothness}\label{section: admissible smoothness}

Our main goal in this section is to introduce a notion of ``sufficiently smoothness'' for certain saturated log adic spaces (resp. saturated log $p$-adic formal schemes) which admit non-finitely generated charts. 

 
\subsection{Divisible log point} We start with the notion of a divisible log point. 

\begin{definition}
A \emph{divisible log point} is a split log point of the form \[\spa(l,l^+)_{N_{\infty}}\] (cf. Definition \ref{example: log points} (2))  where $l$ is a perfectoid field and $N_{\infty}$ is a divisible and saturated (necessarily sharp) monoid. 
\end{definition}

\br The terminology reflects the requirement that $N_\infty$ is uniquely $n$-divisible for all $n\geq 1$ but not the split condition. A slightly better choice of terminology is probably split divisible log point. We decide to stick to the shorter version since we are only interested in the split case. 
\er 

For the rest of the section, we will fix a divisible log point $\spa(l,l^+)_{N_{\infty}}$.

\subsection{Admissibly smooth log adic spaces} 
We would like to consider locally noetherian quasi-fs log adic spaces $X$ that are, in a reasonable sense, log smooth over $\spa(l,l^+)_{N_{\infty}}$, and then study the Kummer \'etale topology on $X$.  It is tempting\footnote{For log schemes instead of log adic spaces, it is possible to define log smoothness by the usual lifting criterion (``formally log smooth'') and some finiteness condition called locally fts (=locally finitely presented up to saturation), see \cite{Dori_Yao} for details. In \textit{loc.cit.}, it is shown that this notion of log smoothness agrees with the requirement that it comes from the saturated base change of a log smooth map of fs log schemes. 
}
to simply require $X\rightarrow \spa(l,l^+)_{N_{\infty}}$ to be the base change from a log smooth morphism of locally noetherian fs log adic space; in other words, $X$ fits into a Cartesian diagram 
\[ 
\begin{tikzcd} 
X \arrow[d] \arrow[r] & X_0 \arrow[d] 
\\ 
\spa(l, l^+)_{N_{\infty}}\arrow[r] & \spa(l, l^+)_N
\end{tikzcd}
\]
in the category of saturated log adic spaces, where
\begin{itemize} 
\item $\spa(l, l^+)_N$ is a split log point modeled on a toric monoid $N$;
\item $X_0\rightarrow  \spa(l,l^+)_N$ is a log smooth morphism of locally noetherian fs log adic spaces; and 
\item the morphism $\spa(l, l^+)_{N_{\infty}}\rightarrow \spa(l, l^+)_N$ is modeled on a chart $N\rightarrow N_{\infty}$.
\end{itemize} 
However, if we start with such morphisms $X_0\rightarrow \spa(l,l^+)_N$ and $\spa(l,l^+)_{N_{\infty}}\rightarrow \spa(l,l^+)_N$, we do not know whether the resulting fiber product \[X_0\times_{\spa(l,l^+)_N}\spa(l,l^+)_{N_{\infty}}\] exists in general (the issue is that due to the saturation step, the resulting pre-adic space might not be sheafy). 
Thanks to Lemma \ref{lemma: tech lemma}, such a fiber product always exists if we further require $X_0\rightarrow \spa(l,l^+)_N$ to be pseudo-saturated.

\begin{proposition}\label{prop:adm_sm_fiber_product_exist}
Let $\spa(l, l^+)_N$ be a split log point modeled on a toric monoid $N$ and let $X_0\rightarrow  \spa(l, l^+)_N$ be a pseudo-saturated morphism of locally noetherian fs log adic spaces. Let $\spa(l,l^+)_{N_{\infty}}\rightarrow \spa(l,l^+)_N$ be a morphism of log points which is identity on the underlying adic space and which admits a chart $N\rightarrow N_{\infty}$. Then the fiber product \[ X:=X_0\times_{\spa(l,l^+)_N}\spa(l,l^+)_{N_{\infty}}\] exists in the category of saturated log adic spaces. Moreover, $X$ is locally noetherian and quasi-fs. If $X_0\rightarrow  \spa(l, l^+)_N$ is locally of finite type, then so is $X\rightarrow \spa(l, l^+)_{N_{\infty}}$.
\end{proposition}

\begin{proof}
It suffices to construct the fiber product $X_0\times_{\spa(l,l^+)_N}\spa(l,l^+)_{N_{\infty}}$ \'etale locally on $X_0$. 
By Proposition \ref{prop: pseudo-sat chart}, we may assume that $X_0=\spa(R, R^+)$ is affinoid and that the map $X_0\rightarrow \spa(l,l^+)_N$ admits a chart modeled on a pseudo-saturated homomorphism $u:N\rightarrow P$ of fs monoids. Let $S$ be the pushout $P\sqcup_N N_{\infty}$ in the category of monoids. Then the (unsaturated) fiber product \[X'_0:=X_0\times_{\spa(l,l^+)_N}\spa(l,l^+)_{N_{\infty}}\] exists in the category of log adic spaces and has underlying adic space $X_0$. It is equipped with the log structure associated with the pre-log structure 
\[S=P\sqcup_N N_{\infty}\rightarrow \mO_{X_{0,\ett}}\] induced by $P\rightarrow\mO_{X_{0,\ett}}$ and $N_{\infty}\rightarrow l\rightarrow \mO_{X_{0,\ett}}$. 

Now consider the completed tensor product of Tate Huber pairs (see \cite[Proposition 5.1.5 (2)]{SW})
\[
(R', R'^+):=(R, R^+)\widehat{\otimes}_{(l\langle S\rangle, l^+\langle S\rangle)}(l\langle S^{\mathrm{sat}}\rangle, l^+\langle S^{\mathrm{sat}}\rangle).
\]
Since $S\rightarrow S^{\mathrm{sat}}$ is of finite type by Lemma \ref{lemma: tech lemma}, the Huber pair $(R', R'^+)$ is topologically of finite type over $(R, R^+)$, and hence \'etale sheafy. Let $X=\spa(R', R'^+)$ and equip it with the log structure associated with the natural pre-log structure $S^{\mathrm{sat}}\rightarrow R'$. One checks that $X$ is indeed the fiber product \[X_0\times_{\spa(l,l^+)_N}\spa(l,l^+)_{N_{\infty}}\] in the category of saturated log adic spaces.

Clearly, $X$ is locally noetherian and it is quasi-fs because it admits a chart modeled on a saturated monoid $S^{\mathrm{sat}}$. If $X_0\rightarrow \spa(l, l^+)_N$ is locally of finite type, then so is the composition  $X\rightarrow \spa(l, l^+)_{N_{\infty}}$ as $X\rightarrow X_0$ is locally of finite type.
\end{proof}

\br\label{remark:finite_type_fiber_product_formal} 
The corresponding assertions of Proposition \ref{prop:adm_sm_fiber_product_exist} hold true for log $p$-adic formal schemes over the map $\spf (l^+, N_\infty)^a \ra \spf (l^+, N)^a$ where $N$ is a sharp fs monoid and $N_\infty$ is perfect. The claim on the existence of fiber product is clear (since there is no issue on \'etale-sheafiness). The last assertion on (local) finite type follows from the same proof, again using Lemma \ref{lemma: tech lemma}. 
\er 

Proposition \ref{prop:adm_sm_fiber_product_exist} inspires us to consider only those ``log smooth'' $X$ of the following type.

\begin{definition}\label{defn: admissiblly log smooth}
\begin{enumerate}
\item A locally noetherian saturated log adic space $X$ is \emph{weakly admissibly smooth} over $\spa(l, l^+)_{N_{\infty}}$ if it fits into a Cartesian diagram
\[ 
\begin{tikzcd} 
X \arrow[d] \arrow[r] & X_0 \arrow[d] 
\\ 
\spa(l, l^+)_{N_{\infty}}\arrow[r] & \spa(l, l^+)_N
\end{tikzcd}
\]
in the category of locally noetherian saturated log adic spaces, where 
\begin{itemize}
\item $\spa(l,l^+)_N$ is a split log point modeled on a toric monoid $N$;
\item $X_0\rightarrow \spa(l, l^+)_N$ is a log smooth and pseudo-saturated morphism of locally noetherian fs log adic spaces;
\item the morphism $\spa(l, l^+)_{N_{\infty}}\rightarrow \spa(l, l^+)_N$ is modeled on a chart $N\rightarrow N_{\infty}$ and is identity on the underlying adic spaces.
\end{itemize}
We refer to $X_0\rightarrow \spa(l, l^+)_N$ as a \emph{weak finite model} of $X\rightarrow \spa(l, l^+)_{N_{\infty}}$.
\item A locally noetherian saturated log adic space $X$ is \emph{admissibly smooth} over $\spa(l, l^+)_{N_{\infty}}$ if it admits a weak finite model as above and, in addition, the map 
\[X\rightarrow \spa(l,l^+)_{N_{\infty}}\] is integral. We refer to $X_0\rightarrow \spa(l, l^+)_N$ as a \emph{finite model} of $X\rightarrow \spa(l, l^+)_{N_{\infty}}$.
\end{enumerate}
\end{definition}

We will also record an integral version of the definition above. 

\bd \label{definition:admissible_smooth_formal}
Let $\ul{l^+} = (l^+, N_\infty)$ be a split  pre-log ring with $l^+$ being $p$-complete and $N_\infty$ being divisible and saturated (cf. Definition \ref{definition:definition_pre_log_rings}).  
\begin{enumerate}
\item A saturated log $p$-adic formal scheme $\fX$ is \emph{weakly admissibly smooth} over $\ul{l^+}$ 
if it fits into a Cartesian diagram
\[ 
\begin{tikzcd} 
\fX \arrow[d] \arrow[r] & \fX_0 \arrow[d] 
\\ 
\spf (l^+, N_\infty)^a \arrow[r] & \spf(l^+, N)^a
\end{tikzcd}
\]
in the category of saturated log $p$-adic formal schemes, where 
\begin{itemize}
\item The map $\spf(l^+, N_\infty)^a \ra \spf(l^+, N)^a$ is modeled on a chart $N \ra N_\infty$ and is identity on the underlying formal schemes, where $\alpha_0: N \ra l^+$ is a split toric pre-log structure;
\item $X_0\rightarrow \spf(l^+, N)^a$ is a ($p$-completely) log smooth and pseudo-saturated morphism of fs log formal schemes.
\end{itemize}
In this case, the map $\fX_0\rightarrow \spf(l^+, N)^a$ is again referred to as a \emph{weak finite model} of $\fX \rightarrow \spf(l^+, N_\infty)^a$.
\item A saturated log $p$-adic formal scheme $\fX$ is \emph{admissibly smooth} over $\ul{l^+} = (l^+, N_\infty)$ if it admits a weak finite model as above and, in addition, the map \[\fX_0\rightarrow \spf(l^+, N)^a\] is integral. In this case,  the map  \[\fX_0\rightarrow \spf(l^+, N)^a\] is again referred to as a \emph{finite model} of $\fX \rightarrow \spf(l^+, N_\infty)^a$.
\end{enumerate}
\ed 

\subsection{Standard finite models}
The following crucial lemma indicates that, \'etale locally, we can always find a ``nice'' finite model for an admissibly smooth log adic space  $X\rightarrow \spa(l,l^+)_{N_{\infty}}$.

\bp \label{lemma: choose X_0 to be X}
Suppose $X$ is admissibly smooth over $\spa(l, l^+)_{N_{\infty}}$. Then, \'etale locally on $X$, there exists a finite model $X_0\rightarrow \spa(l, l^+)_N$ such that 
\begin{enumerate}
\item $N\rightarrow N_{\infty}$ is injective;
\item $X_0\rightarrow \spa(l,l^+)_N$ admits a chart modeled on an injective homomorphism of fs monoids $u:N\rightarrow P$ satisfying
\begin{enumerate}
\item $P$ is torsion-free and $P\cap (-N)=\{0\}$;
\item $u$ is quasi-saturated and integral (i.e., saturated);
\item the cokernel of $u^{\mathrm{gp}}: N^{\mathrm{gp}}\rightarrow P^{\mathrm{gp}}$ is torsion-free;
\item the induced morphism 
\[ X_0\rightarrow\spa(l,l^+)_N\times_{\spa(l\langle N\rangle, l^+\langle N\rangle)}\spa(l\langle P\rangle, l^+\langle P\rangle)\] is a composition of rational localizations and strictly finite \'etale morphisms.
\end{enumerate}
\end{enumerate}
In this case, $X_0\rightarrow \spa(l,l^+)_N$ coincides with $X\rightarrow \spa(l,l^+)_{N_{\infty}}$ on the underlying adic spaces.
\ep

\begin{proof}
We start with condition (1). If we start with an arbitrary finite model \[X_0\rightarrow \spa(l,l^+)_N,\] the image of $N\rightarrow N_{\infty}$ is a finitely generated, sharp, and integral submonoid of $N_{\infty}$. Since $N_{\infty}$ is saturated, the saturation $N'$ of the image of $N\rightarrow N_{\infty}$ is also a submonoid of $N_{\infty}$. In particular, $N'$ is toric and $N'\rightarrow N_{\infty}$ is injective. Consider the split log point associated with the composition $N'\hookrightarrow N_{\infty}\rightarrow l$ and consider the fiber product \[ X'_0:=X_0\times_{\spa(l,l^+)_N}\spa(l,l^+)_{N'}\] in the category of locally noetherian fs log adic spaces. Clearly, $X'_0\rightarrow \spa(l,l^+)_{N'}$ is log smooth by \cite[Proposition 3.1.3]{DLLZ}.
We claim that $X'_0\rightarrow \spa(l,l^+)_{N'}$ is pseudo-saturated. Indeed, if $S$ is the pushout of $P\sqcup_N N'$ in the category of monoids, then
\[
X'_0=X_0\times_{\spa(l\langle S\rangle,l^+\langle S\rangle)}\spa(l\langle S^{\mathrm{sat}}\rangle,l^+\langle S^{\mathrm{sat}}\rangle).
\]
In particular, $X'_0\rightarrow \spa(l,l^+)_{N'}$ admits a chart $N'\rightarrow S^{\mathrm{sat}}$ which is pseudo-saturated by Lemma \ref{lemma: pseudo-saturatedness stable under base change}. Hence, $X'_0\rightarrow \spa(l,l^+)_{N'}$ is pseudo-saturated by Proposition \ref{prop: pseudo-sat chart}, as desired. We point out that, along the same arguments, we have the freedom to further replace $N'$ by any fs (necessarily toric) submonoid of $N_{\infty}$ containing $N'$. Below, we always assume $N\rightarrow N_{\infty}$ is injective.


Secondly, using \cite[Proposition 3.1.4]{DLLZ}, the map $X_0\rightarrow \spa(l,l^+)_N$ admits a chart modeled on an injective homomorphism of fs monoids $u:N\rightarrow P$ such that $P$ is torsion-free, the torsion-part of the cokernel of $u^{\mathrm{gp}}:N^{\mathrm{gp}}\rightarrow P^{\mathrm{gp}}$ is a finite group whose order is invertible in $l$, and the induced morphism \[X_0\rightarrow\spa(l,l^+)_N\times_{\spa(l\langle N\rangle, l^+\langle N\rangle)}\spa(l\langle P\rangle, l^+\langle P\rangle)\] is a composition of rational localizations and strictly \'etale morphisms. We must have \[P\cap (-N)=\{0\},\]  for otherwise, we would have $l\widehat{\otimes}_{l\langle N\rangle}l\langle P\rangle=0$ which is impossible. In fact, we can further require $u$ to be pseudo-saturated and integral using the same argument as in the proof of Proposition \ref{prop: pseudo-sat chart}. More precisely, one can take a localization of $P$ so that $\overline{P}=\mM_{X_0, \bar{x}}$ at a geometric point $\bar{x}\in X_0$. This forces $u:N\rightarrow P$ to be pseudo-saturated and integral. Notice that all of the conditions above are preserved under such a localization. So far, the missing conditions are the following: (i) the cokernel of $u^{\mathrm{gp}}$ is torsion-free; (ii) pseudo-saturated-ness should be upgraded to quasi-saturated-ness. We will fix these below.

To achieve the torsion-free-ness of the cokernel of $u^{\mathrm{gp}}$, consider the monoid
\[N_P:=\{p\in P^{\mathrm{gp}}\,|\, np\in u(N) \textrm{ for some integer }n\geq 1\}\]
which is a toric monoid by \cite[Section I.1, Lemma 2]{KKMSD}. Notice that elements in $N_P$ necessarily lives in $P$ by the saturated-ness of $P$. Also notice that the injection $N\hookrightarrow N_{\infty}$ factors as 
\[N\hookrightarrow N_P\hookrightarrow N_{\infty}.\] By \cite[Lemma 3.3]{Illusie02}, we have
\[
P\sqcup_N^{\textup{sat}} N_P=P\sqcup^{\textup{sat}}_{N_P}(N_P\sqcup_N^{\textup{sat}} N_P)\cong P\sqcup^{\textup{sat}}_{N_P}(N_P\oplus G)=P\oplus G
\]
where $G\cong(N_P)^{\mathrm{gp}}/N^{\mathrm{gp}}$ is a finite group and, as notation indicates, all the pushouts are taken in the category of saturated monoids. Let \[X'_0:=X_0\times_{\spa(l,l^+)_N}\spa(l,l^+)_{N_P}.\] 
Then the map $X'_0\rightarrow \spa(l,l^+)_{N_P}$ is also a finite model of $X\rightarrow \spa(l,l^+)_{N_{\infty}}$ and it admits a chart 
\[N_P\rightarrow P\oplus G\] 
whose composition with the projection $P\oplus G\rightarrow P$ coincides with the inclusion $N_P\hookrightarrow P$. By the argument in the last paragraph of the proof of \cite[Proposition 3.1.4]{DLLZ}, we know that $N_P\rightarrow P$ is also a chart of $X'_0\rightarrow \spa(l,l^+)_{N_P}$. This chart satisfies (a) and (d) and it is pseudo-saturated by Theorem \ref{thm: classification of pseudo-saturated}. We claim that the cokernel of $u^{\mathrm{gp}}$ is indeed torsion-free. To see this, suppose $np=q$ for some $p\in P^{\mathrm{gp}}$, $q\in N_P^{\mathrm{gp}}$, and some integer $n\geq 1$. By Lemma \ref{lemma: Pgp and P}, there exists $q'\in N_P$ such that $q+nq'\in N_P$. Hence, $n(p+q')\in N_P$ which implies $p+q'\in N_P$. Consequently, $p\in N_P$, as desired.

Now, we upgrade pseudo-saturated-ness to quasi-saturated-ness. In fact, by Corollary \ref{corollary: pseudo-sat becomes quasi-sat after finite base change}, there exists an integer $m\geq 1$ such that the map \[\frac{1}{m}N\rightarrow P\sqcup^{\mathrm{sat}}_N\frac{1}{m}N\] is quasi-saturated. Consider the fiber product \[X'_0:=X_0\times_{\spa(l,l^+)_N}\spa(l,l^+)_{\frac{1}{m}N}\] in the category of locally noetherian fs log adic spaces. Then $X'_0\rightarrow \spa(l,l^+)_{\frac{1}{m}N}$ is another finite model and admits a chart modeled on the above quasi-saturated homomorphism $\frac{1}{m}N\rightarrow P\sqcup^{\mathrm{sat}}_N\frac{1}{m}N$. In particular, $X'_0\rightarrow \spa(l,l^+)_{\frac{1}{m}N}$ is quasi-saturated. Repeating the argument in the previous two paragraphs, one can find the desired finite model equipped with a chart satisfying all conditions in (1) and (2).

Finally, we verify that $X_0\rightarrow \spa(l,l^+)_N$ coincides with $X\rightarrow \spa(l,l^+)_{N_{\infty}}$ on the underlying adic spaces. Indeed, since $N\rightarrow P$ is saturated, we have 
\[P\sqcup_N N_{\infty}=P\sqcup^{\mathrm{int}}_NN_{\infty}=P\sqcup^{\mathrm{sat}}_NN_{\infty}\] by Proposition \ref{prop: saturated morphisms} (4). Hence, we have \[X=X_0\times_{\spa(l,l^+)_N}\spa(l,l^+)_{N_{\infty}}\]
where the fiber product could be taken in the category of log adic spaces, integral log adic spaces, or saturated log adic spaces. However, it is clear that $X$ and $X_0$ have the same underlying adic spaces if we take the fiber product in the category of log adic spaces.
\end{proof}

\begin{remark}
The same assertion holds for weakly admissibly smooth case except that we can no longer guarantee that $X_0\rightarrow \spa(l,l^+)_N$ coincides with $X\rightarrow \spa(l,l^+)_{N_{\infty}}$ on the underlying adic spaces.
\end{remark}

\br \label{remark:nice_finite_model_formal}
By the same proof (see Remark \ref{remark:finite_type_fiber_product_formal}), the same assertion holds for admissibly smooth log $p$-adic formal schemes over $(l^+, N_\infty)$. More precisely, if $\fX$ is admissibly smooth over $\spf (l^+, N_{\infty})^a$, then \'etale locally on $\fX$, there exists a finite model $\fX_0\rightarrow \spf(l^+, N)^a$ such that 
\begin{enumerate}
\item $N\rightarrow N_{\infty}$ is injective (in particular $N$ is toric);
\item $\fX_0\rightarrow \spf(l^+,N)$ admits a chart modeled on an injective saturated homomorphism of fs monoids $u:N\rightarrow P$ such that $P$ is torsion-free, $P\cap (-N)=\{0\}$, the cokernel of $u^{\mathrm{gp}}: N^{\mathrm{gp}}\rightarrow P^{\mathrm{gp}}$ is torsion-free, and the induced morphism 
\[ \fX_0\rightarrow\spf(l^+,N)^a\times_{\spf(l^+\gr{N}, N)^a}\spf (l^+ \gr{P}, P)^a \] is a composition of rational localizations and strictly finite \'etale morphisms.
\end{enumerate}
In this case, we have $\fX_0\rightarrow \spf(l^+, N)^a$ coincides with $\fX\rightarrow \spf(l^+, N_{\infty})^a$ on the underlying $p$-adic formal schemes;
\er 

\br \label{remark:induced_Galois_cover_dominant} 
Suppose an injective homomorphism $u: N \ra P$ of fs monoids satisfies the condition (a), (b), and (c) in Proposition \ref{lemma: choose X_0 to be X}. Then the natural map 
\[
P \sqcup_{N} (\frac{1}{n} N)=P \sqcup_{N}^{\textup{int}} (\frac{1}{n} N) = P \sqcup_{N}^{\mathrm{sat}} (\frac{1}{n} N) \lra \frac{1}{n} P
\]
must be injective. It suffices to check this after passing to the group envelops as both sides are integral monoids. But $P^{\gp} \oplus_{N^{\gp}} (\frac{1}{n} N^{\gp}) \hookrightarrow \frac{1}{n} P^{\gp}$ is indeed injective since $P^{\gp} \cong N^{\gp} \oplus \Z^{r}$ for some $r \ge 0$ by the condition that $\text{coker} (u^{\gp})$ is torsion-free. Passing to the limit, we see that the induced map 
\[P\sqcup_{N} {N_{\Q_{\ge 0}}}=P\sqcup_{N}^{\mathrm{int}} {N_{\Q_{\ge 0}}} = P\sqcup_{N}^{\textup{sat}} {N_{\Q_{\ge 0}}} \lra P_{\Q_{\ge 0}} 
\]
is again injective by a similar argument. 
\er 

\begin{lemma}\label{lemma: torsion-free pushout}
Suppose $X$ is admissibly smooth over $\spa(l, l^+)_{N_{\infty}}$. Let \[X_0\rightarrow \spa(l, l^+)_N\] be a finite model of $X$ equipped with a chart $u:N\rightarrow P$ satisfying the conditions in Proposition \ref{lemma: choose X_0 to be X}. 
\begin{enumerate}
\item For every fs submonoid $N'$ of $N_{\infty}$ containing $N$, the pushout \[P'=P\sqcup_NN'=P\sqcup^{\textup{int}}_NN'=P\sqcup^{\textup{sat}}_NN'\] is torsions-free. 
\item $X$ admits a chart modeled on a saturated monoid \[P_{\infty}= P\sqcup_NN_{\infty}=P\sqcup^{\textup{int}}_NN_{\infty}=P\sqcup^{\textup{sat}}_NN_{\infty}.\] Moreover, $P_{\infty}$ is torsion-free and almost $n$-divisible for all $n\geq 1$, and $X\rightarrow X_0$ admits a chart modeled on an injective homomorphism $P\rightarrow P_{\infty}$.
\end{enumerate}
\end{lemma}

\begin{proof}
For (1), notice that $P\sqcup^{\mathrm{sat}}_NN'$ is a submonoid of $P^{\mathrm{gp}}\oplus_{N^{\mathrm{gp}}}N'^{\mathrm{gp}}$. One sees that $P^{\mathrm{gp}}\oplus_{N^{\mathrm{gp}}}N'^{\mathrm{gp}}$ is torsion-free as the cokernel of $u^{\mathrm{gp}}: N^{\mathrm{gp}}\rightarrow P^{\mathrm{gp}}$ is torsion-free and $(N')^{\mathrm{gp}}$ is torsion-free.

For (2), it follows from Lemma \ref{lemma: tech lemma} that $P_{\infty}$ is necessarily almost $n$-divisible for all $n\geq 1$. The injectivity of $P\rightarrow P_{\infty}$ is clear because the composition \[P\rightarrow P_{\infty}\hookrightarrow (P_{\infty})^{\mathrm{gp}}=P^{\mathrm{gp}}\oplus_{N^{\mathrm{gp}}} (N_{\infty})^{\mathrm{gp}}\]
sending $p\mapsto (p,0)$ is injective. Torsion-freeness is also clear as $P_{\infty}$ is the colimit of the submonoids $P\sqcup^{\mathrm{sat}}_NN'$ which are torsion-free, where $N'$ runs through all fs submonoids of $N_{\infty}$ containing $N$.
\end{proof}

\begin{definition}\label{defn: standard finite model}
Suppose $X$ is admissibly smooth over $\spa(l,l^+)_{N_{\infty}}$ and assume that $X$ is affinoid. A \emph{standard finite model} of \[X\rightarrow \spa(l,l^+)_{N_{\infty}}\] is a finite model \[X_0\rightarrow \spa(l,l^+)_N,\] together with a chart $u:N\rightarrow P$ satisfying the conditions in Proposition \ref{lemma: choose X_0 to be X}. We also say that $X$ is modeled on a \emph{standard saturated chart} $P_{\infty}$. A standard finite model of an admissibly smooth map of log $p$-adic formal schemes $\fX \ra \spf (l^+, N_\infty)^a$ is defined similarly. 
\end{definition}

\begin{remark}
By Proposition \ref{lemma: choose X_0 to be X}, every $X$ (resp. $\fX$) that is admissibly smooth over $\spa(l,l^+)_{N_{\infty}}$ (resp. over $\spf (l^+, N_\infty)^a$) \'etale locally admits a standard finite model.
\end{remark}

\newpage 
\section{The Kummer \'etale site} \label{sec:kummer_etale} 
In this (somewhat technical) section, we develop the theory of the Kummer \'etale site of admissibly smooth log adic spaces over a divisible log point $\spa(l, l^+)_{N_{\infty}}$.

\subsection{Kummer \'etale maps}\label{subsection: the kummer etale morphisms} 

Our first goal is to generalize the notion of Kummer \'etale morphisms. Suppose that $X$ is admissibly smooth over $\spa(l, l^+)_{N_{\infty}}$. It is tempting to say that a morphism $Y\rightarrow X$ of saturated log adic spaces is Kummer \'etale if it is the base change from a Kummer \'etale morphism of locally noetherian fs log adic spaces. The following lemma justifies this intuition.

\begin{lemma}\label{lemma: ket morphism from base change}
Suppose that $X$ is admissibly smooth over $\spa(l, l^+)_{N_{\infty}}$. Let $Y_1\rightarrow X_1$ be a Kummer \'etale morphism of locally noetherian fs log adic spaces and let $f:X\rightarrow X_1$ be any morphism of log adic spaces. Then the fiber product 
\[Y:=Y_1\times_{X_1}X\] exists in the category of saturated log adic spaces. Moreover, $Y$ is also admissibly smooth over $\spa(l, l^+)_{N_{\infty}}$ and the morphism $Y\rightarrow X$ is locally of finite type. Furthermore, if $Y_1\rightarrow X_1$ is finite Kummer \'etale, then $Y\rightarrow X$ is finite. 
\end{lemma}

\begin{proof}
Firstly, we may assume that the underlying adic space of $X_1$ coincides with $X$. Indeed, we can replace $X_1$ by $X'_1$ whose underlying adic space is $X$ equipped with the log structure $\mM_{X'_1}:=f^*\mM_{X_1}$, and replace $Y_1$ by \[Y'_1:=Y_1\times_{X_1}X'_1.\]

Secondly, by Proposition \ref{lemma: choose X_0 to be X}, we may assume that, \'etale locally on $X$, the morphism 
\[ X\rightarrow \spa(l,l^+)_{N_{\infty}}\] admits a finite model $X_0\rightarrow \spa(l,l^+)_N$ such that $N\rightarrow N_{\infty}$ is injective and the underlying adic space of $X_0$ coincides with $X$. If necessary, we have the freedom to further replace $N$ by any fs submonoid $N'$ of $N_{\infty}$ containing $N$ and replace $X_0$ by \[X'_0:=X_0\times_{\spa(l,l^+)_N}\spa(l,l^+)_{N'}.\] Such $X'_0$ form a cofiltered system whose limit is $X$.

\'Etale locally, $X_1$ admits a chart modeled on a toric monoid $P$. Since $P$ is finitely generated, the morphism $P_X\rightarrow \mM_X$ factors through $h^*\mM_{X'_0}$ for some $X'_0$ as above, where $h:X\rightarrow X'_0$ is the natural morphism. Consequently, $X\rightarrow X_1$ factors as 
\[X\rightarrow X'_0\rightarrow X_1.
\] Furthermore, we may replace $X_1$ by $X'_0$ and replace $Y_1$ by $X'_0\times_{X_1} Y_1$. 

To sum up, we may assume $X\rightarrow \spa(l,l^+)_{N_{\infty}}$ admits a finite model $X_0\rightarrow \spa(l,l^+)_N$ such that $N\rightarrow N_{\infty}$ is injective and the underlying adic space of $X_0$ is $X$, and also assume that $X_1=X_0$ as log adic spaces. Since a Kummer \'etale morphism is pseudo-saturated, the composition \[Y_1\rightarrow X_1=X_0\rightarrow \spa(l,l^+)_N\] is also pseudo-saturated. By Proposition \ref{prop:adm_sm_fiber_product_exist}, the fiber product \[Y:=Y_1\times_{\spa(l,l^+)_N}\spa(l,l^+)_{N_{\infty}}\]exists and $Y$ is admissibly smooth over $\spa(l,l^+)_{N_{\infty}}$. One checks that $Y$ is indeed the desired fiber product $Y_1\times_{X_1}X$. Since $Y\rightarrow \spa(l,l^+)_{N_{\infty}}$ is locally of finite type, so is $Y\rightarrow X$.

Finally, if $Y_1\rightarrow X_1$ is finite, we show that $Y\rightarrow X$ is finite. Notice that all of the arguments above are \'etale local on $X$. Arguing as in the proof of Proposition \ref{lemma: choose X_0 to be X}, by replacing $N$ with an fs submonoid of $N_{\infty}$ containing $N$, we can choose $Y_1$ so that $Y_1\rightarrow X_1$ coincides with $Y\rightarrow X$ on the underlying adic spaces (\'etale locally on $X$). Since $Y_1\rightarrow X_1$ is finite (as it is the base change from a finite Kummer \'etale morphism), so is $Y\rightarrow X$.
\end{proof}

Lemma \ref{lemma: ket morphism from base change} inspires the following definition of Kummer \'etale morphisms.

\begin{definition}\label{defn: kummer etale morphism over adm log smooth}
Suppose $X$ is admissibly smooth over $\spa(l, l^+)_{N_{\infty}}$ and let $f:Y\rightarrow X$ be a morphism between saturated log adic spaces. 
\begin{enumerate}
\item We say that $f$ is \emph{Kummer \'etale} if \'etale locally on $X$ and $Y$, there exists a Kummer \'etale morphism $Y_1\rightarrow X_1$ of locally noetherian fs log adic spaces and a morphism $X\rightarrow X_1$ such that \[Y=X\times_{X_1}Y_1,\] where the fiber product is taken in the category of saturated log adic spaces. The morphism $Y_1\rightarrow X_1$ is referred to as a \emph{finite model} of $Y\rightarrow X$.
\item We say that $f$ is a \emph{Kummer \'etale cover} if it is Kummer \'etale and surjective.
\item We say that $f$ is \emph{finite Kummer \'etale} if it is Kummer \'etale and finite.
\end{enumerate}
\end{definition}


\bp \label{lemma: choose X_1 Y_1 to be X Y}
Suppose $X$ is admissibly smooth over $\spa(l,l^+)_{N_{\infty}}$ and let $Y\rightarrow X$ be a Kummer \'etale morphism. Then, \'etale locally on $X$ and $Y$, the morphism $Y\rightarrow X$ admits a finite model $Y_1\rightarrow X_1$ fitting into the commutative diagram
\[ 
\begin{tikzcd} 
Y \arrow[d] \arrow[r] & Y_1 \arrow[d] 
\\ 
X \arrow[d] \arrow[r] & X_1 \arrow[d] 
\\ 
\spa(l, l^+)_{N_{\infty}}\arrow[r] & \spa(l, l^+)_N
\end{tikzcd}
\]
such that
\begin{enumerate}
\item $X_1\rightarrow \spa(l,l^+)_N$ is a standard finite model of $X\rightarrow \spa(l,l^+)_{N_{\infty}}$ modeled on a chart $N\rightarrow P$ as in Definition \ref{defn: standard finite model}. In particular, $P$ is torsion-free and the cokernel of $N^{\mathrm{gp}}\rightarrow P^{\mathrm{gp}}$ is torsion-free;
\item $Y_1\rightarrow \spa(l,l^+)_N$ is a standard finite model of $Y\rightarrow \spa(l,l^+)_{N_{\infty}}$ modeled on a chart $N\rightarrow Q$ as in Definition \ref{defn: standard finite model}. In particular, $Q$ is torsion-free and the cokernel of $N^{\mathrm{gp}}\rightarrow Q^{\mathrm{gp}}$ is torsion-free;
\item $Y_1\rightarrow X_1$ admits an fs chart $P\rightarrow Q$ satisfying the conditions in Definition \ref{definition: Kummer etale morphism fs case} (2);
\item $Y_1\rightarrow X_1$ coincides with $Y\rightarrow X$ on the underlying adic spaces.
\end{enumerate}
\ep

\begin{proof}
Following the proof of Lemma \ref{lemma: ket morphism from base change}, the condition (1) and (4) can be achieved. Using \cite[Lemma 4.1.10]{DLLZ}, condition (3) can also be achieved with $Q$ torsion-free. Let $N\rightarrow Q$ be the composition of the injections $N\rightarrow P$ and $P\rightarrow Q$. We must have 
\[Q\cap (-N)=\{0\}.\] Indeed, if $x\in Q\cap (-N)$, then $nx\in P\cap (-N)$ for some integer $n\geq 1$ which forces $x=0$ as $Q$ is torsion-free. The only conditions missing are the quasi-saturated-ness of $N\rightarrow Q$ and the torsion-free-ness of the cokernel of $N^{\mathrm{gp}}\rightarrow Q^{\mathrm{gp}}$.

In fact, the idea of the proof of Proposition \ref{lemma: choose X_0 to be X} remains valid here. To achieve the torsion-free-ness of the cokernel of $N^{\mathrm{gp}}\rightarrow Q^{\mathrm{gp}}$, consider the monoid 
\[N_Q:=\{q\in Q^{\mathrm{gp}}\,|\, nq\in N \textrm{ for some integer }n\geq 1\}.\]
Notice that elements of $N_Q$ necessarily lie in $Q$ by the saturated-ness of $Q$. Consider 
\[X'_1:=X_1\times_{\spa(l,l^+)_N}\spa(l,l^+)_{N_Q}\] and \[Y'_1:=Y_1\times_{\spa(l,l^+)_N}\spa(l,l^+)_{N_Q}.\] The same computation as in the proof of Proposition \ref{lemma: choose X_0 to be X} shows that 
\[Q\sqcup^{\mathrm{sat}}_N N_Q=Q\oplus G\]
where $G\cong (N_Q)^{\mathrm{gp}}/N^{\mathrm{gp}}$ is a finite group, that $Y'_1\rightarrow \spa(l,l^+)_{N_Q}$ admits a chart modeled on the injection $N_Q\rightarrow Q$, and that the cokernel of $(N_Q)^{\mathrm{gp}}\rightarrow Q^{\mathrm{gp}}$ is torsion-free. Notice that the map  $X'_1\rightarrow \spa(l,l^+)_{N_Q}$ admits a chart modeled on $N_Q\rightarrow P\sqcup^{\mathrm{sat}}_NN_Q$ where $P\sqcup^{\mathrm{sat}}_NN_Q$ is torsion-free by Lemma \ref{lemma: torsion-free pushout}. We need to show that the homomorphism $N_Q\rightarrow Q$ factors as the composition of two injections \[N_Q\rightarrow P\sqcup^{\mathrm{sat}}_N N_Q\rightarrow Q.\] The factorization is clear. To see the injectivity of $P\sqcup^{\mathrm{sat}}_NN_Q\rightarrow Q$, it suffices to pass to the group envelopes as both monoids are integral; namely, the canonical homomorphism \[P^{\mathrm{gp}}\oplus_{N^{\mathrm{gp}}} (N_Q)^{\mathrm{gp}}\rightarrow Q^{\mathrm{gp}}\] is injective. Choosing the splitting $P^{\mathrm{gp}}=N^{\mathrm{gp}}\oplus H$ where $H\cong \mathbb{Z}^r$ for some integer $r\geq 0$. We have to show that the map \[H\oplus (N_Q)^{\mathrm{gp}}\rightarrow Q^{\mathrm{gp}}\] is injective. This is equivalent to the condition $H\cap (N_Q)^{\mathrm{gp}}=\{0\}$, which is clear because for any $x\in H\cap (N_Q)^{\mathrm{gp}}$, there exists $n\geq 1$ such that $nx\in H\cap N^{\mathrm{gp}}=\{0\}$ and $Q^{\mathrm{gp}}$ is torsion-free. 

From now on, we assume the torsion-free-ness of the cokernel of $N^{\mathrm{gp}}\rightarrow Q^{\mathrm{gp}}$. By Lemma \ref{corollary: pseudo-sat becomes quasi-sat after finite base change}, there exists a positive integer $m$ such that the map \[\frac{1}{m}N\rightarrow Q\sqcup_N\frac{1}{m}N\] is quasi-saturated. Let 
\[X'_1:=X_1\times_{\spa(l,l^+)_N}\spa(l,l^+)_{\frac{1}{m}N}\] and \[Y'_1:=Y_1\times_{\spa(l,l^+)_N}\spa(l,l^+)_{\frac{1}{m}N}.\] Then the map 
\[ 
X'_1\rightarrow \spa(l,l^+)_{\frac{1}{m}N} \quad (\tu{resp }   Y'_1\rightarrow \spa(l,l^+)_{\frac{1}{m}N} )\] 
are modeled on the charts \[ \frac{1}{m}N\rightarrow P\sqcup_N\frac{1}{m}N  \quad \Big(\tu{resp. } \frac{1}{m}N\rightarrow Q\sqcup_N\frac{1}{m}N.\Big)\]  Notice that both $P\sqcup_N\frac{1}{m}N$ and $Q\sqcup_N\frac{1}{m}N$ are torsion-free by (the proof of) Lemma \ref{lemma: torsion-free pushout}. Consequently, $Y'_1\rightarrow X'_1$ is a finite model of $Y\rightarrow X$ satisfying all desired conditions.
\end{proof}

\begin{remark}\label{remark: n-fold fiber product of finite model}
Suppose $Y_1\rightarrow X_1$ is a finite model of a Kummer \'etale morphism $Y\rightarrow X$ such that the two morphisms coincide on the underlying adic spaces. In this case, we have \[Y_1\times_{X_1}Y_1\times_{X_1}\cdots\times_{X_1}Y_1=Y\times_XY\times_X\cdots\times_XY\]
on the underlying adic spaces. To see this, as every Kummer \'etale morphism is, \'etale locally, a composition of a standard Kummer \'etale morphism and strictly \'etale morphisms, it suffices to assume $Y_1\rightarrow X_1$ is standard Kummer \'etale and admits a chart modeled on $u: P\rightarrow Q$ satisfying the condition as in Definition \ref{definition: Kummer etale morphism fs case} (3). In this case, by \cite[Proposition 4.1.6]{DLLZ},
\begin{align*}
&Y\times_XY\times_X\cdots\times_XY\cong X\times_{X_1}(Y_1\times_{X_1}\cdots\times_{X_1}Y_1)\\
\cong\,\,& X\times_{X_1}(Y_1\times_{X_1}G_{X_1}^D\times_{X_1}\cdots\times_{X_1}G_{X_1}^D)\cong Y\times_X G_X^D\times_X\cdots\times_X G_X^D
\end{align*}
where $G=Q^{\mathrm{gp}}/u^{\mathrm{gp}}(P^{\mathrm{gp}})$ and $G_X^D$ is as in \emph{loc. cit.}. It has the same underlying adic space as \[Y_1\times_{X_1}\cdots\times_{X_1}Y_1\cong Y_1\times_{X_1}G_{X_1}^D\times_{X_1}\cdots\times_{X_1}G_{X_1}^D\] because $G_X^D=G_{X_1}^D$ on the underlying adic spaces and both $G_{X_1}^D\rightarrow X_1$ and $G_X^D\rightarrow X$ are \'etale.
\end{remark}

\begin{definition}\label{definition: standard Kummer etale}
Suppose $X$ is admissibly smooth over $\spa(l, l^+)_{N_{\infty}}$ and let $f:Y\rightarrow X$ be a Kummer \'etale morphism. We say that $f$ is \emph{standard Kummer \'etale} if $f$ is admits a (global) finite model $Y_1\rightarrow X_1$ satisfying all conditions in Proposition \ref{lemma: choose X_1 Y_1 to be X Y} and such that $Y_1\rightarrow X_1$ is standard Kummer \'etale modeled on the chart $u:P\rightarrow Q$ as in Definition \ref{definition: Kummer etale morphism fs case} (3).
\end{definition}

\begin{corollary}\label{cor: ket morphisms are open}
\begin{enumerate}
\item \'Etale locally, Kummer \'etale morphisms are compositions of standard Kummer \'etale morphisms and strictly \'etale morphisms.
\item Kummer \'etale morphisms are open.
\end{enumerate}
\end{corollary}

\begin{proof}
(1) follows from Proposition \ref{lemma: choose X_1 Y_1 to be X Y} and (2) follows from (1) and \cite[Corollary 4.1.9]{DLLZ}.
\end{proof}

Proposition \ref{proposition: stable under composition} and Proposition \ref{prop: base change exists} below show that Kummer \'etale morphisms are compatible with compositions and base change.

\begin{proposition}\label{proposition: stable under composition}
Suppose $X$ is admissibly smooth over $\spa(l, l^+)_{N_{\infty}}$ and let $f:Y\rightarrow X$ and $g:Z\rightarrow Y$ be Kummer \'etale morphisms. Then the composition 
\[g\circ f: Z\rightarrow X\] is also Kummer \'etale.
\end{proposition}

\begin{proof}
By working \'etale locally, we may assume that $Y\rightarrow X$ and $X\rightarrow \spa(l,l^+)_{N_{\infty}}$ admit finite models $Y_1\rightarrow X_1$ and $X_1\rightarrow \spa(l,l^+)_N$ as in Proposition \ref{lemma: choose X_1 Y_1 to be X Y}. Repeating the same argument as in the proof of Lemma \ref{lemma: ket morphism from base change}, by enlarging $N$, we may assume that $Z\rightarrow Y$ admits a finite model $Z_1\rightarrow Y_1$ fitting into the following commutative diagram
\[ 
\begin{tikzcd} 
Z \arrow[d] \arrow[r] & Z_1 \arrow[d] 
\\ 
Y \arrow[d] \arrow[r] & Y_1 \arrow[d] 
\\ 
X \arrow[d] \arrow[r] & X_1 \arrow[d] 
\\ 
\spa(l, l^+)_{N_{\infty}}\arrow[r] & \spa(l, l^+)_N
\end{tikzcd}
\]
where all squares are Cartesian in the category of saturated log adic spaces. In particular, $Z\rightarrow X$ is the base change from a Kummer \'etale morphism $Z_1\rightarrow X_1$ of locally noetherian fs log adic spaces, as desired.
\end{proof}

\begin{lemma}\label{lemma: morphism admits finite model}
Suppose $X$ and $Y$ are both admissibly smooth over $\spa(l,l^+)_{N_{\infty}}$ and let $f:Y\rightarrow X$ be any morphism of log adic spaces. Then, \'etale locally on $X$ and $Y$, $f$ is the base change from a morphism $f_0:Y_0\rightarrow X_0$ of locally noetherian fs log adic spaces. Moreover, one can choose $f_0$ so that it coincides with $f$ on the underlying adic spaces.
\end{lemma}

\begin{proof}
By Proposition \ref{lemma: choose X_0 to be X}, we may assume that $X\rightarrow \spa(l,l^+)_{N_{\infty}}$ admits a finite model 
\[X_0\rightarrow \spa(l,l^+)_N\] which coincides with $X\rightarrow \spa(l,l^+)_{N_{\infty}}$ on the underlying adic spaces and $N\rightarrow N_{\infty}$ is injective. Similarly, we assume that $Y\rightarrow \spa(l,l^+)_{N_{\infty}}$ admits a finite model 
\[Y_0\rightarrow \spa(l,l^+)_{N'}\] with similar properties. By enlarging $N'$, we may assume that $N'$ contains $N$. 

We may further assume that the morphism $Y\rightarrow X\rightarrow X_0$ factors through $Y\rightarrow Y_0$. Indeed, we may replace $N'$ by any fs submonoid $N''$ of $N_{\infty}$ containing $N'$ and replace $Y_0$ by 
\[Y'_0:=Y_0\times_{\spa(l,l^+)_{N'}}\spa(l,l^+)_{N''}.\] Such $Y'_0$ form a cofiltered system with limit $Y$. Suppose $X$ is \'etale locally modeled on an fs chart $P$. Since $P$ is finitely generated, the morphism $P_Y\rightarrow \mM_Y$ factors through $h^*\mM_{Y'_0}\rightarrow \mM_Y$ for some $N''$, where $h: Y\rightarrow Y'_0$ denotes the natural morphism. This means $Y\rightarrow X_0$ factors as $Y\rightarrow Y'_0\rightarrow X_0$, as desired.

Finally, we take \[X'_0:=X_0\times_{\spa(l,l^+)_N}\spa(l,l^+)_{N'}.\]The induced morphism $f_0: Y_0\rightarrow X'_0$ satisfies the requirement.
\end{proof}

\begin{proposition}\label{prop: base change exists}
Suppose $X$, $Y$, and $Z$ are admissibly smooth over $\spa(l,l^+)_{N_{\infty}}$. Let $Y\rightarrow X$ be a Kummer \'etale (resp. finite Kummer \'etale) morphism and let $Z\rightarrow X$ be a morphism locally of finite type. Then the fiber product $W=Y\times_X Z$ exists in the category of saturated log adic spaces and the induced morphism $W\rightarrow Z$ is Kummer \'etale (resp. finite Kummer \'etale).
\end{proposition}

\begin{proof}
By Proposition \ref{lemma: choose X_1 Y_1 to be X Y} and Lemma \ref{lemma: morphism admits finite model}, \'etale locally on $X$, $Y$, and $Z$ (resp. \'etale locally on $X$ and $Z$), $Y\rightarrow X$ and $Z\rightarrow X$ fit into Cartesian diagrams
\[ 
\begin{tikzcd} 
Y \arrow[d] \arrow[r] & Y_1 \arrow[d] 
\\ 
X \arrow[d] \arrow[r] & X_1 \arrow[d] 
\\ 
\spa(l, l^+)_{N_{\infty}}\arrow[r] & \spa(l, l^+)_N
\end{tikzcd}
\,\,\,\,
\begin{tikzcd} 
Z \arrow[d] \arrow[r] & Z_0 \arrow[d] 
\\ 
X \arrow[d] \arrow[r] & X_0 \arrow[d] 
\\ 
\spa(l, l^+)_{N_{\infty}}\arrow[r] & \spa(l, l^+)_{N'}
\end{tikzcd}
\]
such that 
\begin{itemize}
\item $N$ and $N'$ are fs submonoids of $N_{\infty}$;
\item $X_1$, $Y_1$, $X_0$, $Z_0$ are locally noetherian fs log adic spaces and $Y_1\rightarrow X_1$ is Kummer \'etale (resp. finite Kummer \'etale);
\item the underlying adic spaces of $X_1$, $X_0$, $Y_1$, $Z_0$ coincide with $X$, $X$, $Y$, $Z$, respectively. In particular, $Z_0\rightarrow X_0$ is locally of finite type.
\end{itemize}

Moreover, we may assume that $X_0=X_1$. Indeed, by enlarging $N'$ and using the same argument as in the proof of Lemma \ref{lemma: morphism admits finite model}, we may assume that $N'$ contains $N$ and $X\rightarrow X_1$ factors through $X\rightarrow X_0$. Then we can replace $X_1$ and $Y_1$ by $X_0$ and $Y_1\times_{X_1}X_0$, respectively.

Now, let $W_0$ be the fiber product $Y_1\times_{X_0} Z_0$ in the category of locally noetherian fs log adic spaces. Since $Z_0\rightarrow X_0$ is locally of finite type, the fiber product exists. Moreover, the induced morphism $W_0\rightarrow Z_0$ is Kummer \'etale (resp. finite Kummer \'etale). Finally, the fiber product $W:=W_0\times_{Z_0}Z$ exists by Lemma \ref{lemma: ket morphism from base change}. One checks that $W$ is the desired fiber product $Y\times_XZ$.
\end{proof}

We wrap up this subsection with a technical lemma. This is an analogue of \cite[Lemma 4.2.5]{DLLZ}. To state the result, we need to introduce an analogue of ``$X^{\frac{1}{n}}$'' in \cite[Definition 4.1.5 (2)]{DLLZ}.

\begin{construction}\label{construction: X^{1/n}}
Suppose $X\rightarrow \spa(l,l^+)_{N_{\infty}}$ admits a standard finite model and let $X_0$, $N$, $P$, $P_{\infty}$ be as in Definition \ref{defn: standard finite model}. (In particular, $X\cong \spa(R, R^+)$ is affinoid.) For any positive integer $n$, let 
\[X_0^{(n)}:=X_0\times_{\spa(l,l^+)_N}\spa(l,l^+)_{\frac{1}{n}N}\]
where the fiber product is taken in the category of saturated log adic spaces (or, equivalently, in the category of log adic spaces). It admits a chart modeled on a torsion-free monoid \[P\sqcup_N\frac{1}{n}N=P\sqcup^{\mathrm{int}}_N\frac{1}{n}N=P\sqcup^{\mathrm{sat}}_N\frac{1}{n}N\] (cf. Lemma \ref{lemma: torsion-free pushout}). Let 
\[X_0^{(n)'}:=X_0^{(n)}\times_{\spa(l\langle P\sqcup_N\frac{1}{n}N \rangle, l^+\langle P\sqcup_N\frac{1}{n}N\rangle)}\spa(l\langle \frac{1}{n}P\rangle, l^+\langle \frac{1}{n}P\rangle)\]
and
\[X^{\frac{1}{n}}:=X_0^{(n)'}\times_{X_0^{(n)}}X, \]
where all fiber products here are taken in the category of saturated log adic spaces. These log adic spaces fit into the following commutative diagram of locally noetherian saturated log adic spaces 
\[ 
\begin{tikzcd} 
X^{\frac{1}{n}} \arrow[d] \arrow[r] & X_0^{(n)'} \arrow[d] &
\\ 
X \arrow[d] \arrow[r] & X_0^{(n)} \arrow[d] \arrow[r] & X_0 \arrow[d]
\\ 
\spa(l, l^+)_{N_{\infty}}\arrow[r] & \spa(l, l^+)_{\frac{1}{n}N} \arrow[r] & \spa(l, l^+)_N
\end{tikzcd}
\]
where all the squares are Cartesian. A priori, the construction of $X^{\frac{1}{n}}$ depends on the finite model and the chart $u:N\rightarrow P$. 

Note that the morphism \[X_0^{(n)'}\rightarrow X_0^{(n)}\] is standard Kummer \'etale modeled on the chart \[P\sqcup_N\frac{1}{n}N\hookrightarrow \frac{1}{n}P\] (cf. Remark \ref{remark:induced_Galois_cover_dominant}). Also note that the map \[X^{\frac{1}{n}}\rightarrow \spa(l,l^+)_{N_{\infty}}\] is Kummer \'etale with a standard finite model \[X_0^{(n)'}\rightarrow \spa(l,l^+)_{\frac{1}{n}N}\] equipped with chart $\frac{1}{n}N\hookrightarrow \frac{1}{n}P$. Consequently, $X^{\frac{1}{n}}\rightarrow X$ is standard Kummer \'etale (cf. Definition \ref{definition: standard Kummer etale}). In particular, the morphisms $X_0^{(n)'}\rightarrow X_0^{(n)}$ and $X^{\frac{1}{n}}\rightarrow X$ coincide on the underlying adic spaces.
\end{construction}

\bp\label{lemma: Lemma 4.2.5 DLLZ}
Suppose that $X\rightarrow \spa(l,l^+)_{N_{\infty}}$ admits a standard finite model 
\[X_0\rightarrow \spa(l,l^+)_N\] equipped with a chart $u:N\rightarrow P$ so that one can define $X_0^{(n)}$, $X_0^{(n)'}$, and $X^{\frac{1}{n}}$ as in Construction \ref{construction: X^{1/n}}.
\begin{enumerate}
\item The morphism $X^{\frac{1}{n}}\rightarrow X$ only depends on the standard saturated chart $P_{\infty}$ (see Definition \ref{defn: standard finite model}); it does not depend on the choice of $P$ and $N$ as along as \[P\sqcup_NN_{\infty}=P_{\infty}.\]
\item Let $Y\rightarrow X$ be a Kummer \'etale (resp. finite Kummer \'etale) morphism. Then there exists a positive integer $n$ invertible in $l$ such that the natural map \[Y\times_X X^{\frac{1}{n}}\rightarrow X^{\frac{1}{n}}\] is strictly \'etale (resp. strictly finite \'etale).
\end{enumerate}
\ep 

\begin{proof}
For (1), consider the tensor product of Tate Huber pairs
\[
(R', R'^+):=(R, R^+)\widehat{\otimes}_{(l\langle P_{\infty}\rangle, l^+\langle P_{\infty}\rangle)}(l\langle \frac{1}{n}P_{\infty}\rangle, l^+\langle \frac{1}{n}P_{\infty}\rangle).
\]
Since $P_{\infty}$ is almost $n$-divisible, $(R, R^+)\rightarrow (R', R'^+)$ is topologically of finite type, and hence $(R', R'^+)$ is \'etale sheafy. We claim that $X^{\frac{1}{n}}=\spa(R', R'^+)$ equipped with a chart modeled on $\frac{1}{n}P_{\infty}$. Indeed, we have
\[P_{\infty}=P\sqcup_NN_{\infty}=(P\sqcup_N \frac{1}{n}N)\sqcup_{\frac{1}{n}N} N_{\infty}\]
and hence
\[P_{\infty}\sqcup_{P\sqcup_N \frac{1}{n}N}\frac{1}{n}P=\frac{1}{n}P\sqcup_{\frac{1}{n}N}N_{\infty}=\frac{1}{n}P_{\infty}.\]
(There is no difference whether these pushouts are taken in the category of monoids or saturated monoids because $N\rightarrow P$ is saturated.) This yields
\begin{align*}
(R', R'^+)&\,=(R, R^+)\widehat{\otimes}_{(l\langle P_{\infty}\rangle, l^+\langle P_{\infty}\rangle)}(l\langle \frac{1}{n}P_{\infty}\rangle, l^+\langle \frac{1}{n}P_{\infty}\rangle)\\
&\,=(R, R^+)\widehat{\otimes}_{(l\langle P\sqcup_N\frac{1}{n}N\rangle, l^+\langle P\sqcup_N\frac{1}{n}N\rangle)}(l\langle \frac{1}{n}P\rangle, l^+\langle \frac{1}{n}P\rangle)
\end{align*}
which means that 
\[\spa(R', R'^+)=X\times_{X_0^{(n)}}X_0^{(n)'},\] as desired.

As an immediate corollary of (1), we have the freedom to replace $N$ by any fs submonoid $N'$ of $N_{\infty}$ containing $N$ and replace $P$ by the pushout $P\sqcup_NN'$, which is still torsion-free by Proposition \ref{lemma: choose X_0 to be X}.

To prove (2), it suffices to work \'etale locally on $X$. Using Proposition \ref{lemma: choose X_1 Y_1 to be X Y}, we may assume that there is an \'etale covering $\{Y_i\rightarrow Y\}_{i\in I}$ indexed by a finite set $I$ such that each morphism $Y_i\rightarrow X$ admits a finite model $Y_{i,0}\rightarrow X_0$ fitting into the commutative diagram
\[ 
\begin{tikzcd} 
Y_i \arrow[d] \arrow[r] & Y_{i,0} \arrow[d] 
\\ 
X \arrow[d] \arrow[r] & X_0 \arrow[d] 
\\ 
\spa(l, l^+)_{N_{\infty}}\arrow[r] & \spa(l, l^+)_N
\end{tikzcd}
\]
such that
\begin{itemize}
\item $X_0\rightarrow \spa(l,l^+)_N$ is a standard finite model of $X\rightarrow \spa(l,l^+)_{N_{\infty}}$ modeled on a chart $N\rightarrow P$ as in Definition \ref{defn: standard finite model}. In particular, $P$ is torsion-free and the cokernel of $N^{\mathrm{gp}}\rightarrow P^{\mathrm{gp}}$ is torsion-free;
\item $Y_{i,0}\rightarrow \spa(l,l^+)_N$ is a standard finite model of $Y\rightarrow \spa(l,l^+)_{N_{\infty}}$ modeled on a chart $N\rightarrow Q_i$ as in Definition \ref{defn: standard finite model}. In particular, $Q_i$ is torsion-free and the cokernel of $N^{\mathrm{gp}}\rightarrow Q_i^{\mathrm{gp}}$ is torsion-free;
\item $Y_1\rightarrow X_1$ admits an fs chart $P\rightarrow Q_i$ satisfying the conditions in Definition \ref{definition: Kummer etale morphism fs case} (2).
\end{itemize}

We claim that, there exists a positive integer $n$ invertible in $l$ such that the inclusion 
\[P\sqcup_N\frac{1}{n}N\hookrightarrow \frac{1}{n}P\] from Remark \ref{remark:induced_Galois_cover_dominant} factors as 
\[P\sqcup_N\frac{1}{n}N\hookrightarrow Q_i\sqcup_N\frac{1}{n}N\hookrightarrow \frac{1}{n}P\]
for all $i\in I$. Indeed, pick any $n$ invertible in $l$ such that $Q_i\subset \frac{1}{n}P$ for all $i$. To check that the canonical map $Q_i\sqcup_N\frac{1}{n}N\rightarrow \frac{1}{n}P$ is injective, it suffices to check this after passing to the group envelopes as all monoids here are integral. But 
\[Q_i^{\mathrm{gp}}\oplus_{N^{\mathrm{gp}}}\frac{1}{n}N^{\mathrm{gp}}\rightarrow \frac{1}{n}P^{\mathrm{gp}}\] is indeed injective because $Q_i^{\mathrm{gp}}\cap \frac{1}{n}N^{\mathrm{gp}}=N^{\mathrm{gp}}$ by our assumption on $N\rightarrow Q_i$.

Given such $n$, we show that the base change \[Y_i\times_X X^{\frac{1}{n}}\rightarrow X^{\frac{1}{n}}\] of $Y_i\rightarrow X$ is strictly \'etale, for all $i\in I$. Let \[X_{0,i}^{(n)}:=X_0\times_{\spa(l\langle P\sqcup_N \frac{1}{n}N\rangle, l^+\langle P\sqcup_N \frac{1}{n}N\rangle)}\spa(l\langle Q_i\sqcup_N \frac{1}{n}N\rangle, l^+\langle Q_i\sqcup_N \frac{1}{n}N\rangle).\]
Then $X_0^{(n)'}\rightarrow X_0^{(n)}$ factors as $X_0^{(n)'}\rightarrow X_{0,i}^{(n)}\rightarrow X_0^{(n)}$ where the map $X_{0,i}^{(n)}\rightarrow X_0^{(n)}$ is standard Kummer \'etale with Galois group $G\cong Q_i^{\mathrm{gp}}/P^{\mathrm{gp}}$. By assumption, $Y_i$ is strictly \'etale over 
\begin{align*}
&X\times_{\spa(l\langle P\rangle, l^+\langle P\rangle)}\spa(l\langle Q_i\rangle, l^+\langle Q_i\rangle)\\=\,&X\times_{\spa(l\langle P\sqcup_N \frac{1}{n}N\rangle, l^+\langle P\sqcup_N \frac{1}{n}N\rangle)}\spa(l\langle Q_i\sqcup_N \frac{1}{n}N\rangle, l^+\langle Q_i\sqcup_N \frac{1}{n}N\rangle)\\
=\,& X\times_{X_0^{(n)}} X_{0,i}^{(n)}.
\end{align*} Hence, $Y_i\times_XX^{\frac{1}{n}}=Y_i\times_{X_0^{(n)}}X_0^{(n)'}$ is strictly \'etale over
\begin{align*}
& \Big(X\times_{X_0^{(n)}} X_{0,i}^{(n)}\Big)\times_{X_0^{(n)}} X_0^{(n)'}\\
=\,&\Big(\big(X\times_{X_0^{(n)}} X_{0,i}^{(n)}\big)\times_{X_0^{(n)}} X_{0,i}^{(n)}\Big)\times_{X_{0,i}^{(n)}} X_0^{(n)'}\\
\cong\,&\Big(\big(X\times_{X_0^{(n)}} X_{0,i}^{(n)}\big)\times G^D\Big)\times_{X_{0,i}^{(n)}} X_0^{(n)'}\\
=\,&(X\times_{X_0^{(n)}} X_0^{(n)'})\times G^D= X^{\frac{1}{n}}\times G^D
\end{align*} 
where $G^D$ is the diagonalizable group scheme that is Cartier dual to the constant group scheme $G$ (cf. \cite[Proposition 4.1.6]{DLLZ}). This implies that \[Y_i\times_X X^{\frac{1}{n}}\rightarrow X^{\frac{1}{n}}\] is indeed strictly \'etale. 
Finally, by \'etale descent, \[Y\times_XX^{\frac{1}{n}}\rightarrow X^{\frac{1}{n}}\] is \'etale (resp. finite \'etale).
\end{proof}

\subsection{The Kummer \'etale site}\label{subsection: the Kummer etale site}
Now we are ready to study the Kummer \'etale site of $X$. Throughout this subsection, let $X$ be an  admissibly smooth log adic space over a divisible log point $\spa(l, l^+)_{N_{\infty}}$. Let $X_{\ket}$ (resp. $X_{\fket}$) be the category of log adic spaces that are Kummer \'etale (resp. finite Kummer \'etale) over $X$. By Proposition \ref{prop: base change exists} and Proposition \ref{proposition: stable under composition}, fiber products exist in $X_{\ket}$ (resp. $X_{\fket}$). Moreover, since Kummer \'etale morphisms are open by Corollary \ref{cor: ket morphisms are open}, we can equip $X_{\ket}$ (resp. $X_{\fket}$) with the topological coverings to make it into a site, which we call the \emph{Kummer \'etale site} (resp. \emph{finite Kummer \'etale site}). By Proposition \ref{lemma: choose X_1 Y_1 to be X Y}, the Kummer \'etale topology is generated by the standard Kummer \'etale morphisms and strictly \'etale morphisms.

By Proposition \ref{prop: base change exists}, if $f: Y\rightarrow X$ is a morphism of log adic spaces that are admissibly smooth over $\spa(l, l^+)_{N_{\infty}}$ and $f$ is locally of finite type, then $f$ induces a morphism of sites $Y_{\ket}\rightarrow X_{\ket}$ (resp. $Y_{\fket}\rightarrow X_{\fket}$).

The following lemma shows that morphisms between objects in $X_{\ket}$ are necessarily Kummer \'etale.

\begin{lemma}\label{lemma: morphisms in Xket are Kummer etale}
Suppose $X$ is admissibly smooth over $\spa(l, l^+)_{N_{\infty}}$ and let $g:Z\rightarrow Y$ and $f:Y\rightarrow X$ be morphisms of locally noetherian saturated log adic spaces. If both $f$ and $g\circ f$ are Kummer \'etale, so is $g$.
\end{lemma}

\begin{proof}
By choosing finite models $Z_0\rightarrow Y_0$ and $Y_0\rightarrow X_0$ using the same trick as in the proof of Lemma \ref{lemma: ket morphism from base change}, we reduce to \cite[Proposition 4.1.15]{DLLZ}.
\end{proof}

\begin{remark} 
By definition and using the same trick as in the proof of Lemma \ref{lemma: ket morphism from base change}, we know that $X_{\ket}$ coincides with the colimit of sites $X_{0,\ket}$ where $X_0$ runs through all locally noetherian fs log adic spaces with an arrow $X\rightarrow X_0$. In \cite{Dori_Yao}, a similar claim is proven for saturated log schemes (where $X$ is no longer required to be log smooth over a log point). 
\end{remark}

The goal of the rest of the section is to verify that the properties of the Kummer \'etale site studied in \cite[\S4]{DLLZ} remain valid in our new setting.

Firstly, on the Kummer \'etale site $X_{\ket}$, one can define the structure pre-sheaves $\mO_{X_{\ket}}$ and $\mO^+_{X_{\ket}}$ given by $U\mapsto \mO_U(U)$ and $U\mapsto \mO^+_U(U)$, respectively. One can also define the presheaf $\mM_{X_{\ket}}$ given by $U\mapsto \mM_U(U)$.

\begin{proposition}\label{prop: structure pre-sheaves are sheaves}
Suppose $X$ is admissibly smooth over $\spa(l,l^+)_{N_{\infty}}$.
\begin{enumerate}
\item The presheaves $\mO_{X_{\ket}}$, $\mO^+_{X_{\ket}}$ are sheaves.
\item The presheaf $\mM_{X_{\ket}}$ is a sheaf.
\item If $X$ is affinoid, then $H^i(X_{\ket}, \mO_{X_{\ket}})=0$ for all $i\geq 1$.
\end{enumerate}
\end{proposition}

We need the following lemma.

\begin{lemma}\label{lemma: analogue of Thm 4.3.1 DLLZ}
Suppose $X$ is admissibly smooth over $\spa(l,l^+)_{N_{\infty}}$ and suppose $Y\rightarrow X$ is a standard Kummer \'etale cover (cf. Definition \ref{definition: standard Kummer etale}). Then the \v{C}ech complex 
\begin{multline*}
  \quad   C^{\bullet}(Y/X): \:\: 0\rightarrow \mO_{X_{\ket}}(X)\rightarrow\mO_{X_{\ket}}(Y) \rightarrow \mO_{X_{\ket}}(Y\times_XY)  \\  \rightarrow\mO_{X_{\ket}}(Y\times_XY\times_XY)\rightarrow\cdots \qquad 
\end{multline*}
is exact.
\end{lemma}

\begin{proof}
Let $Y_1\rightarrow X_1$ be a finite model of $Y\rightarrow X$ as in Definition \ref{definition: standard Kummer etale}. In particular, $Y_1\rightarrow X_1$ coincides with $Y\rightarrow X$ on the underlying adic spaces. By Remark \ref{remark: n-fold fiber product of finite model}, $C^{\bullet}(Y/X)$ coincides with $C^{\bullet}(Y_1/X_1)$, later of which is exact by \cite[Theorem 4.3.1 (2)]{DLLZ}.
\end{proof}

\begin{proof}[Proof of Proposition \ref{prop: structure pre-sheaves are sheaves}]
\begin{enumerate}
\item It suffices to prove the sheafiness of $\mO_{X_{\ket}}$ as the sheafiness of $\mO^+_{X_{\ket}}$ follows from that of $\mO_{X_{\ket}}$. Recall that the Kummer \'etale topology is generated by standard Kummer \'etale covers and strictly \'etale morphisms. Since the sheafiness of the structure sheaf for the \'etale topology is well-known, it suffices to show the exactness of the \v{C}ech complex \[0\rightarrow \mO_{X_{\ket}}(X)\rightarrow\mO_{X_{\ket}}(Y)\rightarrow \mO_{X_{\ket}}(Y\times_XY)\] 
where $Y\rightarrow X$ is a standard Kummer \'etale cover. This follows from Lemma \ref{lemma: analogue of Thm 4.3.1 DLLZ}.
\item Since $\mM_X$ is a sheaf with respect to the \'etale topology, it suffices to show the exactness of 
\[
0\rightarrow \mM_X(X)\rightarrow \mM_Y(Y)\rightrightarrows \mM_{Y\times_XY}(Y\times_XY)
\]
where $Y\rightarrow X$ is a Kummer \'etale cover such that it admits a finite model $Y_1\rightarrow X_1$ satisfying the conditions in Proposition \ref{lemma: choose X_1 Y_1 to be X Y}. In particular, $X_1\rightarrow \spa(l,l^+)_N$ (resp. $Y_1\rightarrow \spa(l,l^+)_N$) is a finite model of $X\rightarrow \spa(l,l^+)_{N_{\infty}}$ (resp. $Y\rightarrow \spa(l,l^+)_{N_{\infty}}$) and $Y_1\rightarrow X_1$ coincides with $Y\rightarrow X$ on the underlying adic spaces.

For every fs submonoid $N'$ of $N_{\infty}$ containing $N$, let \[X'_1:=X_1\times_{\spa(l,l^+)_N}\spa(l,l^+)_{N'}\] and 
\[Y'_1:=Y_1\times_{\spa(l,l^+)_N}\spa(l,l^+)_{N'}.\]
Then $Y'_1\rightarrow X'_1$ is also a finite model of $Y\rightarrow X$. Such $N'$ form a filtered system and the corresponding filtered colimit of $\mM_{X'_1}$ (resp. $\mM_{Y'_1}$) is equal to $\mM_X$ (resp. $\mM_Y$). By Remark \ref{remark: n-fold fiber product of finite model}, $Y'_1\times_{X'_1}Y'_1$ coincides with $Y\times_XY$ on the underlying adic spaces and hence the filtered colimit of $\mM_{Y'_1\times_{X'_1}Y'_1}$ also coincides with $\mM_{Y\times_XY}$. By \cite[Proposition 4.3.4]{DLLZ},
\[
0\rightarrow \mM_{X'_1}(X'_1)\rightarrow \mM_{Y'_1}(Y'_1)\rightrightarrows \mM_{Y'_1\times_{X'_1}Y'_1}(Y'_1\times_{X'_1}Y'_1)
\]
is exact. Taking the filtered colimit, we obtain the desired exactness.
\item 
By \cite[Proposition A.10]{DLLZ}, we may assume that $X$ is affinoid and $X\rightarrow \spa(l,l^+)_{N_{\infty}}$ admits a standard finite model. Consider the \v{C}ech complex associated with a Kummer \'etale covering $\{U_i\rightarrow X\}_{i\in I}$ indexed by a finite set $I$. By Proposition \ref{lemma: Lemma 4.2.5 DLLZ}, this covering admits a refinement $\{V_j\rightarrow X\}_{j\in J}$ such that $\{V_j\times_X X^{\frac{1}{n}}\rightarrow X^{\frac{1}{n}}\}_{j\in J}$ is a strictly \'etale covering, for some integer $n\geq 1$ invertible in $l$, and such that \[V_j\times_XX^{\frac{1}{n}}\rightarrow V_j\] is a composition of \'etale morphisms and standard Kummer \'etale covers. It follows from Lemma \ref{lemma: analogue of Thm 4.3.1 DLLZ} and \cite[Proposition A.10]{DLLZ} that the \v{C}ech complex
\[0\rightarrow \mO_X(X)\rightarrow \oplus_i \mO_{U_i}(U_i)\rightarrow\oplus_{i,j}\mO_{U_i\times_{X}U_j}(U_i\times_{X}U_j)\rightarrow\cdots\]
is exact. This finishes the proof.
\end{enumerate}
\end{proof}

\begin{corollary}
Suppose $X$ is admissibly smooth over $\spa(l,l^+)_{N_{\infty}}$. Let $\varepsilon_{\mathrm{an}}: X_{\ket}\rightarrow X_{\mathrm{an}}$ and $\varepsilon_{\ett}: X_{\ket}\rightarrow X_{\ett}$ be the natural projection of sites. Then the canonical morphisms \[ \mO_{X_{\mathrm{an}}}\rightarrow R\varepsilon_{\mathrm{an},*}\mO_{X_{\ket}} \:\: \text{and }  \mO_{X_{\ett}}\rightarrow R\varepsilon_{\ett,*}\mO_{X_{\ket}}
\]
are isomorphisms. Consequently, the pullback functor from the category of vector bundles on $X_{\mathrm{an}}$ (resp. $X_{\ett}$) to the category of $\mO_{X_{\ket}}$-modules is fully faithful. 
\end{corollary}

More generally, we also have the following higher vanishing of cohomology of coherent sheaves.

\begin{definition}\label{defn: coherent sheaf}
Let $X$ be admissibly smooth over $\spa(l,l^+)_{N_{\infty}}$.
\begin{enumerate}
\item An $\mO_{X_{\ket}}$-module $\mathcal{F}$ is called an \emph{analytic coherent sheaf} if it is isomorphic to the inverse image of a coherent sheaf on the analytic site of $X$.
\item An $\mO_{X_{\ket}}$-module $\mathcal{F}$ is called a \emph{coherent sheaf} if there exists a Kummer \'etale covering $U_i\rightarrow X$ such that each $\mathcal{F}|_{U_i}$ is an analytic coherent sheaf.
\end{enumerate}
\end{definition}

\begin{proposition}\label{prop: vanishing of cohomology of coherent sheaves}
Let $X$ be admissibly smooth over $\spa(l,l^+)_{N_{\infty}}$. If $X$ is affinoid and $\mathcal{F}$ is a coherent sheaf on $X_{\ket}$, then $H^i(X_{\ket}, \mathcal{F})=0$ for all $i\geq 1$.
\end{proposition}

\begin{proof}
We first prove the statement when $\mathcal{F}$ is an analytic coherent sheaf (i.e., an analogue of \cite[Theorem 4.3.7 (1)]{DLLZ}). By \cite[Proposition A.10]{DLLZ}, it suffices to prove the exactness of the \v{C}ech complex \[C_{\mathcal{F}}^{\bullet}(Y/X): 0\rightarrow \mathcal{F}(X)\rightarrow\mathcal{F}(Y)\rightarrow \mathcal{F}(Y\times_XY)\rightarrow\mathcal{F}(Y\times_XY\times_XY)\rightarrow\cdots\] 
where $Y\rightarrow X$ is a standard Kummer \'etale cover with a finite model $Y_1\rightarrow X_1$ which coincides with $Y\rightarrow X$ on the underlying adic spaces. By Remark \ref{remark: n-fold fiber product of finite model}, $C_{\mathcal{F}}^{\bullet}(Y/X)$ coincides with $C_{\mathcal{F}}^{\bullet}(Y_1/X_1)$,  which is exact by \cite[Theorem 4.3.7]{DLLZ}.

Given this, the proof for a general coherent sheaf is similar to the proof of \cite[Theorem 4.3.7 (2)]{DLLZ}. It suffices to use the arguments in the proof of Proposition \ref{prop: structure pre-sheaves are sheaves} (3) to replace the corresponding argument in \emph{loc. cit.}
\end{proof}

\begin{proposition}\label{prop: representable presheaves are sheaves}
Let $f:Y\rightarrow X$ be a morphism between log adic spaces that are admissibly smooth over $\spa(l, l^+)_{N_{\infty}}$. Then the presheaf $\mathrm{Mor}_X(\cdot, Y)$ is sheaf on $X_{\ket}$.
\end{proposition}

\begin{proof}
Let $\mathcal{F}$ denote the presheaf $\mathrm{Mor}_X(\cdot, Y)$. By \'etale descent, it suffices to show the exactness of the \v{C}ech complex
\[
C^{\bullet}_{\mathcal{F}}(W/U):0\rightarrow \mathcal{F}(U)\rightarrow \mathcal{F}(W)\rightrightarrows \mathcal{F}(W\times_UW)
\]
where $U\in X_{\ket}$ and $W\rightarrow U$ is a standard Kummer \'etale cover. The injectivity of the map $\mathcal{F}(U)\rightarrow \mathcal{F}(W)$ follows from Proposition \ref{prop: structure pre-sheaves are sheaves} (1) \& (2). 

It remains to show that if $f:W\rightarrow Y$ satisfies $\mathrm{pr}_1 \circ f=\mathrm{pr}_2\circ f$ where \[ \mathrm{pr}_1, \mathrm{pr}_2 :W\times_UW\rightarrow W\] are the two natural projections, then $f$ must live in the image of $\mathcal{F}(U)\rightarrow \mathcal{F}(W)$. By the same trick as in the proof of Lemma \ref{lemma: ket morphism from base change}, we may assume that the entire commutative diagram
\[ 
\begin{tikzcd} 
W\times_UW \arrow[d, shift left,"\mathrm{pr}_1"] \arrow[d, shift right, "\mathrm{pr}_2"']\arrow[ddr, shift left] & 
\\ 
W \arrow[d] \arrow[dr, "f"'] & 
\\ 
U \arrow[d] & Y \arrow[dl] 
\\ 
X & 
\end{tikzcd}
\]
is obtained from base change of a diagram of locally noetherian fs log adic spaces
\[ 
\begin{tikzcd} 
W_1\times_{U_1}W_1 \arrow[d, shift left,"\mathrm{pr}_1"] \arrow[d, shift right, "\mathrm{pr}_2"']\arrow[ddr, shift left] & 
\\ 
W_1 \arrow[d] \arrow[dr, "f_1"'] & 
\\ 
U_1 \arrow[d] & Y_1 \arrow[dl] 
\\ 
X_1 & 
\end{tikzcd}
\]
along a morphism $X\rightarrow X_1$, and such that the two diagram coincide on the underlying adic spaces. By \cite[Proposition 4.3.5]{DLLZ}, such an $f_1: W_1\rightarrow Y_1$ satisfying $\mathrm{pr}_1\circ f_1=\mathrm{pr}_2\circ f_1$ must live in the image of the map 
\[\mathrm{Mor}_{X_1}(U_1, Y_1)\rightarrow \mathrm{Mor}_{X_1}(W_1, Y_1).\] Base change along $X\rightarrow X_1$, we conclude that $f$ lives in the image of $\mathcal{F}(U)\rightarrow \mathcal{F}(W)$, as desired.
\end{proof}

\subsection{Kummer \'etale descent} The goal of this subsection is to show that the descent results studied in \cite[Section 4.4]{DLLZ} generalizes to our setup. Recall the following definition of log geometric points.

\begin{definition}\label{defn: log geometric points}
\begin{enumerate}
\item A \emph{log geometric point} is a log point 
\[\zeta=(\spa(k,k^+), M)\] (see Definition \ref{example: log points}) where $k$ is separably closed and $\overline{M}=M/k^{\times}$ is uniquely $n$-divisible for all positive integers $n$ invertible in $k$.
\item Suppose $X$ is admissibly smooth over $\spa(l,l^+)_{N_{\infty}}$. A \emph{log geometric point} of $X$ is a morphism of log adic spaces 
\[\eta: \zeta\rightarrow X\] where $\zeta$ is a log geometric point as in (1).
\item Suppose $X$ is admissibly smooth over $\spa(l,l^+)_{N_{\infty}}$. A \emph{Kummer \'etale neighborhood} of a log geometric point $\eta: \zeta\rightarrow X$ is a factorization of $\eta$ into a composition \[\zeta\rightarrow U\xrightarrow[]{\phi} X\] where $\phi$ is Kummer \'etale.
\end{enumerate}
\end{definition}

\begin{lemma}\label{lemma: log geometric point always exists}
Suppose $X$ is admissibly smooth over $\spa(l,l^+)_{N_{\infty}}$. For every geometric point \[\xi=\spa(k,k^+)\rightarrow X,\] there always exists a log geometric point $\widetilde{\xi}$ above it (i.e., the morphism $\widetilde{\xi}\rightarrow X$ factors through $\xi\rightarrow X$ on the underlying adic spaces).
\end{lemma}

\begin{proof}
This is an analogue of \cite[Construction 4.4.3]{DLLZ}. We may assume that $X$ admits a standard saturated chart $P_{\infty}$ in the sense of Definition \ref{defn: standard finite model}. Equip $\xi$ with the pullback log structure from $X$ so that $\xi$ also admits a chart modeled on $P_{\infty}$. For every positive integer $n$ invertible in $k$, consider $X^{\frac{1}{n}}$ constructed in Construction \ref{construction: X^{1/n}} and consider
\[\xi^{(\frac{1}{n})}:=\big(\xi\times_{X}X^{\frac{1}{n}}\big)_{\mathrm{red}}\]
equipped with the log structure modeled on the chart $\frac{1}{n}P_{\infty}$. Notice that, since we take the reduced subspace, the underlying adic space of $\xi^{(\frac{1}{n})}$ is still isomorphic to $\spa(k,k^+)$. Now consider
\[\tilde{\xi}:=\varprojlim_n \xi^{(\frac{1}{n})}\]
where the inverse limit runs through all positive integers $n$ invertible in $k$. One checks that the underlying adic space of $\tilde{\xi}$ is isomorphic to $\spa(k,k^+)$, and it is equipped with the log structure modeled on the chart $\varinjlim_n \frac{1}{n}P_{\infty}$ which is uniquely $n$-divisible for all $n$ invertible on $k$. Consequently, $\overline{\mM}_{\tilde{\xi}}$ is also uniquely $n$-divisible for all $n$ invertible on $k$ as it is a quotient of $\varinjlim_n \frac{1}{n}P_{\infty}$.
\end{proof}

\begin{lemma}\label{lemma: fiber functor}
Suppose $X$ is admissibly smooth over $\spa(l,l^+)_{N_{\infty}}$ and let $\zeta\rightarrow X$ be a log geometric point. Then the functor $\mathrm{Sh}(X_{\ket})\rightarrow \underline{\mathrm{Sets}}$ sending $\mathcal{F}$ to \[\mathcal{F}_{\zeta}:=\varinjlim \mathcal{F}(U),\] where the colimit runs through all Kummer \'etale neighborhoods $U$ of $\zeta$, is a fiber functor (i.e., a point in the topos $\mathrm{Sh}(X_{\ket})$). Moreover, such fiber functors defined by log geometric points form a conservative system.
\end{lemma}

\begin{proof}
We first claim that the category of Kummer \'etale neighborhoods of $\zeta$ is filtered. Indeed, suppose $U\rightarrow X$ and $U'\rightarrow X$ are two Kummer \'etale neighborhoods of $\zeta$. By the same trick as in the proof of Lemma \ref{lemma: ket morphism from base change}, we may assume that both morphisms admit finite models $U_1\rightarrow X_1$ and $U'_1\rightarrow X_1$, respectively, such that
\begin{itemize}
\item $U_1\times_{X_1}U'_1\rightarrow X_1$ is a finite model of the Kummer \'etale morphism $U\times_XU'\rightarrow X$; and
\item the morphisms 
\[U_1\rightarrow X_1, U'_1\rightarrow X_1, U_1\times_{X_1}U'_1\rightarrow X_1\] coincide with \[ U\rightarrow X, U'\rightarrow X, U\times_XU'\rightarrow X\] on the underlying adic spaces.
\end{itemize}
Using \cite[Proposition 2.3.32 \& Remark 4.1.4]{DLLZ}, we conclude that $U\times_XU'$ is also a Kummer \'etale neighborhood of $\zeta$. This proves the first statement.

The second statement is clear as every point of $X$ admits some log geometric point above it by Lemma \ref{lemma: log geometric point always exists}, and every object in $X_{\ket}$ is covered by liftings of log geometric points of $X$.
\end{proof}

We have the following descent result which is an analogue of \cite[Theorem 4.4.12]{DLLZ}.

\begin{theorem}\label{thm: ket descent of finite ket morphism}
Suppose that $X$ is admissibly smooth over $\spa(l, l^+)_{N_{\infty}}$ and let $f:Y\rightarrow X$ be a Kummer \'etale cover. Let 
\[\mathrm{pr}_1, \mathrm{pr}_2:Y\times_XY\rightarrow Y\] be the two projections. Suppose that  $Y'\in Y_{\fket}$ and suppose that there exists an isomorphism 
\[\mathrm{pr}^{-1}_1(Y')\xrightarrow[]{\sim} \mathrm{pr}^{-1}_2(Y')\] satisfying the usual cocycle condition. Then there exists a unique $X'\in X_{\ket}$ (up to isomorphism) such that 
\[Y'\cong X'\times_X Y.\]
\end{theorem}

\begin{proof}
By the same trick as in the proof of Lemma \ref{lemma: ket morphism from base change}, we may assume that every arrow in the commutative diagram
\[ 
\begin{tikzcd} 
\mathrm{pr}_1^{-1}(Y')\arrow[d] \arrow[dr] \arrow[r, "\sim"]& \mathrm{pr}_2^{-1}(Y')\arrow[ld] \arrow[d]
\\ 
Y\times_XY \arrow[d, shift left,"\mathrm{pr}_1"] \arrow[d, shift right, "\mathrm{pr}_2"'] & Y' \arrow[ld] \\
Y \arrow[d] & \\
X & 
\end{tikzcd}
\]
admits a finite model. In particular, the entire diagram is the base change of a commutative diagram of locally noetherian fs log adic spaces
\[ 
\begin{tikzcd} 
\mathrm{pr}_1^{-1}(Y'_1)\arrow[d] \arrow[dr] \arrow[r, "\sim"]& \mathrm{pr}_2^{-1}(Y'_1)\arrow[ld] \arrow[d]
\\ 
Y_1\times_{X_1}Y_1 \arrow[d, shift left,"\mathrm{pr}_1"] \arrow[d, shift right, "\mathrm{pr}_2"'] & Y'_1 \arrow[ld] \\
Y_1 \arrow[d] & \\
X_1 & 
\end{tikzcd}
\]
along a morphism $X\rightarrow X_1$. Then the theorem follows from \cite[Theorem 4.4.12]{DLLZ}.
\end{proof}

\begin{theorem}\label{thm: equivalence local system}
Let $X$ be admissibly smooth over $\spa(l, l^+)_{N_{\infty}}$. Let $\mathrm{Loc}(X_{\ket})$ denote the category of locally constant sheaves of finite sets on $X_{\ket}$. Then the functor \[\phi:X_{\fket}\rightarrow \mathrm{Loc}(X_{\ket})\] sending $Y\mapsto \mathrm{Mor}_X(\cdot, Y)$ is an equivalence of categories. Moreover
\begin{enumerate}
\item Fiber products exist in both $X_{\fket}$ and $\mathrm{Loc}(X_{\ket})$ and $\phi$ preserves fiber products.
\item Categorical quotients by finite groups exist in both categories, and $\phi$ preserves such quotients.
\end{enumerate}
\end{theorem}

\begin{proof}
This is an analogue of \cite[Theorem 4.4.15]{DLLZ}. Firstly, by Proposition \ref{prop: representable presheaves are sheaves}, $\mathrm{Mor}_X(\cdot, Y)$ is a sheaf. To see that $\mathrm{Mor}_X(\cdot, Y)$ is a locally constant sheaf, it suffices to show that every $Y\in X_{\fket}$ is, Kummer \'etale locally on $X$, a disjoint union of finitely many copies of $X$. This is clear by finding a finite model of $Y\rightarrow X$ and apply \cite[Proposition 4.1.6 \& Remark 4.1.18]{DLLZ}. Hence, the functor $\phi$ is well-defined.

The fully faithfulness of $\phi$ is clear. For the essential surjectivity, notice that every locally constant sheaf in $\mathrm{Loc}(X_{\ket})$ is Kummer \'etale locally represented by objects in $X_{\fket}$ and these objects glue to a global object in $X_{\fket}$ by Theorem \ref{thm: ket descent of finite ket morphism}.

Finally, we prove (1) and (2). We show that categorical quotients by finite groups exists in $X_{\fket}$. Indeed, suppose $Y\in X_{\fket}$ is equipped with an action of a finite group $\Gamma$. Then, $Y\rightarrow X$ admits a finite model $Y_1\rightarrow X_1$ such that the $\Gamma$-action also descends to $Y_1$. By \cite[Corollary 4.4.13]{DLLZ}, the quotient $Y_1/\Gamma$ exists and hence $Y/\Gamma$ also exists and equals to \[(Y_1/\Gamma)\times_{X_1}X.\]

It remains to prove that $\phi$ preserves fiber products and categorical quotients by finite groups. By the Kummer \'etale descent again, we may replace the source and target of $\phi$ with the category of finite disjoint unions of copies of $X$ and the category of constant sheaves of finite sets, respectively. Then the desired statements become trivial.
\end{proof}

\addtocontents{toc}{\protect\setcounter{tocdepth}{1}}
\subsection*{The Kummer \'etale fundamental group}  
Let $\underline{\mathrm{FSets}}$ denote the category of finite sets. To define the Kummer \'etale fundamental group, we need the following analogue of \cite[Lemma 4.4.16]{DLLZ}. 

\begin{lemma}\label{lemma: Galois category}
Let $X$ be admissibly smooth over $\spa(l, l^+)_{N_{\infty}}$ and assume that $X$ is connected. Let $\zeta\rightarrow X$ be a log geometric point. Then $X_{\ket}$ together with the fiber functor 
\[F: X_{\fket}\rightarrow \underline{\mathrm{FSets}}\] sending $Y\mapsto Y_{\zeta}:=\mathrm{Mor}_X(\zeta, Y)$ forms a Galois category. 
\end{lemma}

\begin{proof}
Given Theorem \ref{thm: equivalence local system}, the proof of \cite[Lemma 4.4.16]{DLLZ} applies here verbatim.
\end{proof}

\begin{corollary}\label{cor: ket fundamental group}
Let $X$, $\zeta$, and $F$ be as in Lemma \ref{lemma: Galois category} and let $\pi_1^{\ket}(X, \zeta)$ be the automorphism group of $F$. Then $F$ induces an equivalence of categories
\[
\mathrm{Loc}(X_{\ket})\xrightarrow[]{\sim} \pi_1^{\ket}(X, \zeta)-\underline{\mathrm{FSets}}
\]
sending $\mathcal{F}\mapsto \mathcal{F}_{\zeta}$. Composing with the equivalence of categories $\phi$ in Theorem \ref{thm: equivalence local system}, we obtain an equivalence of categories
\[X_{\fket}\xrightarrow[]{\sim} \pi_1^{\ket}(X, \zeta)-\underline{\mathrm{FSets}}.\]
\end{corollary}

\begin{remark}
Since stalk functors at any two log geometric points $\zeta$, $\zeta'$ are isomorphic, the resulting fundamental groups $\pi_1^{\ket}(X, \zeta)$ and $\pi_1^{\ket}(X, \zeta')$ are isomorphic. Hence, we will omit ``$\zeta$'' whenever the context is clear.
\end{remark}

\begin{example}\label{example: Kummer etale fundamental group of a point}
Suppose $X=\spa(l,l^+)_{N_{\infty}}$. Let \[\widehat{\mathbb{Z}}'(1)(l):=\varprojlim_m \mu_m(l)\] where $\mu_m(l)$ denotes the group of $m$-th roots of unity in $l$ and the limit runs through all positive integers $m$ invertible in $l$. From the construction, we have
\begin{align*}
\pi_1^{\ket}(\spa(l,l^+)_{N_{\infty}}) & \cong \varprojlim \pi_1^{\ket}(\spa(l,l^+)_N) \\ & \cong \varprojlim \mathrm{Hom}(N^{\mathrm{gp}}, \widehat{\mathbb{Z}}'(1)(l)) \\ &  \cong \mathrm{Hom}(N_{\infty}^{\mathrm{gp}}, \widehat{\mathbb{Z}}'(1)(l))
\end{align*}
where the second isomorphism follows from \cite[Corollary 4.4.22]{DLLZ}.
\end{example}

\addtocontents{toc}{\protect\setcounter{tocdepth}{2}}
\subsection{The pro-Kummer \'etale site}  Next we construct the Kummer pro-\'etale site of $X$. Throughout this section, $X$ is assumed to be admissibly smooth over a divisible log point $\spa(l,l^+)_{N_{\infty}}$. We further assume that $l$ is an extension of $\mathbb{Q}_p$.

For any category $\mathcal{C}$, let pro-$\mathcal{C}$ denote the category introduced in \cite[Proposition 3.2]{Scholze}. In particular, pro-$X_{\ket}$ consists of objects of the form $U=\varprojlim_{i\in I} U_i$, where each $U_i\rightarrow X$ is Kummer \'etale. The underlying topological space of such an object is $|U| := \varprojlim_i |U_i|$. We have the following straightforward generalizations of the pro-Kummer \'etale site introduced in \cite[Section 5]{DLLZ}.

\begin{definition}\label{defn: pro-ket morphism}
 \begin{enumerate}
        \item A morphism $U \rightarrow V$ in pro-$X_{\ket}$ is \emph{Kummer \'etale} (resp. \emph{finite Kummer \'etale}, resp. \emph{\'etale}, resp. \emph{finite \'etale}) if it is the pullback under some morphism $V \rightarrow V_0$ in pro-$X_{\ket}$ of some Kummer \'etale (resp. finite Kummer \'etale, resp. strictly \'etale, resp. strictly finite \'etale) morphism $U_0 \rightarrow V_0$ in $X_{\ket}$.

        \item A morphism $U \rightarrow V$ in pro-$X_{\ket}$ is \emph{pro-Kummer \'etale} if it can be written as a cofiltered inverse limit $U = \varprojlim_{i \in I} U_i$ of objects $U_i \rightarrow V$ Kummer \'etale over $V$ such that $U_j \rightarrow U_i$ is finite Kummer \'etale and surjective for all sufficiently large $i$ (i.e., all $i \geq i_0$, for some $i_0 \in I$).  Such a presentation $U = \varprojlim_i U_i \rightarrow V$ is called a \emph{pro-Kummer \'etale presentation}.

        \item  A morphism $U \rightarrow V$ as in (ii) is \emph{pro-finite Kummer \'etale} if all $U_i \rightarrow V$ are finite Kummer \'etale.
    \end{enumerate}
\end{definition}

\begin{definition}\label{defn: pro-ket site}
    The \emph{pro-Kummer \'etale site} $X_{\proket}$ has as underlying category the full subcategory of pro-$X_{\ket}$ consisting of objects that are pro-Kummer \'etale over $X$, and each covering of an object $U \in X_{\proket}$ is given by a family of pro-Kummer \'etale morphisms $\{ f_i: U_i \rightarrow U \}_{i \in I}$ such that $|U| = \cup_{i \in I} f_i(|U_i|)$ and such that each $f_i: U_i \rightarrow U$ can be written as an inverse limit $U_i = \varprojlim_{\mu < \lambda} U_\mu \rightarrow U$ satisfying the following conditions:
    \begin{enumerate}
        \item Each $U_\mu \in X_{\proket}$, and $U = U_0$ is an initial object in the limit.

        \item The limit runs through the set of ordinals $\mu$ less than some ordinal $\lambda$.

        \item For each $\mu < \lambda$, the morphism $U_\mu \rightarrow U_{<\mu} := \varprojlim_{\mu' < \mu} U_{\mu'}$ is the pullback of a Kummer \'etale morphism in $X_{\ket}$, and is the pullback of a surjective finite Kummer \'etale morphism in $X_{\ket}$ for all sufficiently large $\mu$.
    \end{enumerate}
\end{definition}

One checks that \cite[Lemma 5.1.4, Proposition 5.1.5]{DLLZ} remains valid in this generalization. 

We also introduce the notion of log affinoid perfectoid objects in the pro-Kummer \'etale site (cf. \cite[Definition 5.3.1]{DLLZ}).

\begin{definition}\label{defn: log affinoid perfectoid object}
An object $U$ in $X_{\proket}$ is called \emph{log affinoid perfectoid} if it admits a pro-Kummer \'etale presentation
    \[
        U = \varprojlim_{i \in I} U_i = \varprojlim_{i \in I} (\mathrm{Spa}(R_i, R_i^+), \mathcal{M}_i) \rightarrow X
    \]
    satisfying the following conditions:
    \begin{enumerate}
        \item There is an initial object $0 \in I$.

        \item Each $U_i$ admits a standard finite model and a (global) standard saturated chart $P_i$ in the sense of Definition \ref{defn: standard finite model}, and such that each transition morphism $U_j \rightarrow U_i$ is modeled on an injective chart $P_i \rightarrow P_j$ of finite type.

        \item The affinoid algebra 
        \[ (R, R^+) := \big(\varinjlim_{i \in I} \, (R_i, R_i^+)\big)^\wedge,\] where the completion is the $p$-adic completion, is a perfectoid affinoid algebra.

        \item The monoid 
        \[P := \varinjlim_{i \in I} P_i\] is divisible. In other words, $P$ is uniquely $n$-divisible for all $n \geq 1$.
    \end{enumerate}
    In this situation, we say that $U = \varprojlim_{i\in I} U_i$ is a \emph{perfectoid presentation} of the object $U$. The perfectoid space \[\widehat{U}:=\spa(R, R^+)\] is called the \emph{associated affinoid perfectoid space}. We can also equip $\widehat{U}$ with the log structure induced by the chart $P\rightarrow R$. The resulting log adic space is called the \emph{associated log affinoid perfectoid space}, still denoted by $\widehat{U}$ by slight abusing of notation.
\end{definition}

\begin{proposition}\label{prop: ket over log affinoid perfectoid is et}
Let $U=\varprojlim_{i\in I} U_i$ be a log affinoid perfectoid object in $X_{\proket}$ with associated affinoid perfectoid space $\widehat{U}$. Suppose $V\rightarrow U$ is a Kummer \'etale morphism (resp. finite Kummer \'etale morphism) in $X_{\proket}$ that is the pullback of a Kummer \'etale morphism (resp. finite Kummer \'etale morphism) $V_0\rightarrow U_0$ of affinoid log adic spaces in $X_{\ket}$. Then $V\rightarrow U$ is \'etale (resp. finite \'etale) and $V$ is a log affinoid perfectoid object. Moreover, the induced morphism $\widehat{V}\rightarrow \widehat{U}$ is \'etale (resp. finite \'etale).
\end{proposition}

\begin{proof}
Suppose $U_i$'s are modeled on the charts $P_i$'s as in Definition \ref{defn: log affinoid perfectoid object} and $P=\varinjlim_{i\in I}P_i$. We may assume that $0\in I$. By assumption, $U_0$ is affinoid and admits a standard finite model. In particular, $U_0$ admits a standard saturated chart $P_0$. For every $n\geq 1$, let $U_0^{\frac{1}{n}}$ be as in Construction \ref{construction: X^{1/n}}. It is modeled on the chart $\frac{1}{n}P_0$. By Proposition \ref{lemma: Lemma 4.2.5 DLLZ}, there exists $n$ such that \[V_0\times_{U_0}U_0^{\frac{1}{n}}\rightarrow U_0^{\frac{1}{n}}\] is \'etale (resp. finite \'etale).  
Since $P$ is uniquely $n$-divisible, the homomorphism $P_0\rightarrow P$ factors as $P_0\rightarrow\frac{1}{n}P_0\rightarrow P$. Using the fact that $P_0$ is almost $n$-divisible, there exists $i\in I$ such that $P_0\rightarrow P_i$ factors as $P_0\rightarrow\frac{1}{n}P_0\rightarrow P_i$. Hence, $V_0\times_{U_0} U_i\rightarrow U_i$ is \'etale (resp. finite \'etale). Consequently, if we replace $I$ with the cofinal full subcategory of objects that receive morphisms from $i$, we see that \[V=(V_0\times_{U_0}U_i)\times_{U_i}U\rightarrow U\] is \'etale (rsep., finite \'etale). This also implies that $\widehat{V}\rightarrow \widehat{U}$ is \'etale (resp. finite \'etale).
\end{proof}

\begin{corollary}\label{corollary: proket over log affinoid perfectoid is log affinoid perfectoid}
Let $U=\varprojlim_{i\in I} U_i$ be an object in $X_{\proket}$ such that $U_i\rightarrow X$ is a composition of rational localizations and finite Kummer \'etale morphisms, for all sufficiently large $i$. Then, for every log affinoid perfectoid object $V$ in $X_{\proket}$, the fiber product $U\times_XV$ is a log affinoid perfectoid object as well.
\end{corollary}

\begin{proof}
Given Proposition \ref{prop: ket over log affinoid perfectoid is et}, the statement follows from the proof of \cite[Corollary 5.3.9]{DLLZ} verbatim.
\end{proof}

\begin{corollary}\label{corollary: fiber product of log affinoid perfectoid objects}
The subcategory of log affinoid perfectoid objects in $X_{\proket}$ is stable under fiber products.
\end{corollary}

\begin{proof}
Given Proposition \ref{prop: ket over log affinoid perfectoid is et}, the statement follows from the proof of \cite[Proposition 5.3.11]{DLLZ} verbatim.
\end{proof}

We have the following important analogue of \cite[Proposition 5.3.12]{DLLZ}.
\begin{proposition}\label{prop: log affinoid perfectoid objects form a basis}
Suppose $X$ is admissibly smooth over $\spa(l,l^+)_{N_{\infty}}$. Then the subcategory of log affinoid perfectoid objects in $X_{\proket}$ form a basis.
\end{proposition}

\begin{proof}
We have to show that, for every $U=\varprojlim_{i\in I} U_i$ in $X_{\proket}$, \'etale locally on $U$ and $X$, there exists a pro-Kummer \'etale cover of $U$ by log affinoid perfectoid objects. We reduce to the case where 
\begin{itemize}
\item $X$ is affinoid and $X\rightarrow \spa(l,l^+)_{N_{\infty}}$ admits a standard finite model as in Definition \ref{defn: standard finite model};
\item $U_i\rightarrow X$ is a composition of rational localizations and finite Kummer \'etale morphisms for all sufficiently large $i$.
\end{itemize}
By assumption, $X\rightarrow \spa(l,l^+)_{N_{\infty}}$ admits a finite model $X_0\rightarrow \spa(l,l^+)_N$ equipped with a chart $u:N\rightarrow P$ satisfying the conditions in Proposition \ref{lemma: choose X_0 to be X} and such that $X$ admits a chart modeled on \[P_{\infty}=P\sqcup_NN_{\infty}=P\sqcup^{\mathrm{sat}}_NN_{\infty}.\] In particular, the map 
\[X_0 \lra  \mathbb{E}_{10}:=\spa(l,l^+)_N\times_{\spa(l\langle N\rangle, l^+\langle N\rangle)} \spa(l\langle P\rangle, l^+\langle P\rangle)\]
is a composition of rational localizations and strictly \'etale morphisms. The induced morphism
\[X \lra  \mathbb{E}_1:=\spa(l,l^+)_{N_{\infty}}\times_{\spa(l\langle N\rangle, l^+\langle N\rangle)} \spa(l\langle P\rangle, l^+\langle P\rangle)\]
is also a composition of rational localizations and strictly \'etale morphisms. One sees that $\mathbb{E}_1$ is also admissibly smooth over $\spa(l,l^+)_{N_{\infty}}$ and that $\mathbb{E}_{10}\rightarrow \spa(l,l^+)_N$ is a standard finite model of $\mathbb{E}_1\rightarrow \spa(l,l^+)_{N_{\infty}}$. For every integer $n\geq 1$, consider
\[\mathbb{E}_n:=\spa(l,l^+)_{N_{\infty}}\times_{\spa(l\langle \frac{1}{n}N\rangle, l^+\langle \frac{1}{n}N\rangle)} \spa(l\langle \frac{1}{n}P\rangle, l^+\langle \frac{1}{n}P\rangle).\]
Then $\mathbb{E}_n$ is just $(\mathbb{E}_1)^{\frac{1}{n}}$ in the sense of Construction \ref{construction: X^{1/n}}. In particular, $\mathbb{E}_n\rightarrow \mathbb{E}_1$ is a standard Kummer \'etale cover modeled on the chart $P_{\infty}\hookrightarrow \frac{1}{n}P_{\infty}$. 
It is also clear that \[X^{\frac{1}{n}}=X\times_{\mathbb{E}_1}\mathbb{E}_n.\]

We claim that \[\widetilde{\mathbb{E}}:=\varprojlim_n \mathbb{E}_n\in (\mathbb{E}_1)_{\proket}\] is a log affinoid perfectoid object. We know that the transition map $\mathbb{E}_m\rightarrow \mathbb{E}_n$ has a chart 
\[\frac{1}{n}P_{\infty}\rightarrow \frac{1}{m}P_{\infty}\] that is injective and of finite type because $P_{\infty}$ is almost $n$-divisible. It is also clear that $\frac{1}{n}P_{\infty}\rightarrow \frac{1}{m}P_{\infty}$ is uniquely $n$-divisible for all $n\geq 1$. Write $\mathbb{E}_n=\spa(R_n, R^+_n)$, then 
\[R_n=l\widehat{\otimes}_{l\langle \frac{1}{n}N\rangle}l\langle \frac{1}{n}P\rangle=l\widehat{\otimes}_{l\langle N_{\mathbb{Q}_{\geq 0}}\rangle}l\langle \frac{1}{n}P\sqcup_{\frac{1}{n}N}N_{\mathbb{Q}_{\geq 0}}\rangle\]
and the $p$-adic completion $R$ of $\varinjlim_n R_n$ is \[R = (\varinjlim_n R_n)^{\wedge} = 
l\widehat{\otimes}_{l\langle N_{\mathbb{Q}_{\geq 0}}\rangle}l\langle P_{\mathbb{Q}_{\geq 0}}\rangle.
\] 
Since $l$, $l\langle N_{\mathbb{Q}_{\geq 0}}\rangle$ and $l\langle P_{\mathbb{Q}_{\geq 0}}\rangle$ are all affinoid perfectoid algebras (cf. \cite[Lemma 2.2.15]{DLLZ}), we conclude that $R$ is also an affinoid perfectoid algebra by \cite[Proposition 6.18]{Scholze12}.
Now, let 
\[\widetilde{X}:=X\times_{\mathbb{E}_1}\widetilde{\mathbb{E}}\in X_{\proket}.
\] 
Then $\widetilde{X}$ is a log affinoid perfectoid object in $X_{\proket}$. Consequently, 
\[\widetilde{U}:=U\times_X \widetilde{X}\in X_{\proket}\] is a desired pro-Kummer \'etale cover of $U$ by a log affinoid perfectoid object in $X_{\proket}$ by Corollary \ref{corollary: proket over log affinoid perfectoid is log affinoid perfectoid}.
\end{proof}

\begin{proposition}\label{proposition: topos is replete}
Suppose $X$ is admissibly smooth over $\spa(l,l^+)_{N_{\infty}}$. Then the topos $\mathrm{Sh}(X_{\proket})$ is replete.
\end{proposition}

\begin{proof}
By \cite[Proposition 3.2.3]{BS}, it suffices to show that $X_{\proket}$ is locally weakly contractible; namely, every $U\in X_{\proket}$ admits a surjection $\cup_{j\in J} V_i\rightarrow U$ with $V_j$ coherent and weakly contractible in the sense of \cite[Definition 3.2.1]{BS}. To see this, we reduce to the case where 
\begin{itemize}
\item $X=\spa(A, A^+)$ is affinoid and $X\rightarrow \spa(l,l^+)_{N_{\infty}}$ admits a standard finite model as in Definition \ref{defn: standard finite model}; and
\item $U=\varprojlim_{i\in I} U_i$ with $U_i\rightarrow X$ finite Kummer \'etale surjective for all $i$.
\end{itemize}
We can construct the log affinoid perfectoid object $\widetilde{X}$ as in the proof of Proposition \ref{prop: log affinoid perfectoid objects form a basis}. Using \cite[Lemma 2.4.9]{BS}, one can construct an object $\widetilde{X}'\in X_{\proket}$ which is pro-finite-\'etale surjective over $\widetilde{X}$ such that $\widetilde{X}'$ is weakly contractible. By Corollary \ref{corollary: proket over log affinoid perfectoid is log affinoid perfectoid}, $\widetilde{X}'$ is also a log affinoid perfectoid object. Finally, take \[\widetilde{U}:=U\times_X\widetilde{X}'\in X_{\proket}.\] Then $\widetilde{U}$ is once again a log affinoid perfectoid object which is coherent by (an analogue of) \cite[Proposition 5.1.5 (4)]{DLLZ}. On the other hand, since $\widetilde{U}$ is pro-finite-\'etale surjective over a weakly contractible object $\widetilde{X}'$, it must be weakly contractible as well. Therefore, $\widetilde{U}$ is a desired coherent weakly contractible cover for $U$.
\end{proof}

We can define the structure sheaves on the pro-Kummer \'etale site (cf. \cite[Definition 5.4.1]{DLLZ}).

\begin{definition}\label{defn: structure sheaves}
Let $\upsilon: X_{\proket}\rightarrow X_{\ket}$ be the natural projection. We define the following sheaves on $X_{\proket}$.
\begin{enumerate}
\item The \emph{integral structure sheaf} is defined to be 
\[\mathcal{O}_{X_{\proket}}^+ := \upsilon^{-1}(\mathcal{O}_{X_{\ket}}^+),\] and the \emph{structure sheaf} is defined to be \[\mathcal{O}_{X_{\proket}} := \upsilon^{-1}(\mathcal{O}_{X_{\ket}}).\]
\item The \emph{completed integral structure sheaf} is defined to be \[\widehat{\mathcal{O}}_{X_{\proket}}^+ := \varprojlim_n \big(\mathcal{O}_{X_{\proket}}^+ / p^n\big),\] and the \emph{completed structure sheaf} is defined to be \[\widehat{\mathcal{O}}_{X_{\proket}} := \widehat{\mathcal{O}}_{X_{\proket}}^+[1/p].\] When the context is clear, we shall simply write $(\widehat{\mathcal{O}}^+_X, \widehat{\mathcal{O}}_X)$ instead of $(\widehat{\mathcal{O}}_{X_{\proket}}^{+}, \widehat{\mathcal{O}}_{X_{\proket}})$.
\item The \emph{tilted structure sheaves} are defined to be \[\widehat{\mathcal{O}}_{X_{\proket}}^{\flat+} := \varprojlim_\Phi \widehat{\mathcal{O}}_{X_{\proket}}^+ \cong \varprojlim_\Phi \big(\mathcal{O}_{X_{\proket}}^+ / p\big)\] and $\widehat{\mathcal{O}}_{X_{\proket}}^\flat := \varprojlim_\Phi \widehat{\mathcal{O}}_{X_{\proket}}$, where the transition morphisms $\Phi$ are the Frobenius maps given by $x \mapsto x^p$.  When the context is clear, we simply write $(\widehat{\mathcal{O}}^{\flat+}_X, \widehat{\mathcal{O}}^{\flat}_X)$ instead of $(\widehat{\mathcal{O}}_{X_{\proket}}^{\flat+}, \widehat{\mathcal{O}}_{X_{\proket}}^\flat)$.
\item We have 
\[\alpha: \mathcal{M}_{X_{\proket}} := \upsilon^{-1}(\mathcal{M}_{X_{\ket}}) \lra \mathcal{O}_{X_{\proket}}\]
and \[
\alpha^\flat: \mathcal{M}_{X_{\proket}}^\flat := \varprojlim_{a \mapsto a^p} \mathcal{M}_{X_{\proket}} \lra \widehat{\mathcal{O}}_{X_{\proket}}^\flat. \]  When the context is clear, we shall simply write $\mathcal{M}_X$ and $\mathcal{M}_X^\flat$ instead of $\mathcal{M}_{X_{\proket}}$ and $\mathcal{M}_{X_{\proket}}^\flat$, respectively. Moreover, we sometimes use $\underline{\widehat{\mathcal{O}_X}}$ to denote the structure sheaf $\widehat  {\mathcal{O}}_X$ together with the natural map $\mathcal{M}_X\rightarrow \widehat{\mathcal{O}}_X$.
\end{enumerate}
\end{definition}

The following analogue of \cite[Theorem 5.3.4]{DLLZ} remains true.

\begin{theorem}\label{thm: vanishing of cohomology on pro-Kummer etale site}
 Let $U \in X_{\proket}$ be a log affinoid perfectoid object, with associated perfectoid space $\widehat{U} = \mathrm{Spa}(R, R^+)$.  Let $(R^\flat, R^{\flat+})$ be the tilt of $(R, R^+)$.
    \begin{enumerate}
        \item  For each $n \geq 1$, we have \[\mathcal{O}_{X_{\proket}}^+(U) / p^n \cong R^+ / p^n,\] and it is canonically almost isomorphic to $(\mathcal{O}_{X_{\proket}}^+ / p^n)(U)$.

        \item  For each $n \geq 1$, we have \[H^i(U, \mathcal{O}_{X_{\proket}}^+ / p^n)^a = 0,\] for all $i > 0$.  Consequently, 
        \[H^i(U, \widehat{\mathcal{O}}_{X_{\proket}}^+)^a = 0,\] for all $i > 0$.

        \item  We have \[ \widehat{\mathcal{O}}_{X_{\proket}}^+(U) \cong R^+, \qquad \widehat{\mathcal{O}}_{X_{\proket}}(U) \cong R.\] Moreover, the ring $\widehat{\mathcal{O}}_{X_{\proket}}^+(U)$ is canonically isomorphic to the $p$-adic completion of $\mathcal{O}_{X_{\proket}}^+(U)$.

        \item We have \[ \widehat{\mathcal{O}}_{X_{\proket}}^{\flat+}(U) \cong R^{\flat+}, \qquad \widehat{\mathcal{O}}_{X_{\proket}}^\flat(U) \cong R^\flat.\]

        \item We have \[ H^i(U, \widehat{\mathcal{O}}_{X_{\proket}}^{\flat+})^a = 0,\] for all $i > 0$.
    \end{enumerate}
\end{theorem}

\begin{proof}
The proof of \cite[Theorem 5.3.4]{DLLZ} applies here verbatim as long as we replace \cite[Proposition 5.3.12, Lemma 5.3.8]{DLLZ} by Proposition \ref{prop: log affinoid perfectoid objects form a basis} and Proposition \ref{prop: ket over log affinoid perfectoid is et}, respectively.
\end{proof}

\begin{remark}
We remark that most of the results in \cite[Section 5]{DLLZ} remains true in our new context. Through out the paper, we will directly refer to \emph{loc. cit.} whenever we use such a result.
\end{remark}


\newpage 

\section{The complex $A \Omega_{\fX}^{\log}$ and the \'etale comparison}  \label{section:etale_comparison}


In this section, we define the log $\Ainf$-cohomology for admissibly smooth log $p$-adic formal schemes. We then prove a primitive comparison theorem (Theorem \ref{thm:primitive_comparison}) generalizing results from \cite{Scholze, DLLZ}, from which we deduce the \'etale comparison  for log $\Ainf$-cohomology  (Theorem \ref{thm:etale_comparison}).  
 
\subsection{The logarithmic $\Ainf$-cohomology} 
\label{ss:def_Ainf_sheaf}

We first fix the geometric setup. As in the introduction section, let $C$ be a perfectoid field containing all $m$-th power roots of unity for all $m\in \mathbb{Z}_{\geq 1}$ and let $\mO_C$ be its ring of integers. We fix a compatible choice of primitive roots of unity. Let $\ul{\mO_C} = (\alpha: N_\infty \ra \mO_C)$ be a divisible perfectoid pre-log ring over $\mO_C$. Recall that this means that $N_\infty$ is uniquely $n$-divisible for each $n \in \Z_{\ge 1}$ (cf. Definition \ref{definition:definition_pre_log_rings}). In the rest of the paper, we also assume that this pre-log structure on $\mO_C$ is split. In other words, the associated log structure on $\spf \mO_C$ satisfies 
\[N_\infty^a \cong \mO_{\spf \mO_C, \ett}^{\times} \oplus N_\infty\] (cf. Definition \ref{definition:definition_pre_log_rings}). Important examples of such split divisible perfectoid pre-log rings include 
\bi
\item  $\mO_C$ with the pre-log structure $\alpha: \Q_{\ge 0} \ra \mO_C$ sending $x \mapsto 0 $ for all  $x \ne 0 \in \Q_{\ge 0}$.  
\item $\mO_C$ with the pre-log structure $\alpha: \Q_{\ge 0} \ra \mO_C$ sending $a \mapsto p^a$ for all $a \in \Q_{\ge 0}$. (In particular, we need to fix a choice of $p^{\Q_{\ge 0}} \hookrightarrow \mO_C$.) 
\ei 
Let $\fX$ be a saturated log $p$-adic formal scheme that is admissibly smooth over $\ul{\mO_C}$ (cf. Definition \ref{definition:admissible_smooth_formal}). Its log adic generic fiber $X = \fX_\eta$ over the divisible log point $\eta = \spa (C, \mO_C)_{N_\infty}$ exists by Proposition \ref{prop:adm_sm_fiber_product_exist} and is a locally noetherian log adic space that is  admissibly smooth over $\spa (C, \mO_C)_{N_\infty}$ in the sense of Definition \ref{defn: admissiblly log smooth}. In Section \ref{sec:kummer_etale} we have constructed the Kummer pro-\'etale site of $X$. 
Now we are ready to define the logarithmic version of the complex $A \Omega_{\fX}$ considered in \cite{BMS1}. Let 
\begin{equation} \label{eq:the_site_projection}
\nu: X_{\proket} \lra \fX_{\ett}
\end{equation}
be the natural projection from the  Kummer pro-\'etale site of $X$ to the \'etale site of $\fX$ (this is equivalent as the composition of $\upsilon:   X_{\proket} \ra X_{\ket}$ and $X_{\ket} \ra \fX_{\ett}$).  

\begin{definition}\label{defn: AOmega}
Let $\fX$ be an admissibly smooth log $p$-adic formal scheme over $\ul{\mO_C}$. We define $\Ainfx$ to be the period sheaf $W(\widehat{\mO}_X^{\flat+})$ on $X_{\proket}$ and define $\widehat{\Ainfx}$ to be its derived $p$-adic completion; namely, 
\[\widehat{\Ainfx}:=R \varprojlim_n \Big(W(\widehat{\mO}_X^{\flat+})\otimes^\L_{\mathbb{Z}_p}\mathbb{Z}/p^n\mathbb{Z}\Big). \] 
Then we define the following complex of $\Ainf$-modules on $\fX_{\ett}$
$$ A \Omega_{\fX}^{\log} : = L \eta_{\mu} R \nu_* \widehat{\Ainfx},$$ 
and define the logarithmic $\Ainf$-cohomology of $\fX$ as the complex 
\[R \Gamma_{\Ainf} (\fX) :=  R \Gamma (\fX_{\ett}, A \Omega_{\fX}^{\log}) \in \mD(\Ainf). \]  
\end{definition}

\br[Change of perfectoid (pre-)log structure] \label{remark:base_change_N_Q_to_N_infty}
Suppose that $\fX$ is admissibly smooth over $\ul{\mO_C} = (\mO_C, N_\infty)$ and suppose $\fX_0$ is a finite model of $\fX$ over $(\mO_C, N)$, with its log adic generic fiber $X_0$ over $\spa ((C, \mO_C), N)^a$. Let 
\[\fX' = \fX_0 \times_{\spf (\mO_C, N)^a} \spf (\mO_C, N_{\Q \ge 0})^a
\] 
be the base change in the category of saturated log $p$-adic formal schemes, with its log adic generic fiber $X'$ over $\eta' = \spa (C,\mO_C)_{\Q \ge 0}$. Then $\fX'$ and $\fX$ (resp. $X'$ and $X$) have the same underlying formal schemes (resp. adic spaces) by Proposition \ref{lemma: choose X_0 to be X} and its proof. Moreover, base changing along \[\eta = \spa (C, \mO_C)_{N_\infty} \ra \eta' = \spa (C, \mO_C)_{N_{\Q \ge 0}}\] induces an isomorphism  
\[R \Gamma_{\Ainf} (\fX') \isom R\Gamma_{\Ainf} (\fX).\] To see this, we may choose a covering $\{V_{\alpha}'\}$ of $X'_{\proket}$ by log affinoid perfectoid objects (cf. Definition \ref{defn: log affinoid perfectoid object}) and observe that, under the saturated base change along $\eta \ra \eta'$, we get a covering $\{V_\alpha\}$ of $X_{\proket}$ by log affinoid perfectoid objects with the same underlying affinoid adic spaces.   
\er

\subsection{The primitive comparison theorem for log adic spaces} We begin to address the \'etale comparison. 
The key ingredient is the following generalization of \cite[Theorem 6.2.1]{DLLZ}. For the rest of the section, we assume that $C$ is in addition algebraically closed. We also write $\spa \ul{C} := \spa (C, \mO_C)_{N_\infty}$  and $\mO = \mO_C$ to ease notation. 

\begin{theorem}[Primitive Comparison]\label{thm:primitive_comparison} 
Suppose that $X$ is a proper log adic space which is admissibly smooth over $\spa \ul{C}$. Let $\L$ be an $\F_p$-local system on $X_{\ket}$. Then we have the following:
\begin{enumerate}
\item $H^i (X_{\ket}, \L \otimes_{\F_p} (\mO_X^{+a} /p))$ is an almost finitely generated $\mO$-module for every $i\geq 0$, and is almost zero for $i\gg 0$.
\item There is a canonical almost isomorphism
\[H^i (X_{\ket}, \L) \otimes_{\F_p} (\mO^{a}/p) \xrightarrow[]{\sim} H^i (X_{\ket}, \L \otimes_{\F_p} (\mO_X^{+a} /p))\]
of almost $\mO$-modules, for every $i\geq 0$. Consequently, $H^i (X_{\ket}, \L)$ is a finite dimensional $\F_p$-vector space for every $i\geq 0$, and $H^i (X_{\ket}, \L)=0$ for $i\gg 0$.
\item There is a canonical almost isomorphism
\[H^i (X_{\ket}, \L) \otimes_{\F_p} \mO^{\flat a} \: \xrightarrow[]{\sim} H^i (X_{\ket}, \L \otimes_{\F_p} \widehat \mO_{X}^{+\flat a})\]
of almost $\mO^{\flat}$-modules, for every $i\geq 0$. 
\end{enumerate} 
\end{theorem} 

The proof of Theorem \ref{thm:primitive_comparison} is similar to the proof of \cite[Theorem 6.2.1]{DLLZ}. We focus on where our proof differs from \emph{loc. cit}. 

Let us consider the following setup. Suppose that we have an affinoid log adic space $V$ which is admissibly smooth over $\spa \ul{C} = \spa(C,\mO_C)_{N_{\infty}}$ and admits a standard finite model in the sense of Definition \ref{defn: standard finite model}. That is, 
\[V=\spa(S_1, S^+_1)\rightarrow \spa \ul C\] admits a finite model 
\[V_0\rightarrow \spa(C,\mO_C)_N\] such that $N\rightarrow N_{\infty}$ is injective and $V_0\rightarrow \spa(C,\mO_C)_N$ admits a chart modeled on an injective saturated homomorphism of fs monoids $u:N\rightarrow P$,  such that
\begin{itemize}
\item $N$ is toric and $P$ is torsion-free;
\item the cokernel of $u^{\mathrm{gp}}: N^{\mathrm{gp}}\rightarrow P^{\mathrm{gp}}$ is torsion-free, \item the induced morphism \[V_0 \lra \mathbb{E}_{10}:=\spa(C,\mO_C)_N\times_{\spa(C\langle N\rangle, \mO_C\langle N\rangle)}\spa(C\langle P\rangle, \mO_C\langle P\rangle)\] is a composition of rational localizations and strictly \'etale morphisms. 
\end{itemize}
In particular, $V_0\rightarrow \spa(C,\mO_C)_N$ coincides with $V\rightarrow \spa \ul{C}$ on the underlying adic spaces. In fact, we also have $\mathbb{E}_{10}$ coincides with \[\mathbb{E}_1:=\spa(C, \mO_C)_{N_{\infty}}\times_{\spa(C\langle N\rangle, \mO_C \langle N\rangle)}\spa(C\langle P\rangle, \mO_C\langle P\rangle)\] on the underlying adic spaces by our assumption on $u: N\rightarrow P$. Consequently, $V\rightarrow\mathbb{E}_{1}$ coincides with $V_0\rightarrow \mathbb{E}_{10}$ on the underlying adic spaces and thus is also a composition of rational localizations and strictly \'etale morphisms.

 For every integer $n\geq 1$, let 
\[\mathbb{E}_{n0}:=\spa(C,\mO_C)_{\frac{1}{n}N}\times_{\spa(C\langle \frac{1}{n}N\rangle, \mO_C \langle \frac{1}{n}N\rangle)}\spa(C\langle \frac{1}{n}P\rangle, \mO_C\langle\frac{1}{n}P\rangle)\]
and
\[\mathbb{E}_n:=\spa(C,\mO_C)_{N_{\infty}}\times_{\spa(C\langle \frac{1}{n}N\rangle, \mO_C\langle \frac{1}{n}N\rangle)}\spa(C\langle \frac{1}{n}P\rangle, \mO_C\langle\frac{1}{n}P\rangle).\]
They fit into the following commutative diagram of saturated log adic spaces with Cartesian squares
\[ 
\begin{tikzcd} 
\mathbb{E}_n \arrow[d] \arrow[r] & \mathbb{E}_{n0} \arrow[d] & \\ 
\mathbb{E}_1 \arrow[d] \arrow[r] & \mathbb{E}_{10}^{(n)} \arrow[d] \arrow[r] &\mathbb{E}_{10} \arrow[d] \\ 
\spa(C, \mO_C)_{N_{\infty}}\arrow[r] & \spa(C, \mO_C)_{\frac{1}{n}N} \arrow[r] & \spa(C, \mO_C)_N
\end{tikzcd}
\]
where 
\[\mathbb{E}^{(n)}_{10}:=\spa(C,\mO_C)_{\frac{1}{n}N}\times_{\spa(C\langle N\rangle, \mO_C\langle N\rangle)}\spa(C\langle P\rangle, \mO_C\langle P\rangle)\]
is the (saturated) base change of $\mathbb E_{10}$ along $\spa(C, \mO_C)_{\frac{1}{n}N} \ra  \spa(C, \mO_C)_N$. 

Notice that $\mathbb{E}_n$ and $\mathbb{E}_{n0}$ are precisely ``$\mathbb{E}_1^{\frac{1}{n}}$'' and ``$\mathbb{E}_{10}^{(n)'}$'' in the notation of Construction \ref{construction: X^{1/n}}.
In particular, we have
\begin{enumerate}
\item $\mathbb{E}_{n0}\rightarrow \mathbb{E}_{10}^{(n)}$ is a finite model for $\mathbb{E}_n\rightarrow \mathbb{E}_1$. The two morphisms coincide on the underlying adic spaces.
\item $\mathbb{E}_{n0}\rightarrow \mathbb{E}_{10}^{(n)}$ is a standard Kummer \'etale morphism of locally noetherian fs adic spaces, modeled on the chart 
\[ P_n := P\sqcup_N \frac{1}{n}N=P\sqcup^{\textup{sat}}_N \frac{1}{n}N\hookrightarrow \frac{1}{n}P\] (cf. Remark \ref{remark:induced_Galois_cover_dominant}). Hence, $\mathbb{E}_n\rightarrow \mathbb{E}_1$ is a standard Kummer \'etale morphism in the sense of Definition \ref{definition: standard Kummer etale}. 
\end{enumerate}
By \cite[Proposition 4.1.6]{DLLZ} and Remark \ref{remark: n-fold fiber product of finite model}, we know that $\mathbb{E}_n\rightarrow \mathbb{E}_1$ is a finite Kummer \'etale Galois cover with Galois group 
\begin{align*}
&\Gamma_{/n}:=\Hom\Big((\frac{1}{n}P)^{\mathrm{gp}}/(P\sqcup_N^{\textup{sat}}\frac{1}{n}N)^{\mathrm{gp}}, \mu_{\infty}\Big)\\
= & \Hom\Big(\frac{1}{n}P^{\mathrm{gp}}/(P^{\mathrm{gp}}\oplus_{N^{\mathrm{gp}}}\frac{1}{n}N^{\mathrm{gp}}), \mu_{\infty}\Big)\\
\cong & \Hom\Big(\frac{1}{n}(P^{\mathrm{gp}}/N^{\mathrm{gp}})/(P^{\mathrm{gp}}/N^{\mathrm{gp}}), \mu_{\infty}\Big)\\
\cong & \Hom\Big(\frac{1}{n}(P^{\mathrm{gp}}/N^{\mathrm{gp}})/(P^{\mathrm{gp}}/N^{\mathrm{gp}}), \mu_n\Big)\\
\cong & \Hom\Big(P^{\mathrm{gp}}/N^{\mathrm{gp}}, \mu_n\Big)
\end{align*}

\begin{construction} \label{construction:E_infty_cover}
For the underlying adic spaces, we write \[\mathbb{E}_{10}=\mathbb{E}_1 =\spa(R_1, R^+_1)\] to ease notations (resp.  $\mathbb{E}_{n0} = \mathbb{E}_n=\spa(R_n, R^+_n)$). Now consider the inverse system \[\widetilde{\mathbb{E}}:=\varprojlim_n \mathbb{E}_n\in (\mathbb{E}_1)_{\proket}.
\]
Then $\widetilde{\mathbb{E}}$ is a log affinoid perfectoid perfectoid object with associated perfectoid space \[\widehat{\widetilde{\mathbb{E}}}=\spa(R, R^+)=\spa(C,\mO_C)_{N_{\infty}}\times_{\spa(C\langle N_{\mathbb{Q}_{\geq 0}}\rangle, \mO_C\langle N_{\mathbb{Q}_{\geq 0}}\rangle)}\spa(C\langle P_{\mathbb{Q}_{\geq 0}}\rangle, \mO_C\langle P_{\mathbb{Q}_{\geq 0}}\rangle)\]
(cf. the proof of Proposition \ref{prop: log affinoid perfectoid objects form a basis}). The morphism $\widetilde{\mathbb{E}}\rightarrow \mathbb{E}_1$ is a pro-fiinte Kummer \'etale Galois cover (in the sense of \cite[Definition 6.1.2]{DLLZ}) with Galois group
\begin{equation} \label{eq:def_Gamma}
\Gamma:=\varprojlim_n \Gamma_{/n}\cong \Hom(P^{\mathrm{gp}}/N^{\mathrm{gp}}, \varprojlim_n \mu_n)\cong \Hom(P^{\mathrm{gp}}/N^{\mathrm{gp}}, \widehat{\mathbb{Z}}(1)).
\end{equation}
On the other hand, we also have
\begin{align*}
\Gamma=\varprojlim_n \Gamma_{/n} & \cong \varprojlim_n \Hom\Big(\frac{1}{n}(P^{\mathrm{gp}}/N^{\mathrm{gp}})/(P^{\mathrm{gp}}/N^{\mathrm{gp}}), \mu_{\infty}\Big) \\ & \cong \Hom\Big((P_{\mathbb{Q}}^{\mathrm{gp}}/P^{\mathrm{gp}})/(N_{\mathbb{Q}}^{\mathrm{gp}}/N^{\mathrm{gp}}), \mu_{\infty}\Big)
\end{align*}
where $\mu_\infty = \varprojlim_n \mu_n$ denotes the roots of unity in $\mO_C$. 
In particular, a finite order character of $\Gamma$ can be identified with an element in \[\big(P_{\mathbb{Q}}^{\mathrm{gp}}/P^{\mathrm{gp}}\big)/\big(N_{\mathbb{Q}}^{\mathrm{gp}}/N^{\mathrm{gp}}\big).\] This fact will be used later in the section.
Furthermore, observe that $\Gamma$ fits into an exact sequence 
\[0\rightarrow \Gamma\rightarrow \Gamma_P\rightarrow \Gamma_N\rightarrow 0\] of Galois groups where \begin{align} \label{eq:def_Gamma_P}
\Gamma_P & :=\Hom(P^{\mathrm{gp}}, \widehat{\mathbb{Z}}(1))\cong \Hom(P^{\mathrm{gp}}_{\mathbb{Q}}, \mu_{\infty}) \\
\Gamma_N & :=\Hom(N^{\mathrm{gp}}, \widehat{\mathbb{Z}}(1))\cong \Hom(N^{\mathrm{gp}}_{\mathbb{Q}}, \mu_{\infty}). \label{eq:def_Gamma_N}
\end{align}
\end{construction}

We would like to understand the $\Gamma$-action on $(R, R^+)$. Firstly, we give explicit descriptions of $(R, R^+)$ and $(R_1, R^+_1)$.

\begin{proposition}\label{prop: R^+_1}
\begin{enumerate}
\item We have
\[(R_1, R^+_1)=\Big(C\widehat{\otimes}_{C\langle N\rangle}C\langle P\rangle,  \mO_C\widehat{\otimes}_{\mO_C\langle N\rangle} \mO_C\langle P\rangle\Big).\]
In particular, $\mO_C\widehat{\otimes}_{\mO_C\langle N\rangle}\mO_C\langle P\rangle$ is an integrally closed subring of $C\widehat{\otimes}_{C\langle N\rangle}C\langle P\rangle$.
\item We have 
\[(R, R^+)=\Big(C\widehat{\otimes}_{C\langle N_{\mathbb{Q}_{\geq 0}}\rangle}C\langle P_{\mathbb{Q}_{\geq 0}}\rangle, \mO_C\widehat{\otimes}_{\mO_C\langle N_{\mathbb{Q}_{\geq 0}}\rangle}\mO_C\langle P_{\mathbb{Q}_{\geq 0}}\rangle\Big).\]
In particular, \[\mO_C\widehat{\otimes}_{\mO_C\langle N_{\mathbb{Q}_{\geq 0}}\rangle}\mO_C\langle P_{\mathbb{Q}_{\geq 0}}\rangle\] is an integrally closed subring of $C\widehat{\otimes}_{C\langle N_{\mathbb{Q}_{\geq 0}}\rangle}C\langle P_{\mathbb{Q}_{\geq 0}}\rangle$.
\end{enumerate}
\end{proposition}

To prove Proposition \ref{prop: R^+_1}, we need some preparation. We start with an explicit description of the $\mO_C$-algebra \[\mO_C\otimes_{\mO_C[N_{\mathbb{Q}_{\geq 0}}]}\mO_C[P_{\mathbb{Q}_{\geq 0}}].\] According to Theorem \ref{thm: classification of pseudo-saturated}, every element $x\in P$ admits a unique  decomposition \[x=f_N(x)+g_N(x)\] with $f_N(x)\in P_{\mathbb{Q}_{\geq 0}}$ and $g_N(x)\in N_{\mathbb{Q}_{\geq 0}}$ such that $f_N(x)$ is ``minimal'' in the sense that for any $q\in N_{\mathbb{Q}_{\geq 0}}-\{0\}$, we have $f_N(x)-q\not\in P_{\mathbb{Q}_{\geq 0}}$. Consider the subset 
\[Q=Q(N, P, u):=\{f_N(x)\,|\, x\in P_{\mathbb{Q}_{\geq 0}}\}\]
of the rational polyhedral cone $P_{\mathbb{Q}_{\geq 0}}$. Notice that $f_N(\beta x)=\beta f_N(x)$ for all $x\in P$ and $\beta\in \mathbb{Q}_{\geq 0}$. Hence, if $q\in Q$, so is $\beta q$ for all $\beta\in \mathbb{Q}_{\geq 0}$.

\begin{lemma}\label{lemma: Q}
Let $u:N\rightarrow P$ be an injection such that $N$ is toric, $P$ is fs and torsion-free, $P\cap (-N)=\{0\}$, $u$ is pseudo-saturated, and the cokernel of $u^{\mathrm{gp}}$ is torsion-free. (Here we do not assume that $u$ is integral or quasi-saturated.)
\begin{enumerate}
\item Let $\partial P_{\mathbb{Q}_{\geq 0}}$ be the boundary of the rational polyhedral cone $P_{\mathbb{Q}_{\geq 0}}$. Then $Q\subset \partial P_{\mathbb{Q}_{\geq 0}}$.
\item $Q$ is the union of finitely many faces of $P_{\mathbb{Q}_{\geq 0}}$.
\end{enumerate}
\end{lemma}

\begin{proof}[Proof of Part (1)]
Choose an expression
\[P_{\mathbb{Q}_{\geq 0}}=H_{\lambda_1}\cap\cdots\cap H_{\lambda_n}\]
where each $\lambda_j: P_{\mathbb{Q}}^{\mathrm{gp}}\rightarrow \mathbb{Q}$ is a linear form and \[H_{\lambda_j}=\{x\in P_{\mathbb{Q}}^{\mathrm{gp}}\,|\, \lambda_j(x)\geq 0\}.\] 
If $x\in P_{\mathbb{Q}_{\geq 0}}-\partial P_{\mathbb{Q}_{\geq 0}}$, then $\lambda_j(x)>0$ for all $j=1,\ldots, n$. Choose any $q\in N_{\mathbb{Q}_{\geq 0}}-\{0\}$ and choose $\beta\in \mathbb{Q}_{>0}$ so that $\beta\lambda_j(q)\leq \lambda_j(x)$ for all $j$. Then $x-\beta q\in P_{\mathbb{Q}_{\geq 0}}$. This means $x\not\in Q$, as desired.
\end{proof}

\begin{proof}[Proof of Part (2)]
For Part (2), we use induction on $\mathrm{rk}_{\mathbb{Z}}(P^{\mathrm{gp}})$. The case for $\mathrm{rk}_{\mathbb{Z}}(P^{\mathrm{gp}})=1$ is clear. For general $P$, express $\partial P_{\mathbb{Q}_{\geq 0}}$ as a union of (finitely many) facets of $P_{\mathbb{Q}_{\geq 0}}$:
\[\partial P_{\mathbb{Q}_{\geq 0}}=\cup_{i\in I} P_i\]
where each $P_i$ is itself a rational polyhedral cone of dimension equal to $\mathrm{rk}_{\mathbb{Z}} (P^{\mathrm{gp}})-1$. For each $i\in I$, let $Q_i:= Q\cap P_i$. Then 
\[Q=\cup_{i\in I}Q_i\] by part (1). It suffices to show that each $Q_i$ is a union of faces of $P_i$.

To this end, let 
\[P'_i:=P_i\cap P, N'_i:=P_i\cap N=P'_i\cap N,\] and let $u_i:N'_i\hookrightarrow P'_i$ be the injection. To apply the induction hypothesis to $u_i$, we need to check that:
\begin{enumerate}
\item[(a)] $P'_i$ is a torsion-free fs  monoid and $(P'_i)_{\mathbb{Q}_{\geq 0}}=P_i$;
\item[(b)] $N'_i$ is a toric monoid and $(N'_i)_{\mathbb{Q}_{\geq 0}}=P_i\cap N_{\mathbb{Q}_{\geq 0}}$;
\item[(c)] The cokernel of $u_i^{\mathrm{gp}}$ is torsion-free;
\item[(d)] $u_i$ is pseudo-saturated;
\item[(e)] $Q_i=Q(N'_i, P'_i, u_i)$.
\end{enumerate}

For (a), notice that $P_i=P_{\mathbb{Q}_{\geq 0}}\cap H$ for some hyperplane $H$ in the $\mathbb{Q}$-vector space $P^{\mathrm{gp}}_{\mathbb{Q}}$. Since $P'_i=P\cap H$, we know that $P'_i$ is an exact submonoid of the torsion-free fs  monoid $P$ by \cite[Proposition I.2.1.16 (5)]{Ogus}, hence also fs and torsion-free by \cite[Theorem I.2.1.17 (2)]{Ogus}. The identity $(P'_i)_{\mathbb{Q}_{\geq 0}}=P_i$ is clear.

The same argument applies to (b). We have $N'_i=N\cap H$ and hence $N'_i$ is an exact submonoid of a toric monoid. Therefore, $N'_i$ is toric as well. The identity $(N'_i)_{\mathbb{Q}_{\geq 0}}=P_i\cap N_{\mathbb{Q}_{\geq 0}}$ is clear.

Part (c) is also clear as the cokernel of $u^{\mathrm{gp}}$ is torsion-free, and $(N'_i)^{\mathrm{gp}}=N^{\mathrm{gp}}\cap H$, $(P'_i)^{\mathrm{gp}}=P^{\mathrm{gp}}\cap H$ for some hyperplane $H$ in $P^{\mathrm{gp}}_{\mathbb{Q}}$. 

For (d), notice that $P'_i\cap (-N'_i)\subset P_{\mathbb{Q}_{\geq 0}}\cap (-N_{\mathbb{Q}_{\geq 0}})=\{0\}$. By Theorem \ref{thm: classification of pseudo-saturated}, we only need to check the following: every $p\in (P'_i)_{\mathbb{Q}_{\geq 0}}=P_i\cap P_{\mathbb{Q}_{\geq 0}}$ admits a unique minimal decomposition relative to the pair $(P'_i, N'_i)$ (cf. Definition \ref{defn: minimal decomposition}). For this, it suffices to show that $p=p'+q'$ with $p'\in P_i\cap P_{\mathbb{Q}_{\geq 0}}$, $q'\in P_i\cap N_{\mathbb{Q}_{\geq 0}}$ is minimal relative to $(P'_i, N'_i)$ if and only if it is minimal relative to $(N, P)$. Namely, we need to show that
\begin{multline*} 
\qquad  p'-q\not\in P_i\cap P_{\mathbb{Q}_{\geq 0}},\,\,\forall q\in P_i\cap N_{\mathbb{Q}_{\geq 0}}-\{0\}\,\, \\ \Longleftrightarrow\,\, p'-q\not\in P_{\mathbb{Q}_{\geq 0}},\,\,\forall q\in N_{\mathbb{Q}_{\geq 0}}-\{0\}. \quad 
\end{multline*}
The direction $\Leftarrow$ is clear. For the direction $\Rightarrow$, suppose $p'':=p'-q\in P_{\mathbb{Q}_{\geq 0}}$ for some $q\in N_{\mathbb{Q}_{\geq 0}}-\{0\}$. Since $q+p''\in P_i$, we must have both $q, p''\in P_i$ using the basic property of facet. This means $q\in P_i\cap N_{\mathbb{Q}_{\geq 0}}-\{0\}$ and $p'-q\in P_i\cap P_{\mathbb{Q}_{\geq 0}}$, a contradiction. This argument also proves (e).

Consequently, we can apply induction hypothesis to conclude that $Q_i$ is a finite union of faces of $P_i$, which are faces of $P$ as well. Therefore, $Q=\cup_{i\in I}Q_i$ is also a finite union of faces of $P$.
\end{proof}

Consider an $\mO_C$-algebra $\mO_C[Q]$ whose underlying $\mO_C$-module is the free module with basis $e^q$ with $q\in Q$, and the multiplication is given by 
\begin{equation} \label{eq:group_ring_mult}
e^q\cdot e^{q'}:=\alpha(g_N(q+q'))e^{f_N(q+q')},\end{equation} where $\alpha:N_{\mathbb{Q}_{\geq 0}}\rightarrow \mO_C$ is the homomorphism fixed at the beginning of the section. One checks that \[(e^q\cdot e^{q'})\cdot e^{q''}=e^q\cdot (e^{q'}\cdot e^{q''})=\alpha(g_N(q+q'+q''))e^{f_N(q+q'+q'')}\] using Remark \ref{remark: f_N}. By checking the universal property of tensor product, we obtain an identification
\[\mO_C\otimes_{\mO_C[N_{\mathbb{Q}_{\geq 0}}]}\mO_C[P_{\mathbb{Q}_{\geq 0}}] =\mO_C[Q].\]
(There is a similar identification for $C$ in place of $\mO_C$.) As a byproduct, we know that  $\mO_C\otimes_{\mO_C[N_{\mathbb{Q}_{\geq 0}}]}\mO_C[P_{\mathbb{Q}_{\geq 0}}]$ is $p$-torsion free and hence the map
\[\mO_C\otimes_{\mO_C[N_{\mathbb{Q}_{\geq 0}}]}\mO_C[ P_{\mathbb{Q}_{\geq 0}}]\rightarrow C\otimes_{C[ N_{\mathbb{Q}_{\geq 0}}]}C[ P_{\mathbb{Q}_{\geq 0}}]\]
is injective.

There is a $p$-adically completed analogue. Consider an $\mO_C$-algebra $\mO_C\langle Q\rangle$ whose underlying $\mO_C$-module is the $p$-adic completion of $\mO_C[Q]$. It consists of $p$-adically convergent (possibly infinite) formal sums $\sum_{q\in Q} a_q e^q$. The multiplication on $\mO_C\langle Q\rangle$ is still given by the formula in (\ref{eq:group_ring_mult}). 
One checks that this multiplication is well-defined and we obtain an identification
\[\mO_C\widehat{\otimes}_{\mO_C\langle N_{\mathbb{Q}_{\geq 0}}\rangle}\mO_C\langle P_{\mathbb{Q}_{\geq 0}}\rangle =\mO_C\langle Q\rangle.\]

Now we are ready to prove Proposition \ref{prop: R^+_1}. We start with Part (2). 
\begin{proof}[Proof of part (2) of Proposition \ref{prop: R^+_1}]
It follows from the definition that \[R=C\widehat{\otimes}_{C\langle N_{\mathbb{Q}_{\geq 0}}\rangle}C\langle P_{\mathbb{Q}_{\geq 0}}\rangle.\] From the discussion above, we only need to prove that $\mO_C\otimes_{\mO_C[ N_{\mathbb{Q}_{\geq 0}}]}\mO_C[ P_{\mathbb{Q}_{\geq 0}}]$ is integrally closed inside $C\otimes_{C[ N_{\mathbb{Q}_{\geq 0}}]}C[ P_{\mathbb{Q}_{\geq 0}}]$. For this, we have to show that: for any $T\in C[Q]$, integer $d\geq 1$, and any $b_0, \ldots, b_{d-1}\in \mO_C[Q]$ such that 
\[T^d+b_{d-1}T^{d-1}+\ldots+ b_1T+b_0=0\]
in $C[Q]$, we must have $T\in \mO_C[Q]$.

Write $T=\sum_{q\in Q} a_q e^q$ with $a_q\in C$. Let $v$ denote the valuation on the non-archimedean field $C$. Suppose $T\not\in \mO_C[Q]$. Then there exists $q\in Q$ such that $v(a_q)<0$. Let 
\[c:=\min_{q\in Q}v(a_q).
\] 

Choose a basis $e_1, \ldots$, $e_n$ of the $\mathbb{Q}$-vector space $P^{\mathrm{gp}}_{\mathbb{Q}}$. We define the length of an element $x\in P_{\mathbb{Q}_{\geq 0}}$ (using this fixed basis) to be 
\[||x||:=\sqrt{(c_1^2+\ldots+c_n^2)}\] where $x=c_1e_1+\ldots+c_ne_n$. Among all $q\in Q$ such that $v(a_q)=c$, choose one that has the maximal length, say $q_0$. Expanding the term $T^h$, we obtain
\[T^h=\sum_{\substack{s\geq 1, h_1, \ldots, h_s \geq 1\\ h_1+\cdots+h_s=h\\ q_1, \ldots, q_s\in Q}}\binom{h}{h_1, \ldots, h_s}\,\,a^{h_1}_{q_1}\cdots a_{q_s}^{h_s}\,\alpha\big(g_N(h_1q_1+\cdots h_sq_s)\big)e^{f_N(h_1q_1+\cdots h_sq_s)}.\]
Observe that the term with $h=d$, $s=1$, $h_1=d$, and $q_1=q_0$ is $a_{q_0}^d\,e^{dq_0}$ with $v(a_{q_0}^d)=dc$, and every other coefficient in the expansion of 
\[T^d+b_{d-1}T^{d-1}+\ldots+b_1d+b_0\] has valuation greater than or equal to $dc$. Also observe that every coefficient in \[b_{d-1}T^{d-1}+\cdots b_1T+b_0\] has valuation greater than or equal to $(d-1)c$. There must exist another term in the expansion of $T^d$ which also has valuation $dc$. Notice that $v(g_N(x))>0$ for all $x\in N_{\mathbb{Q}_{\geq 0}}-\{0\}$. Therefore, there must exist $s\geq 2$, distinct elements $q_1, \ldots, q_s\in Q$ and $h_1, \ldots, h_s\geq 1$ such that 
\bi 
\item $h_1+\ldots+h_s=d$, \item $h_1v(a_{q_1})+\ldots+h_sv(a_{q_s})=dc$, and 
\item $f_N(h_1q_1+\ldots+h_sq_s)=h_1q_1+\ldots+h_sq_s=dq_0.$
\ei 
The identity \[h_1v(a_{q_1})+\ldots+h_sv(a_{q_s})=dc\] forces $v(a_{q_1})=\cdots=v(q_s)=c$. By assumption, the length of $q_0$ is larger or equal to the length of all of $q_1, \ldots, q_s$. In this case, $h_1q_1+\ldots+h_sq_s=dq_0$ is impossible, a contradiction.
\end{proof}

Now we explain how part (2) of Proposition \ref{prop: R^+_1} implies (1).

\begin{proof}[Proof of part (1) of Proposition \ref{prop: R^+_1}]
It follows from the definition that $R_1=C\widehat{\otimes}_{C\langle N\rangle}C\langle P\rangle$. Since $\mO_C\otimes_{\mO_C\langle N\rangle}\mO_C\langle P\rangle$ (resp. $C\otimes_{C\langle N\rangle}C\langle P\rangle$) is the $p$-adic completion of $\mO_C\otimes_{\mO_C[N]}\mO_C[P]$ (resp. $C\otimes_{C[N]}C[P]$), we reduce to prove that the canonical map \[\mO_C\otimes_{\mO_C[N]}\mO_C[P]\rightarrow C\otimes_{C[N]}C[ P]\]
is injective and that $\mO_C\otimes_{\mO_C[N]}\mO_C[P]$ is integrally closed inside $C\otimes_{C[N]}C[P]$.

Since $N\rightarrow \mO_C$ factors as $N\rightarrow N_{\mathbb{Q}_{\geq 0}}\xrightarrow[]{\alpha} \mO_C$, we have
\[\mO_C\otimes_{\mO_C[N]}\mO_C[P]=\mO_C\otimes_{\mO_C[N_{\mathbb{Q}_{\geq 0}}]} \mO_C[P\sqcup_N N_{\mathbb{Q}_{\geq 0}}]\]
(and similarly for $C$ in place of $\mO_C$). Given (2), it suffices to know that the canonical maps
\[\mO_C\otimes_{\mO_C[N_{\mathbb{Q}_{\geq 0}}]} \mO_C[P\sqcup_N N_{\mathbb{Q}_{\geq 0}}]\rightarrow \mO_C\otimes_{\mO_C[N_{\mathbb{Q}_{\geq 0}}]} \mO_C[P_{\mathbb{Q}_{\geq 0}}]\]
and 
\[C\otimes_{C[N_{\mathbb{Q}_{\geq 0}}]} C[P\sqcup_N N_{\mathbb{Q}_{\geq 0}}]\rightarrow C\otimes_{C[N_{\mathbb{Q}_{\geq 0}}]} C[P_{\mathbb{Q}_{\geq 0}}]\]
are injective, and that
\begin{multline*}
  \quad  \Big(C\otimes_{C[N_{\mathbb{Q}_{\geq 0}}]} C[P\sqcup_N N_{\mathbb{Q}_{\geq 0}}]\Big)\cap \Big(\mO_C\otimes_{\mO_C[N_{\mathbb{Q}_{\geq 0}}]} \mO_C[P_{\mathbb{Q}_{\geq 0}}]\Big) \\    =\mO_C\otimes_{\mO_C[N_{\mathbb{Q}_{\geq 0}}]} \mO_C[P\sqcup_N N_{\mathbb{Q}_{\geq 0}}]. \quad
\end{multline*}

For all of these statements, it is enough to prove that $\mO_C[P\sqcup_N N_{\mathbb{Q}_{\geq 0}}]$ is a direct summand of $\mO_C[P_{\mathbb{Q}_{\geq 0}}]$ as $\mO_C[P\sqcup_N N_{\mathbb{Q}_{\geq 0}}]$-modules, and hence also a direct summand as $\mO_C[N_{\mathbb{Q}_{\geq 0}}]$-modules. To see this, we only need to show that $P\sqcup_N N_{\mathbb{Q}_{\geq 0}}$ is a direct summand of $P_{\mathbb{Q}_{\geq 0}}$ as $P\sqcup_N N_{\mathbb{Q}_{\geq 0}}$-sets in the sense of \cite[Section I.1.2]{Ogus}. By Remark \ref{remark:induced_Galois_cover_dominant}, the homomorphism $P\sqcup_N N_{\mathbb{Q}_{\geq 0}}\rightarrow P_{\mathbb{Q}_{\geq 0}}$ is indeed injective. 

It remains to show: for $x\in P_{\mathbb{Q}_{\geq 0}}$, $y\in P\sqcup_N N_{\mathbb{Q}_{\geq 0}}$, if $x+y\in P\sqcup_N N_{\mathbb{Q}_{\geq 0}}$, then in fact \[x\in P\sqcup_N N_{\mathbb{Q}_{\geq 0}}.\] Note that $P\sqcup_N N_{\mathbb{Q}_{\geq 0}}\rightarrow P_{\mathbb{Q}_{\geq 0}}$ is a direct limit of injective homomorphisms 
\[P\sqcup_N \frac{1}{n}N\rightarrow \frac{1}{n}P\] (cf. Remark \ref{remark:induced_Galois_cover_dominant}). We may assume that $x\in \frac{1}{n}P$ and $y\in P\sqcup_N \frac{1}{n}N$ for some $n$ such that $x+y\in P\sqcup_N \frac{1}{n}N$. In particular, $x\in (P\sqcup_N \frac{1}{n}N)^{\mathrm{gp}}$. Since $nx\in P\subset P\sqcup_N\frac{1}{n}N$, we obtain $x\in P\sqcup_N\frac{1}{n}N$ by the saturated-ness of $P\sqcup_N\frac{1}{n}N$. This finishes the proof.
\end{proof}

Given these preparations, to prove Theorem \ref{thm:primitive_comparison}, the main step is to understand the $\Gamma$-action on \[\mO_C\otimes_{\mO_C[N_{\mathbb{Q}_{\geq 0}}]}\mO_C[P_{\mathbb{Q}_{\geq 0}}].\] 

\bp \label{lemma: Gamma action on R^+}
There is a decomposition of $\mO_C\otimes_{\mO_C[N]}\mO_C[P]$-modules
\[\mO_C\otimes_{\mO_C[N_{\mathbb{Q}_{\geq 0}}]}\mO_C[P_{\mathbb{Q}_{\geq 0}}]=\bigoplus_{\chi}\Big(\mO_C\otimes_{\mO_C[N_{\mathbb{Q}_{\geq 0}}]}\mO_C[P_{\mathbb{Q}_{\geq 0}}]\Big)_{\chi}\]
where the direct sum runs through all finite order characters $\chi$ of $\Gamma$, and the direct summand  $\Big(\mO_C\otimes_{\mO_C[N_{\mathbb{Q}_{\geq 0}}]}\mO_C[P_{\mathbb{Q}_{\geq 0}}]\Big)_{\chi}$ stands for the submodule
\[\Big(\mO_C\otimes_{\mO_C[N_{\mathbb{Q}_{\geq 0}}]}\mO_C[P_{\mathbb{Q}_{\geq 0}}]\Big)_{\chi} :=\Big\{a\in \mO_C\otimes_{\mO_C[N_{\mathbb{Q}_{\geq 0}}]}\mO_C[P_{\mathbb{Q}_{\geq 0}}]\,|\, \gamma\cdot a=\chi(\gamma)a\Big\}.\]
Moreover, we have
\begin{enumerate}
\item $\Big(\mO_C\otimes_{\mO_C[N_{\mathbb{Q}_{\geq 0}}]}\mO_C[P_{\mathbb{Q}_{\geq 0}}]\Big)_{1}=\mO_C\otimes_{\mO_C[N]}\mO_C[P]$.
\item For every $\chi$, $\Big(\mO_C\otimes_{\mO_C[N_{\mathbb{Q}_{\geq 0}}]}\mO_C[P_{\mathbb{Q}_{\geq 0}}]\Big)_{\chi}$ is a finite $\mO_C\otimes_{\mO_C[N]}\mO_C[P]$-module.
\end{enumerate}
\ep

\begin{proof}
Consider the commutative diagram
\[
\begin{tikzcd} 
\bigoplus_{\tilde{\chi}} \mO_C[P_{\mathbb{Q}_{\geq 0}}]_{\tilde{\chi}} \arrow[r, equal] \arrow[d]
& \mO_C[P_{\mathbb{Q}_{\geq 0}}] \arrow[d, two heads] \\
\bigoplus_{\chi} \Big(\mO_C\otimes_{\mO_C[N_{\mathbb{Q}_{\geq 0}}]}\mO_C[P_{\mathbb{Q}_{\geq 0}}]\Big)_{\chi} \arrow[r, hook] & \mO_C\otimes_{\mO_C[N_{\mathbb{Q}_{\geq 0}}]}\mO_C[P_{\mathbb{Q}_{\geq 0}}]
\end{tikzcd}
\]
where the direct sum at the top left corner runs through all finite order character $\tilde{\chi}$ of $\Gamma_P$ (cf. (\ref{eq:def_Gamma_P})). The left vertical arrow is given as follows: notice that the canonical map 
\[ \mO_C[P_{\mathbb{Q}_{\geq 0}}]_{\tilde{\chi}}\rightarrow \mO_C\otimes_{\mO_C[N_{\mathbb{Q}_{\geq 0}}]}\mO_C[P_{\mathbb{Q}_{\geq 0}}]
\] is $\Gamma$-equivariant which factors as
\[\mO_C[P_{\mathbb{Q}_{\geq 0}}]_{\tilde{\chi}}\rightarrow \Big( \mO_C\otimes_{\mO_C[N_{\mathbb{Q}_{\geq 0}}]}\mO_C[P_{\mathbb{Q}_{\geq 0}}]\Big)_{\tilde{\chi}|_{\Gamma}}\hookrightarrow \mO_C\otimes_{\mO_C[N_{\mathbb{Q}_{\geq 0}}]}\mO_C[P_{\mathbb{Q}_{\geq 0}}].\]
Consequently, the injection on the bottom is an isomorphism.

To prove (1), we first claim that 
\[\bigoplus_{\tilde{\chi}|_{\Gamma}=1}\mO_C[P_{\mathbb{Q}_{\geq 0}}]_{\tilde{\chi}}=\mO_C[P\sqcup_N N_{\mathbb{Q}_{\geq 0}}]\]
where the direct sum runs over all finite order characters $\tilde{\chi}$ of $\Gamma_P$ such that $\tilde{\chi}|_{\Gamma}$ is the trivial character. Recall that 
\[\Gamma_P=\Hom (P^{\mathrm{gp}}_{\mathbb{Q}}/P^{\mathrm{gp}}, \mu_{\infty})\]
and hence we can identify a finite order character of $\Gamma_P$ with an element in $P^{\mathrm{gp}}_{\mathbb{Q}}/P^{\mathrm{gp}}$. Consider the natural map $\pi: P_{\mathbb{Q}_{\geq 0}}\rightarrow P^{\mathrm{gp}}_{\mathbb{Q}}/P^{\mathrm{gp}}$. Then we have
\[\mO_C[P_{\mathbb{Q}_{\geq 0}}]_{\tilde{\chi}}=\bigoplus_{a\in P_{\mathbb{Q}_{\geq 0}}, \pi(a)=\tilde{\chi}} \mO_C e^a.\]
Consider the natural map $N^{\mathrm{gp}}_{\mathbb{Q}}\rightarrow P^{\mathrm{gp}}_{\mathbb{Q}}/P^{\mathrm{gp}}$ whose kernel equals to $N^{\mathrm{gp}}_{\mathbb{Q}}\cap P^{\mathrm{gp}}=N^{\mathrm{gp}}$ by our assumption. It induces an injection $N^{\mathrm{gp}}_{\mathbb{Q}}/N^{\mathrm{gp}}\hookrightarrow P^{\mathrm{gp}}_{\mathbb{Q}}/P^{\mathrm{gp}}$.
Notice that $\tilde{\chi}|_{\Gamma}=1$ if and only if $\tilde{\chi}$ corresponds to an element in $P^{\mathrm{gp}}_{\mathbb{Q}}/P^{\mathrm{gp}}$ that lies in 
\[\mathrm{image}(N^{\mathrm{gp}}_{\mathbb{Q}}/N^{\mathrm{gp}}\hookrightarrow P^{\mathrm{gp}}_{\mathbb{Q}}/P^{\mathrm{gp}})=\big(P^{\mathrm{gp}}\oplus_{N^{\mathrm{gp}}}N^{\mathrm{gp}}_{\mathbb{Q}}\big)/P^{\mathrm{gp}}.\]
We arrive at
\[\bigoplus_{\tilde{\chi}|_{\Gamma}=1}\mO_C[P_{\mathbb{Q}_{\geq 0}}]_{\tilde{\chi}}=\bigoplus_{\substack{a\in P_{\mathbb{Q}_{\geq 0}}\\ \pi(a)\in \big(P^{\mathrm{gp}}\oplus_{N^{\mathrm{gp}}}N^{\mathrm{gp}}_{\mathbb{Q}}\big)/P^{\mathrm{gp}}}}\mO_C e^a.\]
Hence, we have to check
\[\big(P^{\mathrm{gp}}\oplus_{N^{\mathrm{gp}}}N^{\mathrm{gp}}_{\mathbb{Q}}\big)\cap P_{\mathbb{Q}_{\geq 0}}=P\sqcup_N N_{\mathbb{Q}_{\geq 0}}\]
where the intersection is taken in $P^{\mathrm{gp}}_{\mathbb{Q}}$. Indeed, we have 
\[P\sqcup_N N_{\mathbb{Q}_{\geq 0}}=P\sqcup^{\mathrm{sat}}_N N_{\mathbb{Q}_{\geq 0}}\subset P^{\mathrm{gp}}\oplus_{N^{\mathrm{gp}}}N^{\mathrm{gp}}_{\mathbb{Q}}\] by definition, and $P\sqcup_N N_{\mathbb{Q}_{\geq 0}}\subset P_{\mathbb{Q}_{\geq 0}}$ by Remark \ref{remark:induced_Galois_cover_dominant}. For the other direction, consider an element $x\in \big(P^{\mathrm{gp}}\oplus_{N^{\mathrm{gp}}}N^{\mathrm{gp}}_{\mathbb{Q}}\big)\cap P_{\mathbb{Q}_{\geq 0}}$. Then 
\[ mx\in P\subset P\sqcup^{\mathrm{sat}}_N N_{\mathbb{Q}_{\geq 0}}\]  for some positive integer $m$. By saturated-ness of $P\sqcup^{\mathrm{sat}}_N N_{\mathbb{Q}_{\geq 0}}$, we conclude that 
\[x\in P\sqcup^{\mathrm{sat}}_N N_{\mathbb{Q}_{\geq 0}}=P\sqcup_N N_{\mathbb{Q}_{\geq 0}}.\] This finishes the proof of the claim.

To finish the proof of (1), it remains to check that the image of $\mO_C[P\sqcup_N N_{\mathbb{Q}_{\geq 0}}]$ under the surjection 
\begin{equation} \label{eq:image_of_surj_lv} \mO_C[P_{\mathbb{Q}_{\geq 0}}]\rightarrow \mO_C\otimes_{\mO_C[N_{\mathbb{Q}_{\geq 0}}]}\mO_C[P_{\mathbb{Q}_{\geq 0}}]
\end{equation} is precisely $\mO_C\otimes_{\mO_C[N]}\mO_C[P]$. Recall the fact that $\mO_C[P\sqcup_N N_{\mathbb{Q}_{\geq 0}}]$ is a direct summand of $\mO_C[P_{\mathbb{Q}_{\geq 0}}]$ as $\mO_C[N_{\mathbb{Q}_{\geq 0}}]$-modules (cf. the second last paragraph of the proof of part (2) of Proposition \ref{prop: R^+_1}). Therefore, the image of $\mO_C[P\sqcup_N N_{\mathbb{Q}_{\geq 0}}]$ under the surjective map  
\[\mO_C[P_{\mathbb{Q}_{\geq 0}}]\rightarrow \mO_C\otimes_{\mO_C[N_{\mathbb{Q}_{\geq 0}}]}\mO_C[P_{\mathbb{Q}_{\geq 0}}]\] in (\ref{eq:image_of_surj_lv}) is 
\[\mO_C\otimes_{\mO_C[N_{\mathbb{Q}_{\geq 0}}]}\mO_C[P\sqcup_N N_{\mathbb{Q}_{\geq 0}}]=\mO_C\otimes_{\mO_C[N]}\mO_C[P],\]
as desired.

Now, we prove (2). Let $\chi$ be a finite order character of $\Gamma$. It suffices to show that 
\[\bigoplus_{\tilde{\chi}|_{\Gamma}=\chi}\mO_C[P_{\mathbb{Q}_{\geq 0}}]_{\tilde{\chi}}=\bigoplus_{\substack{a\in P_{\mathbb{Q}_{\geq 0}}\\ \pi(a)=\tilde{\chi} \textrm{ for some } \tilde{\chi} \textrm{ such that }\tilde{\chi}|_{\Gamma}=\chi}}\mO_C e^a\] is a finite $\mO_C[P\sqcup_N N_{\mathbb{Q}_{\geq 0}}]$-module. Recall that $\chi$ can be identified with an element of the group $\Big(P^{\mathrm{gp}}_{\mathbb{Q}}/P^{\mathrm{gp}}\Big)/\Big(N^{\mathrm{gp}}_{\mathbb{Q}}/N^{\mathrm{gp}}\Big)$. Since
\[\Big(P^{\mathrm{gp}}_{\mathbb{Q}}/P^{\mathrm{gp}}\Big)/\Big(N^{\mathrm{gp}}_{\mathbb{Q}}/N^{\mathrm{gp}}\Big)=\bigcup_{m\geq 1} \Big(\big(\frac{1}{m}P^{\mathrm{gp}}\oplus_{\frac{1}{m}N^{\mathrm{gp}}}N^{\mathrm{gp}}_{\mathbb{Q}}\big)/P^{\mathrm{gp}}\Big)/\Big(N^{\mathrm{gp}}_{\mathbb{Q}}/N^{\mathrm{gp}}\Big),\]
we may assume that 
\[\chi\in \Big(\big(\frac{1}{m}P^{\mathrm{gp}}\oplus_{\frac{1}{m}N^{\mathrm{gp}}}N^{\mathrm{gp}}_{\mathbb{Q}}\big)/P^{\mathrm{gp}}\Big)/\Big(N^{\mathrm{gp}}_{\mathbb{Q}}/N^{\mathrm{gp}}\Big)\] for some $m\geq 1$. Then we have 
\begin{align*}
&\{a\in P_{\mathbb{Q}_{\geq 0}}\,|\, \pi(a)=\tilde{\chi} \textrm{ for some }\tilde{\chi} \textrm{ such that } \tilde{\chi}|_{\Gamma}=\chi\}\\
\subset \,& \pi^{-1}\Big(\big(\frac{1}{m}P^{\mathrm{gp}}\oplus_{\frac{1}{m}N^{\mathrm{gp}}}N^{\mathrm{gp}}_{\mathbb{Q}}\big)/P^{\mathrm{gp}}\Big)\\
= \,& P_{\mathbb{Q}_{\geq 0}}\cap \big( \frac{1}{m}P^{\mathrm{gp}}\oplus_{\frac{1}{m}N^{\mathrm{gp}}}N^{\mathrm{gp}}_{\mathbb{Q}}\big)\\
=\, & \frac{1}{m}P\sqcup_{\frac{1}{m}N}N_{\mathbb{Q}_{\geq 0}}
\end{align*}
Choose a finite set of generators $x_1, \ldots, x_r$ of $P$. Then 
\[\frac{1}{m}P\sqcup_{\frac{1}{m}N}N_{\mathbb{Q}_{\geq 0}}= S_{\chi}+(P\sqcup_N N_{\mathbb{Q}_{\geq 0}})\] where 
\[S_{\chi}=\{\big(\sum_{i=1}^r\frac{a_i}{m}x_i, 0\big)\in \frac{1}{m}P\sqcup_{\frac{1}{m}N}N_{\mathbb{Q}_{\geq 0}}\,|\, 0\leq a_i\leq m-1 \textrm{ for all } i=1, \ldots, r\}\]
is a finite set. This is enough to conclude.
\end{proof}

Taking $p$-adic completions, we obtain:

\begin{corollary}\label{corollary: Gamma action on R^+}
There is a decomposition of $\mO_C\widehat{\otimes}_{\mO_C\langle N\rangle}\mO_C\langle P\rangle$-modules
\[\mO_C\widehat{\otimes}_{\mO_C\langle N_{\mathbb{Q}_{\geq 0}}\rangle}\mO_C\langle P_{\mathbb{Q}_{\geq 0}}\rangle
=\widehat{\bigoplus}_{\chi}\Big(\mO_C\widehat{\otimes}_{\mO_C\langle N_{\mathbb{Q}_{\geq 0}}\rangle}\mO_C\langle P_{\mathbb{Q}_{\geq 0}}\rangle\Big)_{\chi}\]
where the completed direct sum runs through all finite order characters $\chi$ of $\Gamma$, and the term $\Big(\mO_C\widehat{\otimes}_{\mO_C\langle N_{\mathbb{Q}_{\geq 0}}\rangle}\mO_C\langle P_{\mathbb{Q}_{\geq 0}}\rangle\Big)_{\chi}$ stands for the $p$-adic completion of the submodule
\[\Big(\mO_C\widehat{\otimes}_{\mO_C\langle N_{\mathbb{Q}_{\geq 0}}\rangle}\mO_C\langle P_{\mathbb{Q}_{\geq 0}}\rangle\Big)_{\chi}' =\Big\{a\in \mO_C\widehat{\otimes}_{\mO_C\langle N_{\mathbb{Q}_{\geq 0}}\rangle}\mO_C\langle P_{\mathbb{Q}_{\geq 0}}\rangle\,|\, \gamma\cdot a=\chi(\gamma)a\Big\}.\]
Moreover, we have
\begin{enumerate}
\item $\Big(\mO_C\widehat{\otimes}_{\mO_C\langle N_{\mathbb{Q}_{\geq 0}}\rangle }\mO_C\langle P_{\mathbb{Q}_{\geq 0}}\rangle\Big)_{1}=\mO_C\widehat{\otimes}_{\mO_C\langle N\rangle}\mO_C\langle P\rangle$.
\item For every $\chi$, $\Big(\mO_C\widehat{\otimes}_{\mO_C\langle N_{\mathbb{Q}_{\geq 0}}\rangle}\mO_C\langle P_{\mathbb{Q}_{\geq 0}}\rangle \Big)_{\chi}$ is a finite $\mO_C\widehat{\otimes}_{\mO_C\langle N\rangle}\mO_C\langle P\rangle$-module.
\end{enumerate}
\end{corollary}

The rest of the proof of Theorem \ref{thm:primitive_comparison} is essentially the same as \cite[Theorem 6.2.1]{DLLZ}.

\begin{proof}[Proof of Theorem \ref{thm:primitive_comparison}]
Given Proposition \ref{lemma: Gamma action on R^+}, and the fact that \cite[Lemma 6.1.7, Remark 6.1.8, Lemma 6.1.9, Lemma 6.1.11]{DLLZ} remain true in our context, one concludes that the statements of \cite[Proposition 6.1.1]{DLLZ} hold for the morphism $V\rightarrow \mathbb{E}_1$.

To globalize, notice that \cite[Lemma 6.2.4]{DLLZ} generalizes to our situation. Namely, for every integer $M\geq 2$, there exist $M$ affinoid \'etale coverings of $X$
\[\{V_h^{(M)}\}_{h=1}^m, \ldots, \{V_h^{(1)}\}_{h=1}^m\]
such that
\begin{itemize}
\item $V_h^{(N)}\subset\cdots\subset V_h^{(1)}$ is a chain of rational subsets, for every $h=1, \ldots, m$.
\item $V_h^{(j+1)}\subset \overline{V_h}^{(j+1)}\subset V_h^{(j)}$ for all $h=1, \ldots, m$ and $j=1, \ldots, M-1$, where $\overline{V_h}^{(j+1)}$ stands for the closure of $V_h^{(j+1)}$.
\item $V_{h_1}^{(1)}\times_X V_{h_2}^{(1)}\rightarrow V_{h_1}^{(1)}$ is a composition of rational localizations and finite \'etale morphisms, for all $1\leq h_1, h_2\leq m$.
\item Each $V_h^{(1)}$ admits a standard finite model. 
\end{itemize}
Given this, and replacing \cite[Lemma 6.1.6]{DLLZ} with Proposition \ref{lemma: Gamma action on R^+}, the proof of \cite[Theorem 6.2.1]{DLLZ} applies here verbatim.
\end{proof}

\subsection{\'Etale comparison for log $\Ainf$-cohomology} \noindent 

\noindent 

\bp \label{prop:almost_etale_specialization}
Suppose $C$ is algebraically closed and let $X$ be a proper log adic space which is admissibly smooth over $\eta = \spa \ul C = \spa (C, \mO_C)_{N_\infty}$. Then the canonical map 
\[ R \Gamma_{\ket} (X, \Z_p) \otimes_{\Z_p} \Ainf \lra R \Gamma_{\proket} (X, \widehat{\Ainfx}) \] 
is an almost quasi-isomorphism. 
\ep 

\bproof 
First, we claim that the natural maps 
$$ R \Gamma_{\ket} (X, \Z/p^n) \otimes_{\Z/p^n} W_n (\mO^\flat)  \lra R \Gamma_{\proket} (X, W_n (\widehat \mO^{\flat+}_X))  $$
are all almost quasi-isomorphisms. The claim follows by induction on $n$, while the case $n = 1$ follows from the primitive comparison theorem (Theorem \ref{thm:primitive_comparison}). 
Taking derived inverse limits, we have an almost quasi-isomorphism
$$ R \Gamma(X_{\ket}, \Z_p) \otimes \Ainf \aeq R \Gamma (X_{\proket}, R \varprojlim_n W_n (\widehat \mO_X^{\flat+})) =R \Gamma (X_{\proket}, \widehat{\Ainfx}).$$
\eproof 

As a corollary, the cohomology $R \Gamma (\fX_{\ett}, A \Omega^{\log}_{\fX})$ enjoys the following \'etale comparison.  
\bt[\'etale comparison] \label{thm:etale_comparison} Suppose $C$ is algebraically closed and let $\fX$ be a saturated proper log $p$-adic formal scheme that is admissibly smooth over $\spf (\ul{\mO_C})^a$. Let $X = \fX_{\eta}^{\text{ad}}$ be its log adic generic fiber over $\eta = \spa \ul C$, then there is a canonical quasi-isomorphism 
$$ R \Gamma_{\ket} (X, \Z_p) \otimes_{ \Ainf} \Ainf[\frac{1}{\mu}] \isom  R \Gamma (\fX_{\ett}, A \Omega^{\log}_{\fX}) \otimes_{ \Ainf} \Ainf[\frac{1}{\mu}] $$ 
compatible with Frobenius on both sides. 
\et  

\bproof   
By construction, we have a quasi-isomorphism 
$$ A \Omega_{\fX}^{\log} \otimes_{\Ainf} \Ainf [\frac{1}{\mu}] \cong R \nu_* \widehat{\Ainfx} \otimes_{\Ainf} \Ainf[\frac{1}{\mu}] $$
of sheaves on $\fX_{\ett}$. The almost quasi-isomorphism in Proposition \ref{prop:almost_etale_specialization} becomes a quasi-isomorphism after inverting $\mu$, thus we have 
$$  R \Gamma_{\ket} (X, \Z_p) \otimes \Ainf[\frac{1}{\mu}] \isom R \Gamma  (\fX_{\ett}, R \nu_* \widehat{\Ainfx} [\frac{1}{\mu}]) \cong R \Gamma (\fX_{\ett}, A \Omega^{\log}_{\fX}) \otimes \Ainf [\frac{1}{\mu}].$$
The claim about Frobenius compatibility follows from functoriality. 
\eproof

\newpage 

\section{The complex $\sq \Omega^{\log}_{\fX}$ and a primitive Hodge--Tate comparison}  \label{section:HT_primitive}

In this section we address the primitive form of the Hodge--Tate comparison theorem (Theorem \ref{mainthm:primitive_HT} in the introduction), which can be viewed as a form of $p$-adic Cartier isomorphism that lifts the usual Cartier isomorphism from characteristic $p$. 

\subsection{Small charts of admissibly smooth (formal) log schemes} \label{ss:small_charts} 

As in the previous section, let $\ul{\mO_C} = (\alpha: N_\infty \ra \mO_C)$ be a split divisible perfectoid  pre-log ring over $\mO_C$. Recall that this in particular implies that $N_\infty$ is uniquely $n$-divisible for each $n \in \Z_{\ge 1}$. To ease notation, we often  write $\mO$ for $\mO_C$ in the computation that follows.  

Now let $\fX$ be a quasi-fs log $p$-adic formal scheme that is admissibly smooth over $\ul{\mO_C}$ (cf. Definition \ref{definition:admissible_smooth_formal}). In other words, $\fX$ is of the form of the following base change
\[ 
\begin{tikzcd} 
  \fX  \arrow[d] \arrow[r]  &  \fX_0   \arrow[d]  
\\ 
\spf (\mO_C, N_\infty)^a \arrow[r] & \spf (\mO_C, N)^a   
\end{tikzcd}
\]
in the category of saturated log $p$-adic formal schemes, where 
\bi
\item the bottom map is induced by an injective map $N \ra N_\infty$ where $N$ is a toric monoid and the pre-log structure $\alpha_0: N \ra \mO_C$ is split, 
\item $\fX_0 \ra \spf (\mO_C, N)^a$ is a map of fs log $p$-adic formal schemes that is ($p$-completely) log smooth and pseudo-saturated.  
\ei 

By definition, \'etale locally, $\fX_0 \ra \spf (\mO_C, N)^a$ admits a chart modeled on an injective pseudo-saturated and integral map of fs monoids $u: N \ra P$. But according to Remark \ref{remark:nice_finite_model_formal}, we can do much better: \'etale locally on $\fX$, there exists a finite model $\fX_0 \ra \spf (\mO_C, N)^a$ that admits a chart modeled on an injective quasi-saturated and integral (i.e. saturated) map of fs monoids $u: N \ra P$ satisfying conditions (1) and (2) in Remark  \ref{remark:nice_finite_model_formal}. We refer to such a chart as a \emph{small chart}.\footnote{The choice of the terminology here is inspired by Faltings's notion of small affine formal schemes (cf. \cite[Definition 8.5]{BMS1}). This is different from Kato's notion of small maps of monoids (meaning that the cokernel of the associated map on the group envelops is torsion), which does not appear in this article.} More precisely,

\bd  \label{definition:locally_small}
Let $\fX$ be admissbly smooth over $\ul{\mO_C}$. We say that an affine \'etale neighborhood  \[\psi: \fU = \spf R \ra \fX\] is \emph{small} if the log affine formal scheme $(\fU, \psi^* \mM_{\fX})$ is strictly $\etale$ over a log affine formal scheme $\fU^{\square}=\spf (\ul{R^\square})^a$ where $R^{\square}=\mO \widehat{\otimes}_{\mO \gr{N}} \mO\gr{P}$ and $\ul{R^{\square}}$ is a pre-log ring of the form 
$$ \ul{R^{\square}} = ( P \sqcup_{N} N_\infty \lra \mO \widehat{\otimes}_{\mO \gr{N}} \mO\gr{P}),$$
which is obtained as the following base change 
\[
\begin{tikzcd}[column sep = 1em, row sep = 2.5em]
\mO  \gr{\ul P} \arrow[r] 
& (\mO \widehat \otimes_{\mO \gr{N}} \mO\gr{P}, P)  \arrow[r] 
& \ul{R^{\square}} 
\\
\mO \gr{\ul N}  \arrow[u, "u"] \arrow[r, "\iota"]  \arrow[ur, phantom, very near end, "\square"]
& (\mO, N) \arrow[r, "\kappa"] \arrow[u]   \arrow[ur, phantom,   near end, "\square"]
&  (\mO, N_\infty) \arrow[u]  
\end{tikzcd}
\]
in the category of saturated $p$-complete pre-log rings. 
Here  in the diagram, 
\be
\item The map $u$ is induced by an injective saturated homomorphism $u: N \hookrightarrow P$ of fs monoids such that $N$ is sharp, $P$ is torsion-free, $\textup{coker} (u^{\gp}: N^{\gp} \ra P^{\gp})$ is torsion-free and that $P\cap (-N)=\{0\}$.
\item The map $\kappa$ is induced by an injective homomorphism of monoids $N \rightarrow N_\infty$ (which gives rise to an injection $N_{\Q_{\ge 0}} := \varinjlim  \frac{1}{m} N  \hookrightarrow N_\infty$) and is identity on the underlying rings. 
The map $\iota$ is given by $e^n \mapsto \alpha_0(n)$ for all $n \in N$. 
\item The map $\fU\rightarrow \fU^{\square}:=\spf (\ul{R^\square})^a$ is a composition of rational localizations and strictly finite \'etale morphisms.
\ee
In the situation above, we say that $\ul{R^{\square}}$ and the \'etale neighborhood $\fU$ is modeled on the \emph{small chart} $u:N \ra P$ and refer to the map $R^{\square} \ra R$ as a \emph{small coordinate}. 
\ed 

\br 
\be
\item In Definition \ref{definition:locally_small}, since $u$ is injective and the cokernel of $u^{\gp}$ is torsion-free, the map 
$$ \mO \gr{\ul N} = (\mO \gr{N}, N) \lra \mO \gr{\ul P} = (\mO \gr{P}, P)$$ is ($p$-completely) log smooth.
\item In Definition \ref{definition:locally_small}, since $u:N\rightarrow P$ is saturated, the pushout $P_{\infty}:=P \sqcup_{N} N_\infty$ in the category of monoids is already saturated. 
\item By \cite[Theorem II.3.1]{Tsuji}, an injective pseudo-saturated and integral homomorphism $u: N \ra P$ of torsion-free fs monoids is (quasi-)saturated if and only if the mod $p$ map 
\[
u_{\mO/p}: (\mO/p [N], N) \ra ( \mO/p [P], P)
\]
is of Cartier type. 
\ee 
\er 

\br 
Suppose that  $\ul{R^{\square}}$ is modeled on a small chart $N \ra P$ as above, then 
the underlying ring $R^{\square}$ is $p$-torsion free, since 
\[R^\square \cong \mO \widehat \otimes_{\mO \gr{N}} \mO\gr{P},\] 
which is flat over $\mO$ by \cite[Proposition 4.1]{Kato}.  
\er 

By Remark \ref{remark:base_change_N_Q_to_N_infty}, we make the simplifying assumption that $N_{\Q \ge 0} \isom N_\infty$ without loss of generality in the rest of the section. 
 
\subsection{The complex $\sq \Omega_{\fX}^{\log}$} \noindent 

\noindent 
Let us define the complex $\sq \Omega_{\fX}^{\log} \in D(\fX_{\ett})$ discussed in the introduction. Let $X = \fX_{\eta}^{\textup{ad}}$ be the log adic space over $\eta = \spa (C, \mO_C)_{N_\infty}$, obtained as the adic generic fiber of the log $p$-adic formal scheme. Let 
$$\nu: X_{\proket} \lra \fX_{\ett}$$
be the natural projection from the pro-Kummer \'etale site of $X$ to the \'etale site of $\fX$ 
discussed in (\ref{eq:the_site_projection})  in Section \ref{ss:def_Ainf_sheaf}. 

\bd Let $\fX$ be as above. Define 
$$ \sq \Omega^{\log}_{\fX} := L \eta_{(\zeta_p - 1)} R \nu_* \OXplus $$
which is a complex in $D(\fX_{\ett}).$
\ed 

The main result we prove in this section is the following primitive form of the Hodge--Tate comparison. 

\bt \label{thm:primitive_HT}
Let $\fX$ be a quasi-fs log $p$-adic formal scheme that is admissibly smooth over $\ul{\mO_C}$. Then there are canonical isomorphisms of sheaves 
$$ {\Omega}^{i, \textup{ct}}_{\fX/\ul{\mO_C}} \isom \mH^i  (\sq \Omega^{\log}_{\fX}) \{i\}$$
for each $i \ge 0$. Here ${\Omega}^{i, \textup{ct}}_{\fX/\ul{\mO_C}} $ denotes the sheaf of continuous log differentials, i.e., the $p$-adic completion of ${\Omega}^{i}_{\fX/\ul{\mO_C}}$, and $\{i\}$ denotes the Breuil--Kisin--Fargues twist.  
\et 

To prove the theorem, we will construct canonical maps  
\begin{align*}
 \mO_{\fX} \cong \: & {\Omega}^{0, \textup{ct}}_{\fX/\ul{\mO_C}}   \lra  \mH^0 (\sq \Omega^{\log}_{\fX}) \\ 
&{\Omega}^{1, \textup{ct}}_{\fX/\ul{\mO_C}}   \lra  \mH^1 (\sq \Omega^{\log}_{\fX}) \{1\}.
\end{align*}
We shall first construct these maps \'etale locally on $\fX$. The functoriality of our construction will allow us to obtain the desired global maps. To check that these maps are  isomorphisms, we use local coordinates (supplied by the admissibly smooth assumption) and interpret these morphisms via certain continuous group cohomology, following the strategy of \cite{BMS1}. 
The theorem then follows from mulitplicativity.

\subsection{Construction of the map} \label{ss:construction_HT} 


Let us work locally on a small \'etale neighbourhood $\fU=\spf R$ of $\fX$. It follows from the definition that $\fU$ admits a chart given by the pre-log ring $\ul R = (P_{\infty} \ra R)$ where $P_{\infty}$ is the saturated monoid $P \sqcup_{N} N_\infty$.

By Proposition \ref{prop: R^+_1} and the fact that \'etale base change commutes with taking integral closures (cf. \cite[Tag 03GE]{SP}), the log adic generic fiber $U:=\fU^{\textrm{ad}}_{\eta}=\spa (R[\frac{1}{p}], R)$ admits a chart $P_{\infty}\rightarrow R[\frac{1}{p}]$. Consider a log affinoid perfectoid object 
$$ V = \varprojlim_{i \in I} V_i = \varprojlim_{i \in I} (\spa(S_i[\frac{1}{p}], S_i), \mM_i)$$
in the pro-Kummer \'etale site $U_{\proket}$ (cf. Definition \ref{defn: log affinoid perfectoid object}). Let $\widehat{V}=\spa(S[\frac{1}{p}], S)$ be the associated log affinoid perfectoid space where $S$ is the $p$-adic completion of $\varinjlim S_i$. Recall from Definition \ref{defn: log affinoid perfectoid object} that $\widehat{V}$ admits a chart $M_S\rightarrow S$ such that $M_S$ is uniquely $n$-divisible for all $n\geq 1$. We write $\ul S := (S, M_S)$ for the corresponding perfectoid pre-log ring. It follows from Corollary \ref{lemma:perfectoid_monoid}\footnote{Or, by considering the homotopy fiber sequence
\[  \widehat  \L_{\ul{\mO_C}/\Z_p} \otimes_{\mO_C} S \lra  \widehat  \L_{\ul{S}/\Z_p} \lra  \widehat  \L_{\ul{S}/\ul{\mO_C}}\]
associated to the maps of pre-log rings  $ (\Z_p, 0) \ra (\mO_C, N_\infty) \ra (S, M_S),$ and observe that $\widehat \L_{\ul S/\ul{\mO_C}} = 0$ by Corollary \ref{cor:cotangent_for_perfectoid}.} that we have a canonical isomorphism
$$ \widehat  \L_{\ul S/\Z_p} [-1] \isom  \widehat  \L_{\ul{\mO_C}/\Z_p} [-1] \otimes_{\mO_C} S \cong S \{1\}.$$ Note that $\widehat \L_{\ul S/\Z_p}$ is independent of the choice of monoid which gives rise to the log structure of $V$. 

\begin{construction} \label{construction:HT_before_L_eta}
For every log affinoid perfectoid object $V = \varprojlim_{i \in I} V_i$ in $U_{\proket}$ as above, consider a map \footnote{Notice that $\widehat  \L_{\ul{R}/\Z_p}[-1]$ only depends on $\fU$. It does not depend on the log structure on the pre-log ring $\ul{R}$.}
\[  \widehat \L_{\fU/\Z_p} [-1] := \widehat  \L_{\ul{R}/\Z_p}[-1] \lra R \Gamma (V_{\proket}, \OXplus) \{1\}\]
which is functorial in $\fU$ and $V$, given as the composition 
\[  \widehat  \L_{\ul{R}/\Z_p}[-1] \lra  \widehat  \L_{\ul{S}/\Z_p} [-1] \cong \Gamma(V_{\proket}, \OXplus)\{1\} \lra R \Gamma (V_{\proket}, \OXplus) \{1\}.\]
Taking the derived inverse limit among all such log affinoid perfectoid objects $V$ (which form a basis for $U_{\proket}$ by Proposition \ref{prop: log affinoid perfectoid objects form a basis}), we obtain a canonical map 
\begin{equation} \label{map:HT}
  \widehat  \L_{\fU /\Z_p} [-1]  =  \widehat  \L_{\ul{R}/\Z_p} [-1] \lra R \Gamma(U_{\proket}, \OXplus) \{1\}
\end{equation}
which is functorial in $\fU$. 
\end{construction}

\bl \label{lemma:HT_map_factors_via_L_eta} 
 $\widehat  \L_{\fU/\Z_p}[-1]  =\widehat  \L_{\ul R/\Z_p}[-1]$ is concentrated in degree $[0, 1]$, with 
\begin{align*}
& H^0 (  \widehat  \L_{\ul R/\Z_p} [-1]) \cong R \{1\} \\
&  H^1 (  \widehat  \L_{\ul R/\Z_p} [-1]) \cong  {\Omega}^{1, \textup{ct}}_{\ul R/\ul{\mO_C}} = 
\Gamma(\fU, {\Omega}^{1}_{\fU/\ul{\mO_C}}) 
\end{align*} 
\el  

\bproof 
The maps of pre-log rings $ (\Z_p, 0) \ra (\mO_C, N_\infty) \ra \ul R$ give rise to the homotopy fiber sequence
$$
  \widehat  \L_{\ul{\mO_C}/\Z_p}[-1] \otimes^\L_{\mO_C} R \lra 
    \widehat  \L_{\ul R/\Z_p} [-1] \lra 
      \widehat  \L_{\ul R/\ul{\mO_C}}[-1].
$$ 
By assumption, the map $\ul{\mO_C} \ra \ul R$ is an integral map of pre-log rings since $N_\infty \ra P_{\infty}$ is the pushout of the saturated homomorphism of monoids $N \ra P$.  Therefore we have 
$$  \widehat  \L_{\ul R/\ul{\mO_C}} [-1]  \cong {\Omega}^{1, \textup{ct}}_{\ul R/\ul{\mO_C}}[-1] $$ by Lemma \ref{lemma:compare_Gabber_with_Olsson}.  
As $  \widehat  \L_{\ul{\mO_C}/\Z_p}[-1] \cong \mO_C\{1\}$, we conclude that the complex $  \widehat \L_{\ul R/\Z_p} [-1]$ lives in $D^{\ge 0} (\mO_C)$ and satisfies 
$$ H^0 (  \widehat  \L_{\ul R/\Z_p} [-1]) \cong R \{1\}, \qquad H^1 (  \widehat  \L_{\ul R/\Z_p} [-1]) \cong  {\Omega}^{1, \textup{ct}}_{\ul R/\ul{\mO_C}}.$$
This proves the lemma. 
\eproof 

To proceed, let us state the following proposition. 
\bp \label{prop:HT_map_factors_via_L_eta} 
The map constructed in (\ref{map:HT}) factors canonically through 
\begin{equation} \label{map:HT2} 
\widehat \L_{\ul R/\Z_p} [-1] \lra L\eta_{(\zeta_p - 1)} R \Gamma (U_{\proket}, \OXplus) \{1\} 
\end{equation} 
\ep 

We will postpone the proof of Proposition \ref{prop:HT_map_factors_via_L_eta} to the later parts of the section (which makes use of the coordinates). Once Proposition \ref{prop:HT_map_factors_via_L_eta} is proven, in order to formulate Theorem \ref{thm:primitive_HT},  it remains to relate the right hand side of (\ref{map:HT2}) to $R \Gamma (\fU_{\ett}, \sq \Omega^{\log}_{\fX})$. 

\begin{construction}  \label{construction:map_iota}
There is a canonical ``global-to-local'' map 
\begin{equation} \label{eq:switching_L_eta_for_sq_Omega}
 c: L \eta_{(\zeta_p -1)} R \Gamma (U_{\proket}, \OXplus) \lra  R \Gamma (\fU_{\ett}, L \eta_{(\zeta_p -1)} R \nu_* \OXplus) \cong R \Gamma (\fU_{\ett}, \sq \Omega^{\log}_{\fX}).
 \end{equation}
Compose this with the morphism (\ref{map:HT2}) from Proposition \ref{prop:HT_map_factors_via_L_eta}, we obtain a functorial map 
\begin{equation} \label{map:iota}  
\iota:  \widehat \L_{\ul R/\Z_p}[-1] \lra R \Gamma(\fU_{\ett}, \sq \Omega_{\fX}^{\log})\{1\}.
 \end{equation}
Taking cohomology and observing that, by functoriality, this map is compatible under \'etale localization,  we obtain the following natural Hodge--Tate comparison maps 
\begin{align} 
\iota^0: \: &  \mO_{\fX} \lra \mH^0 (\sq \Omega^{\log}_{\fX})  \label{map:iota_0} \\
\iota^1: \: &  {\Omega}^{1, \textup{ct}}_{\ul{\fX}/\ul{\mO_C}} \lra \mH^1 (\sq \Omega^{\log}_{\fX}) \{1\} \label{map:iota_1}
\end{align}
\end{construction}

Our next goal is to show that both of these maps are isomorphisms. 

\bp \label{thm:primitive_HT_01}
Let $\fX$ be a quasi-fs log $p$-adic formal scheme that is  admissibly smooth over $\ul{\mO_C}$, then the maps constructed in \textup{(\ref{map:iota_0})} and \textup{(\ref{map:iota_1})} above are isomorphisms. 
\ep 

Let us give an outline of the proof of Theorem Proposition \ref{thm:primitive_HT_01} for the convenience of the reader. The proof follows a similar strategy as in \cite[Section 8]{BMS1}. 
\begin{outline} \label{outline:HT} 
Let $\ul R$ be as above and let $\square: R^\square \ra R$ be the small coordinate supplied by the admissibly smooth assumption.
\bi 
\item Using the small coordinate, we factor the map in (\ref{map:HT2}) through a certain complex of group cohomology $R \Gamma_{\gp}$ as 
$$ \widehat \L_{\ul R/\Z_p} [-1] \xrightarrow{\: \beta^{\square} \:} L \eta_{(\zeta_p - 1)} R \Gamma_{\gp}\{1\} \xrightarrow{\: \gamma^{\square} \:} L\eta_{(\zeta_p - 1)} R \Gamma (U_{\proket}, \OXplus) \{1\}.$$
\item We then show that $\gamma^{\square}$ (resp. $H^i (\beta^{\square})$ for $i = 0, 1$) are quasi-isomorphism (resp. isomorphisms) by carefully computing $R \Gamma_{\gp}$ using the coordinate $\square$. 
\item The local computations will also imply that the map $c$ in (\ref{eq:switching_L_eta_for_sq_Omega}) is a quasi-isomorphism.
\ei
 This will finish the proof of Proposition \ref{thm:primitive_HT_01} that $\iota^0$ and $\iota^1$ are isomorphisms. Finally we deduce Theorem \ref{thm:primitive_HT} by considering the wedge product. 
 \end{outline}

\subsection{Coordinates for locally small log schemes}    \label{ss:small_coordinates}
 
Recall that $\fU $ is a small affine \'etale neighbourhood of $\fX$ modeled on a small chart $N \ra P$ as in Definition \ref{definition:locally_small}. Let $P_{\infty}:=P \sqcup_N N_\infty$. There is a \emph{coordinate map}  
\begin{equation} \label{eq:coordinate_square}
\square: \ul R^{\square} = (R^{\square}=\mO\widehat{\otimes}_{\mO \gr{N}} \mO \gr{P} , P_{\infty}) \lra \ul R=(R,P_{\infty})
\end{equation}
which induces a strictly \'etale morphism 
\[\fU=(\spf R, P_{\infty})^a \lra \fU^{\square}=(\spf  R^\square, P_{\infty})^a .\]
Let $U = \fU_{\eta}^{\textup{ad}}$ (resp. $U^{\square} = (\fU^{\square})_{\eta}^{\textup{ad}}$) be the log adic generic fiber of $\fU$ (resp. $\fU^{\square}$). The pullback log structure on $U$ (resp. $U^{\square}$) is associated to the pre-log structure 
\[P_{\infty}\lra R^{\square} \lra R\lra R[\frac{1}{p}]\,\,\,\,\,\,\,\,\,\textrm{(resp. }P_{\infty}\lra R^{\square}\lra R^{\square}[\frac{1}{p}]\,\textrm{).}\]

\begin{construction} \label{construction:pro_Kummer_for_small}
Following Construction \ref{construction:E_infty_cover}, we obtain a log affinoid perfectoid object \footnote{$U^{\square}_{\infty}$, $U^{\square}_n$, $U^{\square}$ correspond to $\sq{\mathbb{E}}$, $\mathbb{E}_n$, $\mathbb{E}_1$ in Construction \ref{construction:E_infty_cover}, respectively.}
\[U^{\square}_{\infty}=\varprojlim U^{\square}_n\in (U^{\square})_{\proket}.\]
In particular, $U^{\square}_{\infty}\rightarrow U^{\square}$ is pro-finite Kummer \'etale Galois cover with Galois group
\[\Gamma\cong \Hom(P^{\gp}/N^{\gp}, \widehat{\mathbb{Z}}(1)).\]
By Proposition \ref{prop: R^+_1}, the associated affinoid perfectoid space is \footnote{$\widehat{U^{\square}_{\infty}}$ corresponds to $\widehat{\sq{\mathbb{E}}}$ in Construction \ref{construction:E_infty_cover}.}
\[\widehat{U^{\square}_{\infty}}=\spa(R^{\square}_{\infty}[\frac{1}{p}], R^{\square}_{\infty})\]
where \[R^{\square}_{\infty}=\mO\widehat{\otimes}_{\mO \gr{N_{\mathbb{Q}\geq 0}}} \mO \gr{P_{\mathbb{Q}\geq 0}}.\]

Base change along the coordinate map $\square$ in (\ref{eq:coordinate_square}), we obtain a log affinoid perfectoid object
\[U_{\infty}:=U^{\square}\times_{U^{\square}}U\in U_{\proket}\]
such that $U_{\infty}\rightarrow U$ remains a pro-finite Kummer \'etale Galois covering with Galois group $\Gamma$. Note that this pro-Kummer \'etale cover depends on the coordinate $\square$. The associated affinoid perfectoid space is 
\[\widehat{U_{\infty}}=\spa(R_{\infty}[\frac{1}{p}], R_{\infty})\]
where \[R_{\infty}=R^{\square}_{\infty}\widehat{\otimes}_{R^{\square}}R.\]
Since every Kummer \'etale cover of a log perfectoid object in the pro-Kummer \'etale site $U_{\proket}$ is strictly \'etale by Proposition \ref{prop: ket over log affinoid perfectoid is et}, the almost purity theorem of Faltings and Scholze still applies in this setup. In particular, we obtain an almost quasi-isomorphism (which depends on $\square$): 
\begin{equation} \label{eq:HT_gamma}
\gamma^{\square}:  R \Gamma_{\text{ct}} (\Gamma, R_\infty) \aeq R \Gamma (U_{\proket}, \OXplus).
 \end{equation}

\end{construction}

In the next subsection (Subsection \ref{ss:examples_of_Kummer_covers}) we give some examples of the construction above.  

Now we continue with our roadmap to proving Theorem \ref{thm:primitive_HT} following Outline \ref{outline:HT}. 

\begin{construction}  \label{construction:map_to_group_cohomology}
To utilize the local coordinates further, we apply the constructions in Section \ref{ss:construction_HT} to the (Cech complex of the) pro-Kummer \'etale Galois cover $U_{\infty}\rightarrow U$ supplied by the coordinate $\square$. This gives rise to a map 
\begin{equation} \label{eq:beta_square}
 \widehat \L_{\ul R/\Z_p} [-1] \lra R \Gamma_{\textup{ct}} (\Gamma, R_\infty )\{1\}
\end{equation}
such that its composition with $\gamma^{\square}$ in (\ref{eq:HT_gamma}) is precisely the map obtained in (\ref{map:HT}) in Construction \ref{construction:HT_before_L_eta}. An equivalent way to obtain the map in (\ref{eq:beta_square}) is via 
$$ \widehat \L_{\ul R/\Z_p} \lra R \Gamma_{\text{ct}} (\Gamma, \widehat \L_{\ul{R_\infty}/\Z_p}) \isom R \Gamma_{\text{ct}} (\Gamma, R_\infty [1]) \{1\} $$
where the first map is induced from the $\Gamma$-equivariant map $\widehat \L_{\ul R/\Z_p} \ra \widehat \L_{\ul R_\infty /\Z_p}.$

Note that both $H^0_{\textup{ct}}(\Gamma, R_\infty)$ and $H^0(U_{\proket}, \OXplus)$ are $(\zeta_p - 1)$-torsion free (as they both embed into $R_\infty$), hence there are canonical maps 
\begin{align*}
&  L\eta_{(\zeta_p - 1)} R \Gamma_{\textup{ct}} (\Gamma, R_\infty )   \lra R \Gamma_{\textup{ct}} (\Gamma, R_\infty ) \\
& L \eta_{(\zeta_p - 1)} R \Gamma (U_{\proket}, \OXplus)  \lra 
R \Gamma (U_{\proket}, \OXplus)
\end{align*} 
by \cite[Lemma 6.10]{BMS1}. Following the notation from \cite{BMS1} , we write 
$$ \sq \Omega_{\ul R}^{\square, \gp} := L \eta_{(\zeta_p - 1)} R \Gamma_{\text{ct}} (\Gamma, R_\infty), \qquad \sq \Omega_{\ul R}^{\proket}: =   L \eta_{(\zeta_p - 1)} R \Gamma (U_{\proket}, \OXplus)  $$ 
for simplicity. As notation suggests, $ \sq \Omega_{\ul R}^{\square, \gp}$ depends on the choice of coordinate $\square$ (since our choice of the Kummer pro-\'etale cover $R_\infty$ depends on $\square$). 

\end{construction}

Proposition \ref{prop:HT_map_factors_via_L_eta} will follow from the following.
 
\bp \label{prop:HT_factors_group}
Suppose that $\ul R$ is as above and fix a coordinate $\square: \ul{R^\square} \ra \ul R$ as in (\ref{eq:coordinate_square}). 
\be 
\item The map in (\ref{eq:beta_square}) factors canonically as the composition 
$$ \widehat  \L_{\ul R/\Z_p} [-1] \xrightarrow{ \: \beta^{\square} \: }  \sq \Omega_{\ul R}^{\square, \gp} \{1\} =  L\eta_{(\zeta_p - 1)} R \Gamma_{\textup{ct}} (\Gamma, R_\infty )\{1\}  \lra R \Gamma_{\textup{ct}} (\Gamma, R_\infty )\{1\}.$$
As a result, we obtain the desired map in (\ref{map:HT2}) as the composition
\begin{equation} \label{equation:local_alpha_beta}
\widehat \L_{\ul R/\Z_p} [-1] \xrightarrow{ \: \beta^{\square} \: } L\eta_{(\zeta_p - 1)} R \Gamma_{\textup{ct}} (\Gamma, R_\infty)\{1\} \xrightarrow{\: \gamma^{\square}  \:} L \eta_{(\zeta_p - 1)} R \Gamma (U_{\proket}, \OXplus) \{1\} = \sq \Omega_{\ul R}^{\proket} \{1\}. 
\end{equation}
This factors the canonical map in (\ref{map:HT}). 
\item The map $\beta^{\square}$ in (\ref{equation:local_alpha_beta}) induces isomorphisms on $H^0$ and $H^1$. 
\ee
\ep

\br
In the proposition above, the maps $\beta^{\square}$ and $\gamma^{\square}$ both depend on  $\square$ (to even make sense of them), but the composition $\gamma^{\square} \circ \beta^{\square}$ does not. 
\er 

The primitive Hodge--Tate comparison also relies on the following.

\bp \label{prop:primitive_HT_local} Let $\ul R$ and $\square: \ul{R^{\square}}\lra \ul{R}$ be as above. 
\be
\item The map  $\gamma^{\square}$ in (\ref{equation:local_alpha_beta}) is a quasi-isomorphism (not just in the almost sense). 
\item The map 
$$c : \sq \Omega_{\ul R}^{\proket}  \lra \sq \Omega^{\log}_{\ul R} :=  R \Gamma (\fU_{\ett}, \sq \Omega_{\fX}^{\log})  $$
in (\ref{eq:switching_L_eta_for_sq_Omega}) is a quasi-isomorphism.  
\ee 
\ep 

The proofs of both Proposition \ref{prop:HT_factors_group} and \ref{prop:primitive_HT_local} rely on local computations and fully utilize the choice of the coordinate $\square$, which we finish in the end of this section (Subsection \ref{ss:end_of_prop_local_HT}). This would finish the proof of Proposition \ref{thm:primitive_HT_01}.

\subsection{Examples of pro-Kummer \'etale covers} \label{ss:examples_of_Kummer_covers}
We give some examples before proceeding to compute the continuous group cohomology $R \Gamma_{\textup{ct}} (\Gamma, R_\infty)$.

\begin{example}[Log free algebras]\label{example: log free algebras} \noindent 

\noindent 
Suppose that $\ul \mO_C = (\mO_C, \Q_{\ge 0})$ with pre-log structure  $N_{\infty}=\Q_{\ge_0} \xrightarrow{\alpha \mapsto p^\alpha} \mO_C$. (In particular, this requires the existence of an embedding $p^{\mathbb{Q}_{\geq 0}}\hookrightarrow \mO_C$.) 

Consider the pre-log ring 
$$ \ul{R^{\square}}:=\big( R^\square =  \mO_C \gr{T_1, ..., T_d}, P' = \Q_{\ge 0} \oplus \N^d   \big),$$
where the pre-log structure is given by $\Q_{\ge_0} \xrightarrow{\alpha \mapsto p^\alpha} \mO_C$ and $\N \xrightarrow{1 \mapsto T_i} R^{\square}$ for each copy of $\N$. Suppose there is a strictly \'etale morphism
\[\fU=\spf (\ul{R})^{a}\rightarrow \fU^{\square}=\spf(\ul{R^{\square}})^{a}\]
which is a composition of rational localizations and finite \'etale morphisms. Then $\fU$ is admissibly smooth over $\ul{\mO_C}$ with small chart 
$$N = \N \xrightarrow{1 \mapsto (1, 0, ..., 0)}  P= \N^{d+1}$$ and small coordinate $\square: \ul{R^{\square}}\rightarrow \ul{R}$. Note that in this case, $P_{\infty}=N_\infty \sqcup_{N} P$ is precisely $P' = \Q_{\ge 0} \oplus \N^{d}$. Let $U$ and $U^{\square}$ be the log adic generic fibers of $\fU$ and $\fU^{\square}$, respectively.

For each $m\in \mathbb{Z}_{\geq 1}$, consider the pre-log ring
$$ \ul{R_m^{\square}}=(R^{\square}_m, P'_m):=\big(\mO_C \gr{T_1^{\frac{1}{m}}, ..., T_d^{\frac{1}{m}}}, \,\, \Q_{\ge 0} \oplus \Big(\frac{1}{m} \N \Big)^d \big)$$
where the $P'_m\rightarrow R^{\square}_m$ is induced from $\frac{1}{m} \mapsto T_i^{\frac{1}{m}}$ on the $i$-th copy of $\frac{1}{m}\N$. Pullback along $\square$, we consider 
\[\ul{R_m}=(R_m:=R_m^{\square}\widehat{\otimes}_{R^{\square}}R, P'_m).\]
Let $U^{\square}_m$ and $U_m$ be the log adic generic fibers of $\spf(\ul{R_m^{\square}})^a$ and $\spf(\ul{R_m})^a$, respectively. Then 
\[U^{\square}_{\infty}:=\varprojlim U_m^{\square}\rightarrow U^{\square}\]
and 
\[U_{\infty}:=\varprojlim U_m\rightarrow U\]
are both pro-finite Kummer \'etale Galois covering with Galois group $\Gamma\cong \big(\widehat{\mathbb{Z}}(1)\big)^d$. The associated affinoid perfectoid spaces are
\[\widehat{U_{\infty}^{\square}}=\spa(R_{\infty}^{\square}[\frac{1}{p}], R^{\square}_{\infty})\]
where
\[R_{\infty}^{\square}=\mO_C \gr{T_1^{\frac{1}{\infty}}, ..., T_d^{\frac{1}{\infty}}},\]
and 
\[\widehat{U_{\infty}}=\spa(R_{\infty}[\frac{1}{p}], R_{\infty})\]
where
\[R_{\infty}=R\widehat{\otimes}_{\mO_C\gr{T_1, ..., T_d}}\mO_C \gr{T_1^{\frac{1}{\infty}}, ..., T_d^{\frac{1}{\infty}}}.\]
The natural generators $\gamma_i$ of $\Gamma\cong \big(\widehat{\mathbb{Z}}(1)\big)^d$ acts on $R_{\infty}^{\square}$ and $R_{\infty}$ by 
$$ T_i^{\frac{l}{m}} \mapsto \zeta_{m}^l T_i^{\frac{l}{m}}, \qquad T_j^{\frac{l}{m}} \mapsto T_j^{\frac{l}{m}} \text{ for } j \ne i $$
for each $i = 1, ..., d$. (Recall from the introduction that we have fixed a compatible choice of roots of unity $\zeta_m$ for all $m \in \Z_{\ge 1}$ in $\mO_C$). 
\end{example}

\beg[Standard semistable reduction] \noindent 

\noindent 
Suppose that $\ul \mO_C = (\mO_C, \Q_{\ge 0})$ with pre-log structure  $N_{\infty}=\Q_{\ge_0} \xrightarrow{\alpha \mapsto p^\alpha} \mO_C$. Consider the pre-log ring
\[\ul{R^{\square}}:=(R^{\square}, P')=\big(\mO_C \gr{X_1, ..., X_d}/(X_1 \cdots X_d - p^a,\,\,\,\Q_{\ge 0}\sqcup_{\N} \N^d  \big)\]
where
\bi
\item $a \in \Q_{>0 }$.
\item $P'=\Q_{\ge 0}\sqcup_{\N} \N^d $ is the pushout of the diagonal map $\N \xrightarrow{\Delta} \N^{d}$ along $\N \xrightarrow{1 \mapsto a} \Q_{\ge 0}$.
\item $P'\rightarrow R^{\square}$ is induced from $\Q_{\ge 0} \xrightarrow{\alpha \mapsto p^\alpha} \mO_C$ and $\N^{d} \xrightarrow{1_j \mapsto X_j} R^{\square}$ for each $j$.
\ei
Suppose there is a strictly \'etale morphism
\[\fU=\spf (\ul{R})^{a}\rightarrow \fU^{\square}=\spf(\ul{R^{\square}})^{a}\]
which is a composition of rational localizations and finite \'etale morphisms. Then $\fU$ is admissibly smooth over $\ul{\mO_C}$ with small chart given by the diagonal map
$$N = \N \hookrightarrow P= \N^d$$ and small coordinate $\square: \ul{R^{\square}}\rightarrow \ul{R}$. Note that $P_{\infty}=N_\infty \sqcup_{N} P$ is precisely $P'$. Let $U$ and $U^{\square}$ be the log adic generic fibers of $\fU$ and $\fU^{\square}$, respectively.

For each $m\in \mathbb{Z}_{\geq 1}$, consider the pre-log ring $\ul{R^{\square}_m}=(R^{\square}_m, P'_m)$ where
\bi
\item $R_m^{\square}= \mO_C\gr{X_1^{\frac{1}{m}}, ..., X_d^{\frac{1}{m}}}/(X_1^{\frac{1}{m}} \cdots X_d^{\frac{1}{m}} - p^{\frac{a}{m}}) $
\item $P_m' = (\frac{1}{m} \N)^d \sqcup_{\frac{1}{m}\N} \Q_{\ge 0}$.
\ei 
Let $U_m^{\square}$ be the log adic generic fiber of $\spf(\ul{R^{\square}_m})^a$. Then 
\[U^{\square}_{\infty}:=\varprojlim U^{\square}_m\rightarrow U^{\square}\]
is a pro-finite Kummer \'etale Galois cover with Galois group
$$\Gamma := \ker (s: \widehat \Z(1)^d \ra \widehat \Z(1)) \cong \widehat \Z(1)^{d-1}$$
where $s$ denotes the summation map. Let $e_1, ..., e_d$ be topological generators of $\widehat \Z(1)^d$ that corresponds to the basis for $(P_{\mathbb{Q}_{\geq 0}})^{\gp} \cong \Q^d$, then we may choose 
$$\{\gamma_i := e_i - e_d\}_{i = 1, ..., d-1}$$ as a set of topological generators of $\Gamma$.
Pullback along $\square$, we may define $U$, $U_m$, $U_{\infty}$ as in Example \ref{example: log free algebras}. Then 
\[\widehat{U_{\infty}^{\square}}=\spa(R_{\infty}^{\square}[\frac{1}{p}], R^{\square}_{\infty})\]
where
\[R_{\infty}^{\square}=\mO_C\gr{X_1^{\frac{1}{\infty}}, ..., X_d^{\frac{1}{\infty}}}/(X_1^{\frac{1}{\infty}} \cdots X_d^{\frac{1}{\infty}} - (p^a)^{\frac{1}{\infty}}) ,\]
and 
\[\widehat{U_{\infty}}=\spa(R_{\infty}[\frac{1}{p}], R_{\infty})\]
where
\[R_{\infty}=R\widehat{\otimes}_{R^{\square}}\mO_C\gr{X_1^{\frac{1}{\infty}}, ..., X_d^{\frac{1}{\infty}}}/(X_1^{\frac{1}{\infty}} \cdots X_d^{\frac{1}{\infty}} - (p^a)^{\frac{1}{\infty}}).\]
The generator $\gamma_i$ acts on $R_\infty^{\square}$ (and thus on $R_\infty$) via 
$$ X_i^{\frac{l}{m}} \mapsto \zeta_{m}^{l} X_i^{\frac{l}{m}}, \qquad X_d^{\frac{l}{m}} \mapsto \zeta_{m}^{\;-l} X_d^{\frac{l}{m}},  \qquad X_j^{\frac{l}{m}} \mapsto X_j^{\frac{l}{m}} \text{ for }  j \ne i, d. $$
\eeg

\beg[Horizontal normal crossing divisor]
\noindent 

\noindent This example is similar to the one above. This time $\ul \mO_C = (\mO_C, \Q_{\ge 0})$ where the  pre-log structure sends each nonzero $\alpha \in \Q_{> 0} \longmapsto 0$. Consider pre-log ring $\ul{R^{\square}}=(R^{\square}, P')$ where 
\bi
\item $ R^\square =  \mO_C \gr{X_1, ..., X_d}/(X_1 \cdots X_d)$. 
\item $P' = \N^d \sqcup_{\N} \Q_{\ge 0}$ is the pushout of the diagonal $\N \xrightarrow{\Delta} \N^{d}$ along the natural inclusion $\N \hookrightarrow \Q_{\ge 0}$. 
\item $P'\rightarrow R^{\square}$ is induced from $\Q_{\ge 0} \ra \mO_C$ and $\N^d \xrightarrow{1_j \mapsto X_j} R^{\square}$.
\ei 
For each $m\in \mathbb{Z}_{\geq 1}$, consider
$$ \ul{R_m^{\square}} = (R_m^{\square}, P_m') = \Big( \mO_C\gr{X_1^{\frac{1}{m}}, ..., X_d^{\frac{1}{m}}}/(X_1^{\frac{1}{m}} \cdots X_d^{\frac{1}{m}}),  (\frac{1}{m} \N)^d \sqcup_{\frac{1}{m} \N} \Q_{\ge 0} \Big). $$
Following the same construction as in the previous examples, we may define $U^{\square}, U^{\square}_m, U^{\square}_{\infty}, \widehat{U^{\square}_{\infty}}, R^{\square}_{\infty}$, as well as $U, U_m, U_{\infty}, \widehat{U_{\infty}}, R_{\infty}$. In particular, 
\[R_{\infty}^{\square}=\mO_C\gr{X_1^{\frac{1}{\infty}}, ..., X_d^{\frac{1}{\infty}}}/(X_1^{\frac{1}{\infty}} \cdots X_d^{\frac{1}{\infty}})\]
and
\[R_{\infty}=R\widehat{\otimes}_{R^{\square}}\mO_C\gr{X_1^{\frac{1}{\infty}}, ..., X_d^{\frac{1}{\infty}}}/(X_1^{\frac{1}{\infty}} \cdots X_d^{\frac{1}{\infty}}).\]
This time 
$$\Gamma = \ker (s: \widehat \Z(1)^d \ra \widehat \Z(1)) \cong \widehat \Z(1)^{d-1}.$$  Let $\{\gamma_i := e_i - e_d\}_{i = 1, ..., d-1}$ be a set of topological generators of the Galois group, with $e_i$ the generators of $\widehat \Z(1)^d$, then again, $\gamma_i$ acts on $R_\infty^{\square}$ (and thus on $R_\infty$) via 
$$ X_i^{\frac{l}{m}} \mapsto \zeta_{m}^{l} X_i^{\frac{l}{m}}, \qquad X_d^{\frac{l}{m}} \mapsto \zeta_{m}^{\;-l} X_d^{\frac{l}{m}},  \qquad X_j^{\frac{l}{m}} \mapsto X_j^{\frac{l}{m}} \text{ for }  j \ne i, d. $$ 
\eeg

\subsection{Local computations}  \label{ss:local_computation_mod_xi} 

Retain all notations from the \S \,\ref{ss:small_charts}-\ref{ss:small_coordinates}. Also recall from Section \ref{section:etale_comparison} that $\Gamma$ fits into a short exact sequence
\[0\rightarrow \Gamma\rightarrow \Gamma_P\rightarrow \Gamma_N\rightarrow 0\]
where $\Gamma_P = \Hom (P^{\gp}, \widehat \Z (1))$ and $\Gamma_N = \Hom (N^{\gp}, \widehat \Z (1))$. We start with the computation of the continuous group cohomology $R \Gamma_{\textup{ct}} (\Gamma, R_\infty)$. 

\bl \label{lemma:decomposition}
We have a decomposition 
$$ R_\infty = \widehat \oplus_{\chi} R_{\infty, \chi} $$
where the completed direct sum runs over all finite order characters $\chi$ of $\Gamma$, and $R_{\infty, \chi}$ denotes the $p$-completion of the submodule $R_{\infty, \chi}' = \{a \in R_\infty | \gamma \cdot a = \chi(\gamma) a\}$. In particular, we have a $\Gamma$-equivariant decomposition 
$$ R_\infty = R \oplus \big(\widehat \oplus_{\chi \ne 1} R_{\infty, \chi}\big).$$
\el 

\bproof 
It suffices to prove the statement for $R_\infty^{\square}$ in place of $R_\infty$, which follows from Proposition \ref{lemma: Gamma action on R^+} and Corollary \ref{corollary: Gamma action on R^+}. 
\eproof

\bl \label{lemma:group_coho_local}
\be
\item The inclusion 
\[R = R_{\infty, \chi = 1} \hookrightarrow R_\infty\] induces a split injection 
$$ H^i_{\textup{ct}} (\Gamma, R)  \injlra H^i_{\textup{ct}} (\Gamma, R_\infty) $$
for each $i \ge 0$, with its cokernel entirely $(\zeta_p - 1)$-torsion. In particular, we have a  quasi-isomorphism 
\[L \eta_{(\zeta_p - 1)} R\Gamma_{\textup{ct}} (\Gamma, R) \isom  L \eta_{(\zeta_p - 1)} R\Gamma_{\textup{ct}} (\Gamma, R_\infty).\]
\item For each $i \ge 0$, the abelian groups $H^i_{\textup{ct}} (\Gamma, R_\infty)$ and 
\[H^i_{\textup{ct}} (\Gamma, R_\infty)/(\zeta_p - 1)\] have no (nontrivial) almost zero elements. Slightly more generally, for any $b \in \mO$, the $\mO$-modules $R_\infty/b$ and $H^i_{\text{ct}}(\Gamma, R_\infty/b)$ have no nontrivial almost zero elements. 
\ee
\el   

\bproof 

The first part of (1) follows from the decomposition in Lemma \ref{lemma:decomposition}. It then suffices to show that for each nontrivial primitive character $\chi: \Gamma \ra \mu_m$, the group cohomology 
\[H^i_{\textup{ct}} (\Gamma, R_{\infty, \chi})\] is $(\zeta_m - 1)$-torsion (note that $(\zeta_m - 1)$ is a unit when $p \nmid m$). For this last claim, the same proof of \cite[Lemma 5.5]{Scholze} applies (also compare to \cite[Lemma 6.1.7]{DLLZ}). Namely, upon choosing topological generators $(\gamma_1, ..., \gamma_d)$ of $\Gamma \cong \widehat \Z(1)^d$ for $d = \text{rk}_\Z P^{\gp} - \text{rk}_\Z N^{\gp}$, the group cohomology $R \Gamma_{\textup{ct}} (\Gamma, R_{\infty,\chi})$ is computed by the Koszul complex 
$$ K_{R_{\infty, \chi}} (\gamma_1 - 1, ..., \gamma_d - 1) $$
which is the tensor product of the complexes $R_{\infty, \chi} \xrightarrow{\gamma_i - 1} R_{\infty, \chi}$. The assertion thus follows.  

Now let us prove (2). It suffices to show that for each $\chi$, and each $i$,   $H^i_{\textup{ct}} (\Gamma, R_{\infty, \chi})$ and $H^i_{\textup{ct}} (\Gamma, R_{\infty, \chi})/(\zeta_p - 1)$ have no (nontrivial) almost zero elements. If $\chi = 1$ is the trivial character, then we have 
$$H^i_{\text{ct}} (\Gamma, R_{\infty, \chi = 1}) = H^i_{\text{ct}} (\Gamma, R) = H^i_{\text{ct}} (\Gamma, R^{\square}) \otimes_{R^\square} R \cong  \wedge^i R^d.$$
Note that we have 
\[( \wedge^i R^d)[\fm]  = ( \wedge^i R^d / (\zeta_p - 1))[\fm] = 0.\] This follows, for example, from \cite[Lemma 3.5]{CK_semistable} and the fact that $\mO \ra R^{\square}$ can be obtained as the base change of a map $\mathring{\mO} \ra \mathring{R}^{\square}$ along a discrete valuation subring $\mathring{\mO} \hookrightarrow \mO$ (since $N$ is a finitely generated monoid and $R^{\square}$ is the $p$-completed base change of $\mO [N] \ra \mO [P]$ along $\mO[N] \ra \mO$). Now suppose that $\chi \ne 1$, it remains to show that 
\[H^i_{\text{ct}} (\Gamma, R_{\infty, \chi}) [\fm] = 0.\] By flatness of the map $R^\square  \ra R $ we may replace $R_{\infty, \chi}$ by $R^\square_{\infty, \chi}$. 
 For this claim, let us first observe that the map (from the proof of Proposition \ref{lemma: Gamma action on R^+} and Corollary \ref{corollary: Gamma action on R^+}) 
$$ \mO \gr{P_{\mathbb{Q}_{\geq 0}}}_{\chi} : = \bigoplus_{\sq \chi |_{\Gamma} = \chi} \mO_{C} \gr{P_{\mathbb{Q}_{\geq 0}}}_{\sq \chi} \lra R^{\square}_{\infty, \chi} $$ 
admits a $\Gamma$-equivariant section (note that $R_{\infty, \chi}^{\square}$ is topologically free as an $\mO$-module). Therefore, it suffices to show that  
\[H^i_{\text{ct}} (\Gamma, \mO\gr{P_{\mathbb{Q}_{\geq 0}}}_{\chi}) [\fm] = 0.\] This again can be deduced from \cite[Lemma 3.5]{CK_semistable} since in this last statement, we may replace $\mO$ by a  discrete valuation ring (for example, $W(k)$), and observe that  
\[H^i_{\text{ct}} (\Gamma, \mO\gr{P_{\mathbb{Q}_{\geq 0}}}_{\chi}) \cong H^i_{\text{ct}} (\Gamma, W(k)\gr{P_{\mathbb{Q}_{\geq 0}}}_{\chi}) \otimes_{W(k)} \mO.\] The corresponding statements for $R_\infty/b$ follow by similar arguments. 
\eproof

As before, we write 
\[\sq \Omega^{\square, \gp}_{\ul R} : =  L \eta_{(\zeta_p - 1)} R \Gamma_{\textup{ct}} (\Gamma, R_\infty) \quad 
\] and write 
\[\sq \Omega_{\ul R}^{\proket} := L \eta_{(\zeta_p - 1)} R \Gamma(U_{\proket}, \widehat \mO_X^+).\] 

\bc \label{cor:L_eta_turns_into_isom} The almost quasi-isomorphism $\gamma^{\square}$ in (\ref{eq:HT_gamma}) becomes a quasi-isomorphism 
$$ \gamma^{\square}: \sq \Omega^{\square, \gp}_{\ul R} \isom \sq \Omega_{\ul R}^{\proket}  $$
after applying $L\eta_{(\zeta_p - 1)}$. 
\ec  
 
\bproof 
This is a direct consequence of Proposition  \ref{lemma:decalage_turns_almost_into_actual} and Lemma \ref{lemma:group_coho_local}. 
\eproof 

\bc \label{cor:base_change_sq_Omega}
For a (strictly) \'etale morphism $\ul R \ra \ul {R'}$, the natural map 
$$\sq \Omega_{\ul R}^{\proket} \otimes_R^\L R' \lra \sq \Omega_{\ul{R'}}^{\proket}$$
is a quasi-isomorphism. 
\ec 

\bproof 
It suffices to prove the claim assuming $R = R^{\square}$ (in this case $R^{\square} \ra R'$ provides a coordinate for $\ul R'$). By Corollary \ref{cor:L_eta_turns_into_isom} it suffices to show that the natural map 
$$  \sq \Omega^{\square, \gp}_{\ul{R^\square}} \otimes_{R^\square} R' \lra \sq \Omega^{\square, \gp}_{\ul{R'}} $$
is a quasi-isomorphism. By Part (1) of Lemma \ref{lemma:group_coho_local}, it suffices to show that the map 
$$ ( L \eta_{(\zeta_p - 1)} R \Gamma_{\text{ct}} (\Gamma, R^{\square})) \otimes_{R^\square} R' \lra   L \eta_{(\zeta_p - 1)} R \Gamma_{\text{ct}} (\Gamma, R') $$
is a quasi-isomorphism. Since $R^\square \ra R'$ is flat and 
\[H^i (L \eta_{(\zeta_p - 1)} C) = H^i(C)/H^i(C)[\zeta_p - 1],
\]
it suffices to show that the natural of group cohomology 
$$ R \Gamma_{\text{ct}} (\Gamma, R^{\square}) \otimes_{R^\square} R' \isom  R \Gamma_{\text{ct}} (\Gamma, R') $$
is a quasi-isomorphism. But this follows from the description of $H^i (\Gamma, R') \cong \wedge^i ((R')^{d})$ as in the proof of Lemma \ref{lemma:group_coho_local}. 
\eproof

We will need the following lemma, which is an analog of \cite[Proposition 8.17]{BMS1}. For this let us consider the following (slightly simplified) situation (which corresponds to the case where $N= N_\infty = 0$). Let $\mO \rightarrow \mO \gr{\ul P} = (\mO \gr{P}, P)$ be a $p$-complete pre-log ring where $P$ is a sharp fs monoid  
with $\text{rk}_\Z P^{\gp} = n$. 
The pre-log ring 
\[\mO \gr{\ul{P_{\mathbb{Q}_{\geq 0}}}} = (\mO \gr{P_{\mathbb{Q}_{\geq 0}}}, P_{\mathbb{Q}_{\geq 0}})\] provides a Kummer pro-\'etale cover of Galois group $\Gamma_P \cong \Hom(P^{\gp}, \Z(1))$ on the associated log adic spaces. 
 From Construction \ref{construction:map_to_group_cohomology} we obtain a map 
\begin{equation} \label{eq:image_of_map_on_grp_cohomology_P}
 \widehat \L_{\mO\gr{\ul P} /\Z_p} \lra R \Gamma_{\text{ct}} (\Gamma_P, \mO \gr{P_{\mathbb{Q}_{\geq 0}}}[1] \{1\}).
 \end{equation} 

\bl  \label{lemma:image_lies_in_zeta_minus_1}
Let us fix generators $\delta_1, ..., \delta_n \in P^{\gp} \cong \Z^{n}$, which determines (topological) generators 
\[\gamma_1, ..., \gamma_n \in \Gamma_P \cong \Hom (P^{\gp}, \Z(1)).\]  After taking $H^0$, the map in (\ref{eq:image_of_map_on_grp_cohomology_P}) can be described as follows. Write $\delta_i \in  \widehat {\Omega}^1_{\mO \gr{\ul P}/\mO}$ for the element corresponding to $\delta_i \otimes 1$ via the identification 
$$ \widehat {\Omega}^1_{\mO \gr{\ul P}/\mO} \cong P^{\gp} \otimes_{\Z} \mO \gr{P},$$
then its image in $H^1_{\text{ct}} (\Gamma_P, \mO \gr{P_{\mathbb{Q}_{\geq 0}}} \{1\})$ under the map in (\ref{eq:image_of_map_on_grp_cohomology_P}) is given by the image of 
$$ \text{dlog}_i \otimes 1 \in H^1_{\text{ct}} (\Gamma_P, \mO\{1\}) \otimes \mO \gr{P} \hookrightarrow H^1_{\text{ct}} (\Gamma_P, \mO \gr{P_{\mathbb{Q}_{\geq 0}}} \{1\}), $$
where 
$$\text{dlog}_i \in H^1_{\text{ct}} (\Gamma_P, \mO \{1\}) \cong \Hom_{\text{ct}} (\Gamma_P, T_p (\Omega^1_{\mO/\Z_p})) $$
is the map that sends $\gamma_i$ to the element $\text{dlog} (\zeta_{p^\infty}) \in T_p (\Omega^1_{\mO/\Z_p}).$ Here the element $\text{dlog} (\zeta_{p^\infty})$ can be viewed as the image of our choice of topological generator (determined by $p$-power roots of unity) of $\Z_p (1) \cong T_p (\mu_{p^\infty} (\mO))$ under Fontaine's dlog map 
$$\text{dlog}: \mu_{p^\infty} \lra \Omega^1_{\mO/\Z_p}.$$
\el 

\bproof 
The proof is analogous to the proof of \cite[Proposition 8.17]{BMS1}. Let us describe the map mod $p^n$. The right hand side is 
\[H^0_{\text{ct}}(\Gamma_P,  \L_{\mO \gr{\ul{P_{\mathbb{Q}_{\geq 0}}}}/\Z_p} \otimes^\L_{\Z_p} \Z/p^n).\] There is a commutative diagram 
\[
\begin{tikzcd} 
\L_{\mO \gr{\ul P}/\Z_p} \arrow[r] \arrow[d]
& \L_{\mO \gr{\ul{P_{\mathbb{Q}_{\geq 0}}}}/\Z_p} \arrow[r]  \arrow[d]
& \L_{\mO \gr{\ul{P_{\mathbb{Q}_{\geq 0}}}}/\Z_p} \otimes^\L_{\Z_p} \Z/p^n \arrow[d, "\rotatebox{90}{$\sim$}"]  \\
{\Omega}^1_{\mO \gr{\ul P}/\Z_p}   \arrow[r] \arrow[d, "g"]
&   {\Omega}^1_{\mO \gr{\ul{P_{\mathbb{Q}_{\geq 0}}}}/\Z_p} \arrow[r] 
& {\Omega}^1_{\mO \gr{\ul{P_{\mathbb{Q}_{\geq 0}}}}/\Z_p} \otimes^\L_{\Z_p} \Z/p^n \arrow[r, "\sim"] 
&  {\Omega}^1_{\mO \gr{\ul{P_{\mathbb{Q}_{\geq 0}}}}/\Z_p} [p^n] \\
{\Omega}^1_{\mO \gr{\ul P}/\mO}  
\end{tikzcd}
\] 
where, after derived $p$-completion,  the top arrow induces the map 
\begin{equation} \label{eq:image_of_map_on_grp_cohomology_P_mod_p}
 \widehat \L_{\mO\gr{\ul P} /\Z_p} \lra R \Gamma_{\text{ct}} (\Gamma_P,  \widehat \L_{\mO \gr{\ul{P_{\mathbb{Q}_{\geq 0}}}}/\Z_p} ) \otimes^\L_{\Z_p} \Z/p^n   
\end{equation}
(which is mod $p^n$ version of \ref{eq:image_of_map_on_grp_cohomology_P}). 
Write $\textup{dlog } t_i \in {\Omega}^1_{\mO \gr{\ul P}/\Z_p}$ for the image of the element $(0, \delta_i)$ along the surjection 
$$ \Omega^1_{\mO \gr{P}/\Z_p} \midoplus  P^{\gp}  \otimes_\Z \mO \gr{P} \twoheadrightarrow {\Omega}^1_{\mO \gr{\ul P}/\Z_p}, $$ which maps to $\delta_i$ under the map $g$ in the diagram. It thus suffices to understand where $\text{dlog } t_i$ gets sent to under the map on group cohomology induced by the second (horizontal) row of maps in the diagram.  The rest of the proof is essentially the same proof of \cite[Proposition 8.17]{BMS1}. The group cohomology 
\[R \Gamma_{\text{ct}} (\Gamma_P,  \widehat \L_{\mO \gr{\ul{P_{\mathbb{Q}_{\geq 0}}}}/\Z_p} ) \otimes^\L_{\Z_p} \Z/p^n \] is computed by the double complex 
\[
\begin{tikzcd}[column sep = 1cm]
W  \arrow[rr, "{\gamma_1-1, ..., \gamma_n - 1}"]  && \bigoplus_{1 \le k \le n} W \arrow[r]  &  \bigoplus_{1 \le k_1 < k_2 \le n} W \arrow[r] & \cdots \\ 
W  \arrow[rr, "{\gamma_1-1, ..., \gamma_n - 1}"]  \arrow[u, swap, "p^n"] && \bigoplus_{1 \le k \le n} W \arrow[r] \arrow[u, swap, "p^n"]  &  \bigoplus_{1 \le k_1 < k_2 \le n} W \arrow[r] \arrow[u, swap, "p^n"]  & \cdots
\end{tikzcd}
\]
where we write $W :=  {\Omega}^1_{\mO \gr{\ul{P_{\mathbb{Q}_{\geq 0}}}}/\Z_p}$ and the top (and bottom) row is the Koszul complex \[\text{Kos}(W; \gamma_1 -1, ..., \gamma_n - 1)\] on the action of $\gamma_i - 1$ on $W$. The top left corner is placed in (cohomological) degree $(0, 0)$. To understand the map (\ref{eq:image_of_map_on_grp_cohomology_P_mod_p}), we also need to understand the action of $\Gamma_P$ on $W$. This is better formulated using log structures rather than pre-log structures. For each $n \in \Z_{\ge 1}$, let $P_n := \frac{1}{p^n} P$ and let $\alpha_n: P_n \ra \mO \gr{P_n}$ denote the pre-log structure. Let 
\[P_n^a \ra \mO \gr{P_n}\] be the corresponding log ring, namely, $ P_n^a$ is the pushout of $\alpha_n^{-1} (\mO \gr{P_n} ^\times) \hookrightarrow P_n$ along 
\[\alpha_n: \alpha_n^{-1} (\mO \gr{P_n} ^\times) \ra  \mO \gr{P_n} ^\times\] in the category of monoids. The action of each $\gamma \in \Gamma_P$ on the log algebra $(\mO \gr{P_n}, P_n^a)$ can be described as follows
\[
\begin{tikzcd}
\alpha \: \in P_n^a \:\: \qquad \arrow[r, mapsto, "\gamma \cdot "] 
\arrow[d, mapsto, shift right=6ex] 
& \quad  \gamma(\alpha) + \alpha \in P_n^a \arrow[d, mapsto, shift right=3ex] \\
e^{\alpha} \in \mO \gr{P_n} \quad   \arrow[r, mapsto, "\gamma \cdot "] & \qquad \gamma(\alpha)  e^\alpha \in \mO \gr{P_n}  
\end{tikzcd}
\]
where $\gamma (\alpha) \in \mu_{p^n} (\mO)$ is the image of $\alpha$ under $\gamma$, viewed as a map on $P_n^{\gp}$ via 
$$\Gamma_P= \Hom(P^{\gp}, \Z(1)) \cong \Hom  (P_{\mathbb{Q}_{\geq 0}}^{\gp}/P^{\gp}, \mu_{\infty} (\mO)) \ra \Hom (P_n^{\gp}/P^{\gp}, \mu_{\infty} (\mO)).$$
In particular, for each $\alpha \in P_n$, if we write $\text{dlog } \alpha$ for the image of $(0, \alpha)$ under the map 
\begin{multline}  \label{eq:describing_omega_1_log}
\quad  \Omega^1_{\mO \gr{P_n}/\Z_p} \midoplus \Big(P_n^{\gp} \Big)  \midotimes_\Z \mO \gr{P_n} \\  \surjlra {\Omega}^1_{\mO \gr{\ul{P_n}}/\Z_p} \cong  {\Omega}^1_{(\mO \gr{P_n}, P_n^a)/\Z_p}  \lra  {\Omega}^1_{\mO \gr{\ul{P_{\mathbb{Q}_{\geq 0}}}}/\Z_p} = W, \quad  
\end{multline}
then $\gamma \in \Gamma_P$ acts on $\textup{dlog } \alpha$ by 
$$ \gamma \cdot \text{dlog } \alpha  = \text{dlog } \gamma(\alpha) + \text{dlog } \alpha.$$
Now, to finish the proof, we observe that $\text{dlog } t_i \in W$ can be written as 
\[p^n \text{dlog } (t_i)^{1/p^n},\] where 
$\text{dlog } (t_i)^{1/p^n}$ denotes the image of $(0, \delta_i/p^n)$ under the map in (\ref{eq:describing_omega_1_log}). 
This implies that $\text{dlog } t_i$ is cohomologous to 
$$(\gamma_i -1) \cdot \text{dlog } t_i^{1/p^n} = \gamma_i \cdot \text{dlog }( t_i^{1/p^n}) - \text{dlog } t_i^{1/p^n} = \textup{dlog } \zeta_{p^n}.$$
This finishes the proof of the lemma. 
\eproof

\subsection{The end of the proof}  \label{ss:end_of_prop_local_HT}

Now we are ready to prove Proposition \ref{prop:HT_factors_group} (and thus Proposition \ref{prop:HT_map_factors_via_L_eta}). 
 
\bproof[Proof of Proposition \ref{prop:HT_factors_group}] 

By Lemma \ref{lemma:HT_map_factors_via_L_eta} and \cite[Lemma 8.16]{BMS1}, there exists a unique such factorization as required by Part (1) of the proposition if the induced map
\begin{equation} \label{eq:induced_map_on_H1}
H^1( \widehat \L_{\ul R/\Z_p} [-1]) \cong \widehat {\Omega}_{\ul R/\ul{\mO_C}}^1 \lra  H^1_{\textup{ct}} (\Gamma, R_\infty) \{1\} \end{equation}
factors through $(\zeta_p -1) H^1_{\textup{ct}} (\Gamma, R_\infty) \{1\}$.  We claim that the map in  (\ref{eq:beta_square}) in fact induces
\bi
\item an isomorphism $H^0(\widehat \L_{\ul R/\Z_p} [-1]\{-1\}) \cong R \isom   H^0_{\textup{ct}} (\Gamma, R_\infty) \cong R$, and 
\item an isomorphism $H^1(\widehat \L_{\ul R/\Z_p} [-1]\{-1\})  \isom (\zeta_p - 1) H^1_{\textup{ct}} (\Gamma, R_\infty)$.  
\ei
This would finish the proof of both parts of Proposition \ref{prop:HT_factors_group}. The first one is an isomorphism by construction. For the second map, we note that 
\begin{align} \label{eq:identifying_bases_1}
 H^1(\widehat \L_{\ul R/\Z_p} [-1]) \cong \widehat {\Omega}_{\ul R/\ul{\mO_C}}^1 & \cong  \widehat {\Omega}_{ \mO_C \gr{\ul P}/\mO_C \gr{\ul N}}^1 \widehat \otimes_{\mO_C \gr{P}} R \\
\nonumber &  \cong (P^{\gp}/N^{\gp}) \otimes_\Z R  \cong 
\Z^d \otimes_{\Z} R 
\cong R^d
\end{align}
is a free $R$-module of rank $d$ (on basis $\textup{dlog} \gamma_i'$ for a basis $\gamma_i' $ of $P^{\gp}/N^{\gp} \cong \Z^d$). 
Also note that, by Lemma \ref{lemma:group_coho_local}, we have 
\begin{align} \label{eq:identifying_bases_2}
 (\zeta_p - 1) H^1_{\textup{ct}} (\Gamma, R_\infty) & \cong (\zeta_p - 1) H^1_{\textup{ct}} (\Gamma, R) \\
\nonumber & \cong (\zeta_p -1) (P^{\gp}/N^{\gp} \otimes_{\Z} R ) \cong (\zeta_p - 1) R^d 
 \end{align}
is also a free $R$-module of rank $d$. Therefore it suffices to identify basis under the map on $H^1$. Note that the map (\ref{eq:induced_map_on_H1}) is compatible with base change along the \'etale map $R^{\square} \ra R$ given by the coordinate $\square$, so we may assume that $R = R^{\square}$. 

Note that we have an analogous map as in (\ref{eq:induced_map_on_H1}) but with $R = R^{\square}$ replaced by $\mO \gr{P}$: 
\begin{equation} \label{eq:induced_map_on_H1_P}
H^1( \L_{\mO \gr{\ul P}/\Z_p} [-1]) \cong \widehat {\Omega}_{\mO\gr{\ul P}/\mO}^1 \lra  H^1_{\textup{ct}} (\Gamma_P, \mO \gr{P_{\mathbb{Q}_{\geq 0}}}) \{1\}. 
\end{equation}
Base changing along $\mO \gr{P} \ra R^{\square} = \mO \gr{P} \otimes_{\mO \gr{N}} \mO$ and the surjection $P^{\gp} \ra P^{\gp}/N^{\gp}$ identifies the isomorphisms
$$\widehat {\Omega}_{\mO\gr{\ul P}/\mO}^1 \cong P^{\gp} \otimes_\Z \mO \gr{P} \cong \mO \gr{P}^{n} $$
with the isomorphisms in (\ref{eq:identifying_bases_1}), and the isomorphisms 
$$ (\zeta_p - 1)  H^1_{\textup{ct}} (\Gamma_P, \mO \gr{P_{\mathbb{Q}_{\geq 0}}})  \cong  H^1_{\textup{ct}} (\Gamma_P, \mO \gr{P}) \cong (\zeta_p - 1)  (P^{\gp} \otimes_\Z \mO \gr{P})   $$
with the isomorphisms in (\ref{eq:identifying_bases_2}). 
By the construction of the maps  (\ref{eq:induced_map_on_H1})  and  (\ref{eq:induced_map_on_H1_P}), we further reduce to the case of $\mO_C \gr{P}$. Now the assertion follows from Lemma \ref{lemma:image_lies_in_zeta_minus_1}.  
\eproof

\bproof[Proof of Proposition \ref{prop:primitive_HT_local}] 

Part (1) is Corollary \ref{cor:L_eta_turns_into_isom}. For Part (2), we may argue exactly as in \cite[Lemma 8.13 (iv)]{BMS1}. For completeness let us recall their argument. It suffices to show that the natural map 
$$ \sq \Omega^{\proket}_{\ul R} \otimes_{R} \mO_{\fU_{\ett}} \lra \sq \Omega^{\log}_{\fU}  $$
is a quasi-isomorphism in $D(\fU_{\ett})$, for which we check on stalks at points. For any $x \in \fU$, the stalk of the left side is 
$$ \varinjlim_{x \in \fU' = \spf \ul{R'}} (\sq \Omega^{\proket}_{\ul R} \otimes_R R' ) \cong  \varinjlim_{x \in \fU' = \spf \ul{R'}} \sq \Omega^{\proket}_{\ul{R'}}  =    \varinjlim_{x \in \fU' = \spf \ul{R'}} L \eta_{(\zeta_p - 1)} R \Gamma(U'_{\proket}, \OXplus)  $$
where the colimit is taken over all affine \'etale neightborhood $\fU'$ of $\fU = \spf \ul{R}$ and $U$ is the log adic generic fiber of $\fU$. The first isomorphism is given by Corollary \ref{cor:base_change_sq_Omega}. The stalk on the right side is  
$$   \varinjlim_{x \in \fU' = \spf \ul{R'}} R \Gamma(\fU',  L \eta_{(\zeta_p -1)} R \nu_* \OXplus) \cong  L \eta_{(\zeta_p -1)}  \varinjlim_{x \in \fU' = \spf \ul{R'}} R \Gamma(\fU',  R \nu_* \OXplus),$$ 
since taking stalks commute with $L\eta$ (\cite[Lemma 6.14]{BMS1} applied to the $x \ra \fU$). The claim then follows, since $L \eta$ commutes with filtered colimits by \cite[Corollary 6.5]{BMS1}. 
\eproof

\subsection{Multplicativity}  \noindent 

\noindent 
Finally we show that the construction is this section is compatible with taking wedge product, and then prove a Kunneth formula for $\sq \Omega^{\log}_{\fX}$. First let us consider the cup product. Let 
$$\textup{cup}: \: \big( R^1 \nu_* \OXplus \big)^{\otimes i} \lra R^i \nu_* \OXplus $$
be the cup product induced from 
\begin{equation} \label{eq:cup_map}
\cup: \: R^j \nu_* \OXplus \otimes^\L_{\mO_{\fX_{\ett}}} R^k \nu_* \OXplus \lra H^{j+k} ( R \nu_* \OXplus \otimes^\L_{\mO_{\fX_{\ett}}}  R \nu_* \OXplus) 
\end{equation}
(\cite[068G]{SP}, see also \cite[Section 4.7]{CK_semistable}) and the canonical map 
$$R \nu_* \OXplus \otimes^\L_{\mO_{\fX_{\ett}}}  R \nu_* \OXplus \ra R \nu_* \OXplus $$
coming from adjunction. 

\bp \label{prop:wedge_iso}
Let $\fX$ be a locally small quasi-fs log formal scheme over $\ul \mO$ as in the beginning of the section.  The cup product map in (\ref{eq:cup_map}) induces a natural isomorphism 
$$ \wedge^i \mH^1(\sq \Omega^{\log}_{\fX}) \isom \mH^i (\sq \Omega^{\log}_{\fX}). $$
In particular, the map $\iota^1$ in (\ref{map:iota_1}) obtained in Construction \ref{construction:map_iota} (and considered in Proposition \ref{thm:primitive_HT_01}) induces natural isomorphisms 
$$\iota^i: {\Omega}^{i, \text{ct}}_{\fX/\ul \mO} = \widehat {\Omega}^i_{\fX/\ul{\mO}} \isom \mH^i (\sq \Omega^{\log}_{\fX}) \{i\}.$$
\ep

\bproof 
Let us first construct the natural map locally on $\fU= \spf R$ which is small (and modeled on a monoid map $N \ra P$). We need to construct a natural map (functorial on $R$)
\begin{equation} \label{eq:wedge_of_H1}
 \wedge^i H^i (\sq \Omega^{\log}_{\ul R}) \lra H^i (\sq \Omega^{\log}_{\ul R})
\end{equation}
which is an isomorphism. By Proposition \ref{prop:primitive_HT_local} it suffices a natural map 
\begin{multline*}
\:\:\: \bigwedge^i H^i (\sq \Omega^{\proket}_{\ul R}) \cong \bigwedge^i \Big( \frac{H^1(U_{\proket}, \OXplus)}{H^1(U_{\proket}, \OXplus) [\zeta_p - 1]} \Big) \\ \lra H^i (\sq \Omega^{\proket}_{\ul R}) \cong  \frac{H^i(U_{\proket}, \OXplus)}{H^i(U_{\proket}, \OXplus) [\zeta_p - 1]} \:\:\:     
\end{multline*}
where $U = \fU_{\eta}^{\text{ad}}$ is the log adic generic fiber as before. This comes from the cup product map 
\[\text{cup}: (H^1 (U_{\proket}, \OXplus))^{\otimes i} \ra H^i (U_{\proket}, \OXplus)\] above, and the graded-commutativity $x \cup y = (-1)^{jk} y \cup x$ of the map (\ref{eq:cup_map}), which descends to the desired wedge product since each $H^i (\sq \Omega^{\proket}_{\ul R})$ is torsion free (in particular contains no nonzero $2$-torsion) by Lemma \ref{lemma:group_coho_local} and Corollary \ref{cor:L_eta_turns_into_isom}. Finally, to show that the map (\ref{eq:wedge_of_H1}) is an isomorphism, we are allowed to use coordinates and (by Part (1) of Lemma \ref{lemma:group_coho_local}) it suffices to show that the cup product on group cohomology induces an isomorphism 
$$ \wedge^i H^1_{\text{ct}}(\Gamma, R) \isom  H^i_{\text{ct}}(\Gamma, R).$$
But this follows from the description of $R\Gamma_{\text{ct}}(\Gamma, R)$ as the Koszul complex $\text{Kos}(R; 0, ..., 0)$ with trivial operators (which identifies with the $d$-fold tensor product of the complex $R \xrightarrow{0} R$, where $d = \text{rk}_{\Z} (P^{\gp}/N^{\gp})$), and the description of the cup product on the Koszul complex 
(for example, see  \cite[Lemma 7.5]{BMS1}). 
\eproof 

In particular, this finishes the proof of Theorem \ref{thm:primitive_HT}, which we summarize as follows. 
\bproof[Proof of Theorem \ref{thm:primitive_HT}] \indent 
We obtain the canonical map $\iota^0$ (resp. $\iota^1$) on $H^0$ (resp. $H^1$) in Construction \ref{construction:map_iota}, which makes use of the factorization in Proposition \ref{prop:HT_map_factors_via_L_eta}, which in turn follows from Proposition \ref{prop:HT_factors_group}. The fact that $\iota^0$ and $\iota^1$ are isomorphisms (content of Proposition \ref{thm:primitive_HT_01}) is a consequence of Proposition \ref{prop:HT_factors_group} and \ref{prop:primitive_HT_local}. Finally, Proposition \ref{prop:wedge_iso} allows us to construct isomorphism $\iota^i$ on higher degrees. 
\eproof 

Finally, we end this section with the following 

\br \label{remark:multiplicativity_tilde}
Let $\ul{R_1}, \ul{R_2}$ be two $p$-complete pre-log rings over $\ul \mO$ which give rise to small \'etale neighbourhoods of $\fX$, then the natural maps $\ul{R_1} \ra \ul{R_1}\otimes_{\ul \mO} \ul{R_2}$ and $\ul{R_2} \ra \ul{R_1}\otimes_{\ul \mO} \ul{R_2}$ induce a quasi-isomorphism 
\[ \sq \Omega^{\log}_{\ul{R_1}} \: \widehat \otimes^\L_{\mO} \: \sq \Omega^{\log}_{\ul{R_2}} \isom \sq \Omega^{\log}_{ \ul{R_1}\otimes_{\ul \mO} \ul{R_2} }.
\]
For this it suffices to choose coordinates $\square_i: R_i^{\square} \ra R_i$ for $i = 1, 2$ where $R_i^{\square}$ is modeled on a small chart $u_i: N_i \ra P_i$. Note that $R_1^\square \widehat \otimes_{\mO} R_2^{\square}$ is then modeled on the small chart $N_1 \otimes N_2 \ra P_1 \oplus P_2$. By the same argument of \cite[Proposition 8.14]{BMS1} (using Proposition 6.8 and Lemma 6.20 of \emph{loc.cit.}) it suffices to prove this before applying the operator $L\eta_{(\zeta_p -1)}$ on the level of group cohomology. In other words,  
it suffices to show that 
\[
R \Gamma_{\text{ct}} (\Gamma, R^{\square}_{1, \infty} \widehat \otimes_{\mO} R^{\square}_{2, \infty}) \cong R \Gamma_{\text{ct}} (\Gamma_1 , R^{\square}_{1, \infty})  \widehat \otimes_{\mO} R \Gamma_{\text{ct}} (\Gamma_2 , R^{\square}_{2, \infty}) 
\]
where $\Gamma_i := \ker (\Gamma_{P_i} \ra \Gamma_{N_i})$ and 
\[\Gamma := \ker (\Gamma_{P_1 \oplus P_2} \ra \Gamma_{N_1 \oplus N_2}) \cong \Gamma_1 \times \Gamma_2.\] The desired isomorphism thus follows from the description of the group cohomology using Koszul complexes. 
\er


\newpage 

\section{The Hodge--Tate and de Rham comparisons} \label{section:Hodge_Tate_and_dR}

The goal of this section is to prove the Hodge--Tate and de Rham comparisons for $A \Omega_{\fX}^{\log}$. To formulate, first note that we have the following map of pro-Kummer \'etale sheaves  
$$\sq \theta = \theta \circ \varphi^{-1}: \Ainfx  \lra \OXplus  $$
(cf. Subsection \ref{section:review_Ainf})
which induces the following Hodge--Tate comparison map (of sheaves on $\fX_{\ett}$): 
\begin{equation} \label{eq:HT_comp} 
A \Omega^{\log}_{\fX} \otimes^\L_{\Ainf} \Ainf/\varphi(\xi) \lra  \sq \Omega_{\fX}^{\log}=  L \eta_{(\zeta_p -1)} R \nu_* \OXplus.
\end{equation}
 
\bt[Hodge--Tate comparison] \label{thm:HT_comp}
Let $\fX$ be a quasi-fs log $p$-adic formal scheme that is admissibly smooth over $\ul \mO = \ul{\mO_C}$. Then the the Hodge--Tate comparison map 
\[A \Omega^{\log}_{\fX} \otimes^\L_{\Ainf} \Ainf/\varphi(\xi) \isom  \sq \Omega_{\fX}^{\log}\]
is a quasi-isomorphism. 
\et 

This will then allow us to prove 

\bt[de Rham comparison] \label{thm:dR_comp}
Let $\fX$ be a quasi-fs log $p$-adic formal scheme that is admissibly smooth over $\ul \mO = \ul{\mO_C}$, then there is a natural quasi-isomorphism 
\[ A \Omega^{\log}_{\fX} \otimes^\L_{\Ainf} \Ainf/\xi \cong   \Omega^{\bullet, (\textup{ct})}_{\fX/\ul \mO}.\]
In particular, we have a quasi-isomorphism 
\[
R \Gamma_{\Ainf} (\fX) \otimes_{\Ainf}^\L \Ainf/\xi \isom R \Gamma_{\textup{logdR}}(\fX/{\ul\mO}). 
\]
\et

The proof of Theorem \ref{thm:HT_comp} ultimately relies on local analysis of $A \Omega_{\fX}^{\log}$, as done in \cite{BMS1}. However, as we wish to avoid the analysis of the analog of $\sq W_r \Omega_{\fX}$ in \textit{loc.cit}, we instead use a strategy that is similar to \cite{CK_semistable} (which is a slight modification of \cite{BMS1}) to reduce to local computations. 

\begin{construction}

Suppose that $\fX$ is admissibly smooth over $\ul{\mO}$ as in the previous section. Let $\fX_{\ett, \text{small}}$ denote the site formed by small $p$-adic formal affine opens $\fU = \spf R \ra \fX$, so that there exists a coordinate map from $\ul{R^{\square}}$ to $\ul R$ in the sense of Subsection \ref{ss:small_coordinates} (where $\ul{R^\square}$ is a small pre-log ring modeled on $N \ra P$ as in Definition \ref{definition:locally_small}).  One checks that $\fX_{\ett, \text{small}}$ indeed forms a site under the \'etale topology. Note that such affine opens form a basis of $\fX_{\ett}$, so we have canonical isomorphism of topoi $\fX_{\ett, \text{small}}^{\sim} \cong \fX_{\ett}^{\sim}$, which we denote by $\fX_{\ett}$ again by an abuse of notation.  
 Let 
$\fX_{\ett}^{\text{psh}}$ denote the presheaf topos on the site formed by the affine opens in $\fX_{\ett, \text{small}}$; i.e., the topos associated to the site with the same objects in $\fX_{\ett, \text{small}}$ but equipped with the trivial topology. We have a natural map of topoi 
$$ j =  (j^{-1}, Rj_*): \fX_{\ett} \lra \fX_{\ett}^{\text{psh}} $$
where $Rj_*$ is the forgetful functor and $j^{-1}$ is given by sheafification. Following \cite{BMS1}, we define the following (pre)sheaves of complexes on $\fX_{\ett}^{\text{psh}}$
\begin{align*} 
 A \Omega_{\fX}^{\log, \text{psh}} & : = L \eta_{\mu} R \nu_*^{\text{psh}} \widehat{\Ainfx} , \qquad \\ 
 (\text{resp. } \sq \Omega^{\log, \text{psh}}_{\fX} & := L \eta_{(\zeta_p - 1)} R \nu_*^{\text{psh}} \OXplus ) 
\end{align*}
where $ \nu^{\text{psh}} $ is the composition 
\[\nu^{\text{psh}} := j \circ \nu:  X_{\proket} \ra \fX_{\ett}^{\text{psh}}.\]  By construction, on each small neighborhood $\fU = \spf (\ul{R})^a \in \fX_{\ett, \text{small}}$ with log adic generic fiber $U$, we have 
\begin{equation}\label{eq:definition_AOmega_proket_R} 
R \Gamma (\fU,  A \Omega_{\fX}^{\log, \text{psh}}) = A \Omega^{\proket}_{\ul R} :=  L \eta_{\mu} R \Gamma (U_{\proket}, \widehat{\Ainfx}).
\end{equation} 
Let us also define the complex 
\begin{equation}\label{eq:definition_of_AOmega_R}
A \Omega^{\log}_{\ul R} := R \Gamma (\fU_{\ett}, A \Omega^{\log}_{\fX}). 
\end{equation}
Note that the construction $\ul R \mapsto A \Omega_{\ul R}^{\proket}$ and $\ul R \mapsto A \Omega_{\ul R}^{\log}$ are both functorial in $\ul R$. 
\end{construction} 

\br \label{remark:derived_complete}
In particular, the complex $ A \Omega^{\proket}_{\ul R} \in \mD(\Ainf)$ is derived $(p, \mu)$-adically complete (and also derived $\varphi(\xi)$-adically complete). This is because $\Ainfx$ is derived $(p, \mu)$-adically complete (and derived $\varphi(\xi)$-complete), and $L\eta_\mu$ preserves derived completeness in $\mD(\Ainf)$. More generally, $L \eta$ preserves derived completeness in any replete topos by \cite[Subsection 6.2]{BMS1}, but not necessarily on $\fX_{\ett}$. Thus it is \textit{not a priori} clear why $A \Omega^{\log}_{\ul R}$ is derived $(p, \mu)$-adically complete. We will prove this as a consequence of Proposition \ref{prop:presheaf_is_sheaf} below. 
\er 

Note that, we have 
\begin{align}  \label{eq:presheaf_sheafifies_to_AOmega}
 j^{-1} A \Omega_{\fX}^{\log, \text{psh}} & \isom  A \Omega_{\fX}^{\log}, \qquad \qquad  \\ (\text{resp. }  j^{-1} \sq \Omega_{\fX}^{\log, \text{psh}} & \isom \sq \Omega_{\fX}^{\log})  \nonumber
 \end{align}
since $L \eta_{\mu}$ commutes with $j^{-1}$ by \cite[Lemma 6.14]{BMS1}. Therefore, Theorem \ref{thm:HT_comp} follows from its following presheaf variant (by applying $j^{-1}$): 

\bp \label{prop:HT_presheaf}
Under the same assumptions of Theorem \ref{thm:HT_comp}, the natural map 
$$ A \Omega_{\fX}^{\log, \text{psh}} \otimes^{\L}_{\Ainf} \Ainf/\varphi(\xi) \lra \sq \Omega_{\fX}^{\log, \text{psh}} $$ 
is a quasi-isomorphism. 
\ep

The proof of Proposition \ref{prop:HT_presheaf} is then reduced to local computations of $A \Omega_{\fX}^{\log, \text{psh}}$ (in Subsection \ref{ss:local_computation_AOmega}) and $\sq \Omega_{\fX}^{\log, \text{psh}}$ (which is already done in Section \ref{section:HT_primitive}). Slightly more precisely, for each small neighbourhood $\fU=\spf(\ul{R})^a$ of $\fX$, we will use coordinates (cf. Subsection  \ref{ss:small_coordinates}) to obtain a quasi-isomorphism 
\begin{equation} \label{eq:qi_group_proket_Ainf}
A \Omega^{\square, \gp}_{\ul R} \isom A \Omega^{\proket}_{\ul R}
\end{equation}
where $A \Omega^{\square, \gp}_{\ul R}$ is certain group cohomology. Proposition  \ref{prop:HT_presheaf}  will then follow from Corollary \ref{cor:L_eta_turns_into_isom} and a group cohomology computation, which asserts that 
\begin{equation} \label{eq:qi_group_Ainf_mod_xi}
A \Omega^{\square, \gp}_{\ul R} \otimes_{\Ainf}^\L \Ainf/\varphi(\xi) \isom \sq \Omega^{\square, \gp}_{\ul R}.  
\end{equation}
Finally, in the end of this section, we prove the following claim and deduce some corollaries that will be used later.  

\bp \label{prop:presheaf_is_sheaf}
Keep the assumptions of Theorem \ref{thm:HT_comp}. The natural map in (\ref{eq:presheaf_sheafifies_to_AOmega}) gives rise to a quasi-isomorphism 
$$ A \Omega_{\fX}^{\log, \text{psh}} \isom Rj_* A \Omega_{\fX}^{\log}. $$ 
In other words, the presheaf $ A \Omega_{\fX}^{\log, \text{psh}}$ is already a sheaf. In particular, on small affine neighbourhoods $\fU = \spf R$ of $\fX$, we have a quasi-isomorphism  
\[
A \Omega^{\proket}_{\ul R} = L \eta_{\mu} R \Gamma (U_{\proket}, \widehat{\Ainfx}) \isom A \Omega^{\log}_{\ul R} = R \Gamma (\fU_{\ett}, A \Omega^{\log}_{\fX}). 
\]
using notations from (\ref{eq:definition_AOmega_proket_R}) and (\ref{eq:definition_of_AOmega_R}). 
\ep 

The proof of Proposition \ref{prop:presheaf_is_sheaf} will be given in Subsection \ref{ss:proof_of_more_on_AOmega}. 

\subsection{Local computation of group cohomology} \label{ss:proket_cover}  


Keep notations from Subsection \ref{ss:small_coordinates} and \ref{ss:local_computation_mod_xi}. Let $\fU = \spf (\ul{R})^a \in \fX_{\ett, \text{small}}$ be a small affine \'etale neighbourhood of $\fX$. It comes with a small coordinate
\[\square: \ul{R^{\square}}=(R^{\square}=\mO \widehat{\otimes}_{\mO \gr{N}} \mO\gr{P}, N_{\infty}\sqcup_{N} P)\rightarrow \ul{R}=(R, N_{\infty}\sqcup_{N} P)\]
such that the induced morphism
\[\fU=\spf(\ul{R})^a\rightarrow \fU^{\square}=\spf(\ul{R^{\square}})^a\]
is a composition of rational localizations and finite \'etale morphisms. By Remark \ref{remark:base_change_N_Q_to_N_infty}, we may assume $N_{\mathbb{Q}_{\geq 0}}\isom N_{\infty}$ from now on.

Also recall that the coordinate provides a perfectoid pre-log ring 
\[(R_\infty = R^{\square}_{\infty} \widehat{\otimes}_{R^\square} R, P_{\mathbb{Q}_{\geq 0}}) \] with $R_\infty^\square = \mO \widehat{\otimes}_{\mO \gr{N_{\mathbb{Q}_{\geq 0}}}} \mO \gr{P_{\mathbb{Q}_{\geq 0}}} $, which comes from a pro-finite Kummer \'etale Galois cover of $\fU$ with Galois group $\Gamma$. Recall that 
$\Gamma$ fits into a short exact sequence
\[0\rightarrow \Gamma\rightarrow \Gamma_P\rightarrow \Gamma_N\rightarrow 0\]
where $\Gamma_P = \Hom (P^{\gp}, \widehat \Z (1))$ and $\Gamma_N = \Hom (N^{\gp}, \widehat \Z (1))$. The tilt of $R_\infty^{\square}$ can be identified with 
$$
(R_\infty^\square)^{\flat} \cong \mO^\flat \widehat \otimes_{\mO^{\flat} \gr{N_{\mathbb{Q}_{\geq 0}}}} \mO^\flat \gr{P_{\mathbb{Q}_{\geq 0}}}
$$
where the map $N_{\mathbb{Q}_{\geq 0}} \ra \mO^\flat$ is induced from $N_{\mathbb{Q}_{\geq 0}} \ra \mO$ (by taking inverse limit perfection) and the completion is $p^{\flat}$-adic.  The tilt $R_\infty^\flat$ of $R_\infty$ can thus be identified with the $(p^{1/p^\infty})$-adic completion of any lift of the \'etale morphism $R_\infty^{\square} /p \ra R_\infty /p$ along the projection 
\[(R_\infty^{\square} )^\flat \ra (R_\infty^{\square})^\flat /p^\flat \cong R_\infty^{\square} /p.\] This allows us to identify  
\begin{equation} 
\Ainf(R_\infty^{\square}) \cong \Ainf \widehat \otimes_{\Ainf \gr{N_{\mathbb{Q}_{\geq 0}}}} \Ainf \gr{P_{\mathbb{Q}_{\geq 0}}}
\end{equation}
where the map $N_{\mathbb{Q}_{\geq 0}} \ra \Ainf(\mO^\flat)$ is obtained from the Teichmuller lift of  $N_{\mathbb{Q}_{\geq 0}} \ra \mO^\flat$ and  the completion is $(p, \mu)$-adic. Let us also consider the following $(p, \mu)$-adically complete $\Ainf$-algebra (as a sub-algebra of $\Ainf (R_\infty^{\square})$): 
\[A (R^{\square}) := \Ainf \widehat \otimes_{\Ainf \gr{N}} \Ainf \gr{P}.\] 
The natural surjection $\Ainf \ra \Ainf/\xi \cong \mO$ allows us to identify $A(R^\square) /\xi \cong R^\square$. Along this pro-nilpotent thickening, there exists a unique formally \'etale map 
$$A(R^\square) \ra A (R)$$ where $A(R)$ is $(p, \mu)$-adically complete which lifts the formally \'etale map $R^\square \ra R$. This then allows us to identify $\Ainf (R_\infty)$ with the $(p, \mu)$-adically completed base change 
$$ \Ainf (R_\infty) \cong \Ainf (R_\infty^{\square})\widehat \otimes_{A (R^\square)} A(R).$$

\addtocontents{toc}{\protect\setcounter{tocdepth}{0}}
\subsection*{The  action of $\Gamma$ on  $\Ainf(R_\infty^\square)$} \noindent 

\noindent 
Next we describe the action of $\Gamma$ on $(R_\infty)^\flat$ and $\Ainf(R_\infty^\square)$. The identification 
$$ (\mO \gr{P_{\mathbb{Q}_{\geq 0}}})^\flat \cong \mO^\flat \gr{(P_{\mathbb{Q}_{\geq 0}})^\flat} \cong \mO^\flat \gr{P_{\mathbb{Q}_{\geq 0}}} $$
is induced from the isomorphism $P_{\mathbb{Q}_{\geq 0}} \cong (P_{\mathbb{Q}_{\geq 0}})^\flat$ of monoids, given by 
$$ \alpha \in P_{\mathbb{Q}_{\geq 0}} \longmapsto (\alpha, \alpha/p, \alpha/p^2, ... ) \in (P_{\mathbb{Q}_{\geq 0}})^\flat. $$
Thus the action of $\Gamma_P$ on $\mO^\flat \gr{P_{\mathbb{Q}_{\geq 0}}}$ is given by 
$$\gamma \cdot e^{\alpha} = \gamma(\alpha)^\flat \cdot e^{\alpha} $$
where $\gamma (\alpha)^\flat \in \mO^\flat$ is the element 
$$\gamma (\alpha)^{\flat} := (\gamma (\alpha), \gamma (\alpha/p), \gamma(\alpha/p^2), ... ).$$
Here $\Gamma_P$ is identified with $\Hom (P_{\mathbb{Q}}^{\gp} /P^{\gp}, \mu_\infty)$. This action in turn induces an action of $\Gamma$ on the tilt
\[(R_\infty^{\square})^\flat \cong \mO^\flat \gr{P_{\mathbb{Q}_{\geq 0}}} \widehat{\otimes}_{\mO^\flat \gr{N_{\mathbb{Q}_{\geq 0}}}} \mO^\flat,\] since by construction, $\Gamma$ acts on $\mO^\flat \gr{N_{\mathbb{Q}_{\geq 0}}}$ trivially. Likewise, $\Gamma_P$ acts on 
\[\Ainf (\mO \gr{P_{\mathbb{Q}_{\geq 0}}}) \cong \Ainf \gr{P_{\mathbb{Q}_{\geq 0}}}\]by the formula 
$$\gamma \cdot e^{\alpha} = [\gamma(\alpha)^\flat] \cdot e^{\alpha}, $$
where $[\gamma (\alpha)^\flat] \in \Ainf$ denotes the Teichmuller lift of $\gamma(\alpha)^\flat$. This then induces a continuous action of $\Gamma$ on $\Ainf (R_\infty^\square)$, which further extends to a continuous action on $\Ainf (R_\infty)$, and is compatible with the action on $R_\infty$ (resp. on $R_\infty^\flat$) modulo $\xi$ (resp, modulo $p$).

In order to prove Proposition \ref{prop:HT_presheaf}, we would like to analyze
$$ A \Omega_{\ul R}^{\square, \gp} :=  L \eta_{\mu} R \Gamma_{\text{ct}} (\Gamma, \Ainf (R_\infty)).$$

Let us fix topological generators $\gamma_1, ..., \gamma_d$ of $\Gamma \cong \widehat \Z(1)^d$.  
For a continuous character $\psi: \Gamma \ra \Ainf^\times$, let us define 
$$ \Ainf (R_\infty)_{\psi}  := \{ x \in \Ainf (R_\infty)  \: | \:  \gamma \cdot x = \psi (\gamma) x \: \text{ for all } \gamma \in \Gamma \}. $$
Note that $\Ainf(R_\infty)$ is  $(p, \mu)$-adically complete, and $(p, \mu)$-adically formally flat over $\Ainf$ by \cite[Lemma 3.13]{CK_semistable}. 
The same proof of Proposition \ref{lemma: Gamma action on R^+} (also see Corollary \ref{corollary: Gamma action on R^+}) provides a  decomposition 
\begin{equation} \label{eq:decompose_by_Ainf_char} 
\Ainf (R_\infty) = \widehat \midoplus_{\psi} \Ainf (R_{\infty})_{\psi} 
\end{equation}
where the completion is $(p, \mu)$-adic, and $\psi$ runs over all continuous characters of $\Gamma$ which send each $\gamma_i$ to an element of the form 
\begin{equation} \label{eq:form_of_char}
[\epsilon]^{\frac{a_i}{p^{r_i}}} = [(\zeta_{p^{r_i}}, \zeta_{p^{r_i+1}}, \zeta_{p^{r_i+2}}, ...)]^{a_i} 
\end{equation}
 where $a_i \in \Z_p^\times, r_i \in \Z$ (for a negative integer $b \le 0$ we set $\zeta_{p^b} = 1$). As in Lemma \ref{lemma:decomposition}, the decomposition in (\ref{eq:decompose_by_Ainf_char}) follows from it $\Gamma_P$-variant: 
$$ \Ainf (\mO\gr{P_\infty}) \cong \Ainf \gr{P_\infty} \cong \widehat \midoplus_{\sq \psi} \Ainf \gr{P_\infty}_{\sq \psi},$$
here $\sq \psi: \Gamma_P \ra \Ainf^\times$ are characters of $\Gamma_P$ which can be described similarly as $\psi$.   Likewise we have 
\[\Ainf(R_\infty)/\mu = \widehat \oplus_{\psi} \Ainf (R_\infty)_{\psi}/\mu \] with the completion being $p$-adic.

\br 
We say that a character $\psi: \Gamma \ra \Ainf^\times$ \emph{lifts} $\chi: \Gamma \ra \mO^\times$ if $\chi = \theta \circ\psi$, where $\theta: \Ainf \ra \mO$ is the natural mod $\xi$ quotient map.  Let us remark that the character $\psi$ of $\Gamma$ which lifts a finite order character $\chi$ is generally of infinite order.  Likewise we say that $\psi^\flat: \Gamma \ra (\mO^\flat)^\times$ lifts $\chi$ if $\chi = \theta^\flat \circ \psi^\flat$ for the natural (monoid) map $\theta^\flat: \mO^\flat \ra \mO$ which sends $(x_0, x_1, ...) \mapsto x_0$. For a finite order character $\chi$, we write 
\begin{equation} \label{eq:component_over_chi} 
\Ainf (R_\infty)_{\chi} = \widehat \oplus_{\psi \text{ lifts } \chi} \Ainf (R_\infty)_{\psi}.
\end{equation} Note that when $\chi = 1$ is trivial, $\Ainf (R_\infty)_{1} = A (R)$.  Note that each $\Ainf (R_\infty)_{\chi}$ is a module over $A(R)$. In fact we have 
\[\Ainf (R_\infty)_{\chi} \cong \Ainf (R^\square_\infty)_{\chi} \widehat \otimes_{A(R^\square)} A(R)\] and (by considering Koszul complexes)
\begin{equation} \label{eq:chi_component}
R \Gamma(\Gamma, \Ainf(R_\infty)_{\chi}) \cong R \Gamma(\Gamma, \Ainf(R^\square_\infty)_{\chi}) \widehat \otimes_{A(R^\square)} A(R) 
\end{equation} 
\er

The following lemma is an analog of \cite[Lemma 9.6]{BMS1}
\bl \label{lemma:killing_junk_torsion}
For a nontrivial finite order character $\chi: \Gamma \ra \mu_\infty$, the group cohomology \[H^i_{\text{ct}} (\Gamma, \Ainf (R_\infty)_\chi)\] is entirely $\varphi^{-1} (\mu)$-torsion. In particular, we have 
\[L \eta_{\mu} R \Gamma_{\text{ct}} (\Gamma,  \widehat \oplus_{\chi \ne 1} \Ainf (R_\infty)_{\chi}) = 0 \] and 
\[ L \eta_\mu R \Gamma_{\text{ct}} (\Gamma, A(R)) \isom L \eta_{\mu} R \Gamma_{\text{ct}} (\Gamma,  \Ainf (R_\infty)). \]
\el 

\bproof 
Fix a nontrivial $\chi: \Gamma \ra \mu_\infty$. As in \cite{BMS1}, we prove the slightly stronger statement that multiplication by $\varphi^{-1} (\mu)$ on $R \Gamma_{\text{ct}}(\Gamma, \Ainf(R_\infty)_{\chi})$ is homotopic to $0$. By the identification in (\ref{eq:chi_component}) it suffices to prove this claim for $\Ainf (R_\infty^\square)_\chi$. Let $\psi: \Gamma \ra \Ainf^\times$ be a character that appears in (\ref{eq:decompose_by_Ainf_char}) and lifts $\chi$. It suffices to show that multiplication by $\varphi^{-1} (\mu)$ on $R \Gamma_{\text{ct}}(\Gamma, \Ainf(R_\infty^\square)_{\psi})$ is homotopic to $0$. 
On the generators  $\gamma_1, ..., \gamma_d$ of $\Gamma$, we have 
\[\psi (\gamma_i) = [\epsilon]^{\frac{a_i}{p^{r_i}}}\] as in (\ref{eq:form_of_char}). Without loss of generality, let us assume that $r_1 \ge r_2 \ge ... \ge r_d$, so in particular $r_1 \ge 1$ since $\chi$ is nontrivial. By modifying the generators $\gamma_i$ if necessary, we may also assume that each $a_i = 1$. Now the complex $R \Gamma_{\text{ct}}(\Gamma, \Ainf(R_\infty^\square)_{\psi})$ can be computed by the completed tensor product of the complexes 
$$ \Ainf (R_\infty^\square)_{\psi} \xrightarrow{ [\epsilon]^{\frac{1}{p^{r_i}}} - 1} \Ainf (R_\infty^\square)_{\psi}. $$
Note that $\big([\epsilon]^{\frac{1}{p^{r_1}}} - 1 \big) \vline \big([\epsilon]^{\frac{1}{p}} - 1 \big) $, it suffices to show that multiplication by $\varphi^{-1}(\mu) = [\epsilon]^{\frac{1}{p}} - 1$ on the complex 
$$ \Ainf (R_\infty^\square)_{\psi} \xrightarrow{ [\epsilon]^{\frac{1}{p}} - 1} \Ainf (R_\infty^\square)_{\psi} $$
is homotopic to $0$. This follows from (the proof of) \cite[Lemma 9.6]{BMS1}. 
\eproof

We will also need the following lemma.  

\bl \label{lemma:each_chi_component_is_tf}
Let $\chi: \Gamma \ra \mu_\infty$ be a nontrivial character as above and let $\psi: \Gamma \ra \Ainf^\times$ be a continuous character lifting $\chi$.  Then \[H^i_{\text{ct}} (\Gamma, \Ainf (R_\infty)_\psi /\mu)\] is $p$-torsion free. 
\el

\bproof 
The proof is similar to the proof  of \cite[Proposition 3.19]{CK_semistable}. Arguing as in the proof of Lemma \ref{lemma:killing_junk_torsion}, the cohomology group $R \Gamma_{\text{ct}}(\Gamma, \Ainf(R_\infty)_{\psi}/\mu)$ can be computed by the $p$-completed tensor product of the complexes 
$$ \Ainf (R_\infty)_{\psi}/\mu \xrightarrow{ [\epsilon]^{\frac{1}{p^{r_i}}} - 1} \Ainf (R_\infty)_{\psi}/\mu $$
with $r_1 \ge r_2 \ge \cdots \ge r_d$ and $r_1 \ge 1$. When $i =1$, write $r = r_1$, then the complex is given by 
$$ \Ainf (R_\infty )_{\psi}/\mu  \otimes_{\Ainf}^\L  \Ainf (R_\infty)_{\psi}/ ([\epsilon]^{\frac{1}{p^r}} - 1)$$
which is quasi-isomorphic to 
$$  \Big( \Ainf (R_\infty)_{\psi}/ ([\epsilon]^{\frac{1}{p^r}} - 1)  \xrightarrow{ \:\: \: 0 \: \: \:  }  \Ainf (R_\infty)_{\psi}/ ([\epsilon]^{\frac{1}{p^r}} - 1)  \Big).$$
Since $r_1 \ge r_2 \ge \cdots$, we know that each $H^i_{\text{ct}} (\Gamma, \Ainf (R_\infty)_\psi /\mu)$ is a direct sum of some copies of 
\[\Ainf (R_\infty)_{\psi}/ ([\epsilon]^{\frac{1}{p^r}} - 1).\] It then suffices to show that $\Ainf (R_\infty)_{\psi}/\varphi^{-r} (\mu)$ is $p$-torsion free. For this we may reduce to $\Ainf (R_\infty)_{\chi}/\varphi^{-r} (\mu)$ (cf.  (\ref{eq:component_over_chi})), and then to $\Ainf(R_\infty^{\square})/\varphi^{-r}(\mu)$. But we know that $\Ainf(R_\infty^\square)$ is $(p, \mu)$-completely flat (in fact free) over $\Ainf$, the claim thus follows.  
\eproof

\addtocontents{toc}{\protect\setcounter{tocdepth}{2}}
\subsection{Local computation of $A \Omega_{\fX}^{\log, \mathrm{psh}}$}
\label{ss:local_computation_AOmega}  

Our main goal of this subsection is to setup the proofs for the quasi-isomorphisms in (\ref{eq:qi_group_proket_Ainf}) and (\ref{eq:qi_group_Ainf_mod_xi}), from which Proposition \ref{prop:HT_presheaf} would follow. Similar to \cite{CK_semistable} (and slightly different from \cite{BMS1}), our proof relies on the following:

\bl \label{lemma:background_lemma_on_Ainf_gp_computation}
\noindent 
\be
\item Let $\mB \ra \mB'$ be a map in $\mD(\Ainf)$ and let $\mC$ be its cone. Suppose that  
    \be
     \item $W(\fm^\flat)$ kills each $H^i (\mC)$
     \item  $H^i(\mB \otimes^\L_{\Ainf} \Ainf/\mu)$ has no nontrivial $W(\fm^\flat)$-torsion for each $i$.  
     \ee 
Then $L \eta_{\mu}$ turns the map $\mB \ra \mB'$ into a quasi-isomorphism 
$$L \eta_\mu \mB \isom  L \eta_\mu \mB' .$$
\item Let $\mB \in \mD(\Ainf)$ be a complex such that $H^i (\mB \otimes^\L_{\Ainf} \Ainf/\mu)$ has no nontrivial $p$-torsion, then the natural map 
$$ ( L \eta_{(\mu)} \mB ) \otimes^\L_{\Ainf} \Ainf/\varphi(\xi) \lra L \eta_{(\zeta_p - 1)} (\mB \otimes^\L_{\Ainf} \Ainf/\varphi(\xi)) $$ is a quasi-isomorphism. 
\ee 
\el 

\bproof 
The first claim is \cite[Lemma 3.19]{CK_semistable}, the second is a special case of \cite[Lemma 5.16]{Bhatt}
\eproof

From Lemma \ref{lemma:background_lemma_on_Ainf_gp_computation}, it is clear that we need to analyze torsion in (various versions of) group cohomology, let us start with

\bl \label{lemma:no_almost_zero_for_flat}
For any $b \in \mO$, $i \in \N$, $R_\infty^\flat/b$ and $H^i_{\text{ct}}(\Gamma, R_\infty^\flat/b)$ both have no nontrivial almost zero elements (i.e., $\fm^\flat$-torsion).  
\el 

\bproof 
This follows from Lemma \ref{lemma:group_coho_local}, exactly as in \cite[Lemma 3.12]{CK_semistable}.
\eproof

Next we establish the following analog of \cite[Proposition 3.19]{CK_semistable}: 

\bp \label{prop:over_Ainf_no_almost_zero}
Keep notation from above. For each $i$, the cohomology 
\[H^i_{\text{ct}} (\Gamma, \Ainf(R_\infty)/\mu)
\]
\be
\item  is $p$-torsion free, 
\item is $p$-adically complete,
\item 
has no nontrivial $W(\fm^\flat)$-torsion. 
\ee
\ep 

\bproof  The proof of Part (1) is entirely analogous to that of \cite[Proposition 3.19]{CK_semistable}.  To show that $H^i_{\text{ct}} (\Gamma, \Ainf (R_\infty)/\mu)$ is $p$-torsion free,  it suffices to show that the group cohomology \[H^i_{\text{ct}} (\Gamma, \Ainf (R_\infty)_\chi/\mu)\] for each summand in the completed direct sum is $p$-torsion free (by \cite[Lemma 3.6]{CK_semistable}). If $\chi = 1$,  then $H^i_{\text{ct}} (\Gamma, A(R)/\mu)$ is a direct sum of copies of $A(R)/\mu$, which is $p$-torsion free. If $\chi \ne 1$ is nontrivial, then this is the content of Lemma \ref{lemma:each_chi_component_is_tf}.

Now we prove Part (2). We know that $H^i_{\text{ct}}(\Gamma, \Ainf(R_\infty)/\mu)$ is $p$-torsion free, we have (by the long exact sequence associated to group cohomology) 
\begin{equation}  \label{eq:mod_mu_is_torsion_free}
 H^i_{\text{ct}} (\Gamma, \Ainf (R_\infty)/\mu) \otimes_{\Ainf/\mu} \Ainf/(\mu, p^n) \cong H^i_{\text{ct}} (\Gamma, \Ainf (R_\infty)/(\mu, p^n)),
 \end{equation}
since $\Ainf(R_\infty)/\mu$ is $p$-torsion free. This further implies that the following sequence
\begin{multline} \label{eq:no_R1_lim} \quad 
0 \lra H^i_{\text{ct}} (\Gamma, \Ainf (R_\infty)/(\mu, p^n)) [p] \lra 
H^i_{\text{ct}} (\Gamma, \Ainf (R_\infty)/(\mu, p^n)) \\  \lra 
H^i_{\text{ct}} (\Gamma, \Ainf (R_\infty)/(\mu, p^{n-1})) \lra 0 
\end{multline}
 is short exact for each $i$. In particular, we have 
 $$R^1\text{lim} H^i_{\text{ct}} (\Gamma, \Ainf(R_\infty)/(\mu, p^n)) = 0. $$
Therefore, we have 
$$ H^i_{\text{ct}} (\Gamma, \Ainf(R_\infty)/\mu) \cong \varprojlim H^i_{\text{ct}} (\Gamma, \Ainf(R_\infty)/(\mu, p^n)).$$ 
Combining this with the isomorphism (\ref{eq:mod_mu_is_torsion_free}), we know that $H^i_{\text{ct}} (\Gamma, \Ainf (R_\infty^\square)/\mu)$ is indeed $p$-adically complete. 

For part (3), it suffices to show that for each $n$, $H^i_{\text{ct}} (\Gamma, \Ainf(R_\infty)/(\mu, p^n))$ has no nontrivial $W(\fm^\flat)$-torsion. Let us consider the short exact sequence 
$$ 0 \ra \Ainf (R_\infty)/(\mu, p) \xrightarrow{ p^{n-1}}  \Ainf (R_\infty)/(\mu, p^n) \ra  \Ainf (R_\infty)/(\mu, p^{n-1}) \ra 0. $$ 
Applying $H^i_{\text{ct}} (\Gamma, -)$ to the sequence remains short exact by (\ref{eq:no_R1_lim}). Thus it suffices to show that 
\[H^i_{\text{ct}} (\Gamma, \Ainf (R_\infty)/(\mu, p) ) = H^i_{\text{ct}} (\Gamma,  R_\infty^\flat/\mu)\] has no nontrivial $W(\fm^\flat) = \fm^\flat$-torsion, 
but this is the content of Lemma \ref{lemma:no_almost_zero_for_flat}. 
\eproof

Now we are ready to show that the map in  (\ref{eq:qi_group_proket_Ainf})  is a quasi-isomorphism. 
\bc \label{cor:qi_group_proket_Ainf}
Then the coordinate $\square$ induces a quasi-isomorphism 
$$ \gamma^{\square}:  A \Omega_{\ul R}^{\square, \gp}=L \eta_{\mu} R \Gamma_{\text{ct}} (\Gamma, \Ainf (R_\infty))  \isom A \Omega^{\proket}_{\ul R} = L \eta_{\mu} R \Gamma (U_{\proket}, \widehat{\Ainfx}) $$
in $\mD(\Ainf)$. 
\ec 

\bproof 
The Kummer pro-\'etale cover considered in Subsection \ref{ss:proket_cover} induces a map 
\begin{equation} \label{eq:almost_qis_over_Ainf} 
e: R \Gamma_{\text{ct}}(\Gamma, \Ainf (R_\infty))  \lra R \Gamma (U_{\proket}, \widehat{\Ainfx})
\end{equation}
which is an almost quasi-isomorphism. In other words, let $\mC = \text{Cone}(e)$ be the cone of this map in $\mD(\Ainf)$, then $[\fm^\flat]$ kills each $H^i (\mC)$.  This 
reduces to prove that the natural map 
\[
H^i_{\text{ct}}(\Gamma, R_\infty^\flat)  \lra H^i (U_{\proket}, \widehat \mO_X^{\flat +} )\] is an almost isomorphism  (with respect to $\fm^\flat \subset \mO^\flat$), which follows from almost purity. 

Next, note that, since both $A \Omega_{\ul R}^{\square, \gp}$ and $A \Omega^{\proket}_{\ul R}$are derived $p$-complete (see Remark \ref{remark:derived_complete}), so the complex $\mC = \text{Cone} (e)$, and thus each $H^i (\mC)$ is derived $p$-complete. Therefore, $W(\fm^\flat)$ kills each $H^i (\mC)$ by \cite[Lemma 3.17]{CK_semistable}. This shows that $\mC$ satisfies the first condition of Part (1) of Lemma  \ref{lemma:background_lemma_on_Ainf_gp_computation}. For the second requirement, note that 
$$ R \Gamma_{\text{ct}}(\Gamma, \Ainf (R_\infty)) \otimes^\L \Ainf /\mu \cong R \Gamma_{\text{ct}}(\Gamma, \Ainf (R_\infty)/\mu) $$
so we can apply Proposition \ref{prop:over_Ainf_no_almost_zero}. The corollary now follows from Lemma  \ref{lemma:background_lemma_on_Ainf_gp_computation}. 
\eproof

\subsection{Hodge--Tate and de Rham comparison}


Now we put the ingredients together to prove the Hodge--Tate and de Rham comparison. First let us take care of the map (\ref{eq:qi_group_Ainf_mod_xi}). Retain notations from the beginning of this section. 

\bl \label{lemma:qi_group_Ainf_mod_xi}
The $\Gamma$-equivariant surjection 
\[\sq \theta = \theta \circ \varphi^{-1}: \Ainf (R_\infty) \ra R_\infty\] 
 induces a quasi-isomorphism 
$$ L \eta_\mu R \Gamma_{\text{ct}} (\Gamma, \Ainf (R_\infty)) \otimes^\L \Ainf/\varphi (\xi) \isom L \eta_{(\zeta_p-1)} R \Gamma_{\text{ct}} (\Gamma, R_\infty)  $$
\el

\bproof
By Lemma \ref{lemma:group_coho_local} and Lemma \ref{lemma:killing_junk_torsion}, it suffices to show that $\sq \theta$ induces a quasi-isomorphism 
$$  L \eta_\mu R \Gamma_{\text{ct}} (\Gamma, A(R)) \otimes^\L \Ainf/\varphi (\xi) \isom L \eta_{(\zeta_p-1)} R \Gamma_{\text{ct}} (\Gamma, R).$$
Note that 
\[A(R) =\widehat \midoplus_{\psi} \Ainf (R_\infty)_{\psi }\] where the direct sum is taken over the set $\mS$ of all $\psi: \Gamma \ra \Ainf^\times$ which lift the trivial character $\text{triv}: \Gamma \ra \mO^\times$ and are of the form described in Subsection \ref{ss:proket_cover}. Let us partition the set $\mS = \mS_1 \sqcup \mS_2$ as follows. Let $\gamma_1, ..., \gamma_d$ be the topological generators of $\Gamma$ as before, for any $\psi \in \mS$ and each $\gamma_i$, we have 
\begin{equation} \label{eq:defining_S_1}
\psi (\gamma_i) = [\epsilon]^{a_i \cdot p^{s_i}} 
\end{equation}  where $s_i \ge 0$ and $a_i \in \Z_p^\times$ (see (\ref{eq:form_of_char}) and the discussion above it). We say that $\psi \in \mS_1$ if for each $i$, $s_i \ge 1$ in the expression (\ref{eq:defining_S_1}), otherwise it lies in $\mS_2$. We claim that 
$$    L \eta_\mu R \Gamma_{\text{ct}} (\Gamma, \widehat \oplus_{\psi \in \mS_1} \Ainf (R_\infty)_{\psi})   \isom L \eta_\mu R \Gamma_{\text{ct}} (\Gamma, A(R))  $$
is a quasi-isomorphism. The proof of this claim is essentially the same as the proof of Lemma \ref{lemma:killing_junk_torsion}, namely each $H^i_{\text{ct}}(\Gamma, \Ainf(R_\infty)_{\psi})$ is entirely $\mu$-torsion for $\psi \in \mS_2$. Therefore, it suffices to show that 
\begin{equation} \label{eq:S_1_component}    L \eta_\mu R \Gamma_{\text{ct}} (\Gamma, \widehat \oplus_{\psi \in \mS_1} \Ainf (R_\infty)_{\psi})  \otimes_{\Ainf}^\L \Ainf/\varphi (\xi) \isom L \eta_{(\zeta_p-1)} R \Gamma_{\text{ct}} (\Gamma, R).
\end{equation}
Now let us analyze the $\psi$-component of the group cohomology. By the proof of Lemma \ref{lemma:killing_junk_torsion}, we see that (up to modifying the generators $\gamma_i$ by $p$-adic units), for $\psi \in \mS_1$, 
\[R \Gamma_{\text{ct}} (\Gamma,   \Ainf (R_\infty)_{\psi})\] is computed by the Koszul complex 
\[  \text{Kos}(\Ainf(R_\infty)_{\psi}; \gamma_i -1 ) \cong  \text{Kos}(\Ainf(R_\infty)_{\psi}; \: [\epsilon]^{p^{s_1}} -1, \cdots , [\epsilon]^{p^{s_d}} -1 ).\]
Note that $\mu = [\epsilon] - 1$ divides each $[\epsilon]^{p^{s_1} - 1}$, thus by \cite[Lemma 7.9]{BMS1} we have 
$$ L \eta_{\mu} R \Gamma_{\text{ct}} (\Gamma,   \Ainf (R_\infty)_{\psi}) \cong 
 \text{Kos}(\Ainf(R_\infty)_{\psi}; \: \frac{[\epsilon]^{p^{s_1}} -1}{[\epsilon]-1}, \cdots , \frac{[\epsilon]^{p^{s_d}} -1}{[\epsilon]-1})
$$
Note that each $\Ainf (R_\infty)_\psi$ is $(p, \mu)$-adically flat over $\Ainf$, and $\varphi(\xi) =  \frac{[\epsilon]^{p -1}}{[\epsilon]-1}$ divides $ \frac{[\epsilon]^{p^{s_i}} -1}{[\epsilon]-1}$ for each $i$, so the left side of (\ref{eq:S_1_component}) can be identified with 
\begin{multline*}
 \qquad     L \eta_\mu R \Gamma_{\text{ct}} (\Gamma, \widehat \oplus_{\psi \in \mS_1} \Ainf (R_\infty)_{\psi})  \otimes_{\Ainf}^\L \Ainf/\varphi (\xi) \\
     \cong \text{Kos} ( \widehat \oplus_{\psi \in \mS_1} \Ainf (R_\infty)_{\psi}; \: 0, \cdots, 0). \qquad 
\end{multline*}
By the computation of the right side of  (\ref{eq:S_1_component}) using Koszul complexes, it suffices to show that the map 
$$ \widehat \oplus_{\psi \in \mS_1} \Ainf (R_\infty)_{\psi} \hookrightarrow A(R) \hookrightarrow \Ainf(R_\infty) \xrightarrow{\: \sq \theta \: } R_\infty $$
induces  isomorphisms
$$ \widehat \oplus_{\psi \in \mS_1}  \Ainf (R_\infty)_{\psi} /\varphi(\xi) \xrightarrow[\sim]{\varphi^{-1}}  A(R)/\xi \isom R. $$
This reduces to show that the map 
$$\varphi^{-1}:  \widehat \oplus_{\psi \in \mS_1}  \Ainf (R^\square_\infty)_{\psi} /\varphi(\xi) \lra  A(R^\square)/\xi  = \Ainf \gr{P } \otimes_{\Ainf \gr {N}} \Ainf /\xi $$ 
is an isomorphism, which is clear. 
\eproof

Now we finish the proof of Theorem \ref{thm:HT_comp}. 

\bproof[Proof of Proposition \ref{prop:HT_presheaf} (and Theorem \ref{thm:HT_comp})] Theorem \ref{thm:HT_comp} comes from Proposition \ref{prop:HT_presheaf} by applying the sheafification functor $j^{-1}$, while the latter statement follows from combining the isomorphisms (\ref{eq:qi_group_proket_Ainf}), (\ref{eq:qi_group_Ainf_mod_xi}) (Corollary \ref{cor:qi_group_proket_Ainf} and Lemma \ref{lemma:qi_group_Ainf_mod_xi}) and Corollary \ref{cor:L_eta_turns_into_isom}. 
\eproof 

The de Rham comparison is then a consequence of the Hodge--Tate comparison. The proof is exactly the same as in \cite{BMS1}, which we include for completeness. 

\bproof[Proof of Theorem \ref{thm:dR_comp}]
This follows from the following sequence of isomorphisms 
\begin{align*} 
A \Omega^{\log}_{\fX} \otimes^\L_{\Ainf} \Ainf/\xi \: \: & = \: \:  (L \eta_\mu R \nu_* \widehat{\Ainfx}) \otimes^\L_{\Ainf} \Ainf/\xi \\ & \xrightarrow[\raise.5ex\hbox{$\sim$}]{ \: \varphi \: } \: \:   
 (L \eta_{\varphi(\mu)} R \nu_* \widehat{\Ainfx}) \otimes^\L_{\Ainf} \Ainf/\varphi(\xi)  \\
 &  \isom \: (L \eta_{\varphi(\xi)} L \eta_\mu R \nu_*  \widehat{\Ainfx} ) \otimes^\L_{\Ainf} \Ainf/\varphi(\xi) \\ 
 & \isom \mH^{\bullet} (A \Omega_{\fX}^{\log} \otimes_{\Ainf}^\L \Ainf/\varphi(\xi)) \{\bullet\} \\
  & \isom \mH^\bullet  (\sq \Omega^{\log}_{\fX}) \{\bullet\} \\
  & \cong \Omega^{\bullet, (\textup{ct})}_{\fX/\ul \mO}.
\end{align*}
Here the third last isomorphism follows from \cite[Proposition 6.12]{BMS1}, the second last isomorphism is the Hodge--Tate comparison (Theorem \ref{thm:HT_comp}) and the last identification is the primitive Hodge--Tate comparisom (Theorem \ref{thm:primitive_HT}). 
\eproof

\subsection{More on $A \Omega^{\log}_{\fX}$} \label{ss:proof_of_more_on_AOmega}

The goal of this subsection is to prove Proposition \ref{prop:presheaf_is_sheaf}, which asserts that the presheaf $A \Omega^{\log, \psh}_{\fX}$ is a sheaf. 
Let us first note that $A \Omega^{\log, \psh}_{\fX}$ is derived $(p, \xi)$-adically complete (thus derived $(p, \varphi(\xi))$-adically complete).  

\bl \label{lemma:presheaf_is_sheaf}
Suppose that $\fX$ is a locally small quasi-fs log formal scheme over $\ul \mO_C$, then the natural map 
$$ \sq \Omega_{\fX}^{\log, \psh} \lra R j_*  \sq \Omega_{\fX}^{\log}$$
is a quasi-isomorphism. 
\el 

\bproof 
Evaluating on an affine open $\fU = \spf R \in \text{Ob}(\fX_{\ett, \text{small}})$, the map becomes 
$$c: L \eta_{(\zeta_p -1)} R \Gamma (U_{\proket}, \OXplus) \lra  R \Gamma (\fU_{\ett}, L \eta_{(\zeta_p - 1)} R \nu_* \OXplus),$$
 namely the map in (\ref{eq:switching_L_eta_for_sq_Omega}), which is a quasi-isomorphism by Proposition  \ref{prop:primitive_HT_local}.   
\eproof 

By the same argument of \cite[Corollary 4.6]{CK_semistable}, this implies that $ A \Omega^{\log, \psh}_{\fX} $ is already a sheaf. For completeness let us include the argument below. 

\bproof[Proof of Proposition \ref{prop:presheaf_is_sheaf}] 
Since $ A \Omega^{\log, \psh}_{\fX} $ is derived $\varphi(\xi)$-complete, by (the \'etale analog of) \cite[Lemma 9.15]{BMS1}, it suffices to show that 
$$ A \Omega^{\log, \psh}_{\fX} \otimes^\L \Ainf/\varphi(\xi)^r \isom R j_*  j^{-1}(A  \Omega^{\log, \psh}_{\fX} \otimes^\L \Ainf/\varphi(\xi)^r)$$
is a quasi-isomorphism for each $r \ge 1$. This further reduces to the case where $r = 1$. The claim thus follows from Proposition \ref{prop:HT_presheaf} and Lemma \ref{lemma:presheaf_is_sheaf}. \eproof

\bc \label{cor:AOmega_derived_complete}
As before, assume that $\fX$ is admissibly smooth over $\ul{\mO_C}$. Then $A \Omega_{\fX}^{\log}$ is derived $(p, \xi)$-adically complete (thus also derived $\varphi(\xi)$-complete). 
\ec 

\bproof Let us prove that $A \Omega_{\fX}^{\log}$ is derived $\xi$-complete (the argument for derived $p$-completeness is similar). We need to show that 
\begin{equation} \label{eq:derived_xi_complete}
A \Omega_{\fX}^{\log} \lra \underset{\longleftarrow}{\text{Rlim}} (A \Omega_{\fX}^{\log} \otimes^\L_{\Ainf} \Ainf/\xi^n) 
\end{equation} 
is a quasi-isomorphism. As $j^{-1} \circ R j_* \cong \textup{id}$,  it suffices to check this after applying $R j_*$ to (\ref{eq:derived_xi_complete}). This indeed becomes an isomorphism, since we have
\begin{align*} 
R j_* A \Omega_{\fX}^{\log} \;\xleftarrow{\: {}_{\sim} \:} \;  A \Omega_{\fX}^{\log, \psh}
& \isom  \underset{\longleftarrow}{\text{Rlim}} (A \Omega_{\fX}^{\log, \psh} \otimes^\L_{\Ainf} \Ainf/\xi^n) \\
& \; \cong \:  \: 
R j_*  \underset{\longleftarrow}{\text{Rlim}} (A \Omega_{\fX}^{\log} \otimes^\L_{\Ainf} \Ainf/\xi^n).
\end{align*}
Here the first isomorphism is Proposition \ref{prop:presheaf_is_sheaf} and the middle isomorphism uses that $A \Omega_{\fX}^{\log, \psh}$ is derived $(p, \mu)$-adically complete (cf. Remark \ref{remark:derived_complete}). 
\eproof 

\bc 
Same assumption as above. The log $\Ainf$-cohomology  
\[R \Gamma_{\Ainf} (\fX) = R \Gamma( \fX,  A \Omega_{\fX}^{\log})\] is a perfect complex in $\mD(\Ainf)$. 
\ec 

\bproof 
This is a consequence of the previous corollary and the de Rham comparison theorem (cf. Theorem \ref{thm:dR_comp}). 
\eproof 

\bc[Multiplicativity] \label{cor:multiplicative} 
Consider the setup of Remark \ref{remark:multiplicativity_tilde}, let $\ul{R_1}, \ul{R_2}$ be two $p$-complete pre-log rings over $\ul \mO$ which give rise to small \'etale neighbourhoods of $\fX$.  
There is a natural quasi-isomorphism 
\[ A \Omega^{\log}_{\ul{R_1}} \: \widehat \otimes^\L_{\Ainf} \:  A \Omega^{\log}_{\ul{R_2}} \isom A \Omega^{\log}_{ \ul{R_1}\otimes_{\ul \mO} \ul{R_2} }.
\] 
Here the derived completion is $(p, \xi)$-adic. 
\ec 

\bproof 
The map comes from functoriality of the construction $\ul R \mapsto A \Omega^{\log}_{\ul R}$. By Corollary \ref{cor:AOmega_derived_complete} and derived Nakayama's lemma it suffices to show that the map is a quasi-isomorphism after (derived) mod $\varphi(\xi)$. Using the Hodge--Tate comparison, this follows from Remark \ref{remark:multiplicativity_tilde}. (One may also argue mod $\xi$ and use the de Rham comparison instead).   
\eproof


\newpage

\section{Derived log $\Ainf$-cohomology}
\label{section:derived}

In this short section we define a derived version of log $\Ainf$-cohomology (compare with \cite[Construction 7.6]{BS} for the nonlog case and \cite[6.8]{Bhatt_dR} for the derived log de Rham cohomology).  
Note that this is analogous to the procedure of obtaining a derived version of log prismatic cohomology in \cite{logprism}.

\subsection{Derived log $\Ainf$-cohomology}

\begin{construction} \label{construction:derived_Ainf_on_prelog}
Fix a split pre-log ring $\ul{\mO_C} = (\mO_C, N_\infty)$ as in Section \ref{section:HT_primitive}.  The category of derived $p$-complete pre-log rings over $\ul{\mO_C}$ is complete and generated under colimits by pre-log rings of the form of
\[
\Sigma(S, \ul T) :=  (\mO_C \gr{(X_s)_{ s\in S}, \N^T }, N_\infty \oplus \N^T)
\]
for finite sets $S,T$. 
We consider the $\infty$-category of derived $p$-complete simplicial pre-log rings over $\ul{\mO_C}$ (these are ``animated $p$-complete pre-log rings''). 
We have a functor
\begin{equation}\label{eq:functor_on_affine}
 \Sigma(S, \ul T) 
\longmapsto
R \Gamma_{\Ainf} (\spf (\Sigma (S, \ul T))^a) \end{equation} 
to the $\infty$-category $\mathcal{D}(\Ainf)$ of $A$-modules. In fact, each value is a derived $(p, \mu)$-complete commutative algebra in $\mathcal{D}(\Ainf)$ and is equipped with a $\varphi_A$-semilinear map $\varphi$, so we may regard the functor above as a functor to the $\infty$-category of such objects. 

\begin{definition}
Fix $\ul{\mO_C}$ as above. The \emph{derived log $\Ainf$-cohomology} is the functor obtained as the left Kan extension from $p$-complete pre-log rings 
\[\Sigma (S, \ul T) =  (\mO_C \langle (X_s)_{ s\in S}, \N^T \rangle, N_\infty \oplus\N^T)\] to \emph{all}  derived $p$-complete simplicial pre-log rings over  $\ul{\mO_C}$. For a derived $p$-complete (simplicial) pre-log ring $\ul R = (R, M)$ over $\ul{\mO_C}$, we write $A \Omega^{\L,\log}_{\ul R/\ul{\mO_C}}$ (or simply $A \Omega^{\L,\log}_{\ul R}$ if the base pre-log ring  is understood) for the derived log $\Ainf$-cohomology of $\ul R$, equipped with the Frobenius map
\[
\varphi: A \Omega^{\L,\log}_{\ul R} \to \varphi_{\Ainf,*} A \Omega^{\L,\log}_{\ul R}.
\]  

We similarly define the functor $\ul R \mapsto \sq \Omega_{\ul R}^{\L, \log}$. 
\end{definition}
\end{construction}

\bc[Derived Hodge--Tate comparison] \label{cor:derived_HT}
For any derived $p$-complete (simplicial) pre-log ring $\ul R$ over $\ul{\mO_C}$, the derived $\Ainf$-cohomology $A \Omega^{\log}_{\ul R}$ satisfies 
\[A \Omega_{\ul R}^{\L, \log} \otimes_{\Ainf}^\L \Ainf/\varphi(\xi) \cong \sq \Omega_{\ul R}^{\L, \log}.\]
Moreover, $\sq \Omega_{\ul R}^{\L, \log}$
admits an $\N$-indexed increasing exhaustive filtration $\tu{Fil}^{\tu{conj}}_\bullet \sq \Omega_{\ul R}^{\L, \log}$ (called the \textit{conjugate filtration} on $\sq \Omega_{\ul R}^{\L, \log}$) such that its $i^{th}$-graded piece is given by sending $\ul R$ to $\widehat \L^i_{\ul R/\ul{\mO_C}} [-1]$,  the derived $p$-completion of $\wedge^i \L_{\ul R/\ul{\mO_C}}[-i]$. 
\ec 

\bproof 
This follows from the Hodge--Tate comparison (Theorem \ref{thm:HT_comp} and Theorem \ref{thm:primitive_HT}). 
\eproof 

\bc[Derived de Rham comparison] \label{cor:derived_dR}
For any derived $p$-complete (simplicial) pre-log ring $\ul R$ over $\ul{\mO_C}$, the derived $\Ainf$-cohomology $A \Omega^{\log}_{\ul R}$ satisfies the derived de Rham comparison 
\[ A \Omega^{\L, \log}_{\ul R}   \otimes^\L_{\Ainf} \mO_C \isom \widehat \L {\Omega}_{\ul R/ \ul \mO_C}.\]
\ec 

\bproof 
This follows from de Rham comparison (Theorem \ref{thm:dR_comp}) for $\ul R = \Sigma(S, \ul{T})$. 
\eproof 

\bc[Derived multiplicativity] \label{cor:derived_multiplicative}
For any derived $p$-complete (simplicial) pre-log rings  $\ul{R_1}, \ul{R_2}$ over $\ul{\mO_C}$, we have a natural  quasi-isomorphism 
\[
A \Omega^{\L, \log}_{\ul{R_1}} \widehat \otimes^\L_{\Ainf} A \Omega^{\L, \log}_{\ul{R_2}} \isom 
A \Omega^{\L, \log}_{\ul{R_1} \widehat \otimes_{\ul{\mO}} \; \ul{R_2} }. 
\]
\ec 

\bproof 
This follows from the statement for $\ul{R_1} = \Sigma(S_1, \ul{T_1})$ and $\ul{R_2} = \Sigma (S_2, \ul{T_2})$, which is Corollary \ref{cor:multiplicative}. 
\eproof 

Let us record some consequences of the derived Hodge--Tate comparison. 

\bc[Change of base] \label{cor:change_of_base}
For any split perfectoid pre-log ring $\ul{\mO_C}$, and any derived $p$-complete simplicial pre-log ring $\ul R$ over $\ul{\mO_C}$, the map $\mO_C \ra \ul{\mO_C}$ of pre-log rings (where $\mO_C$ is viewed as a pre-log ring equipped with the trivial pre-log structure) induces a canonical isomorphism 
\[
A \Omega^{\L, \log}_{\ul R/\mO_C} \isom  A \Omega^{\L, \log}_{\ul R/\ul{\mO_C}}  
\]
\ec 

\bproof 
This follows from Corollary \ref{cor:derived_HT} and Corollary \ref{cor:cotangent_for_perfectoid}. 
\eproof

\begin{proposition}[Quasisyntomic descent]
Fix $\ul{\mO_C}$ as above.  On the log quasisyntomic site $\QSyn_{\ul{\mO_C}}^{\log, \opp}$, the presheaf
\[
\ul R \longmapsto A \Omega_{\ul R}^{\L, \log} 
\]
is a sheaf. 
\end{proposition}

\begin{proof}
Note that the derived $p$-completion of $\wedge^i \L_{\ul R/\ul{\mO_C}}[-i]$ lies in $D^{\geq 0}(R)$ for every $i$. Therefore, by derived Nakayama and the Hodge--Tate comparison, the problem reduces to the quasisyntomic descent for exterior powers of the cotangent complex, which is Proposition \ref{lem-descent-lfpqc-hodge-graded-piece}. 
\end{proof}

\subsection{Global derived log $\Ainf$-cohomology} \label{ss:global_derived}
\noindent 

\noindent 
Now consider a log $p$-adic formal scheme instead of a pre-log ring. 
Let $\fX = (\fX, \mM_{\fX})$ be a(ny) log $p$-adic formal scheme over $\mO_C$ (without restrictions on its singularities). Then we have a functor which assigns an affine object of the \'etale site $\fX_{\ett}$:
\[
\fU=\spf (R) \longmapsto   A \Omega^{\L, \log}_{(R, \Gamma (\fU, \mM_{\fX}))}. 
\]
We denote this functor by $A \Omega_{\fX}^{\L, \log, \tu{pre}}$. Note that, by Corollary \ref{cor:change_of_base}, the value of this functor does not depend on the log structure on the base $\mO_C$ (in other words, for a log $p$-adic formal scheme $\fX$ over a split perfectoid base $\ul{\mO_C}$, we may view it as a log formal scheme over either $\ul{\mO_C}$ or $\mO_C$, the value of the pre-sheaf will not change). We also remark that this presheaf valued on $\mD(\Ainf)$ may not be an \'etale sheaf.  

\begin{definition}[Derived log $\Ainf$-cohomology]
Let $\fX$ be a log $p$-adic formal scheme over $\mO_C$. We define $A \Omega_{\fX}^{\L,\log} \in \mD(\fX_{\ett}, \Ainf)$ to be the \'etale sheafification of the presheaf $A \Omega_{\fX}^{\L, \log, \tu{pre}}$. Its cohomology 
\[R \Gamma_{\Ainf}^{\L} (\fX):=R \Gamma (\fX_{\ett}, A \Omega_{\fX}^{\L, \log})\] is refered to as the \emph{derived log $\Ainf$-cohomology} of $\fX$. 
\end{definition}

Next we compare the derived log $\Ainf$-cohomology to the non-derived version when $\fX$ is reasonably smooth over some base $\ul{\mO_C}$. 

The following is an analog of \cite[Proposition 4.9]{logprism}. 

\begin{proposition} \label{prop:comparing_derived_with_non_derived_on_charts}
Suppose that $\fX$ is admissibly smooth over $\ul{\mO_C}$. 
\be
\item 
For any affine object $\fU=\spf (R)$ of $\fX_{\ett}$, the natural map
\[
A \Omega^{\L, \log}_{\fX} (\fU) \lra R \Gamma_{\Ainf} (\fU) =  R\Gamma (\fU_{\ett}, A \Omega_{\fX}^{\log}) 
\]
is an isomorphsim in $\mD(\Ainf)$. In other words, there is a natural isomorphism of \'etale sheaves in $\mD(\fX_{\ett}, \Ainf)$
\[
A \Omega_{\fX}^{\L, \log} \isom A \Omega_{\fX/\ul{\mO_C}}^{\log}.
\]
\item If $P \to \Gamma (\fU, \mM_{\fX})$ is a  chart for the log structure, then the natural map
\[
A \Omega^{\L, \log}_{(R, P)} \lra 
A \Omega^{\L, \log}_{\fX} (\fU) 
\]
is an isomorphism. 
\ee
\end{proposition}


\bproof 
Let us first prove (1). The map comes from functoriality (of left Kan extension) and is compatible with Hodge--Tate comparison on both sides, therefore it suffices to show that the \'etale sheafification of the functor that sends 
\[\spf R \in \fX_{\ett} \longmapsto \widehat \L_{(R, \Gamma (\fU, \mM_{\fX}))}\] agrees with the \'etale sheaf $\Omega_{\fX/\ul{\mO_C}}^{\bullet, (\tu{ct})}.$ To this end, the same argument of \cite[Proposition 4.9]{logprism} applies.  For (2), we again apply Hodge--Tate comparison, which reduces us to show that the natural map
\[
\widehat \L_{(R, P)/\ul{\mO_C}} \lra \widehat \Omega^1_{(R, P)/\ul{\mO_C}}
\]
is an isomorphism. This follows from \cite[Lemma 2.10]{logprism}. 
\eproof

\newpage

\section{Comparison with log prismatic and log crystalline cohomology} 
  \label{section:compare_AOmega}

In this section, we prove a comparison result between log $\Ainf$-cohomology and log prismatic cohomology (developed in \cite{Koshikawa, logprism}) over the base log prism $(\Ainf, (\xi), N_\infty)$ (see below for a brief review of the definition), where $N_\infty \cong N_\infty^\flat \ra \Ainf$ is induced from a split divisible pre-log structure $N_\infty \ra \mO_C$ as before. 
The first main result of this section is 

\bt \label{theorem:logprism_comparison}
 Let $\fX$ be a quasi-fs log $p$-adic formal scheme that is admissbly smooth over $\ul{\mO_C} = (\mO_C, N_\infty)$,   then there exists a functorial Frobenius equivariant isomorphism 
\[ \varphi^* \Prism_{ \fX/(\Ainf, N_\infty)} = \Prism_{ \fX/(\Ainf,   N_\infty)} \widehat \otimes^{\L}_{\Ainf, \varphi} \Ainf \isom  A \Omega^{\log}_{ \fX} \]
between log prismatic cohomology and  log $\Ainf$-cohomology (as sheaves $E_\infty$-$\Ainf$-algebras on ${\fX}_{\text{\'et}}$, where the former is defined below in (\ref{eq:sheaf_of_logprism_cohomology})). 
In particular, 
we have a functorial Frobenius equivariant isomorphism 
\begin{equation}
\varphi^* R \Gamma_{\Prism} (\fX / (\Ainf, (\xi), N_\infty) \isom 
 R \Gamma_{\Ainf} (\fX/\ul{\mO_C}).  
\end{equation}
\et

Similar to \cite{BS}, our proof goes through the log $q$-crystalline site constructed in \cite{Koshikawa}. As a result, we deduce the following comparison theorem over $\Acris$.

\bt  \label{thm:Acris_comparison}
Let $\fX$ be a log $p$-adic formal scheme that is admissibly smooth over $\ul{\mO_C}$, then there exists a functorial quasi-isomorphism 
\begin{equation} \label{eq:Acris_comparison} 
R \Gamma_{\Ainf} (\fX/\ul{\mO_C}) \widehat \otimes^\L_{\Ainf} \Acris 
\isom R \Gamma_{\textup{logcrys}} (\fX_{\mO_C/p}/(\Acris, N_\infty)) 
\end{equation}
which is compatible with the Frobenius on both sides. 
\et

\bproof[Proof of Theorem \ref{thm:Acris_comparison} from Theorem \ref{theorem:logprism_comparison}] Write the base change of $\fX$ from $\ul{\mO_C}$ to $\ul{\mO_C}/p$ (resp. to $(\Acris/p, N_\infty)$) by $\fX_{\mO_C/p}$ (resp. $\fX_{\Acris/p}$). Moreover, write $\sq \fX^{(1)}$ for the Frobenius base change \[
\sq \fX^{(1)} := \fX \times_{(\spf \mO_C, N_\infty)^a, \varphi} (\spec \Acris/p, N_\infty)^a  
\]
where $\varphi$ is induced by the Frobenius map $\mO_C = \Ainf/\xi \ra \Acris/p$ on the rings and multiplication by $p$ on the monoid. Equivalently, 
\[\sq \fX^{(1)} = \fX_{\Acris/p} \times_{(\spec \Acris/p, N_\infty)^a \varphi} (\spec \Acris/p, N_\infty)^a.\]
Using Theorem \ref{theorem:logprism_comparison} and properties of log prismatic cohomology (see \cite[Theorem 1.2]{logprism}), we have functorial quasi-isomorphisms 
\begin{align}
\nonumber R \Gamma_{\Ainf} (\fX/\ul{\mO_C}) \widehat \otimes^\L_{\Ainf} \Acris   
& \isom R \Gamma_{\Prism} (\fX/(\Ainf, (\xi), N_\infty)) \widehat \otimes_{\Ainf, \varphi} \Acris\\ 
\label{eq:Acris_comparison_arrow2} & \isom R \Gamma_{\Prism}(\sq \fX^{(1)}/(\Acris, (p), N_\infty)) \\
\label{eq:Acris_comparison_arrow3} & \isom R \Gamma_{\Prism} (\fX_{\Acris/p}/(\Acris, (p), N_\infty)) \widehat \otimes^\L_{\Acris, \varphi} \Acris \\ 
\label{eq:Acris_comparison_arrow4} & \isom R \Gamma_{\textup{logcrys}} (\fX_{\Acris/p}/(\Acris, N_\infty))  \\
\nonumber & \isom R \Gamma_{\textup{logcrys}} (\fX_{\mO_C/p}/(\Acris, N_\infty)). 
\end{align}
Where the isomorphisms (\ref{eq:Acris_comparison_arrow2}) and (\ref{eq:Acris_comparison_arrow3}) follow from base change of log prismatic cohomology, and the isomorphism (\ref{eq:Acris_comparison_arrow4}) comes from the crystalline comparison of log prismatic cohomology (cf. \cite{Koshikawa, logprism}). 
\eproof

\br 
There are several ways to prove Theorem \ref{thm:Acris_comparison}. One method is to apply an ``all possible coordinate" method similar to \cite{BMS1}. Another possible approach is to follow the strategy of \cite{Yao_Acris}. The key point there is to construct a canonical map locally on pre-log rings using log quasisyntomic descent. However,  there are some additional complications caused by the presence of log strctures. For this approach, one should replace $A \Omega_{S/\mO_C}$ from \cite{Yao_Acris} by a certain twist of the derived log $\Ainf$ cohomology 
\begin{equation} \label{eq:Frob_twisted_over_nonperfect_base}
  A \Omega^{\L, \log}_{\ul S/\ul{\mO_C}} \otimes^\L_{\Z_p \gr{\sq M}, \varphi} \Z_p \gr{\sq M}  
\end{equation}
for  quasiregular semiperfectoid pre-log rings, where $\sq M \ra M$ is the exactification of $M^\flat \ra M$ and $\varphi$ is induced by $\sq M \xrightarrow{p} \sq M$. It is possible to show that for quasiregular semiperfectoid pre-log rings, the base change of the object in (\ref{eq:Frob_twisted_over_nonperfect_base})
to $\Acris$ can be upgraded to a ``universal log PD thickening" of $\ul S/p$ (see Lemma \ref{lemma:universal_PD_thickening} for a more precise statement. The natural object we should consider is actually the left hand side of (\ref{eq:universal_log_PD_compare_to_prism})). 
Another approach (probably the  shortest proof) is to use the comparison with log prismatic cohomology (Theorem \ref{theorem:logprism_comparison}) and apply the crystalline comparison for the latter (see \cite{Koshikawa, logprism}). In this paper, we decide to take the prismatic approach, and study  $A \Omega_{\ul S/\ul{\mO_C}}^{\L, \log} \widehat \otimes^\L_{\Ainf} \Acris$ for  quasiregular semiperfectoid pre-log rings \textit{a posteriori}. 
\er

As a consequence, we immediatly obtain Part (4) of Theorem \ref{mainthm:comparison} from the introduction, which states that for an admissibly smooth log $p$-adic formal scheme $\fX$ over $\ul{\mO_C}$, the log $\Ainf$-cohomology specializes to the log crystalline cohomology of the special fiber. More precisely, 

\bc \label{thm:Acris_for_locally_small} 
Let $\fX$ be a locally small log formal scheme over $\ul{\mO_C}$, then there exists a functorial quasi-isomorphism 
$$  R \Gamma_{\Ainf} (\fX/\ul{\mO_C}) \widehat \otimes^\L_{\Ainf} W(k) \isom R \Gamma_{\textup{logcrys}} (\fX_{k}/W(k)).$$
compatible with Frobenius. 
\ec

\subsection{Log prismatic and derived log prismatic cohomology} \label{ss:log_prism} \noindent 

\noindent 
Let us first briefly review the theory of log prismatic cohomology. Recall that, a $\delta_{\log}$-\textit{ring} is pre-log ring $(A, \alpha: M_A \ra A)$ equipped with additional structures: a map $\delta: A \ra A$ such that $(A, \delta)$ becomes a $\delta$-ring and a map $\delta_{\log}: M \ra A$ which satisfies the following conditions 
\bi
\item $\delta_{\log} (0) = 0,$ 
\item $\delta_{\log}(m+m') = \delta_{\log} (m) + \delta_{\log}(m') + p \delta_{\log} (m) \delta_{\log}(m') $, 
\item $\alpha(m)^p \cdot \delta_{\log}(m) = \delta (\alpha(m))$.
\ei
A \textit{pre-log prism} is a triple $(A, I, M_A)$ where $(A, M_A)$ is a $\delta_{\log}$-ring and $(A, I)$ is a prism in the sense of \cite{BS}. A \textit{log prism} is a bounded\footnote{This means that $(A, I)$ is a bounded prism in the sense of \cite[Defintiion 1.6]{BS}, in other words, $A/I$ has bounded $p^{\infty}$-torsion.} pre-log prism $(A, I, M_A)$ considered up to taking the associated log structure.  

Let $(A, I, M_A)$ be a pre-log prism and 
 $\fX$ be a log $p$-adic formal scheme   over $(A/I, M_A)$.  
Consider the category $(\fX/(A,M_A))_{\Prism}$, which is defined to be (the opposite of) the category of integral log prisms $(B, IB, \mM_{ B})$ equipped with 
\bi
\item a map $g: (A, I, M_A) \ra (B, IB, M_B)$ of pre-log prisms, where $M_B = \Gamma (\spf B, \mM_{ B}))$; 
\item a map $f\colon \spf (B/I) \ra \fX$ of log $p$-adic formal schemes, and
\item a strict closed immersion $(\spf (B/I), f^* \mM_{\fX}) \ra (\spf B,\mM_{ B})$ of   log formal schemes (here the target is a log  $(p, I)$-adic formal scheme).
\footnote{This is equivalent to saying that the closed immersion is exact.}
\ei 
A morphism 
\[(B, IB, \mM_{ B}) \ra (C, IC, \mM_{ C})\] in $(\fX/(A,M_A))_{\Prism}$ is an \'etale cover if $B \ra C$ is $(p, I)$-completely \'etale and faithfully flat, and the map on the associated log $(p, I)$-adic formal schemes is strict. 
The logarithmic prismatic site of $\fX$ over $(A, I, M_A)$ is $(\fX/(A,M_A))_{\Prism}$ equipped with the \'etale topology. It has a structure sheaf $\mO_{\Prism}$ (resp. reduced structure sheaf $\cl \mO_{\Prism}$) defined by sending the object 
$(B, IB, \mM_{ B}) \mapsto B$ (resp. $(B, IB, \mM_{ B}) \mapsto B/IB$). The \'etale sheaf that appears in the statement of Theorem \ref{theorem:logprism_comparison} is defined as 
\begin{equation} \label{eq:sheaf_of_logprism_cohomology}
\Prism_{\fX/(A, M_A)} := R \mu_{*} \mO_{\Prism} 
\end{equation} 
where $\mu:\tu{Shv} ((\fX/(A, M_A))_{\Prism}) \ra \tu{Shv} (\fX_{\ett})$ is the natural projection of topoi. The logarithmic prismatic cohomology of $\fX$ is defined as  
\[
R \Gamma_{\Prism}(\fX/(A, M_A)) := R \Gamma ((\fX/(A, M))_{\Prism}, \mO_{\Prism}) = R \Gamma (\fX_{\ett}, \fX/(A, M_A))_{\Prism}).
\]
This is an $E_\infty$-$A$-algebra equipped with a $\varphi_A$-semilinear endomorphism $\varphi$. 

\br \label{remark:derived_logprismatic} 
There is a derived theory of log prismatic cohomology which parallels the theory of derived $\Ainf$-cohomology developed in Section \ref{section:derived}. For a derived $p$-complete (simplicial) pre-log ring $(\ul R)$ over $(A/I, M_A)$, we denote its derived log prismatic cohomology by $\Prism^{\L}_{\ul R/(A, M_A)}$. For a log $p$-adic formal scheme $\fX$ over $(A/I, M_A)$, there is also a sheaf $\Prism^{\L}_{\fX/(A, M_A)}$ defined in a similar fashion as in Subsection \ref{ss:global_derived}. We refer the reader to \cite[Section 4]{logprism} for the detail of this construction. 
\er

\br \label{remark:relevant_perfectoid_log_prism}
The relevant pre-log prism for us is the following. Let $\ul{\mO_C} = (\mO_C, N_\infty)$
be a split divisible perfectoid pre-log ring as in previous sections. Let $N_\infty^{\flat} \cong N_\infty \ra \Ainf$ be the pre-log structure induced from $\alpha^{\flat}: N_\infty^{\flat} \ra \mO_C^\flat$ by composing with the Teichmuller lift $\mO_C^\flat \ra \Ainf$. Let $I = (\xi)$. Then $(\Ainf, I, N_\infty^{\flat})$ is a perfect pre-log prism in the sense of \cite[Section 2]{logprism}. Note that, for a derived $p$-complete $\ul R$ over $\ul{\mO_C}$, we have a natural isomorphism 
\[ 
\Prism^\L_{\ul R/\Ainf} \isom \Prism^{\L}_{\ul R/(\Ainf, N_\infty)},
\]
similar to Corollary \ref{cor:change_of_base}.
\er 

\br 
Moreover, $(\Ainf, I, N_\infty)$ is a pre-log $q$-PD triple in the sense of \cite[Definition 7.1]{Koshikawa}, with the map from $\Z_p [\![q-1]\!] \ra \Ainf$ given by $q \mapsto [\epsilon]$. In particular, $[p]_q \mapsto \varphi(\xi) \in \Ainf$ and $(\Ainf, (\varphi(\xi)), N_\infty^{\flat})$ is a bounded pre-log prism over $(\Z_p[\![q-1]\!], [p]_q)$. 
\er

\subsection{The comparison with log prismatic cohomology} 

For the proof of Theorem \ref{theorem:logprism_comparison}, it suffices to work \'etale locally on $\fX$. We will make use of the derived theory, and show that for any pre-log ring $\ul R$ over $\ul{\mO_C} = (\mO_C, N_\infty)$, we have a functorial $\varphi$-equivariant isomorphism 
\begin{equation} \label{eq:derived_iso} \varphi^* \Prism^\L_{\ul R/(\Ainf, N_\infty)}  \cong  A \Omega^{\L, \log}_{\ul R} 
\end{equation}
between derived prismatic cohomology and derived $\Ainf$-cohomology, and then deduce Theorem \ref{theorem:logprism_comparison} from this using de Rham comparison. To prove the isomorphism in (\ref{eq:derived_iso}), we will first deal with the log free case. In what follows we write $\mO = \mO_C$.

\bp \label{prop:Ainf_comp_free}
Let $S, J$ be finite sets and let 
\[\ul R_{S, J} := (\mO \gr{\{X_j\}_{j \in J}, \N^S}, N_\infty \oplus \N^S)\] be the pre-log ring which is log-free over $(\mO\gr{X_j}_{j \in J}, N_\infty)$. Then there exists a $\varphi$-equivariant isomorphism 
\[    \varphi^* \Prism_{\ul R_{S, J}/ (\Ainf,  N_\infty)} \isom A \Omega^{\log}_{\ul R_{S, J}} \] 
which is functorial in $S$ and $J$. 
\ep 

\bproof 
By multplicativity on both sides (\cite[Proposition 4.6]{logprism} and Corollary \ref{cor:derived_multiplicative}) and the comparison between the non-log version of prismatic and $\Ainf$-cohomology (\cite[Theorem 17.2]{BS}), we are reduced to the case when $J = \emptyset$. 

Now we write $\ul R_S := (\mO \gr{\N^S}, N_\infty \oplus \N^S)$.  In \cite{Koshikawa}, it is shown that there exists a canonical isomorphism of $E_\infty$-$\Ainf$-algebras 
\[  \varphi^* \Prism_{\ul R_S/ (\Ainf, I, N_\infty)} \isom q \Omega_{\ul R_S/(\Ainf, N_\infty)} \]
between (Frobenius-twisted) log prismatic cohomology and the log $q$-crystalline cohomology of $\ul R_S$, which is in particular functorial in $S$. Thus it suffices to prove Proposition \ref{prop:Ainf_comp_free} with the left side replaced by the log $q$-crystalline cohomology $q \Omega_{\ul R_S/(\Ainf, N_\infty)}$.  

For each $S$, write $D_S = \Ainf \gr{\N^S}$ for the $(p, \mu)$-adically completed polynomial ring over $\Ainf$. We have an exact surjection 
\[ 
\ul D_S  = (\Ainf \gr{\N^S}, N_\infty \oplus \N^S ) \lra \ul R_S = (\mO \gr{\N^S}, N_\infty \oplus \N^S) 
\]
of pre-log rings, which identifies the source $\ul D_S$ as the log $q$-PD envelop of this surjection (see \cite[Lemma 7.4]{Koshikawa}). By \cite[Theorem 7.17]{Koshikawa}, there is a canonical isomorphism 
\[ 
f_S: q\Omega_{\ul R_S/(\Ainf, N_\infty)} \isom q \Omega^\bullet_{\ul D_S /(\Ainf, N_\infty)} 
\]
where the right side is the log $q$-de Rham complex of $\ul D_S$ with respect to $(\Ainf, N_\infty)$ constructed in 7.15 of \textit{loc.cit.} The isomorphism $f_S$ is functorial with respect to the surjection $\ul D_S \ra \ul R_S$, in particular, it is functorial with respect to $S$. Therefore, it suffices to construct an isomorphism 
\[ g_S: q\Omega^\bullet_{\ul D_S/ (\Ainf, N_\infty)} \lra A \Omega^{\log}_{\ul R_S} \]
that is functorial in $S$. To this end, let $R_{S, \infty} = \mO \gr{\Q_{\ge 0}^{\oplus S}}$, then 
\[(R_{S, \infty}, N_\infty \oplus \Q_{\ge 0}^{\oplus S})\] gives a Kummer pro\'etale cover of $\ul R_S$ on the associated log adic generic fibers, with Galois group 
\[ \Gamma_S  \cong \Hom ((\N^S)_{\Q}^{\text{gp}}/(\N^S)^{\text{gp}}, \mu_\infty(\mO)) \cong \Hom ((\Q/\Z)^S, \mu_\infty(\mO)) \cong \widehat \Z(1)^{ S}
\] (see Subsection \ref{ss:small_coordinates}). The continuous action of $\Gamma_S$ on $R_{S, \infty}$ induces a continous action on $\Ainf (R_{S, \infty}) \cong \Ainf \gr{\Q_{\ge 0}^{\oplus S}}$. Note that we have an injection 
\begin{equation} \label{eq:D_S_inside_Ainf}
D_S = \Ainf \gr{\N^S} \lra \Ainf (R_{S, \infty})
\end{equation}
induced by the monoid map $\N^{S} \subset \Q_{\ge 0}^S$, which is stable under the action of $\Gamma_S$. For each $i \in \{1, ...,  s = |S| \}$, let $\gamma_i \in \Gamma_S$ be the element which sends the copy of $\Q/\Z$ indexed by $s_i \in S$ to $\mu_\infty(\mO)$ by $\gamma_i (\frac{m}{n}) = \zeta_{n}^m$, and sends all other copies of $\Q/\Z$ to $1$. Then $\gamma_1, ..., \gamma_s $ is a set of topological generators of $\Gamma_S$, and acts on $\Ainf (R_{S, \infty})$ as follows: for each $i \in \{1, ..., s\}$ and $\alpha = \frac{m}{n} \in \Q_{\ge 0}$, write $e_i^{\alpha}$ for the element in 
\[\Ainf (R_{S, \infty}) \cong \Ainf \gr{\Q_{\ge 0}^{\oplus S}}\] which corresponds to the element $\alpha_i = (0, ..., \alpha, ... , 0) \in \Q_{\ge 0}^{\oplus S}$ in the monoid where $\alpha$ sits in $i^{th}$ position, then 
\[ 
\begin{cases} 
\gamma_i (e_i^\alpha)  = [\epsilon(\alpha)] \cdot e_i^{\alpha} \\ 
\gamma_i (e_j^{\alpha}) = e_j^{\alpha} & \text{ for all } j \ne i 
 \end{cases}
\] 
where $\epsilon (\alpha) \in \mO^\flat$ denotes the element 
\[ 
\epsilon (\alpha) := (\zeta_{n}^m, \zeta_{np}^m, \zeta_{np^2}^m, \cdots) \in \mO^\flat. 
\]
In particular, if we write 
\[D_S = \Ainf \gr{X_j}_{j = 1, ..., s}\] with $X_j = e_j^{1}$, then the action of $\Gamma_S$ on $D_S$ (via the injection in (\ref{eq:D_S_inside_Ainf})) is given by \[ \gamma_i (X_j) = X_j \tu{ if } j \ne i, \tu{  and }\gamma_i (X_i) = [\epsilon] \cdot X_i.\]  By Proposition \ref{prop:presheaf_is_sheaf}, Lemma \ref{lemma:killing_junk_torsion} and Corollary \ref{cor:qi_group_proket_Ainf}, we have quasi-isomorphisms 
\begin{equation} \label{eq:beta_S}
\beta_S: L \eta_{\mu} R \Gamma_{\text{ct}} (\Gamma_S, D_S) \isom L \eta_{\mu} R \Gamma_{\text{ct}} (\Gamma_S, \Ainf (R_{S, \infty})) \isom A \Omega^{\log}_{\ul R_S}.
\end{equation}
The first arrow in (\ref{eq:beta_S}) is induced by the $\Gamma_S$-equivariant map (\ref{eq:D_S_inside_Ainf}) and the second is described in Section \ref{section:Hodge_Tate_and_dR}. The map $\beta_S$ is functorial with respect to $S$ (since both maps are). The left side $L \eta_\mu R \Gamma_{\text{ct}} (\Gamma_S, D_S)$ is computed by the Koszul complex 
\[
\text{Kos}(D_S;  \frac{\gamma_1 -1}{[\epsilon] - 1}, \cdots, \frac{\gamma_s - 1}{[\epsilon] - 1}).
\] 
But this is precisely the $q$-de Rham complex of $D_S$ with respect to $(\Ainf, N_\infty)$ (see  \cite[Construction 7.15]{Koshikawa}). In other words, we have identifications 
\[
h_S: q \Omega^\bullet_{D_S/(\Ainf, N_\infty)} \cong \text{Kos}(D_S;  \frac{\gamma_1 -1}{[\epsilon] - 1}, \cdots, \frac{\gamma_s - 1}{[\epsilon] - 1}) \cong L \eta_\mu R \Gamma_{\text{ct}} (\Gamma_S, D_S)
\]
of $E_\infty$-$\Ainf$-algebras, which is clearly functorial in $S$. This proves the proposition (the $\varphi$-compatibility is clear). 
\eproof

\bc \label{cor:derived_Ainf_comparison}
Retain notations from before. Let $\ul R$ be a derived $p$-complete (simplicial) pre-log ring over $\ul \mO$, then there is a functorial $\varphi$-equivariant isomorphism 
\[ 
 \varphi^* \Prism^\L_{\ul R/(\Ainf, N_\infty)}  \cong  A \Omega^{\L, \log}_{\ul R}, 
\]
which is compatible with the derived Hodge--Tate comparison on both sides. 
\ec

\bproof 
This is immediate from Proposition \ref{prop:Ainf_comp_free} under left Kan extensions.  
\eproof 

\br
From the proof of Proposition \ref{prop:Ainf_comp_free} above, we see that the log prismatic cohomology of the log affine space $\A^{n, \log}_{\mO}$ is computed by the log $q$-de Rham complex $q\Omega^\bullet_{D_S}$, which can be identified with the Koszul complex 
\[
\text{Kos}(D_S;  \frac{\gamma_1 -1}{[\epsilon] - 1}, \cdots, \frac{\gamma_s - 1}{[\epsilon] - 1}).
\] 
where $D_S = \Ainf \gr{X_j}_{j = 1, ..., s}$. In other words, the log $q$-differential is given by 
\[
\nabla_{q, i}^{\log} = \frac{\gamma_i - 1}{[\epsilon] - 1}. 
\]
For comparison, we note that the (non-log) prismatic cohomology of $\A^n_{\mO}$ is computed by the non-log q-de Rham complex, which identifies with the Koszul complex 
\[
\text{Kos}(D_S;  \frac{\gamma_1 -1}{[\epsilon]X_1 - X_1}, \cdots, \frac{\gamma_s - 1}{[\epsilon] X_s - X_s}).
\] In other words, the non-log $q$-differential is given by 
\[
\nabla_{q, i}  = \frac{\gamma_i - 1}{([\epsilon] - 1) X_i}. 
\]
\er 

\br \label{remark:Nygaard_on_Ainf}  Let $\ul R$ be as above. 
Using the Nygaard filtration on derived log prismatic cohomology (constructed in \cite[Theorem 5.1]{logprism}), Corollary \ref{cor:derived_Ainf_comparison} allows us to equip the derived log $\Ainf$-cohomology $A \Omega^{\L, \log}_{\ul R}$ with a derived Nygaard filtration $\tu{Fil}_N^\bullet A \Omega^{\L, \log}_{\ul R}$ by derived $(p, \mu)$-completed objects, equipped with an isomorphism 
\[
\tu{gr}^i_{N} A \Omega^{\L, \log}_{\ul R} \isom \tu{Fil}^{\tu{conj}}_\bullet \sq \Omega_{\ul R}^{\L, \log} \{i\}
\]
where the right hand side denotes the conjugate filtration from Corollary \ref{cor:derived_HT}.  
\er

\bproof[Proof of Theorem \ref{theorem:logprism_comparison}]
It suffices to work \'etale locally on $\fX$ and construct a functorial isomorphism 
\[
\varphi^* \Prism_{\ul{R}/(\Ainf, N_\infty)} \isom A \Omega_{\ul{R}}^{\log}
\]
for all pre-log rings $\ul{R}$ associated with a small chart as described in \S \ref{ss:construction_HT}. By Corollary 
\ref{cor:derived_Ainf_comparison} it suffices to show that under our assumption, we have natural isomorphisms 
\[
\Prism^\L_{\ul{R}/(\Ainf, N_\infty)} \isom \Prism_{\ul{R}/(\Ainf, N_\infty)}, \qquad A \Omega^{\L, \log}_{\ul{R}} \isom A \Omega^{\log}_{\ul{R}}.
\]
The first isomorphism follows from \cite[Proposition 4.5]{logprism} and the second follows from Proposition \ref{prop:comparing_derived_with_non_derived_on_charts}. 
\eproof

\newpage 
\section{Comparison with prismatic cohomology of the infinite root stack} 
  \label{section:infinite_root_stack}
In this  section, we compare our construction of (derived) log $\Ainf$-cohomology with a variant constructed by considering the infinite root stack of an fs log formal scheme. We also relate the logarithmic $p$-adic Tate twists to the derived Hodge filtration on derived log de Rham cohomology, which can be viewed as a logarithmic extension of the Beilinson fiber square obtained in \cite{Beilinson_fiber}. 

\subsection{The infinite root stack} 

Let us first recall the construction of (a variant of) the infinite root stack by Talpo--Vistoli \cite{Talpo_Vistoli}. Let $(\fX, \mM)$ be a log (formal) scheme, then the structural monoid map $\alpha: \mM \ra \mO_{\fX}$ induces a map (of symmetric monoidal functors):
\[\cl \alpha: \cl{\mM}= \mM/\mO_{\fX}^\times \lra \tu{Div}_{\fX} := [\mO_{\fX}/\mO_{\fX}^\times] \]
where the target can be identified with the fibred category over $\fX_{\ett}$ consisting of pairs $(\mL, s)$ where $\mL$ is a line bundle on $\fX$ and $s$ is a global section of $\mL$. 

\begin{construction}[Talpo--Vistoli] 
For a log (formal) scheme $(\fX, \mM)$, there is a stack $\sqrt[\infty]{(\fX, \mM)}$ fibred over $\tu{Aff}_{\fX}$ in groupoids, called the infinite root stack of $(\fX, \mM)$, such that for an affine (formal) scheme $f: T \ra \fX$,  a lift $T \ra \sqrt[\infty]{(\fX, \mM)}$ of $f$ consists of a pair $(\psi, \iota)$ where 
\[
\psi: (f^* \cl{\mM})_{\infty} \lra \tu{Div}_{T} 
\]
is a symmetric monoidal functor and $\iota$ is an isomorphism between $f^* \cl \alpha$ and the composition 
\[
f^* \cl{\mM} \lra (f^* \cl{\mM})_{\infty} \xrightarrow{\: \psi \: }  \tu{Div}_{T}. 
\]
Here we write $\mP_{\infty} :=\varinjlim_n (\frac{1}{p^n}\mP)$ for a sheaf of monoid $\mP$. 
\end{construction}

Let us describe the simplest example of the infinite root stack (which is essentially the only infinite root stack that we need in this article). 
\beg[The infinite root stack of the logarithmic affine line] \label{example:root_stack} Let $\A^{1,\log}$ denote the log $p$-dic formal scheme associated to the pre-log ring $(\Z_p \gr{t}, \N \xrightarrow{1 \mapsto t} \Z_p \gr{t})$. By \cite[Proposition 3.10]{Talpo_Vistoli}, we may identify
\[
\sqrt[\infty]{\A^{1, \log}} \cong \varprojlim_n \: [\A_{\Z_p}^1/\mu_n] 
\]
where each $[\A^1_{\Z_p}/\mu_n]$ denotes the quotient stack \footnote{which can also be thought of as the formal completion of the algebraic stack $[ \A^1/\mu_n]$.}
\[ \pi_n: [\A_{\Z_p}^1/\mu_n] = [\spf \Z_p \gr{t^{1/p^n}}/\mu_{p^n}] \lra \A^1_{\Z_p} :=  \spf \Z_p \gr{t},
\] 
where $\pi_n$ is naturally induced from the inclusion $\Z_p[t] \ra \Z_p [t^{1/p^n}]$. In fact, the map $\pi_n$ can be upgraded to a map of log (formal) stacks, 
where the log structure on $[\A^1_{\Z_p}/\mu_{p^n}]$ comes from the pre-log structure $\frac{1}{p^n} \N$. In the limit we have a natural map 
\begin{equation} \label{eq:pi_inf_A1}
\pi_\infty: 
\sqrt[\infty]{\A^{1,\log}} \lra \A^1_{\Z_p} = \spf \Z_p \gr{t}
\end{equation}
of $p$-adic formal stacks, which can be upgraded to a map of log formal stacks, where the target is the log formal affine line $\A^{1, \log}_{\Z_p}$, and the infinite root stack $\sqrt[\infty]{\A^{1,\log}}$ is equipped with the log structure associated to the pre-log structure $\N_\infty = \N[\frac{1}{p}]$. 

This example easily generalizes to $\A^{n,\log} = \spf \Z_p \gr{t_i}_{1 \le i \le n}$ with log structure coming from $\N^n \ra \Z_p \gr{t_i}_{1 \le i \le n}$ sending each $1 \in \N^n$ at position $i$ to $t^i$. 
\eeg

The following observation is due to Bhatt--Clausen--Mathew. We learned this statement from Mathew. For the lack of references we also provide a proof. 

\bl \label{lemma:cohomology_of_infinite_root_stack}
The map (\ref{eq:pi_inf_A1}) induces an isomorphism 
\[
\widehat \Omega^{1}_{(\Z_p \gr{t}, \N)/\Z_p} \isom R \Gamma (
\sqrt[\infty]{\A^{1,\log}}, \widehat \L_{
\sqrt[\infty]{\A^{1,\log}}/\Z_p}). 
\]
\el 

\bproof 
First we note that, in order to obtain the natural map above, we view (\ref{eq:pi_inf_A1}) as a map of log stacks and identify its log cotangent complex over $\Z_p$ with the (nonlog) cotangent complex $ \widehat \L_{
\sqrt[\infty]{\A^{1,\log}}/\Z_p}$ by Corollary \ref{cor:cotangent_for_perfectoid} (see also Lemma \ref{lemma:cotangent_for_perfect_maps}). Since both sides are derived $p$-complete it suffices to check this after mod $p$, in other words, it suffices to show that after reduction mod $p$, (\ref{eq:pi_inf_A1}) induces an isomorphism 
\[
  \Omega^{1}_{(\F_p [t], \N)/\F_p} = \F_p [t] \cdot\frac{dt}{t} \isom R \Gamma (
\sqrt[\infty]{\A_{\F_p}^{1,\log}},  \L_{
\sqrt[\infty]{\A_{\F_p}^{1,\log}}/\F_p}). 
\]
Now let us compute the target. 
First we note that 
\begin{equation} \label{eq:cotangent_complex_for_B_mu_p_infty}
    \L_{B \mu_{p^\infty}/ \F_p} \cong \F_p,
\end{equation}
which is concentrated in degree $0$. To see this, recall that \[\L_{\mu_{p^n}/\F_p} \cong [(x-1) \F_p [x^{1/p^n}]/(x-1)^2 \xrightarrow{ \:\: 0 \: \: }  \F_p [x^{1/p^n}]/(x-1)  \textup{ d} x^{1/p^n}] \]
is a complex in (cohomological) degree $-1$ and $0$; and for every $n \in \Z_{\ge 1}$, we have 
\begin{equation} \label{eq:cotangent_complex_for_B_mu_p}
\L_{ B\mu_{p^n}/F_p} \cong [\F_p \xrightarrow{\:\: 0 \: \: } \F_p]
\end{equation}
which lives in degree $0$ and $1$. Now observe that, when we take colimit as $n \mapsto \infty$, the maps in degree $1$ of (\ref{eq:cotangent_complex_for_B_mu_p}) is given by $\tu{d} x^{1/p^{n-1}} = \tu{d} (x^{1/p^n})^p = 0 $, therefore the colimit vanishes in degree $1$, thus the desired isomorphism (\ref{eq:cotangent_complex_for_B_mu_p_infty}). Therefore, writing  $\A^1_{\infty} = \spec \F_p [t^{1/p^\infty}]$, we know that $\L_{ [\A^1_{\infty}/\mu_{p^\infty}]/\F_p}$ is concentrated in degree $0$ and is given by  
\[\L_{ [\A^1_{\infty}/\mu_{p^\infty}]/\F_p} \cong \F_p [t^{1/p^\infty}],
\]
with the coaction of $\F_p [x^{1/p^\infty}]/(x-1)$ given by $t^{1/p^n} \mapsto t^{1/p^n} \cdot x^{1/p^n}$. To compute its derived global section (in the fpqc topology, for example), we consider the flat cover 
\[\A^1_{\infty} \lra \sqrt[\infty]{\A_{\F_p}^{1,\log}} = [\A^1_{\infty}/\mu_{p^\infty}]\]
and its Cech nerve $Y_\infty^{\bullet}$. Each term $Y_\infty^{i} \cong \A_\infty^1 \times \mu_{p^\infty}^{i}$ is affine. Let $\iota^i: Y_\infty^i\ra [\A^1_{\infty}/\mu_{p^\infty}]$ denote the canonical projection, then we have 
\[
 R \Gamma ([\A^1_{\infty}/\mu_{p^\infty}], \L_{[\A^1_{\infty}/\mu_{p^\infty}]/\F_p}) \cong \tu{Tot }\Gamma (Y_\infty^\bullet, (\iota^\bullet)^* \L_{[\A^1_{\infty}/\mu_{p^\infty}]/\F_p}) \cong \F_p [t].
\]
Finally, we need to identify this copy of $\F_p [t]$ with $\Omega^1_{(\F_p [t], \N)/\F_p}$ by identifying its generator with $dt /t$. This can be done via a diagram chase on the finite level. In fact, one may argue after the pullback along $\G_m \hookrightarrow \A^{1, \log}$, and consider the resulting map (by a similar consideration as above)  
\begin{equation} \label{eq:identifying_generator}
\F_p [t^{\pm1}] dt   \lra  \F_p [t^{\pm 1}].
\end{equation}
On the finite level, we have the following commutative diagram 
\[
\begin{tikzcd}
   Y =  \spec \F_p [u^{\pm 1}] \arrow[d, "f"] \arrow[dd, dashed, swap, near start,  "\sq f", bend right= 57]  & Y \times \mu_p = \spec \F_p [u^{\pm 1}, v]/(v^p - 1) \arrow[l, "p_1"] \arrow[d, "p_2"] \\
    X = [ \spec \F_p [u^{\pm 1}]/\mu_p] \arrow[d, "\pi"] & Y =  \spec \F_p [u^{\pm 1}] \arrow[l] \\
    \G_m = \spec \F_p [t^{\pm 1}] 
\end{tikzcd}
\]
where $p_1$ sends $u \mapsto uv$, $p_2$ sends $u \mapsto u$, and $\pi$ is an isomorphism. On the level of rings, the vertical arrows further fit into the following commutative diagram 
\[
\begin{tikzcd}[column sep = 2.5em]
\F_p [u^{\pm 1}, t^{\pm 1}]/(\frac{u^p}{t} - 1) \arrow[r, "{\:\: u \mapsto uv \: \:}"] & \F_p [u^{\pm 1}, v]/(v^p - 1)  \\
\F_p [u^{\pm 1}, t^{\pm 1}] \arrow[r, "\:\: {u \mapsto uv, \: t \mapsto u^p}\: \: "] \arrow[u, "g"]   & \F_p [u^{\pm 1}, v]  \arrow[u, "h"]  \\ 
\F_p[t^{\pm 1}]  \arrow[r, " t \mapsto u^p"]  \arrow[u] \arrow[uu, near start, bend left = 60, "\sq f"] & \F_p [u^{\pm 1}] \arrow[u] \arrow[uu, near start, swap, "p_2", bend right = 60]. 
\end{tikzcd}
\]
By functoriality we have a commutative diagram 
\[
\begin{tikzcd}
H^0(\L_{g} [-1] ) \cong  (\frac{u^p}{t}-1) k[t^{\pm}, u^{\pm}]/(\frac{u^p}{t}-1)^2  \arrow[d, shift right = 6em] \arrow[d, shift left = 3em]   \arrow[r] & \sq f^* \L_{\F_p[t^{\pm 1}]/\F_p} \arrow[d] \\ 
H^0(\L_{h} [-1] )  \cong (v^p - 1) k[u^{\pm 1}, v] /(v^p-1)^2 \arrow[r] &  p_1^* f^* \L_{X/\F_p}. 
\end{tikzcd}
\]
Tracing this diagram, we know that starting from $t$ in the top left corner, we get $dt = t \frac{dt}{t}$ in the top right corner, and $u^p $ in the bottom left corner. Now identifying $u = t^{1/p}$, we know that the map (\ref{eq:identifying_generator}) is given by $dt = t \cdot dt /t \mapsto t$ and thus $dt/t \mapsto 1$. This finishes the proof of the lemma.
\eproof

\subsection{Log prismatic cohomology via the infinite root stack}

Let us put ourselves in the following setup, as in  Construction \ref{construction:derived_Ainf_on_prelog}. We 
consider pre-log rings of the form of
\[
\Sigma(S, \ul T) :=  (\mO_C \gr{(X_i)_{ i \in S}, \N^T }, \N^T)
\]
for finite sets $S,T$ with cardinality $s$ and $t$, as in Construction \ref{construction:derived_Ainf_on_prelog} (except that here we take $N_\infty = 0$). 

We write $\fX_{S,T}$ for the associated log affine formal scheme $\spf \Sigma (S, \ul T)$, and write $\sqrt[\infty]{\fX_{S, T}}$ for the corresponding infinite root stack. From Example \ref{example:root_stack}, we have 
\begin{equation} \label{eq:root_stack_lim}
\sqrt[\infty]{\fX_{S, T}} \cong \varprojlim_n \:  \sqrt[n]{\fX_{S, T}} =  \varprojlim_n \: \big(\A^{s} \times [\A^t/\mu_n^t] \big) 
\end{equation}
where we identify $\A^t$ with $\spf \mO_C \gr{T_j^{1/n}}_{j \in T}$. Moreover, as in (\ref{eq:pi_inf_A1}), we have a canonical map 
\[
\pi_\infty: \sqrt[\infty]{\fX_{S, T}}  \lra \spf (\Sigma(S, \ul T))
\]
of stacks. This can be upgraded to a map of log stacks 
\begin{equation} \label{eq:pi_inf}
\pi_\infty: (\sqrt[\infty]{\fX_{S, T}}, \mM_{S, T})  \lra \fX_{S, T},
\end{equation}
where $\sqrt[\infty]{\fX_{S, T}}$ is viewed as a log stack with log structure $\mM_{S,T}$ coming from the monoid $\Q_{\ge 0}^{T}$ (see Example \ref{example:root_stack}).

\begin{construction} 
\label{construction:via_root_stack}
We consider the derived prismatic cohomology $R \Gamma_{\Prism} (\sqrt[\infty]{\fX_{S, T}})$ by the left Kan extension in the category of derived $(p, \mu)$-completed objects in $\mD(\Ainf)$  along the Yoneda embedding $\tu{Aff}_{\mO_C} \hookrightarrow \tu{PStk}_{\mO_C} $ into prestacks over $\mO_C$ (see \cite[Definition 2.2.1]{Kubrak}). From (\ref{eq:root_stack_lim}) we can write $R \Gamma_{\Prism} (\sqrt[\infty]{\fX_{S, T}})$ as the filtered colimit 
\[
R \Gamma_{\Prism} (\sqrt[\infty]{\fX_{S, T}}) \cong \varinjlim_n R \Gamma_{\Prism} (\sqrt[n]{\fX_{S, T}}) 
\]
where each term $\sqrt[n]{\fX_{S, T}} =  \A^{s} \times [\A^t/\mu_n^t]$ is a (formally) smooth Artin stack over $\mO_C$. Equivalently, we may consider consider the Cech nerve $\mC^\bullet$ of the quasisyntomic cover 
\[\A^s \times \spf \mO_C \gr{T_j^{1/\infty}}_{j \in T} \ra \sqrt[\infty]{\fX_{S, T}} \] 
and observe that the prismatic cohomology of $\sqrt[\infty]{\fX_{S, T}}$ can be computed by  
\[
R \Gamma_{\Prism} (\sqrt[\infty]{\fX_{S, T}}) \cong \lim \:  R \Gamma_{\Prism} (\mC^\bullet)
\]
as in Subsection \ref{construction:derived_Ainf_on_prelog} by quasisyntomic descent. This construction is  functorial with respect to $S$ and $T$, so we can further left Kan extend to all derived $p$-complete simplicial pre-log rings over $\mO_C$, as in Construction \ref{construction:derived_Ainf_on_prelog}, thus resulting in a functor 
\[
\ul R \longmapsto \Prism^{\L, \infty}_{\ul R/\Ainf}.
\]
taking values in  derived $(p, \mu)$-complete commutative algebra objects in $\mathcal{D}(\Ainf)$ that are equipped with a $\varphi_{\Ainf}$-semilinear Frobenius map $\varphi$.  
\end{construction}

Our main result in this section is that this construction agrees with the derived log prismatic cohomology constructed in \cite{logprism} (see Section \ref{section:compare_AOmega}). 

\bt \label{thm:compare_with_infinite_root}
Let $\ul{\mO_C} = (\mO_C, N_\infty)$ be a split pre-log ring (see Section \ref{section:HT_primitive}) and let $\ul R$ be a $p$-complete pre-log $\ul{\mO_C}$-algebra. Then there is a natural $\varphi$-equivariant isomorphism 
\[
\Prism^{\L}_{\ul{R}/\Ainf} \isom  \Prism^{\L, \infty}_{\ul R/\Ainf}
\]
compatible with the (derived) Hodge--Tate and de Rham comparisons, and preserving the Nygaard filtrations on both sides. In other words, the construction of derived prismatic cohomology via the infinite root stack  agrees with the derived log prismatic cohomology of $\ul R$. 
\et 

\bproof 
By (the prismatic version of) Corollary \ref{cor:change_of_base}, we may assume that $N_\infty = 0$. Moreover, it suffices to check the claim on $\Sigma (S, \ul T)$. For this, we first observe that using the map (\ref{eq:pi_inf}), the construction of derived log prismatic cohomology (as a left Kan extension) then supplies a natural map 
\[
\Prism^{\L}_{\Sigma(S, \ul T)/\Ainf} \lra R \Gamma_{\Prism}^{\log} (\sqrt[\infty]{\fX_{S,T}}, \mM_{S, T})/\Ainf).
\]
Note that the right hand side can be identified with 
\[
 \lim \:  R \Gamma_{\Prism} (\mC^{\log, \bullet})
\]
where $\mC^{\log, \bullet}$ denotes the ($p$-completed) Cech nerve of the log quasisyntomic cover 
\[\A^s \times (\spf \mO_C \gr{T_j^{1/\infty}}_{j \in T}, \Q_{\ge 0}^T)^a \lra (\sqrt[\infty]{\fX_{S, T}}, \mM_{S, T}).\] 
Note that each term in $\mC^{\log, \bullet}$ is the formal spectrum of a log quasiregular semiperfectoid $\mO_C$-algebra. Also note that, forgetting log structures induces a map 
\[
\Prism^{\L, \infty}_{\Sigma(S, \ul T)/\Ainf} \lra R \Gamma_{\Prism}^{\log} (\sqrt[\infty]{\fX_{S,T}}, \mM_{S, T})/\Ainf)
\]
which is an equivalence by the derived Hodge--Tate comparison (Corollary \ref{cor:derived_dR} and \cite[Proposition 2.4.3, Corollary 2.4.4 and its proof]{Kubrak}) and Corollary \ref{cor:cotangent_for_perfectoid}. Thus, inverting this equivalence, we obtain a map 
\[
\Prism^{\L}_{\Sigma(S, \ul T)/\Ainf} \lra\Prism^{\L, \infty}_{\Sigma(S, \ul T)/\Ainf} 
\]
which is $\varphi$-equivariant and is functorial in $S$ and $T$. To show that this is an equivalence, by the Kunneth formula (Corollary \ref{cor:derived_multiplicative} 
and \cite[Proposition 2.2.16]{Kubrak}), it suffices to show that, when $S = \O$ and $T = \{*\}$ is a singleton, the corresponding map 
\begin{equation}\label{eq:inf_root_A1}
\Prism^{\L}_{(\mO_C \gr{T}, T^{\N})/\Ainf} \lra R \Gamma_{\Prism} (\sqrt[\infty]{\A^{1, \log}_{\mO_C}}/\Ainf)
\end{equation}
is an equivalence. Applying the Hodge--Tate comparison once again (using \cite[Proposition 2.4.3, Corollary 2.4.4 and its proof]{Kubrak} for the right hand side), it suffices to show that (\ref{eq:inf_root_A1}) induces an isomorphism
\[
\widehat \L^1_{(\mO_C \gr{T}, T^{\N})/\mO_C} \isom R \Gamma (\sqrt[\infty]{\A_{\mO_C}^{1, \log}}, \widehat \L^1_{\sqrt[\infty]{\A_{\mO_C}^{1, \log}}/\mO_C}).
\]
But this follows from Lemma \ref{lemma:cohomology_of_infinite_root_stack}. 
\eproof

\bc 
Let $\ul R$ be a $p$-complete (simplicial) pre-log $\ul{\mO_C}$-algebra, then we have functorial $\varphi$-equivariant isomorphisms 
 
\[
  A \Omega^{\L, \log}_{(R, P)} \isom  \varphi_{\Ainf}^* \Prism^\L_{(R, P)/ \Ainf} \isom \varphi_{\Ainf}^* \Prism^{\L, \infty}_{(R, P)/\Ainf}
\] 
compatible with the (derived) Hodge--Tate and de Rham comparisons. 
\ec 

 \br 
It is not hard to see Construction \ref{construction:via_root_stack} works equally well over any base log prism. In particular, by inspecting the proof, Theorem \ref{thm:compare_with_infinite_root} holds over an arbitrary base pre-log prism (instead of just $(\Ainf, (\xi))$). The same remark applies to the dicussion in Subsection \ref{ss:log_Beilinson} below. 
\er

\subsection{A logarithmic Beilinson fiber sequence} \label{ss:log_Beilinson}

We deduce a logarithmic version of the Beilinson fiber square (on graded terms) obtained in \cite{Beilinson_fiber}. Let  $\ul S$ be a (simplicial) pre-log $\mO_C$-algebra. 
Recall that 
\[
\Z_p (n) (\ul S ):= \tu{Fib} \big(\tu{Fil}_{N}^n \Prism^{\L}_{\ul S} \{n\} \xrightarrow{\varphi - \tu{can}} \Prism^{\L}_{\ul S} \{n\}\big). 
\]
and that $\Q_p (n) = \Z_p (n)[1/p].$

\bt Let $\ul S = (S, M)$ be a pre-log ring that is log quasisyntomic over $\Z_p$.  
    \be  
      \item There exists a natural pullback square 
     \[ \begin{tikzcd}[row sep = 2em]
          \Q_p (n) (\ul S) \arrow[d] \arrow[r] & \Q_p (n) (\ul S/p) \arrow[d] \\ 
           \tu{Fil}_{H}^n \widehat \L \Omega_{\ul S/\Z_p} \{n\}_{\Q_p} \arrow[r]  &  \widehat \L \Omega_{\ul S/\Z_p} \{n\}_{\Q_p}.
       \end{tikzcd} \]
       in the derived $\infty$-category $\mD(\Q_p)$, where $  \tu{Fil}_{H}^n$ denotes the ($p$-completed) Hodge filtration on the derived log de Rham cohomology, and $()_{\Q_p}$ denotes the base change $\otimes_{\Z_p} \Q_p.$
      \item There is a functorial isomorphism
      \[ 
      \Q_p (n) (\ul S) \cong \textup{Fib} \Big(
           \tu{Fil}_{H}^n \widehat \L \Omega_{\ul S/\Z_p} \{n\}  \xrightarrow{\varphi - p^n} 
           \widehat \L \Omega_{\ul S/\Z_p} \{n\} \Big)_{\Q_p}.
      \]
    \ee 
\et 

\bproof 
It suffices to check the assertions for $\spf (\Sigma(S, \ul T))^a$ and we can further reduce to the case of the log affine line (where $S = \O$ and $T = \{*\}$ is a singleton). By Theorem \ref{section:infinite_root_stack}, we may 
replace $\Prism_{\ul S}^{\L}$ by $\Prism_{\ul S}^{\L, \infty}$.  Consider the flat cover 
\[\A_\infty
^1 = \spf \mO_C \gr{t^{1/p^\infty}} \ra \sqrt[\infty]{\A^{1,\log}}\] as in the proof of Lemma \ref{lemma:cohomology_of_infinite_root_stack} and takes the $p$-completed Cech nerve $\mC^\bullet$, then each $\mC^i = \spf S^i$ where $S^i$ is a quasisyntomic $\mO_C$-algebra. From the Hodge--Tate comparison and the fact that 
\[
\tu{gr}_N^i \Prism_{S^i}^\L \isom \tu{Fil}_\bullet^{\tu{conj}} ({\Prism}_{S^i}^\L \otimes^\L_{\Ainf, \theta} \mO_C) \{i\}
\]
(for example, see Remark \ref{remark:Nygaard_on_Ainf}), we know that $\Q_p (n) (\ul S) \cong \tu{Tot } \Q_p (n) (S^\bullet)$, and similarly for $\Q_p (n) (\ul S/p)$, $ \tu{Fil}_{H}^n \widehat \L \Omega_{\ul S/\Z_p}$, and $\widehat \L \Omega_{\ul S/\Z_p}$. The assertions thus follow from the corresponding assertions for $S^i$ and functoriality, which are \cite[Theorem 6.17, Theorem 6.22]{Beilinson_fiber}. 
\eproof

\newpage 
\section{Derived log $\Ainf$-cohomology on qrsp objects}

In this section, we describe the derived log $\Ainf$-cohomology for quasiregular semiperfectoid pre-log $\mO_C$-algebras and compares this construction with certain universal log PD algebras, first considered in \cite{BMS2}.  

\subsection{A universal log prism} \label{ss:universal_log_prism} 

Let us record the following notion of a   ``universal  log prism'' associated to a quasiregular semiperfectoid pre-log ring from \cite{logprism}, which will be used later.

\begin{proposition}\label{initial object}
Let $\ul S = (S, N)$ be a  quasiregular semiperfectoid integral pre-log ring. 
Then there exists a ``universal bounded pre-log prism" $(\Prism^{\init}_{\ul S}, (\xi), M^{\init}_{\ul S})$ equipped with a map 
\[ S \ra \Prism^{\init}_{\ul S}/\xi
\] and an exact surjection
\[
(\Prism^{\init}_{\ul S}, M^{\init}_{\ul S} ) \to (\Prism^{\init}_{\ul S}/\xi, N\to \Prism^{\init}_{\ul S}/\xi)^a, 
\] 
where the underlying pre-log ring $(\Prism^{\init}_{\ul S}, M^{\init}_{\ul S})$ of the pre-log prism  $(\Prism^{\init}_{\ul S}, (\xi), M^{\init}_{\ul S})$ is a log ring, which satisfies the following ``universal property": for any bounded integral pre-log prism $(A, I, M_A)$ such that $(A, M_A)$ is a log ring, equipped with a map $S\to A/I$ and an exact surjection
\[
(A, M_A) \to (A/I, N\to S \to A/I)^a,
\]
there is a unique map of pre-log prisms
\[
(\Prism^{\init}_{\ul S}, (\xi), M^{\init}_{\ul S}) \to  (A, I, M_A)
\] 
compatible with the exact surjections above.  
\end{proposition}

\bproof 
This follows \cite[Proposition 4.10 and Proposition 4.12]{logprism}. 
\eproof

\br \label{remark:initial_object_independent_of_R}
Our notation is slightly different from \cite{logprism}. There one fixes a ``surjective" map $\ul R \ra \ul S$ from a  perfectoid pre-log ring $\ul R$ and our $(\Prism^{\init}_{\ul S}, (\xi), M^{\init}_{\ul S})$ is denoted by $(\Prism^{\init}_{\ul S/\ul R}, (\xi), M^{\init}_{\ul S/\ul R})^a$ there, which \emph{a priori} depends on the map $\ul R \ra \ul S$. In Proposition 4.12 of \emph{loc.cit.,} it is shown that (after taking associated log rings) this construction is independent of the choice of $\ul R \ra \ul S.$
\er

\begin{proposition}
Assume that $\ul S = (S, N)$ is a quasiregular semiperfectoid integral pre-log ring over $\mO_C$. Then $\Prism^{\L}_{\ul S/ \Ainf}$ is discrete and naturally has the structure of a $\delta$-ring. Moreover, there is an isomorphism of $\delta$-rings
\[
\Prism^\L_{\ul S/\Ainf}  \cong \Prism^{\init}_{\ul S}. 
\]
\end{proposition}
 
\bproof 
This is (a special case of) \cite[Proposition 4.12]{logprism}. 
\eproof 

\begin{remark} 
The pre-log prism  $(\Prism^{\init}_{\ul S}, (\xi), M^{\init}_{\ul S})$ is the initial object in the category   $\mC_{/(S, N)}^{\tu{str}}$ whose objects are diagrams of the form 
\[
 \begin{tikzcd}
& (A, I , M_A)  \arrow[d, "i"] \\
(S, N) \arrow[r] & (A/I, N)^a,
\end{tikzcd}
\]
where $(A, I, M_A)$ is a bounded  integral pre-log prism such that $(A, M_A)$ is a log ring and $i$ is an exact surjection of log rings. This is again slightly different from \cite{logprism}. There the construction is made for a larger class of pre-log rings (where $\ul S$ is only assumed to be semiperfectoid) and the resulting pre-log prism might not be bounded.  
\end{remark}
 
In the end of this section, let us briefly recall the construction of $(\Prism_{\ul S}^{\init},(\xi), M_{\ul S}^{\init})$ for the convenience of the reader. 

\begin{construction} \label{construction:initial_logprism}
Pick any map $\ul R \to \ul S$ where  $\ul R = (R, M)$ is a perfectoid integral pre-log ring, such that the map is surjective on rings and surjective modulo invertible elements on monoids.  Take the exactification  \[M^{\flat} \to \widetilde{M} \to N\] of the composition $M^{\flat}\to M \to N$ and set
\bi
\item $\widetilde{R}\coloneqq R\widehat{\otimes}_{\Z_p} \Z_p \langle \widetilde{M} \rangle,$
\item $\Ainf(R, M^{\flat})\coloneqq \Ainf(R)\widehat{\otimes}_{\Z_p} \Z_p \langle M^\flat \rangle,$ 
\item $\Ainf(R, \widetilde{M}) \coloneqq \Ainf(R)\widehat{\otimes}_{\Z_p} \Z_p \langle \widetilde{M} \rangle. $
\ei 
The ring $\Ainf(R, M^{\flat})$ is perfectoid and $(\Ainf(R, M^{\flat}), (\xi), M^{\flat})$ is a perfect pre-log prism.  We also have a bounded pre-log prism $(\Ainf(R, \widetilde{M}), (\xi), \widetilde{M})$ with a map \[ (\Ainf(R, M^{\flat}), (\xi), M^{\flat}) \lra (\Ainf(R, \widetilde{M}), (\xi), \widetilde{M}).  \] The map of pre-log rings \[ (\Ainf(R, \widetilde{M}), \widetilde{M}) \to (S, N) \]  induces an exact surjection on the associated log rings.  Applying the existence of (nonlog) prismatic envelopes to the surjection $\Ainf(R, \widetilde{M})\to S$, we obtain a pre-log prism  
\[(\Prism^{\init}_{\ul S/ \ul R}, (\xi), \sq M^{\init}_{\ul S/ \ul R}\coloneqq \widetilde{M})\] 
over $(\Ainf(R, \widetilde{M}), (\xi), \widetilde{M})$ equipped with a map $S\to \Prism^{\init}_{\ul S/ \ul R}/\xi$ and an exact surjection 
\[ (\Prism^{\init}_{\ul S/ \ul R}, \sq M^{\init}_{\ul S/ \ul R})^a \lra (\Prism^{\init}_{\ul S/ \ul R}/\xi, N\to \Prism^{\init}_{\ul S/ \ul R}/\xi)^a. \]  
Finally, we take the associated log ring and define 
\[
(\Prism_{\ul S}^{\init}, M_{\ul S}^{\init}) := (\Prism^{\init}_{\ul S/ \ul R}, \sq M^{\init}_{\ul S/ \ul R})^a.
\]
By \cite[Proposition 4.12]{logprism} (also see Remark \ref{remark:initial_object_independent_of_R}), $(\Prism_{\ul S}^{\init}, M_{\ul S}^{\init})$ is independent of the choice of the map $\ul R \ra \ul S$.  
\end{construction}

\br[Initial object in ``relative log prismatic site"] \label{remark:initial_object_in_relative_site}
Suppose that $\ul S = (S, N)$ is a quasiregular semiperfectoid algebra over $\ul{\mO_C} = (\mO_C, N_\infty)$, then there is a natural map 
\begin{equation} \label{eq:initial_object_in_relative_site}
(\Ainf, (\xi), N_\infty = N_\infty^\flat) \lra (\Prism^{\init}_{\ul S}, (\xi), M_{\ul S}^{
\init})
\end{equation}
of pre-log prisms and is functorial in $\ul S$. For this, we may take the following canonical choice of $\ul R$ in Construction \ref{construction:initial_logprism}: let
\[\sq R = \mO_C \widehat \otimes_{\Z_p} \Ainf(S^\flat)\] and $M = N^\flat := \varprojlim_{[p]} N$ with the monoid structure given by the Teichmuller lift of $N^\flat \ra S^\flat$. From the construction, we have the following canonical map 
\[
N_\infty \cong N_\infty^\flat \ra N^\flat \cong M^\flat \ra \sq M \ra M_{\ul S}^{\init}
\]
which gives rise to the desired map in (\ref{eq:initial_object_in_relative_site}).
\er


\subsection{$A \Omega_{\ul S}^{\log} \widehat \otimes^\L A_{\mathrm{cris}}$ for log quasiregular semiperfectoid rings} \noindent 

\noindent In this subsection, we further analyse 
$A \Omega_{\ul S}^{\log} \widehat \otimes^\L \Acris$ for log quasiregular semiperfectoid rings, 
which is a relatively simple object. 

\begin{lemma} \label{lemma:top_free}
For $\ul S \in \qrsP_{\ul{\mO_C}/}$ such that $S$ is $p$-torsion free, $A \Omega^{\log}_{\ul S} \widehat \otimes^\L \Acris$ is a topologically free discrete $\Acris$-module concentrated in degree $0$, in particular $p$-torsion free. 
\end{lemma}
 
\begin{proof} 
It suffices to show that the derived mod $p$ reduction  $\L {\Omega}_{(\ul{S}/p)/(\ul{\mO_C}/p)}$ is a free $\mO_C/p$-module concentrated in degree $0$. For this we consider the graded pieces \[\wedge^i \L_{(\ul S/p)/(\ul{\mO_C}/p) } [-i]\] for the conjugate filtration on $\L {\Omega}_{(\ul{S}/p)/(\ul{\mO_C}/p)}$. Observe that by Lemma \ref{lemma:log_BMS_4.25}, we know that  
\[\L_{(\ul S/p)/(\ul{\mO_C}/p) } [-1]\] is flat over $S/p$, thus free over $\mO_C/p$ by the assumption on $S$. 
\end{proof} 

Similar to Construction \ref{construction:derived_Ainf_on_prelog}, we can derive both sides of the isomorphism (\ref{eq:Acris_comparison}).

\begin{construction}

Fix a split pre-log ring $\ul{\mO_C} = (\mO_C, N_\infty)$ as in Section \ref{section:HT_primitive}. Recall from Construction \ref{construction:derived_Ainf_on_prelog} that we have 
\begin{equation} \label{eq:Sigma_S_T_2}
\Sigma(S, \ul T)  =  (\mO_C \gr{(X_s)_{ s\in S}, \N^T}, N_\infty \oplus \N^T)
\end{equation}
for finite sets $S,T$.  The \emph{derived absolute log crystalline cohomology} $\L R \Gamma_{\crys} (-/\Z_p)$ is the functor on the $\infty$-category of  derived $p$-complete simplicial pre-log rings over $\ul{\mO_C}/p$ (in other words, animated $p$-complete pre-log rings), obtained as the derived $p$-completed left Kan extension of the functor 
\begin{align*}
\Sigma(S, \ul T) \:  \longmapsto \: & \: \:  R \Gamma_{\logcrys} ((\spec \Sigma(S, \ul T)/p)^a /(\Acris, N_\infty)) \\ & \qquad \qquad \qquad   \cong 
R \Gamma_{\logcrys} ((\spec \Sigma(S, \ul T)/p)^a /\Z_p)
\end{align*}
from the $p$-complete pre-log rings $\Sigma (S, \ul T)$ 
to the category of \emph{all} derived $p$-complete simplicial pre-log rings over  $\ul{\mO_C}$. \footnote{From the construction, this functor only depends on $\ul S \otimes^\L_{\Z} \Z/p$.}
\end{construction}

\bc Let $\ul S$ be a derived $p$-completed (simplicial) pre-log ring. There exists functorial isomorphisms 
\begin{equation} \label{eq:AOmega_Acris_QRSP}
\widehat \L \Omega_{\ul S/\Z_p} \isom \L R \Gamma_{\crys} (\ul S/\Z_p) \isom A \Omega^{\log}_{\ul S} \widehat \otimes^\L_{\Ainf} \Acris. 
\end{equation}
In particular, when $\ul S$ is quasiregular semiperfectoid, all the objects in (\ref{eq:AOmega_Acris_QRSP}) are discrete. 
\ec 

\bproof 
The second isomorphism follows from left Kan extending the (inverse of the) isomorphism in (\ref{eq:Acris_comparison}). For the first arrow, note that from \cite[Proposition 7.18]{Bhatt_dR}, we have a natural map 
\[ 
\widehat \L \Omega_{\Sigma (S, \ul T)/\Z_p} \lra 
R \Gamma_{\logcrys} ((\Sigma(S, \ul T)/p)/\Z_p)
\]
which is an isomorphism by \cite[Theorem 7.22]{Bhatt_dR}. 
Now let $\widehat \L\Omega'_{-/\Z_p}$ denote the derived $p$-completed left Kan extension of the functor 
\[
\Sigma(S, \ul T) \:  \longmapsto \: \widehat \L\Omega_{\Sigma(S, \ul T)/\Z_p} 
\]
to all derived $p$-complete simplicial pre-log $\ul{\mO_C}$-algebras. It suffices to show that for all such pre-log rings $\ul S$ over $\ul{\mO_C}$, we have a natural isomorphism 
\begin{equation} \label{eq:left_Kan_extend_twice}
\widehat \L\Omega'_{\ul S/\Z_p} \lra \widehat \L\Omega_{\ul S/\Z_p}.
\end{equation}
Note that the left hand side is the derived $p$-completion of the homotopy colimit 
\[
\varinjlim_{\Sigma \ra \ul{S}} \varinjlim_{\Sigma_0 \ra \Sigma } \Omega^\bullet_{\Sigma_0/\Z_p} \cong \varinjlim_{\Sigma_0 \ra \Sigma \ra \ul{S}} \Omega^\bullet_{\Sigma_0/\Z_p}  
\]
where the colimit is taken over all $\Sigma$ of the form of $\Sigma(S, \ul T)$ in (\ref{eq:Sigma_S_T_2}) and all $\Sigma_0$ of the form 
\[\Sigma_{0}(S, T):= (\Z_p \gr{(X_s)_{s \in S}, \N^{T}}, \N^T)\] where $S, T$ are finite sets. The natural functor from the category of diagrams of the form $\Sigma_0  \ra \Sigma \ra \ul S$ to the category of diagrams of the form $\Sigma_0 \ra S$ sending 
\[(\Sigma_0 \ra \Sigma \ra \ul S) \mapsto (\Sigma_0 \ra \ul S)\] is cofinal, thus (\ref{eq:left_Kan_extend_twice}) is indeed an isomorphism. The final claim can be checked after taking derived reduction mod $p$, and follows from the (derived) Hodge--Tate comparison. 
\eproof

\subsection{A digression on log PD envelops} 
\label{ss:remark_on_Acris}
\noindent

\noindent 
It is possible to relate the objects in (\ref{eq:AOmega_Acris_QRSP}) with a certain universal log PD evenlop of $\ul S$ when $\ul S$ is quasiregular semiperfectoid. 
\bd 
Let $\ul S = (S, M) \in \QRSP $ be a log quasiregular semiperfectoid ring. We define $\Acris (\ul S)$ to be the $p$-completed log PD envelop of 
\[
W(\ul S^\flat)= (W(S^\flat), M^\flat) \lra \ul S/p = (S/p, M).
\]
More precisely, it is given by the $p$-completed (non-log) PD envelop of the surjection 
\begin{equation} \label{eq:exactification_A_tilde}
\sq A:= W(S^\flat) \widehat \otimes_{\Z_p} \Z_p \gr{\sq M} \lra S
\end{equation}
where $\sq M$ is the exactification of the map $M^\flat \ra M$. 
\ed 

\br Equivalent, $\Acris(\ul S)$ can be defined as the $p$-completed log PD envelop of $\ul S/p$ with respect to the PD thickening $\Z_p \ra \F_p$. 
Moreover, $\Acris(\ul S)/p^n$ 
is the log PD envelop of $W_n(\ul S^\flat) \ra \ul S/p$. 
\er 

\br 
We warn the reader that, unlike in the non-log case, in general it is not true that the natural map 
\[ 
A \Omega^{\log}_{\ul S} \widehat  \otimes_{\Ainf}^\L \Acris \cong \L R \Gamma_{\tu{crys}} (\ul S/\Z_p) \lra   \Acris(\ul S). 
\]
is an isomorphism, even though both sides are PD thickenings of $S/p$ (at least when $S$ has no $p$-torsion, see Remark \ref{remark:log_PD_thickening}).  The issue is that $\sq M$ is typically not $p$-divisible. This is related to the problem that the naive Nygaard filtration in \cite[Example 5.2]{logprism} does not yield the correct definition. 
\er 

Nevertheless, we have the following relation between $A \Omega_{\ul S}^{\log}$ and $\Acris(\ul S)$. Keep notations from above, in particular, let $\sq A$ be as in (\ref{eq:exactification_A_tilde}). Recall from
Subsection \ref{ss:universal_log_prism} that $\Prism_{\ul S}^\L \cong \Prism_{\ul S}^{\init}$ is naturally an $\sq A$-algebra. Write 
\[
\sq A_{\tu{cris}} := \sq A \widehat \otimes_{\Ainf} \Acris \cong \Acris (S^\flat) \widehat \otimes_{\Z_p} \Z_p \gr{\sq M}.
\]

\bl \label{lemma:universal_PD_thickening} Let $\ul S = (S, M) \in \QRSP $ be a log quasiregular semiperfectoid ring.
There is a natural Frobenius equivariant isomorphism
\begin{equation} \label{eq:universal_log_PD_compare_to_prism}
\Prism^\L_{\ul S} \widehat \otimes_{\sq A, \varphi_{\sq A}}^\L \sq A_{\tu{cris}}  \isom  \Acris(\ul S).     
\end{equation} 
\el 

\bproof 
From Subsection \ref{ss:universal_log_prism}, $\Prism^{\L}_{\ul S}$ is canonically isomorphic to the prismatic envelop 
\[
W(S^\flat) \widehat \otimes_{\Z_p} \Z_p \gr{\sq M}\{\frac{J}{\xi}\}^{\delta} = \sq A  \{\frac{J}{\xi}\}^{\delta},
\]
where $J := \ker (\sq A \twoheadrightarrow S)$. The right hand side is PD envelop $D_{\sq A} (J)$, which is isomorphic to 
\[
\sq A \{\frac{\varphi_{\sq A}(J)}{p}\} \cong \sq A_{\tu{cris}} \{\frac{\varphi_{\sq A}(J)}{\varphi(\xi)}\} \cong  \sq A  \{\frac{J}{\xi}\}^{\delta} \widehat \otimes_{\sq A, \varphi_{\sq A}} \sq A_{\tu{cris}}
\]
by $p$-completely flat base change of prismatic envelops. 
\eproof 

\br 
The object $\Prism^\L_{\ul S} \widehat \otimes_{\sq A, \varphi_{\sq A}}^\L \sq A_{\tu{cris}}$ on the left hand side of the isomorphism (\ref{eq:universal_log_PD_compare_to_prism}) is isomorphic $A \Omega_{\ul S}^{\log} \widehat \otimes^\L_{\Ainf \gr{\sq M}, \tu{Fr}} \Acris \gr{\sq M}$ where $\tu{Fr}$ is the ``relative Frobenius'' induced by $\sq M \xrightarrow{p} \sq M$. 
It also appears in the definition of the Nygaard filtration on derived log prismatic cohomology in \cite[Definition 5.14]{logprism}. In fact, Lemma \ref{lemma:universal_PD_thickening} \emph{explains} why the somewhat strange definition in \textit{loc.cit.} is in fact natural (once we extend scalars to $\Acris$). However, in the logarithmic context, it is not entirely clear to us whether there is a more canonical and satisfying interpretation of   $\Acris (\ul S)$. In particular, unlike in \cite{BMS2},  it is not clear to us how to relate $\Acris (\ul S)$ to logarithmic topology Hochschild homology. 
\er 

\br 
The various objects discussed above fit into the following commutative diagram. 
\[
\begin{tikzcd}[column sep = 0.5cm, row sep = 0.3cm]
\sq M \arrow[dd] \arrow[rd, "\cdot p"] \arrow[rrr] & &  & \Prism_{\ul{S}}^\L \cong \Prism_{\ul S}^{\init} \arrow[rd] \arrow[dd] \\
& \sq M \arrow[dd]  \arrow[rrr] & & & \Acris (\ul S) \arrow[r, dash, "\sim"] \arrow[dd] &  \Prism^\L_{\ul S} \widehat \otimes_{\sq A, \varphi_{\sq A}}^\L \sq A_{\tu{cris}}   \\
M \arrow[rd, "\cdot p"]  \arrow[rr, "\alpha \qquad "] & & S \arrow[r] \arrow[rrd, swap,  "\tu{Fr}"] &  \Prism_{\ul S}^\L /\xi \arrow[rd, "\beta_{\varphi}"] \\
& M \arrow[rrr, "\alpha \qquad "] & & & S/p
\end{tikzcd}
\]
Here the map $\beta_{\varphi}$ is the composition 
\[\Prism_{\ul S}^\L/\xi \isom \varphi_{A}^* \Prism_{\ul S}^\L/\varphi(\xi) \cong A\Omega_{\ul S}^{\log}/\varphi(\xi) \lra A \Omega_{\ul{S}}^{\log}/(p, \xi) \cong \L \Omega_{(\ul S/p)/\mO_C} \lra S/p. 
\]
\er


\br \label{remark:log_PD_thickening}  
As remarked earlier, for $\ul S = (S, M) \in \qrsP_{\ul{\mO_C}/}$ such that $S$ is $p$-torsion free and $M$ is integral, the natural projection map $\beta$ in the diagram
\[
\begin{tikzcd} 
A \Omega^{\log}_{\ul S}  \widehat \otimes^\L \Acris \arrow[r] \arrow[rd, dashed, "\beta"] &  \widehat \L {\Omega}_{\ul{S}/\ul{\mO_C}} \arrow[d, two heads] \\
& S/p
\end{tikzcd} 
\] 
 is also a PD-thickening. Here the horizontal arrow is induced from the (derived) log de Rham comparison $A \Omega^{\log}_{\ul S}  \otimes^\L \Ainf/\xi \cong \widehat \L {\Omega}_{\ul{S}/\ul{\mO_C}}$. 
The proof is similar to that of \cite[Lemma 2.7]{Yao_Acris}. For completeness, let us sketch the argument.  Since $S$ is $p$-torsion free, it suffices to show that, for each element
$$x \in I: = \ker (\theta: A \Omega^{\log}_{\ul S}  \ra S),$$  $x^p$ is divisible by $p$ in $A \Omega^{\log}_{\ul S} \widehat \otimes^\L \Acris$.  Let $\sq S = \Ainf(S^\flat)/\xi$ and let  
$$ K := \ker (\sq S \gr{M^\flat} \ra S \gr{M}) $$
be the kernel of the natural map.  
Let $\ul{S_1'} = (S_1', M^\flat)$ be the pre-log ring where  
\[S_1' := \sq S \gr{M^\flat,  X_{j}^{1/p^\infty}}_{j \in J} = \sq S \gr{M^\flat} \gr{X_j^{1/p^\infty}}_{j \in J}\] is the perfectoid ring obtained by adjoining all $p$-power roots of $X_j$ for each $j \in J$ to $\sq S \gr{M^\flat}$
 (in the $p$-adic complete sense).   Let $\ul{S'} = (S', M^\flat)$ be the quotient pre-log ring of $\ul{\sq S_1'}$, where
\[S' :=  \sq S \gr{M^\flat,  X_{j}^{1/p^\infty}}_{j \in J}/(X_j).\]
Note that $\ul{S'}$ is a quasiregular semiperfectoid log $\ul{\mO_C}$-algebra. Now we choose a homomorphism  $\sq S' \lra \sq S \gr{M^\flat} $ which sends $X_j \mapsto j$ for each $ j\in J$. 
This determines a map 
$\sq \delta: \ul{\sq S'} \ra \ul{\sq S}$, which then determines a map 
\begin{equation} \label{map:delta}
 \delta:  \ul{S'} = \big( \sq S \gr{M^\flat,  X_{j}^{1/p^\infty}}_{j \in J}/(X_j),  M^\flat \big)  \lra \ul{S} = (S, M). 
\end{equation} 
The map  $\delta$ induces a surjective map 
$$  A \Omega^{\log}_{\ul S'}    \widehat \otimes^\L \Acris \longrightarrow  A \Omega^{\log}_{\ul S}  \widehat \otimes^\L \Acris$$ of rings. Moreover, the induced map
$$\ker \big(A \Omega^{\log}_{\ul{S'}} \widehat \otimes^\L \Acris \ra S'/p \big) \longrightarrow \ker \big(A \Omega^{\log}_{\ul{S}} \widehat \otimes^\L \Acris \ra S/p \big)$$
is surjective. This allows us to reduce to the case where  $\ul S$ is replaced with 
\[\ul{S'} = ( \sq S \gr{\sq M, X_{j}^{1/p^\infty}}/(X_j), \sq M).\]
Consider the quasiregular semiperfectoid algebra $S_0 =  \mO_C \gr{X_{j}^{1/p^\infty}}_{j \in J}/(X_j)$ (viewed as pre-log ring with the trivial pre-log structure).  Observe that $\ul S' = S_0 \widehat \otimes_{\mO_C} (\sq S \gr{M^\flat}, M^\flat)$ and that 
the natural map 
$$ A \Omega^{\log}_{S_0} \widehat \otimes^\L_{\Ainf} \Ainf(\sq S \gr{ M^\flat}) \lra  A\Omega^{\log}_{\ul{S'}}$$
is an isomorphism by considering reduction mod $\xi$ and using base change for derived log de Rham cohomology. 
Hence we are further reduced to prove the proposition for $S_0$\footnote{
Note that, again, here we use the fact that $S_0$ is $p$-torsion free. 
}. 
For this, we observe that the natural map 
$$ A \Omega_{S_0} \lra A \Omega^{\log}_{S_0} $$ from the derived $\Ainf$-cohomology to the derived log $\Ainf$-cohomology  is an isomorphism. 
But $A \Omega_{S_0} \widehat \otimes_{\Ainf} \Acris \ra S_0/p$ is a PD thickening by  \cite[Lemma 2.7 \& 2.8]{Yao_Acris}. 
\er

\newpage 
\section{Breuil--Kisin--Fargues modules} \label{ss:BKF}
In this final section of the paper, we show that when $\fX$ is a proper quasi-fs log $p$-adic formal scheme that is admissibly smooth over $\ul \mO$, its log $\Ainf$-cohomology takes values in the category of Breuil--Kisin--Fargues modules. From this we deduce $p$-torsion discrepancies among the Kummer \'etale cohomology, log de Rham and log crystalline cohomology, generalizing results from \cite{BMS1, CK_semistable}. 

\subsection{Breuil--Kisin--Fargues modules} 
First let us recall the definition. 

\bd \label{definiton:BKF_module}
A \emph{Breuil--Kisin--Fargues module} is a finitely presented $\Ainf$-module $M$ equipped with a $\varphi$-linear map $\varphi_M: M \ra M$, such that $M[\frac{1}{p}]$ is a finite free $\Ainf[1/p]$-module and that $\varphi_M$ induces an isomorphism 
$$\varphi_M: M [\frac{1}{\xi}] \isom M[\frac{1}{\varphi(\xi)}].$$
\ed 

Let us first record a criterion to test finite-freeness over $\Ainf[1/p].$

\bp  Let $M$ be a finitely presented $\Ainf$-module. Assume either of the following
\be \item $M [1/{p \mu}]$ is finite projective over $\Ainf [1/{p \mu}]$ and $M \otimes_{\Ainf} \Bcrisp$ is finite projective over $\Bcrisp$; or 
\item $M [1/{p \mu}]$ is finite free over $\Ainf [1/{p \mu}]$ of the same rank as $\textup{rk}_{W(k)} M \otimes_{\Ainf} W(k)$, and there exists a $\varphi$-linear map $\varphi_M: M \ra M$ that becomes an isomorphism after inverting $\xi$.  \ee Then $M[1/p]$ is finite projective (equivalently, finite free) over $\Binf = \Ainf[1/p]$.  \ep 

\bproof This is \cite[Lemma 4.20]{BMS1} and Lemma A.4 of \cite{Morrow}. 
\eproof 

To apply (either part of) the proposition above, we shall need the following result.  

\bp \label{prop:free_over_B_cris}
Let $\fX$ be a proper quasi-fs log $p$-adic formal scheme that is admissibly smooth over $\ul \mO = (\mO, N_\infty)$. Then  for each $i$, the cohomology group $H^i (R \Gamma_{\Ainf} (\fX) \otimes_{\Ainf}^\L \Bcrisp )$ is a finite free module over $\Bcrisp$. Moreover, we have  
\[\textup{rk}_{W(k)}  H^i_{\tu{logcrys}} (\fX_{k}/W(k))  = \textup{rk}_{\Z_p} H^i_{\ket} (\fX_{\eta}, \Z_p).\]
\ep

\bproof 
Let $\fX_0 \ra \spa(\mO, N)^a$ be a finite model of $\fX \ra \spa(\mO, N_\infty)^a$ as in the beginning of Subsection \ref{ss:small_charts}. Then by the absolute crystalline comparison (Theorem \ref{eq:Acris_comparison}) and base change for log crystalline cohomology, we have 
\begin{align*}   
R \Gamma_{\Ainf} (\fX/\ul{\mO_C})   \otimes^\L_{\Ainf} \Acris 
&\cong R \Gamma_{\textup{logcrys}} (\fX_{\mO_C/p}/(\Acris, N_\infty)) \\
& \cong R \Gamma_{\textup{logcrys}} ({\fX_0}_{,\mO_C/p}/(\Acris, N)).
\end{align*}
Let $\fX_{0, k}$ denote the base change of $\fX_0$ along $(\mO, N) \ra (k, N)$, where $ k$ is the residue field of $\mO$ and is equipped with the pre-log structure induced from the composition. After replacing $N \ra k$ by $\cl N \ra k$ we may assume that $x \mapsto 0 \in k$ for all nonzero $x \in N$. Let $(W(k), N)$ denote the pre-log ring where the pre-log structure is obtained from $(k, N)$ by taking Techmuller representatives. We fix a section $k \ra \mO_C/p$. 
By \cite[Proposition 8.9]{logprism} (and its proof), we know that there exists some large enough $n$ such that the section $k \ra \mO_C/p^{1/p^n}$ can be upgraded to a section $(N, k) \ra (N, \mO_C/p^{1/p^n})$ and that 
\[
\fX_{0, \mO_C/p^{1/p^n}} \cong \fX_{0, k} \times_{\spec (k, N)^a} \spec (\mO_C/p^{1/p^n}, N)^a.
\]
Therefore, after applying iterated powers of Frobenius, we have an isomorphism  
\begin{multline*}
\quad   R \Gamma_{\textup{logcrys}} (\fX_{0, k}/(W\gr{N}^{\small \tu{PD}}, N)) \otimes^\L_{W\gr{N}^{\tu{PD}}} \Acris [\frac{1}{p}] \\ 
\cong  R \Gamma_{\textup{logcrys}} ({\fX_0}_{,\mO_C/p^{1/p^n}}/(\Acris, N)) [\frac{1}{p}]  
 \cong   R \Gamma_{\textup{logcrys}} ({\fX_0}_{,\mO_C/p}/(\Acris, N)) [\frac{1}{p}]  \quad 
\end{multline*}
induced from the base change of pre-log PD algebras 
\[
(W\gr{N}^{\small \tu{PD}}, N) \lra (\Acris, N). 
\]
Here $W\gr{N}^{\small \tu{PD}}$ denotes the $p$-completed PD evenlop of $W(k)[N]$ with respect to the kernel of the natural map  $W(k)[N] \ra W(k)$. 
The first claim then follows from this and the Hyodo--Kato isomorphism (\cite[Proposition 4.13]{HK})
\[
R \Gamma_{\textup{logcrys}} (\fX_{0, k}/ (W, N))  \otimes_{W} W\gr{N}^{\tu{PD}} [\frac{1}{p}] \cong R \Gamma_{\textup{logcrys}} (\fX_{0, k}/(W\gr{N}^{\small \tu{PD}}, N)) [\frac{1}{p}].
\]
The second claim follows from the first one,  Theorem \ref{thm:etale_comparison}, and Corollary \ref{thm:Acris_for_locally_small}. 
\eproof

\bc \label{cor:valuation_in_BKF_modules}  
Let $\fX$ be a proper quasi-fs log $p$-adic formal scheme that is admissibly smooth over $\ul \mO = (\mO, N_\infty)$. Then for each $i$, the log $\Ainf$-cohomology $H^i_{\Ainf} (\fX)$ with the Frobenius $\varphi$ is a Breuil--Kisin--Fargues modules. 
\ec 

\bproof 
This follows from \cite[Lemma 4.20]{BMS1}, using the \'etale comparison (Theorem \ref{thm:etale_comparison}) and Proposition \ref{prop:free_over_B_cris}.
\eproof

\subsection{Torsion discrepancy} \label{ss:torsion}

One of the original motivations of \cite{BMS1} is to study relations between torsions in integral $p$-adic cohomologies. In this direction, let us record the following corollary on relations among the ``sizes'' of $p$-adic cohomology groups with torsion coefficients, generalizing \cite{BMS1, CK_semistable} (see also \cite{logprism}).

\bc \label{theorem:torsion} 
Let $\fX$ be a proper quasi-fs log $p$-adic formal scheme that is admissibly smooth over $\ul \mO = (\mO, N_\infty)$. Let $W(\ul k)$ denote the pre-log ring $(W(k), N_\infty)$ where $k$ is the residue field of $\mO$ and the pre-log structure $N_\infty \ra W(k)$ is obtained from taking Teichmuller representatives. Let $\fX_{k}$ (resp. $\fX_{\mO/p^n}$) denote the base change of $\fX$ to the log point $\spec (k, N_\infty)^a$ (resp. to the log scheme $\spec (\mO/p^n, N_\infty)^a$). Moreover, write $H^i_{\textup{logdR}}(\fX)$ (resp. $H^i_{\textup{logdR}}(\fX_{\mO/p^n})$) for the log de Rham cohomology 
\[H^i_{\textup{logdR}}(\fX/(\mO, N_\infty)) \quad \big(\tu{resp. } H^i_{\textup{logdR}}(\fX_{\mO/p^n}/(\mO/p^n, N_\infty)) \big)\]. Then we have  
\be
\item For each $i$ and $n \ge 0$, we have 
\begin{align*} 
\textup{length}_{\Z_p} (H^i_{{\tu{k\'et}}}(\fX_{\eta}^{\tu{ad}}, \Z_p)_{\textup{tor}}/p^n) & \le \textup{length}_{W(k)} (H^i_{\textup{logcrys}}(\fX_k/W(\ul{k}))_{\textup{tor}}/p^n) \\
\textup{length}_{\Z_p} (H^i_{{\tu{k\'et}}}(\fX_{\eta}^{\tu{ad}}, \Z/p^n)) & \le \textup{length}_{W(k)} (H^i_{\textup{logcrys}}(\fX_k/W_n(\ul{k}))) 
\end{align*}
\item  For each $i$ and $n \ge 0$, we have 
\begin{align*} 
\textup{length}_{\Z_p} (H^i_{{\tu{k\'et}}}(\fX_{\eta}^{\tu{ad}}, \Z_p)_{\textup{tor}}/p^n) & \le \textup{val}_{\mO} (H^i_{\textup{logdR}}(\fX)_{\textup{tor}}/p^n) \quad \\
\textup{length}_{\Z_p} (H^i_{{\tu{k\'et}}}(\fX_{\eta}^{\tu{ad}}, \Z/p^n)) & \le  \textup{val}_{\mO} (H^i_{\textup{logdR}}(\fX_{\mO/p^n}))
\end{align*}
where $\textup{val}_{\mO}$ denotes the normalized length defined in \cite[Subsection 7.10]{CK_semistable}. 
\item Moreover, for each $i$, the cohomology group $H^i_{\textup{logcrys}}(\fX_k/W(\ul{k}))$ is $p$-torsion free if and only if $H^i_{\textup{logdR}}(\fX)$ is $p$-torsion free, and (either) implies that $H^i_{{\tu{k\'et}}}(\fX_{\eta}, \Z_p)$ is $p$-torsion free and that $H^i_{\Ainf}(\fX)$ is a free $\Ainf$-module. 
\ee
\ec 

\bproof This follows from the same proof of \cite[Theorem 7.9, Theorem 7.12 and Theorem 7.7]{CK_semistable}, making use of Corollary \ref{cor:valuation_in_BKF_modules}, as well as the \'etale, de Rham  and   crystalline comparison of log $\Ainf$ cohomology. 
\eproof

\newpage

\addcontentsline{toc}{section}{References}
\bibliographystyle{plain}

\bibliography{ref}

\end{document}